\documentclass[11pt,oneside]{amsart}
\usepackage[T1]{fontenc}
\usepackage{textcomp}
\usepackage[utf8]{inputenc}
\usepackage[a4paper]{geometry}
\usepackage{xcolor}
\usepackage{booktabs}
\usepackage{mathtools}
\usepackage{amstext}
\usepackage{amsthm}
\usepackage{amssymb}
\PassOptionsToPackage{normalem}{ulem}
\usepackage{ulem}
\usepackage[pdfusetitle,
 bookmarks=true,bookmarksnumbered=true,bookmarksopen=false,
 breaklinks=false,pdfborder={0 0 0},pdfborderstyle={},backref=false,colorlinks=true]
 {hyperref}
\hypersetup{
 linkcolor=blue,  citecolor=blue, urlcolor=blue}

\makeatletter

\providecommand{\tabularnewline}{\\}
\providecolor{lyxadded}{rgb}{0,0,1}
\providecolor{lyxdeleted}{rgb}{1,0,0}
\DeclareRobustCommand{\mklyxadded}[1]{\textcolor{lyxadded}\bgroup#1\egroup}
\DeclareRobustCommand{\mklyxdeleted}[1]{\textcolor{lyxdeleted}\bgroup\mklyxsout{#1}\egroup}
\DeclareRobustCommand{\mklyxsout}[1]{\ifx\\#1\else\sout{#1}\fi}

\numberwithin{equation}{section}
\numberwithin{figure}{section}
\numberwithin{table}{section}
\theoremstyle{plain}
\newtheorem{thm}{\protect\theoremname}[section]
\theoremstyle{plain}
\newtheorem{conjecture}[thm]{\protect\conjecturename}
\theoremstyle{remark}
\newtheorem{notation}[thm]{\protect\notationname}
\theoremstyle{definition}
\newtheorem{defn}[thm]{\protect\definitionname}
\theoremstyle{plain}
\newtheorem{cor}[thm]{\protect\corollaryname}
\theoremstyle{definition}
\newtheorem{example}[thm]{\protect\examplename}
\theoremstyle{plain}
\newtheorem{prop}[thm]{\protect\propositionname}
\theoremstyle{remark}
\newtheorem{rem}[thm]{\protect\remarkname}
\theoremstyle{plain}
\newtheorem{lem}[thm]{\protect\lemmaname}
\theoremstyle{remark}
\newtheorem*{rem*}{\protect\remarkname}
\newenvironment{lyxlist}[1]
	{\begin{list}{}
		{\settowidth{\labelwidth}{#1}
		 \setlength{\leftmargin}{\labelwidth}
		 \addtolength{\leftmargin}{\labelsep}
		 }}
	{\end{list}}
\theoremstyle{remark}
\newtheorem{note}[thm]{\protect\notename}

\usepackage{graphicx}
\usepackage{setspace}
\usepackage{url}
\usepackage[perpage,symbol]{footmisc}
\usepackage{amssymb} 
\usepackage{amsmath} 
\usepackage{multicol}
\usepackage{etoolbox}
\usepackage{bbm}
\usepackage{enumitem}
\usepackage[font=footnotesize,labelfont=bf, width=.9\textwidth]{caption}
\usepackage{etoolbox}
\usepackage{fancyhdr}
\usepackage{mathdots}
\usepackage[abbrev,alphabetic,nobysame,bibtex-style]{amsrefs}
\usepackage{tikz}

\usepackage{mleftright}
\mleftright

\usepackage{enumitem}
\setlist[itemize]{label={$\vcenter{\hbox{\tiny$\bullet$}}$}}
\setlist[enumerate,1]{label={\textnormal{(\arabic*)}}}
\setlist[enumerate,2]{label={\textnormal{(\alph*)}}}

\usepackage[protrusion=true,expansion=true]{microtype}
\date{}

\DefineFNsymbols*{lamportnostar}[math]{{(\dagger)}{(\ddagger)}{(\S)}{(\P)}{(\|)}{(\dagger\dagger)}{(\ddagger\ddagger)}}
\setfnsymbol{lamportnostar}

\DeclareMathOperator{\vol}{vol}

\theoremstyle{plain}

\theoremstyle{definition}
\theoremstyle{definition}

\allowdisplaybreaks[2]

\newcommand{\R}{\operatorname{Re}}

\AtBeginDocument{%
   \def\MR#1{}
}

\makeatother

\providecommand{\conjecturename}{Conjecture}
\providecommand{\corollaryname}{Corollary}
\providecommand{\definitionname}{Definition}
\providecommand{\examplename}{Example}
\providecommand{\lemmaname}{Lemma}
\providecommand{\notationname}{Notation}
\providecommand{\notename}{Note}
\providecommand{\propositionname}{Proposition}
\providecommand{\remarkname}{Remark}
\providecommand{\theoremname}{Theorem}

\begin{document}
\title{The Cohomological Sarnak--Xue Density Hypothesis for~$SO_{5}$}
\author{Shai Evra}
\address{Einstein Institute of Mathematics, The Hebrew University of Jerusalem, Jerusalem, 9190401, Israel}
\email{shai.evra@mail.huji.ac.il}
\author{Mathilde Gerbelli-Gauthier}
\address{Department of Mathematics and Statistics, McGill University, Montreal, Quebec H3A 0B, Canada}
\email{mathilde.gerbelli-gauthier@mail.mcgill.ca}
\author{Henrik P. A. Gustafsson}
\address{Department of Mathematics and Mathematical Statistics, Umeå University, SE-901 87 Umeå, Sweden}
\email{henrik.gustafsson@umu.se}
\begin{abstract}
We prove the cohomological version of the Sarnak--Xue Density Hypothesis
for $SO_{5}$ over a totally real field and for inner forms split
at all finite places. The proof relies on recent lines of work in
the Langlands program: (i) Arthur's Endoscopic Classification of Representations
of classical groups, extended to inner forms by Taïbi and its explicit
description for $SO_{5}$ by Schmidt, and (ii) the Generalized Ramanujan--Petersson
Theorem, proved for cohomological self-dual cuspidal representations
of general linear groups. We give applications to the growth of cohomology
of arithmetic manifolds, density-Ramanujan complexes, cutoff phenomena
and optimal strong approximation. 
\end{abstract}

\maketitle
\tableofcontents{}

\section{Introduction}

\subsection{The Sarnak-Xue Density Hypothesis}

Let $G\leq GL_{n}$ be a connected semisimple linear algebraic group
defined over $\mathbb{Q}$, such that $G_{\infty}:=G\left(\mathbb{R}\right)$
is non-compact. For any $q\in\mathbb{N}$, let $\Gamma\left(q\right):=G\left(\mathbb{Q}\right)\cap\ker\bigl(GL_{n}\left(\mathbb{Z}\right)\rightarrow GL_{n}\left(\mathbb{Z}/q\mathbb{Z}\right)\bigr)$
be the level $q$ congruence subgroup. The space $L^{2}\left(\Gamma\left(q\right)\backslash G_{\infty}\right)$
is equipped with the right regular $G_{\infty}$-action. Denote by
$\Pi^{\mathrm{unit}}\left(G_{\infty}\right)$ the set of irreducible
unitary representations of $G_{\infty}$, and for any $\pi\in\Pi^{\mathrm{unit}}\left(G_{\infty}\right)$,
let $r(\pi)\in[2,\infty]$ be its rate of decay of matrix coefficients
(defined in Section~\ref{sec:Local-rep-par}). Denote the multiplicty
with which $\pi$ occurs in the decomposition of $L^{2}\left(\Gamma\left(q\right)\backslash G_{\infty}\right)$
by 
\[
m\left(\pi;q\right):=\dim\mbox{Hom}_{G_{\infty}}\left(\pi,L^{2}\left(\Gamma\left(q\right)\backslash G_{\infty}\right)\right).
\]
The following is the classical Sarnak--Xue Density Hypothesis (SXDH)
for $G$.
\begin{conjecture}[SXDH \cite{SX91}]
\label{conj:SXDH-=00005Cinfty} For any $\pi\in\Pi^{\mathrm{unit}}\left(G_{\infty}\right)$,
\begin{equation}
m\left(\pi;q\right)\ll\vol\left(\Gamma\left(q\right)\backslash G_{\infty}\right)^{\frac{2}{r(\pi)}}.\label{eq:SXDH}
\end{equation}
\end{conjecture}

\begin{notation}
\label{not:intro-asymptotic} Throughout this work we shall use the
following asymptotic notations: Given $f_{1},f_{2}\,:\,\mathbb{N}\rightarrow\mathbb{R}_{>0}$,
we denote $f_{1}\ll f_{2}$, if for any $\epsilon>0$, there exists
$c\in\mathbb{R}_{>0}$, such that $f_{1}(q)\leq cq^{\epsilon}\cdot f_{2}(q)$
for any $q\in\mathbb{N}$. Denote $f_{1}\asymp f_{2}$ if $f_{1}\ll f_{2}$
and $f_{2}\ll f_{1}$. 
\end{notation}

Sarnak and Xue \cite{SX91} proved the SXDH for $G$ such that $G_{\infty}=SL_{2}\left(\mathbb{R}\right)$
or $SL_{2}\left(\mathbb{C}\right)$ and $\Gamma\left(q\right)\leq G_{\infty}$
cocompact (see Huntley and Katznelson \cite{HK93} for the case where
$\Gamma\left(q\right)\leq G_{\infty}$ is non-compact). Frączyk, Harcos,
Maga and Milićević \cite{FHMM20} extended this result of Sarnak--Xue
and proved the SXDH for all $G$ such that $G_{\infty}=SL_{2}\left(\mathbb{R}\right)^{a}\times SL_{2}\left(\mathbb{C}\right)^{b}$.

In \cite{Blo23,Man22,Ass23,AB22,JK22}, a slight variant of the SXDH
was considered, which we call the Satake variant of the SXDH, in which
the exponent $\frac{2}{r(\pi)}$ is replaced by $1-\frac{\sigma_{\pi}}{\sigma_{\mathbf{1}}}$,
where $\sigma_{\pi}=\max_{i}\left|\mathrm{Re}\mu_{\pi}(i)\right|$,
$\mu_{\pi}=\left(\mu_{\pi}(i)\right)\in\mathbb{C}^{n}$ is the Satake
parameters of $\pi$, and $\mathbf{1}$ is the trivial representation
(so $\mu_{\mathbf{1}}=\left(\frac{n-1}{2},\ldots,\frac{1-n}{2}\right)$
and $\sigma_{\mathbf{1}}=\frac{n-1}{2}$). In \cite{Blo23}, Blomer
proved the Satake variant of  the SXDH for $G=SL_{n}$, $n\geq3$,
for the parahoric-level congruence subgroups $\Gamma_{0}\left(q\right)=\left\{ g\in SL_{n}\left(\mathbb{Z}\right)\,:\,g_{n*}\equiv(0,\ldots0,1)\mod q\right\} $,
and in \cite{Man22,Ass23}, Man and Assing proved the Satake variant
of the SXDH for $G=Sp_{4}$ and Iwahori-levels. In \cite{GK23}, Golubev
and Kamber emphasize the importance of proving the SXDH for principal
levels and showed that it implies the optimal lifting property (see
\cite{Sar15}). Shortly after in \cite{AB22,JK22}, Assing--Blomer
and Jana--Kamber proved the Satake variant of the SXDH for $G=SL_{n}$,
$n\geq3$, for principal levels, still assuming $q$ square free. 

We note that this line of work of Blomer \emph{et al} relies on a
Kuznetsov-type relative trace formula, and as such is more analytical
than our work, which follows the work of Marshall and others \cite{CM13,Mar14,Mar16,MS19,GG21,DGG22}
relying on results coming from the Langlands program. In particular,
we invoke Arthur's endoscopic classification of automorphic representations
of classical groups (see Theorem~\ref{thm:Arthur-Taibi}) and the
generalized Ramanujan--Petersson conjecture (see Theorem~\ref{thm:GRPC}).
Consequently, our results consider \emph{cohomological} representations,
which we now describe.

Let $\tilde{X}\cong G_{\infty}/K_{\infty}$ be the symmetric space
associated to the Lie group $G_{\infty}$, where $K_{\infty}\leq G_{\infty}$
is a maximal compact subgroup. Let $K\left(q\right):=K_{\infty}K_{f}\left(q\right)$
where $K_{f}\left(q\right):=\prod_{p}K_{p}\left(q\right)$ and $K_{p}\left(q\right):=G\left(\mathbb{Q}_{p}\right)\cap\ker\bigl(GL_{n}\left(\mathbb{Z}_{p}\right)\rightarrow GL_{n}\left(\mathbb{Z}_{p}/q\mathbb{Z}_{p}\right)\bigr)$.
Note that $\Gamma\left(q\right)=G\left(\mathbb{Q}\right)\cap K_{f}\left(q\right)$
and that $K_{p}\left(q\right)=K_{p}=G\left(\mathbb{Q}_{p}\right)\cap GL_{n}\left(\mathbb{Z}_{p}\right)$
for any $p\nmid q$. For simplicity of exposition, assume in this
introduction that $G$ satisfies the strong approximation property
(see \cite{Rap14}). Then for any $q\in\mathbb{N}$, define the level
$q$ congruence arithmetic manifold 
\[
X\left(q\right):=\Gamma\left(q\right)\backslash\tilde{X}\cong\Gamma\left(q\right)\backslash G_{\infty}/K_{\infty}\cong G\left(\mathbb{Q}\right)\backslash G\left(\mathbb{A}\right)/K\left(q\right).
\]
Let $\Pi^{\mathrm{alg}}\left(G\right)$ be the set of finite-dimensional
algebraic representations of $G$. Any $E\in\Pi^{\mathrm{alg}}\left(G\right)$
defines a local system on $X\left(q\right)$ (see e.g. \cite{BW00}).
Denote by $H_{(2)}^{*}\left(X\left(q\right);E\right)$ the $L^{2}$-cohomology
space of $X\left(q\right)$ with coefficients in $E$. By Matsushima's
formula \cite{Mat69,Bor80} we get the following decomposition
\begin{equation}
H_{(2)}^{*}\left(X\left(q\right);E\right)\cong\bigoplus_{\pi=\pi_{\infty}\otimes\pi_{f}}m\left(\pi\right)\cdot\left(H^{*}\left(\pi_{\infty};E\right)\otimes\pi_{f}^{K_{f}\left(q\right)}\right).\label{eq:Matsushima}
\end{equation}
In the above sum, $\pi$ runs over irreducible admissible representations
of $G\left(\mathbb{A}\right)$, the factors $\pi_{\infty}$ and $\pi_{f}:=\otimes_{p}\pi_{p}$
are respectively the infinite and finite parts of $\pi$ with $\pi_{f}^{K_{f}\left(q\right)}$
the subspace of $K_{f}\left(q\right)$-fixed vectors of $\pi_{f}$.
For each $\pi_{\infty}$, we have $H^{*}\left(\pi_{\infty};E\right):=H^{*}\left(\mathfrak{g},K_{\infty};\pi_{\infty}\otimes E\right)$
the $\left(\mathfrak{g},K_{\infty}\right)$-cohomology of $\pi_{\infty}$
with coefficients in $E$, where $\mathfrak{g}$ is the Lie algebra
of $G_{\infty}$. Finally, we denote the multiplicity of $\pi$ in
the discrete automorphic spectrum of $G$ by
\[
m\left(\pi\right):=\dim\mbox{Hom}_{G(\mathbb{A})}\left(\pi,L^{2}\left(G\left(\mathbb{Q}\right)\backslash G\left(\mathbb{A}\right)\right)\right).
\]
For a given $E$, there are only finitely many irreducible representations
$\sigma$ of $G_{\infty}$ for which $H^{*}\left(\sigma;E\right)\ne0$;
these are called \emph{cohomological representations} with coefficients
in $E$. Denote the set of cohomological representations of $G_{\infty}$
with coefficients in $E$ by $\Pi^{\mathrm{coh}}\left(G_{\infty};E\right)$.
This set was classified by Vogan--Zuckerman \cite{VZ84}. The following
is a special case of the SXDH, which we call the Cohomological Sarnak--Xue
Density Hypothesis (CSXDH). 
\begin{conjecture}[CSXDH]
\label{conj:CSXDH-=00005Cinfty} For any $E\in\Pi^{\mathrm{alg}}\left(G\right)$
and any $\pi\in\Pi^{\mathrm{coh}}\left(G_{\infty};E\right)$, 
\[
m\left(\pi;q\right)\ll\vol\left(\Gamma\left(q\right)\backslash G_{\infty}\right)^{\frac{2}{r(\pi)}}.
\]
\end{conjecture}

See the works of Marshall and his collaborators \cite{CM13,Mar14,Mar16,MS19},
who proved this conjecture for unitary groups of the form $U\left(1,n\right)$,
$n\geq2$, and proved positive results towards this conjecture in
more general settings. 

\subsection{Conditionality\protect\label{subsec:conditional}}

All of our results in this paper rely on Arthur’s endoscopic classification
for split classical groups \cite{Art13}, which was itself originally
conditional on the stabilization of the twisted trace formula, the
twisted weighted fundamental lemma, and on some unpublished references
therein. The stabilization was established in \cite{MW16-1,MW16-2}
and the proof of the unpublished references was announced in a recent
preprint \cite{AGIKMS}. As such, our results are now conditional
only on the twisted weighted fundamental lemma. In what follows, we
will refer to this assumption as \textbf{(EC)}.

\subsection{Results}

We now specialize to the case of split $5\times5$ special odd orthogonal
group 
\[
SO_{5}=SO\left(J_{5}\right)=\left\{ g\in SL_{5}\,:\,gJ_{5}g^{t}=J_{5}\right\} ,\qquad J_{5}=\left(\begin{array}{ccc}
 &  & 1\\
 & \iddots\\
1
\end{array}\right).
\]

We shall prove several CSXDH results for inner forms of the split
group $SO_{5}$ defined over a number field $F/\mathbb{Q}$. In all
of these results we shall require the level $q$ to run exclusively
over integers which are only divisible by primes $>10\left[F:\mathbb{Q}\right]+1$.
For example, for $F=\mathbb{Q}$ we require that $q$ is only divisible
by primes $\geq13$. we will refer to this assumption as \textbf{(DP)}.
This assumption is related to the best known results toward depth
preservation in the local Langlands correspondence (see \cite{Oi22}).

The following is a special case of our main result, which proves the
CSXDH for $G=SO_{5}/\mathbb{Q}$. 
\begin{thm}[CSXDH for $SO_{5}$]
\label{thm:CSXDH-SO5-=00005Cinfty} Let $G=SO_{5}/\mathbb{Q}$. Assume
\textbf{(EC)} and \textbf{(DP)} holds. Then Conjecture~\ref{conj:CSXDH-=00005Cinfty}
holds.
\end{thm}

As a consequence we also get upper bounds on the growth of $L^{2}$-Betti
numbers of $X\left(q\right)$. For simplicity of exposition, we consider
the trivial coefficient system $E=\mathbb{C}\in\Pi^{\mathrm{alg}}\left(G\right)$.
Note that $\dim\tilde{X}=6$. For $0\leq k\leq6$, denote the $k$-th
$L^{2}$-Betti number of $X\left(q\right)$ by 
\[
h^{k}\left(G;q\right):=\dim H_{(2)}^{k}\left(X\left(q\right);\mathbb{C}\right),\qquad H_{(2)}^{*}\left(X\left(q\right);\mathbb{C}\right)=\bigoplus_{k=0}^{6}H_{(2)}^{k}\left(X\left(q\right);\mathbb{C}\right).
\]
Poincaré duality states that $h^{k}\left(G;q\right)=h^{6-k}\left(G;q\right)$
for any $0\leq k\leq6$, hence it suffices to describe $h^{k}\left(G;q\right)$
for $k=0,1,2,3$. By the connectivity of the manifolds $X\left(q\right)$,
the classification of cohomological representations of Vogan--Zuckerman
\cite{VZ84}, and the limit multiplicity results of DeGeorge--Wallach
and Savin \cite{DW78,Sav89}, it follows that
\begin{equation}
h^{0}\left(G;q\right)=1,\qquad h^{1}\left(G;q\right)=0\qquad\mbox{and}\qquad h^{3}\left(G;q\right)\asymp\vol\left(X\left(q\right)\right).\label{eq:Betti-local}
\end{equation}
It is therefore natural to ask about the growth rate of $h^{2}\left(G;q\right)$.
Unlike the above estimates of $h^{k}\left(G;q\right)$ for $k=0,1,3$,
which are local results, meaning that the congruence subgroup $\Gamma\left(q\right)$
can be replaced by any lattice, the following estimate on $h^{2}\left(G;q\right)$
is global and relies on the arithmetic nature of $\Gamma\left(q\right)$. 
\begin{thm}
\label{thm:CSXDH-SO5-Betti} Let $G=SO_{5}$. Under the assumptions
of Theorem \ref{thm:CSXDH-SO5-=00005Cinfty}, we have 
\[
h^{2}\left(G;q\right)\ll\vol\left(X\left(q\right)\right)^{\frac{1}{2}}.
\]
\end{thm}

Theorems~\ref{thm:CSXDH-SO5-=00005Cinfty} and \ref{thm:CSXDH-SO5-Betti}
use the endoscopic classification of the discrete automorphic spectrum
of the split classical group $SO_{5}$, which was proved by Arthur
\cite{Art13} for all quasi-split special orthogonal and symplectic
groups over number fields. In the last chapter of \cite{Art13}, Arthur
discusses the extension of his work to inner forms of quasi-split
classical groups. In \cite{Tai18}, Taïbi extends Arthur's endoscopic
automorphic classification to certain inner forms of special orthogonal
and symplectic groups (see also \cite{KMSW14} for the case of unitary
groups). Taïbi's work \cite{Tai18} applies to inner forms which are
quasi-split at all finite places and at the infinite place the group
has a compact maximal torus, which includes the following family of
inner forms which were studied by Gross \cite{Gro96}.
\begin{defn}
\label{def:Gross-form} Let $G^{*}$ be a split classical group defined
over a number field $F$. An inner form $G$ of $G^{*}$ is called
a \emph{Gross inner form} if $G\left(F_{v}\right)\cong G^{*}\left(F_{v}\right)$
for any finite place $v$ of $F$. Denote $G_{\infty}=\prod_{v\mid\infty}G\left(F_{v}\right)$.
A Gross inner form $G$ is called \emph{uniform} if $F\ne\mathbb{Q}$
is a totally real number field and $G_{\infty}\cong G'\left(\mathbb{R}\right)\times\mbox{compact}$,
where $G'\left(\mathbb{R}\right)$ is a non-compact inner form of
$G^{*}\left(\mathbb{R}\right)$. A Gross inner form $G$ is called
\emph{definite} if $G_{\infty}$ is compact (which implies that $F$
is a totally real number field). 
\end{defn}

Let $G$ be a uniform Gross inner form of a split classical group.
Then we can again define the congruence subgroups $\Gamma\left(q\right)\leq G_{\infty}$,
$q\in\mathbb{N}$, and the congruence manifolds $X\left(q\right)=\Gamma\left(q\right)\backslash G_{\infty}/K_{\infty}$,
exactly as before. Since $G$ is uniform, the congruence manifolds
$X\left(q\right)$ are compact, hence the $L^{2}$-cohomology $H_{(2)}^{k}\left(X\left(q\right);\mathbb{C}\right)$
coincides with the Betti cohomology $H^{k}\left(X\left(q\right);\mathbb{C}\right)$
and we let $h^{k}\left(G;q\right):=\dim H^{k}\left(X\left(q\right);\mathbb{C}\right)$.
Note that we still have $h^{0}\left(G;q\right)=1$ and $h^{\left\lfloor \frac{1}{2}\dim X\right\rfloor }\left(G;q\right)\asymp\vol\left(X\left(q\right)\right)$
(the analogue of \eqref{eq:Betti-local}).

The following result extends Theorem~\ref{thm:CSXDH-SO5-Betti} to
uniform Gross inner forms of $SO_{5}$ over more general totally real
number fields. 
\begin{thm}
\label{thm:CSXDH-SO5-Betti-Gross} Let $G$ be a uniform Gross inner
form of $SO_{5}$. Assume \textbf{(EC)} and \textbf{(DP)} holds. Then
\[
h^{d}\left(G;q\right)\ll\vol\left(X\left(q\right)\right)^{\frac{1}{2}},\qquad d=\begin{cases}
2 & \textrm{if }G_{\infty}\cong SO\left(3,2\right)\times\mbox{compact}\\
1 & \textrm{if }G_{\infty}\cong SO\left(1,4\right)\times\mbox{compact}.
\end{cases}
\]
\end{thm}

Note that in the case of $G_{\infty}\cong SO\left(1,n\right)\times\mbox{compact}$,
then $\tilde{X}=\mathbb{H}^{n}$ is the $n$-dimensional real hyperbolic
space and $X\left(q\right)$ is a compact congruence hyperbolic $n$-manifold.
Gromov conjectured (see \cite[Equation (12)]{SX91}) the following
asymptotic bound on the Betti numbers of congruence real hyperbolic
$n$-manifolds 
\begin{equation}
h^{d}\left(G;q\right)\ll\vol\left(X\left(q\right)\right)^{\frac{2d}{n-1}},\qquad\forall d\leq\left(n-1\right)/2.\label{eq:Gromov}
\end{equation}
In \cite{CM13} (see also \cite[Section 4.2]{Ber18}), Cossutta and
Marshall suggested a tighter bound
\begin{equation}
h^{d}\left(G;q\right)\ll\vol\left(X\left(q\right)\right)^{\frac{2d}{n}},\qquad\forall d<\left(n-1\right)/2,\label{eq:Cossutta-Marshall}
\end{equation}
and in \cite[Corollary 1.4]{CM13}, they proved this conjecture in
the smaller range of $d<\frac{\left\lfloor n/2\right\rfloor }{2}$,
and for levels of the form $q=\mathfrak{c}\mathfrak{p}^{m}$, where
$\mathfrak{c}$ and $\mathfrak{p}$ are fixed. In our case of interest,
$n=4$ and $d=1$, Gromov's conjecture~\eqref{eq:Gromov} reads $h^{1}\left(G;q\right)\ll\vol\left(X\left(q\right)\right)^{\frac{2}{3}}$,
while Cossutta--Marshall's stronger conjecture~\eqref{eq:Cossutta-Marshall}
reads $h^{1}\left(G;q\right)\ll\vol\left(X\left(q\right)\right)^{\frac{1}{2}}$,
and because $d=1=\frac{\left\lfloor n/2\right\rfloor }{2}$, \cite[Corollary 1.4]{CM13}
falls just short of proving Gromov's conjecture. As a consequence
of Theorem~\ref{thm:CSXDH-SO5-Betti-Gross} we get the following:
\begin{cor}
Let $G$ be a uniform Gross inner form of $SO_{5}$ and $G_{\infty}\cong SO\left(1,4\right)\times\mbox{compact}$.
Assume \textbf{(EC)} and \textbf{(DP)} holds. Then the compact congruence
hyperbolic $4$-manifolds $\left\{ X\left(q\right)\right\} _{q}$
satisfy Gromov's conjecture~\eqref{eq:Gromov}, as well as the stronger
Cossutta--Marshall's conjecture~\eqref{eq:Cossutta-Marshall}. 
\end{cor}

Consider now the case of $G$ a definite Gross form of a split classical
group over $F$ and fix a finite place $p$. For simplicity of exposition
we assume in this introduction that $F=\mathbb{Q}$. For any $p\nmid q\in\mathbb{N}$,
denote the $p$-adic group $G_{p}:=G\left(\mathbb{Q}_{p}\right)$
and its level $q$ congruence $p$-arithmetic subgroup by $\Gamma_{p}\left(q\right):=\mbox{\ensuremath{G\left(\mathbb{Q}\right)\cap\ker\bigl(GL_{n}\left(\mathbb{Z}[1/p]\right)\rightarrow GL_{n}\left(\mathbb{Z}/q\mathbb{Z}\right)\bigr)}}\leq G_{p}$.
Denote by $L^{2}\left(\Gamma_{p}\left(q\right)\backslash G_{p}\right)$
the right regular $G_{p}$-representation and let $\Pi\left(G_{p}\right)$
be the set of irreducible admissible representations of $G_{p}$.
For any $\pi\in\Pi\left(G_{p}\right)$, denote the rate of decay of
its matrix coefficients by $r(\pi)\in[2,\infty]$ and denote its multiplicity
in $L^{2}\left(\Gamma_{p}\left(q\right)\backslash G_{p}\right)$ by
\[
m\left(\pi;q\right):=\dim\mbox{Hom}_{G_{p}}\left(\pi,L^{2}\left(\Gamma_{p}\left(q\right)\backslash G_{p}\right)\right).
\]
The following is the $p$-adic analogue of the Cohomological Sarnak--Xue
Density Hypothesis for definite Gross inner forms. We note that it
can be stated more generally for any semisimple linear algebraic group
$G\leq GL_{n}$ and prime $p$, such that $G_{\infty}=G\left(\mathbb{R}\right)$
is compact and $G_{p}=G\left(\mathbb{Q}_{p}\right)$ is not.
\begin{conjecture}[$p$-adic CSXDH]
\label{conj:CSXDH-p} Let $G$ be a definite Gross inner form of
$SO_{5}$ over a number field and $p$ a prime of $F$. Assume \textbf{(EC)}
and \textbf{(DP)} holds. Then for any $\pi\in\Pi\left(G_{p}\right)$,
\[
m\left(\pi;q\right)\ll\vol\left(\Gamma_{p}\left(q\right)\backslash G_{p}\right)^{\frac{2}{r(\pi)}}.
\]
\end{conjecture}

By the work of \cite{GK23} the $p$-adic CSXDH is reduced to the
special case of spherical representations: If the $p$-adic CSXDH
holds for any $\pi\in\Pi^{\mathrm{ur}}\left(G_{p}\right)$, where
$\Pi^{\mathrm{ur}}\left(G_{p}\right)=\left\{ \pi\in\Pi\left(G_{p}\right)\,:\,\pi^{K_{p}}\ne0\right\} $
and $K_{p}=G_{p}\cap GL_{n}\left(\mathbb{Z}_{p}\right)$, then it
holds for any $\pi\in\Pi\left(G_{p}\right)$ (see \cite[Theorems 1.6 and 1.8]{GK23}).
We note that for a given $\pi\in\Pi\left(G_{p}\right)$, and $U\leq K_{p}$
a compact open subgroup such that $\pi^{U}\ne0$, then the implicit
constant in the above inequality depends on the index $[K_{p}:U]$. 

The following result is another special case of our main result, which
proves the $p$-adic CSXDH for definite Gross inner forms of $G=SO_{5}$.
\begin{thm}[$p$-adic CSXDH for $SO_{5}$]
\label{thm:CSXDH-SO5-p} Let $G$ be a definite Gross inner form
of $SO_{5}$ over a number field $F$ and $p$ a prime of $F$. Assume
\textbf{(EC)} and \textbf{(DP)} holds. Then Conjecture~\ref{conj:CSXDH-p}
holds.
\end{thm}

From the $p$-adic CSXDH we obtain the following applications to the
field of high-dimensional expanders and Ramanujan complexes (see \cite{Lub14}).
Let $\tilde{X}_{p}$ be the Bruhat-Tits building associated to $G_{p}$,
which is an infinite, locally finite, contractible, poly-simplicial
complex, admitting a strongly transitive action of $G_{p}$ (see \cite{Tit79}).
Denote the set of maximal faces (a.k.a. chambers) of $\tilde{X}_{p}$
by $\tilde{Y}_{p}$, which is identified as a $G_{p}$-set with $G_{p}/I_{p}$,
where $I_{p}\leq G_{p}$ is an Iwahori subgroup. Let $K\left(q\right):=G_{\infty}I_{p}K_{f}^{p}\left(q\right)$
where $K_{f}^{p}\left(q\right):=\prod_{\ell\ne p}K_{\ell}\left(q\right)$
and $K_{\ell}\left(q\right):=G\cap\ker\bigl(GL_{n}\left(\mathbb{Z}_{\ell}\right)\rightarrow GL_{n}\left(\mathbb{Z}_{\ell}/q\mathbb{Z}_{\ell}\right)\bigr)$.
Note that $\Gamma_{p}\left(q\right)=G\left(\mathbb{Q}\right)\cap K_{f}^{p}\left(q\right)$,
and that for any $p\ne\ell\nmid q$, then $K_{\ell}\left(q\right)=K_{\ell}=G\cap GL_{n}\left(\mathbb{Z}_{\ell}\right)$.
Define the level $q$ congruence complex and its set of maximal faces
to be
\[
X_{p}\left(q\right):=\Gamma_{p}\left(q\right)\backslash\tilde{X}_{p},\qquad Y_{p}\left(q\right):=\Gamma_{p}\left(q\right)\backslash\tilde{Y}_{p}\cong\Gamma_{p}\left(q\right)\backslash G_{p}/I_{p}\cong G\mathbb{\left(Q\right)}\backslash G\left(\mathbb{A}\right)/K\left(q\right).
\]

\begin{defn}
\label{def:Ram-comp} (see \cite{LLP20}) Call $X_{p}\left(q\right)$
Ramanujan if for any $\pi\in\Pi\left(G_{p}\right)$ such that $m\left(\pi;q\right)\dim\pi^{I_{p}}\ne0$,
then $\pi$ is either $1$-dimensional or tempered, i.e. $r(\pi)=2$.
\end{defn}

Ramanujan complexes were constructed for the groups $G_{p}=PGL_{n}\bigl(\mathbb{F}_{p}((t))\bigr)$
in \cite{LSV05-1,LSV05-2,Li04,Sarv07} and for $G_{p}=PU_{3}\left(\mathbb{Q}_{p}\right)$
in \cite{BC11,BFGMKW15,EP22,EFMP23}. In \cite{FLW13,KLW18} the authors
have studied the case of $G_{p}=PSp_{4}\left(\mathbb{Q}_{p}\right)\cong SO_{5}\left(\mathbb{Q}_{p}\right)$.
However, there is currently no known construction of a family of Ramanujan
complexes for any group which is not a form of $PGL_{n}$, and in
particular there is no known construction for $SO_{5}$. In fact in
\cite{HPS77}, the authors constructed a counter-example to the Naive
Ramanujan Conjecture (NRC) for $Sp_{4}$. Using a similar counter-example
to the NRC of $SO_{5}\cong PSp_{4}$, in Proposition~\ref{prop:Non-Ram}
we construct non-Ramanujan congruence complexes for definite Gross
inner forms of $SO_{5}$.

Due to the failure of the NRC, and motivated by the use of the Bombieri--Vinogradov
theorem as a replacement of the Riemann hypothesis for Dirichlet $L$-functions,
Sarnak suggested that in many applications one can replace the Ramanujan
condition with the following slightly weaker notion of density-Ramanujan
(see \cite{Sar15}). 
\begin{defn}
\label{def:den-Ram-comp} Call $\left\{ X_{p}\left(q\right)\right\} _{q}$
density-Ramanujan if for any fixed $r\in(2,\infty)$, 
\[
\sum_{r(\pi)\geq r}m\left(\pi;q\right)\dim\pi^{I_{p}}\ll|Y_{p}\left(q\right)|^{\frac{2}{r}}.
\]
\end{defn}

In \cite{GK23}, Golubev and Kamber gave various useful criteria (see
Corollary~\ref{cor:sph-SXDH}), for proving the density-Ramanujan
property and the $p$-adic CSXDH, namely Theorem~\ref{thm:CSXDH-SO5-p}.
Therefore, as a consequence of the $p$-adic CSXDH, we get the following:
\begin{thm}
\label{thm:den-Ram-SO5} Let $G$ be a definite Gross inner form of
$SO_{5}$ and $p$ a prime. Under the assumptions of Theorem \ref{thm:CSXDH-SO5-p},
the family of congruence complexes $\left\{ X_{p}\left(q\right)\right\} _{q}$
is density-Ramanujan.
\end{thm}

In \cite{LLP20}, the authors proved that a family of Ramanujan complexes
exhibit the cutoff phenomenon (in total variation) for the mixing
time for its Non Backtracking Random Walk (NBRW) on the maximal faces
(see Section~\ref{sec:Gross-forms} for the precise definitions).
In Theorem~\ref{thm:den-Ram->cutoff}, we strengthen this result
of \cite{LLP20} and show that for congruence complexes the density-Ramanujan
property implies the cutoff phenomenon, vindicating Sarnak's strategy
for this application. Currently this is the only application in the
theory of high dimensional expanders which requires stronger results
towards the Ramanujan Conjecture than Kazhdan's property (T), see
\cite[Section 2.4]{Lub18}. Combined with the previous result we get:
\begin{thm}
\label{thm:cutoff-SO5} Let $G$ be a definite Gross inner form of
$SO_{5}$ and $p$ a prime. Under the assumptions of Theorem \ref{thm:CSXDH-SO5-p},
the family of congruence complexes $\left\{ X_{p}\left(q\right)\right\} _{q}$
exhibits the cutoff phenomenon.
\end{thm}

Two immediate applications of the cutoff phenomenon are the following:
(i) The congruence complexes $\left\{ X_{p}\left(q\right)\right\} _{q}$
satisfy the optimal almost diameter property (Theorem~\ref{cor:OAD-comp-SO5}),
and (ii) the $p$-arithmetic subgroup $\Gamma_{p}=G\left(\mathbb{Q}\right)\cap GL_{n}\left(\mathbb{Z}[1/p]\right)$
satisfies the optimal strong approximation property (Theorem~\ref{cor:OSA-comp-SO5}).
See \cite{GK23} for a more detailed description of the connection
between the SXDH and these applications.

\subsection{Strategy}

Let us now explain our strategy for proving the classical and the
$p$-adic CSXDH, namely, Theorems~\ref{thm:CSXDH-SO5-=00005Cinfty}
and \ref{thm:CSXDH-SO5-p} respectively, and their generalizations
to all Gross inner forms of $SO_{5}$, both uniform and definite,
defined over a totally real number field. 

Our techniques rely on Arthur's endoscopic classification of automorphic
representations of split classical groups \cite{Art13}, and its extension
by Taïbi to Gross inner forms \cite{Tai18}, using the notion of global
$A$-parameters. Following Marshall and Shin \cite{MS19}, we emphasize
the importance of the $A$-shape of the $A$-parameter, and its closely
related Arthur $SL_{2}$-type, as invariants that govern the behavior
of both the multiplicity and the rate of decay of matrix coefficients. 

Using Arthur's endoscopic classification we state in Conjecture~\ref{conj:CSXDH-shape}
a uniform version of the CSXDH, which generalizes both Conjecture~\ref{conj:CSXDH-=00005Cinfty}
(CSXDH) and Conjecture~\ref{conj:CSXDH-p} ($p$-adic CSXDH). Our
main result in this work is Theorem~\ref{thm:CSXDH-SO5-shape}, which
proves Conjecture \ref{conj:CSXDH-shape} for Gross inner forms of
$SO_{5}$ over totally real number fields, therefore implying both
Theorems~\ref{thm:CSXDH-SO5-=00005Cinfty} (CSXDH for $SO_{5}$)
and \ref{thm:CSXDH-SO5-p} ($p$-adic CSXDH for $SO_{5}$), and hence
also their consequences Theorems~\ref{thm:CSXDH-SO5-Betti}, \ref{thm:CSXDH-SO5-Betti-Gross},
\ref{thm:den-Ram-SO5} and \ref{thm:cutoff-SO5}, described above.\medskip{}

\noindent The proof of the CSXDH, namely Theorem~\ref{thm:CSXDH-SO5-shape},
is split into two parts: 

First, for each $A$-shape $\varsigma$ denote by $r\left(\varsigma\right)$
the worst-case rate of decay of matrix coefficients of any unramified
or archimedean representation sitting inside a local $A$-packet associated
to a local $A$-parameter of a cohomological global $A$-parameter
whose $A$-shape is $\varsigma$. Motivated by the Arthur--Ramanujan
conjectures \cite{Art89} (see also \cite{Clo02}), we use the best
known results towards the generalized Ramanujan--Petersson conjecture
for general linear groups \cite{Clo13}, combined with Arthur's endoscopic
classification for classical groups \cite{Art13}, to give upper bounds
on~$r\left(\varsigma\right)$ (see Theorem~\ref{thm:r-shape}).

Second, for each $A$-shape $\varsigma$, each finite-dimensional
representation $E$ and each level~$q$, denote by $h\left(G,\varsigma;q,E\right)$
the dimension of the contribution to Matsushima's formula~\eqref{eq:Matsushima}
coming from automorphic representations with cohomology $E$ and level
$q$, whose global $A$-parameters are of $A$-shape~$\varsigma$.
In order to bound $h\left(G,\varsigma;E,q\right)$, we work case by
case on the six possible $A$-shapes of $SO_{5}$, relying on Schmidt's
explicit description of Arthur's endoscopic classification of $SO_{5}$
\cite{Sch18,Sch20}. The $A$-shapes of $SO_{5}$ are as follows:
$(\mathbf{G})$ and $(\mathbf{Y})$, the General and Yoshida types,
which contribute the generic automorphic representations, $(\mathbf{F})$,
the finite type, which contributes the one-dimensional automorphic
representations, $(\mathbf{B})$ and $(\mathbf{Q})$, the Howe--Piatetski-Shapiro
and Soudry types, and $(\mathbf{P})$, the Saito--Kurokawa type.
The main work is done on the last three shapes. 

For the $A$-shapes $(\mathbf{B})$ and $(\mathbf{Q})$ we implement
a ``divide and conquerer'' strategy of bounding $h\left(G,\varsigma;q,E\right)$
by $h_{1}\left(G,\varsigma;q,E\right)\cdot h_{2}\left(G,\varsigma;q,E\right)$,
where $h_{1}\left(G,\varsigma;q,E\right)$ counts the number of automorphic
representations with cohomology $E$, level $q$ and $A$-shape $\varsigma$,
and $h_{2}\left(G,\varsigma;q,E\right)$ gives upper bounds on the
contribution coming from a single such automorphic representation
to Matsushima's formula~\eqref{eq:Matsushima}. By Arthur's classification,
$h_{1}\left(G,\varsigma;q,E\right)$ is essentially the number of
global $A$-parameters of $G$ of $A$-shape $\varsigma$ which satisfies
certain local restrictions coming from $E$ and $q$, and we get that
$h_{1}\left(G,\varsigma;q,E\right)$ is bounded above by the total
contribution of automorphic representations of $G^{\{\varsigma\}}$
at level $q$ to cohomology with coefficients in $E$, where $G^{\{\varsigma\}}$
is a specific smaller classical group whose dual $\widehat{G^{\{\varsigma\}}}$
is the ``natural cone'' inside $\hat{G}$ of global $A$-parameters
of $A$-shape $\varsigma$. The bound on $h_{2}\left(G,\varsigma;q,E\right)$
requires a uniform bound on the Gelfand--Kirillov dimension of representations
appearing in the local $A$-packets of $A$-parameters of $A$-shape
$(\mathbf{B})$ or $(\mathbf{Q})$, where uniform means in both the
representations and the local field. Such a bound is available for
the group $GL_{2}$ \cite[Corollary A.3]{MS19}, and using Schmidt's
explicit description of the local $A$-packets of $SO_{5}$, we get
similar such bounds for the $A$-shape $(\mathbf{B})$ or $(\mathbf{Q})$
(see Corollary~\ref{cor:h-BQ}). 

For the $A$-shape $(\mathbf{P})$ we use the endoscopic character
relations and implement the strategy introduced by \cite{GG21}. More
generally, we use the fundamental lemma \cite{Wal97,Ngo10}, and its
strengthening by Ferrari \cite{Fer07} to indicator functions of congruence
subgroups to compute $h\left(G,\varsigma;q,E\right)$ using similar
terms associated to endoscopic groups for any classical group $G$.
The $A$-shape $(\mathbf{P})$, is a special case of $A$-shapes studied
in \cite{GG21,DGG22}, called Generalized Saito--Kurokawa. In analogy
with \cite{GG21} on unitary groups, we prove upper bounds for $h\left(G,\varsigma;q,E\right)$,
for any split classical group $G$ and any Generalized Saito--Kurokawa
$A$-shape~$\varsigma$ (see Corollary~\ref{cor:h-P}).

\subsection{Outline of the paper}

The paper is structured as follows. In Section~\ref{sec:Preliminaries},
we recall the necessary background about groups and representations.
We then recall Arthur's classification; first in full generality in
Section~\ref{sec:Arthur-classification}, and then Schmidt's explicit
description of $A$-packets for $SO_{5}$ in Section~\ref{sec:Schmidt}.
In Section~\ref{sec:Depth-cohomology} we begin the argument in earnest:
we collect and prove local results (on depth at finite places and
cohomology at infinite places) that allows us to state our main theorem
precisely. Section~\ref{sec:Bounds-on-rate} contains the computation
of the rate of decay of matrix coefficients for local factors of automorphic
representations. These are then compared with the growth of multiplicities
which we compute from the endoscopic classification in Section~\ref{sec:Bounds-on-cohomology}
and this comparison is the statement of the main theorem. Lastly,
Section~\ref{sec:Gross-forms} is dedicated to applications to density-Ramanujan
complexes, the cutoff phenomenon, the optimal almost diameter and
the optimal strong approximation. 

\subsection{Acknowledgements}

We wish to thank Peter Sarnak for his continuous encouragements and
support throughout this project. We thank Amitay Kamber and Simon
Marshall for reading an earlier version of our paper and suggesting
valuable improvements to it. We also thank Edgar Assing, Rahul Dalal
and Ori Parzanchevski for many interesting and useful discussions. 

Evra is supported by the Azrieli Foundation, by the National Science
Foundation (NSF) grant DMS-1638352 and by the Israel Science Foundation
(ISF) grant 1577/23. During parts of the project Gustafsson was supported
by the Swedish Research Council (Vetenskapsrådet) grant 2018-06774. 

We are grateful for the hospitality of the Institute for Advanced
Study in Princeton, where this project originated. The third author
wishes to thank the Centre de Recherches Mathématiques in Montréal.

\section{Preliminaries \protect\label{sec:Preliminaries}}

In this section we set up notations and collect known facts regarding
split classical groups, their Gross inner forms, and their local and
global representations and Langlands parameters. Throughout this paper,
by a split classical group we mean either a special orthogonal group
or a symplectic group (excluding the special and general linear groups). 

\subsection{Classical groups and root data\protect\label{subsec:Classical-groups}}

Let $G$ be a split special orthogonal or symplectic group defined
over a local or global field $F$, namely
\[
SO_{n}:=\left\{ g\in SL_{n}\,:\,gJ_{n}g^{t}=J_{n}\right\} ,\quad J_{n}=\left(\begin{array}{ccc}
 &  & 1\\
 & \iddots\\
1
\end{array}\right),
\]
\[
Sp_{2n}:=\left\{ g\in SL_{2n}\,:\,gJ'_{2n}g^{t}=J'_{2n}\right\} ,\quad J'_{2n}=\left(\begin{array}{cc}
 & J_{n}\\
-J_{n}
\end{array}\right).
\]
The dual group of $G$ is is the complex algebraic group $\hat{G}$
whose root datum is dual to that of $G$. For example, $\widehat{SO_{2n+1}}=Sp_{2n}\left(\mathbb{C}\right)$
and $\widehat{Sp_{2n}}=SO_{2n+1}\left(\mathbb{C}\right)$. Since $G$
is split we can use the dual group $\hat{G}$ instead of the $L$-group
$^{L}G=\hat{G}\times W_{F}$, where $W_{F}$ is the Weil group of
$F$.

Denote by $N'=N(G)$ and $N=N(\hat{G})$ the dimensions of the standard
representations of $G$ and $\hat{G}$, respectively, 
\begin{equation}
\mathrm{Std}_{G}\,:\,G\rightarrow GL_{N'}\qquad\mbox{and}\qquad\mathrm{Std}_{\hat{G}}\,:\,\hat{G}\rightarrow GL_{N}\left(\mathbb{C}\right).\label{eq:std-rep}
\end{equation}
For example, if $G=SO_{2n+1}$ then $N'=2n+1$ and $N=2n$, and if
$G=Sp_{2n}$ then $N'=2n$ and $N=2n+1$. 

Let $T_{n}=\left\{ \mbox{diag}(x_{1},\ldots,x_{n})\,:\,x_{i}\in GL_{1}\right\} \leq GL_{n}$
be the subgroup of diagonal matrices, and define the standard split
maximal tori of $G$ and $\hat{G}$, respectively, by $T=T_{N'}\cap G$
and $\hat{T}=T_{N}(\mathbb{C})\cap\hat{G}$. Note that $T$ and $\hat{T}$
are dual to one another, hence 
\[
X^{*}\left(T\right)\cong X_{*}(\hat{T})\qquad\mbox{and}\qquad X_{*}\left(T\right)\cong X^{*}(\hat{T}),
\]
where $X^{*}\left(\cdot\right)=\mbox{Hom}\left(\cdot,\mathbb{G}_{m}\right)$
and $X_{*}=\mbox{Hom}\left(\mathbb{G}_{m},\cdot\right)$ are the groups
of algebraic characters and cocharacters. Denote the root datum of
$\left(G,T\right)$ to be 
\[
\mathcal{R}\left(G,T\right)=\left(X^{*}\left(T\right),X_{*}\left(T\right),\Phi^{*}\left(G,T\right),\Phi_{*}\left(G,T\right)\right),
\]
where $\Phi^{*}\left(G,T\right)\subset X^{*}\left(T\right)$ is the
set of roots and $\Phi_{*}\left(G,T\right)\subset X_{*}\left(T\right)$
the set of coroots. Note that the root datum of $(\hat{G},\hat{T})$
is $\mathcal{R}(\hat{G},\hat{T)}=\left(X_{*}\left(T\right),X^{*}\left(T\right),\Phi_{*}\left(G,T\right),\Phi^{*}\left(G,T\right)\right).$

Let $B_{n}\leq GL_{n}$ be the subgroup of upper triangular matrices,
and denote $B=B_{N'}\cap G$ and $\hat{B}=B_{N}(\mathbb{C})\cap\hat{G}$,
the standard Borel subgroups of $G$ and $\hat{G}$, respectively.
Let $\Delta^{*}\left(G,B,T\right)\subset\Phi^{*+}\left(G,B,T\right)$
and $\Delta_{*}\left(G,B,T\right)\subset\Phi_{*}^{+}\left(G,B,T\right)$
be the sets of simple and positive roots and coroots, respectively,
with respect to the Borel subgroup $B$. Denote the Weyl vector to
be 
\[
\rho:=\frac{1}{2}\sum_{\alpha\in\Phi^{*+}\left(G,B,T\right)}\alpha\in X^{*}\left(T\right)\otimes_{\mathbb{Z}}\mathbb{R}.
\]

Let $\mathcal{W}=\mathcal{W}\left(G,T\right)=N_{G}\left(T\right)/T$
be the Weyl group of $G$, where $N_{G}\left(T\right)$ is the normalizer
of $T$ in $G$. Note that $\mathcal{W}\left(G,T\right)\cong\mathcal{W}(\hat{G},\hat{T})$,
the Weyl group of $\hat{G}$. The action of $\mathcal{W}$ on $T$
induces an action on $X^{*}\left(T\right)$ and $X_{*}\left(T\right)$.

\subsection{Gross inner forms\protect\label{subsec:Gross-forms}}

Let $G^{*}$ be a split classical group defined over a totally real
field $F$, and let $G$ be an inner form of $G^{*}$ (see \cite[Section 3.1]{Tai18}
for the precise definition). 
\begin{defn}
Call $G$ a Gross inner form if $G\left(F_{v}\right)\cong G^{*}\left(F_{v}\right)$
for any finite place $v$ of $F$. 
\end{defn}

Let $S$ be the set of infinite places of $F$. Since $F$ is totally
real then $F_{v}\cong\mathbb{R}$ for any $v\in S$. The inner forms
of $G^{*}$ defined over $\mathbb{R}$ are completely classified by
$H^{1}\left(\mathbb{R},G_{\mathrm{ad}}^{*}\right)$, the first cohomology
of the Galois group $\mbox{Gal}\left(\mathbb{C}/\mathbb{R}\right)\cong\mathbb{Z}/2\mathbb{Z}$
acting on the adjoint group $G_{\mathrm{ad}}^{*}=G^{*}/Z\left(G^{*}\right)$.
Therefore a Gross inner form $G$ of $G^{*}$ over $F$ is completely
determined by the finite collection of cohomology classes $\left\{ h_{v}\right\} _{v\in S}\subset H^{1}\left(\mathbb{R},G_{\mathrm{ad}}^{*}\right)$,
such that $G\left(F_{v}\right)$ is the real inner form of $G^{*}\left(F_{v}\right)$
of type $h_{v}$ for any $v\in S$. For more details see \cite[Section 3]{Tai18}.

Let use make two important remarks: 1) There is a unique compact inner
form of $G^{*}$ defined over $\mathbb{R}$; this fact is equivalent
to $G^{*}$ admitting discrete series. 2) Not every collection of
cohomology classes $\left\{ h_{v}\right\} _{v\in S}\subset H^{1}\left(\mathbb{R},G_{\mathrm{ad}}^{*}\right)$
defines a Gross inner form: certain sign conditions described \cite[Section 3]{Tai18}
must be satisfied. 

For any $a,b\in\mathbb{N}_{0}$, denote the special orthogonal real
group of signature $\left(a,b\right)$ by
\[
SO\left(a,b\right)=\left\{ g\in SL_{a+b}\left(\mathbb{R}\right)\,:\,gI_{a,b}g^{t}=I_{a,b}\right\} ,\quad I_{a,b}=\left(\begin{array}{cc}
I_{a}\\
 & -I_{b}
\end{array}\right).
\]
We note that $SO\left(a,b\right)\cong SO\left(b,a\right)$ for any
$a,b$, and that the split special orthogonal real group $SO_{N}\left(\mathbb{R}\right)$
is isomorphic to $SO\left(\left\lceil \frac{N}{2}\right\rceil ,\left\lfloor \frac{N}{2}\right\rfloor \right)$.
Finally when $b=0$, we abbreviate $SO\left(a\right)=SO\left(a,0\right)$,
which is a compact Lie group. 
\begin{example}
For the split special odd orthogonal group $G^{*}=SO_{2n+1}$, then
$H^{1}\left(\mathbb{R},G_{\mathrm{ad}}^{*}\right)$ is classified
by the set $[n]=\left\{ 0,1,\ldots,n\right\} $, where for the inner
form associated to $a\in[n]$ is the special odd orthogonal group
of signature $\left(2a+1,2n-2a\right)$. The split inner form corresponds
to $a=\left\lfloor \frac{n}{2}\right\rfloor $, in which case $G_{v}:=G(F_{v})\cong SO\left(n+1,n\right)$,
while the compact inner form corresponds to $a=n$, in which case
$G_{v}\cong SO\left(2n+1,0\right)=SO\left(2n+1\right)$. Note that
if $G_{v}\cong SO\left(2a+1,2n-2a\right)$, then $K_{v}:=S\left(O\left(2a+1\right)\times O\left(2n-2a\right)\right)$
is a maximal compact subgroup of $G_{v}$.
\end{example}

For an inner form $G$ of a quasi-split group $G^{*}$ over $\mathbb{R}$
, with a maximal compact subgroup $K\leq G\left(\mathbb{R}\right)$,
its Kottwitz sign is define to be 
\begin{equation}
e\left(G\right):=\left(-1\right)^{q\left(G^{*}\right)-q\left(G\right)},\qquad q\left(G\right):=\frac{1}{2}\left(\dim G\left(\mathbb{R}\right)-\dim K\right).\label{eq:Kottwitz-sign}
\end{equation}
For example if $G$ is the inner form of $G^{*}=SO_{2n+1}$ of signature
$\left(2a+1,2n-2a\right)$, then $q\left(G\right)=\left(2a+1\right)\left(n-a\right)$,
and 
\[
e\left(G\right)=\left(-1\right)^{q\left(G^{*}\right)-q\left(G\right)}=\left(-1\right)^{\frac{1}{2}n\left(n+1\right)-\left(2a+1\right)\left(n-a\right)}=\left(-1\right)^{\frac{1}{2}n\left(n-1\right)-a}.
\]

\begin{prop}
\label{prop:Gross-forms}\cite[Proposition 3.1.2.1]{Tai18} Let $\left\{ a_{v}\right\} _{v\in S}\subset[n]$.
Then there exists a Gross inner form $G$ of $SO_{2n+1}$ over $F$,
such that $G\left(F_{v}\right)\cong SO\left(2a_{v}+1,2n-2a_{v}\right)$
for any $v\in S$ if and only if 
\[
\sum_{v\in S}\left(\frac{1}{2}n\left(n-1\right)+a_{v}\right)\equiv0\mod 2,
\]
which is also equivalent to $\prod_{v\in S}e\left(G\left(F_{v}\right)\right)=1$. 
\end{prop}

\begin{defn}
Let $G/F$ be a Gross inner form of the split group $G^{*}$. Let
$v_{1},\ldots,v_{s}$ be the infinite places of $F$. Say that $G$
is definite if $G\left(F_{v_{i}}\right)$ is compact for any $i=1,\ldots,s$.
Say that $G$ is uniform of type $G'\left(\mathbb{R}\right)$, a non-compact
inner form of $G^{*}$ over $\mathbb{R}$, if $F\ne\mathbb{Q}$, $G\left(F_{v_{1}}\right)\cong G'\left(\mathbb{R}\right)$
and $G\left(F_{v_{i}}\right)$ is compact for any $i=2,\ldots,s$.
\end{defn}

\begin{cor}
\label{cor:Gross-definite-SO5} Let $G^{*}=SO_{5}$. Then there exists
a definite Gross inner form if and only if $[F:\mathbb{Q}]$ is even.
\end{cor}

\begin{proof}
The signatures of a definite inner form are $a_{v_{i}}=n=2$, for
any $i=1,\ldots,s$, $s=[F:\mathbb{Q}]$, and since $F$ is totally
real, hence $\sum_{v\in S}\left(\frac{1}{2}n\left(n-1\right)+a_{v}\right)\equiv[F:\mathbb{Q}]$,
and by Proposition~\ref{prop:Gross-forms}, we get the claim.
\end{proof}
\begin{cor}
\label{cor:Gross-uniform-SO5} Let $G^{*}=SO_{5}$ and let $a\in\left\{ 0,1\right\} $.
Then there exists a uniform Gross inner form of type $SO\left(2a+1,4-2a\right)$
if and only if $F\ne\mathbb{Q}$ and $a\equiv[F:\mathbb{Q}]\mod 2$. 
\end{cor}

\begin{proof}
The signatures of a uniform Gross inner form of type $SO\left(2a+1,4-2a\right)$
are determined by $a_{v_{1}}=a$ and $a_{v_{i}}=n=2$, for any $i=2,\ldots,s$,
hence $\sum_{v\in S}\left(\frac{1}{2}n\left(n-1\right)+a_{v}\right)\equiv a+[F:\mathbb{Q}]\mod 2$,
and by Proposition~\ref{prop:Gross-forms}, we get that claim.
\end{proof}

\subsection{Local representations and parameters\protect\label{sec:Local-rep-par}}

Let $G$ be a connected reductive group over a local field $F$. Let
$K\leq G\left(F\right)$ be a maximal compact subgroup.

Denote by $\Pi\left(G\left(F\right)\right)$ the set of equivalence
classes of irreducible admissible (complex) representations of $G\left(F\right)$,
defined as follows: For $F$ non-archimedean, 
\[
\Pi\left(G\left(F\right)\right):=\left\{ \pi\,:\,G\left(F\right)\rightarrow GL\left(V_{\pi}\right)\,:\,\begin{array}{l}
\forall\left\{ 0\right\} \lvertneqq U\lvertneqq V_{\pi}\mbox{ closed},\;\pi\left(G\left(F\right)\right)\left(U\right)\ne U\\
\forall v\in V_{\pi},\;\mbox{Stab}_{G\left(F\right)}\left(v\right)\mbox{ is an open subgroup}\\
\forall C\leq G\left(F\right)\mbox{ compact open},\;\dim\left(V_{\pi}^{C}\right)<\infty
\end{array}\right\} /\cong.
\]
For $F$ archimedean, $\Pi\left(G\left(F\right)\right)$ will denote
the set of equivalence classes of irreducible admissible $\left(\mathfrak{g},K\right)$-modules,
where $\mathfrak{g}$ is the Lie algebra of $G\left(F\right)$.

Denote the set of (unitary) equivalence classes of unitary irreducible
representations of $G\left(F\right)$, for both archimedean and non-archimedean
$F$, by
\begin{multline}
\Pi^{\mathrm{unit}}\left(G\left(F\right)\right)\\
:=\left\{ \pi\,:\,G\left(F\right)\rightarrow U\left(H_{\pi}\right)\,:\,\begin{array}{l}
H_{\pi}\mbox{ Hilbert space},\;U\left(H_{\pi}\right)\mbox{ unitary operators},\\
\forall\left\{ 0\right\} \lvertneqq U\lvertneqq H_{\pi}\mbox{ closed},\;\pi\left(G\left(F\right)\right)\left(U\right)\ne U
\end{array}\right\} /\cong.\label{eq:Pi-unit}
\end{multline}
Then there is a natural injection of $\Pi^{\mathrm{unit}}\left(G\left(F\right)\right)$
into $\Pi\left(G\left(F\right)\right)$, for both archimedean and
non-archimedean $F$.

For any representation $\pi\in\Pi^{\mathrm{unit}}\left(G\left(F\right)\right)$
denote by $c_{v,u}^{\pi}\,:\,G\left(F\right)\rightarrow\mathbb{C}$,
$c_{v,u}^{\pi}\left(g\right)=\langle\pi(g)v,u\rangle_{\pi}$ the matrix
coefficient of $\pi$ w.r.t. the vectors $v,u\in V$. Let $V_{K}\subset V$
be the subspace of $K$-finite vectors and for any $v,u\in V_{K}$,
call $c_{v,u}^{\pi}$ a $K$-finite matrix coefficient. For $r\geq2$,
denote by $L^{r}\left(G\left(F\right)\right)$ the space of $r$-integrable
functions on $G\left(F\right)/Z\left(F\right)$, where $Z$ is the
center of $G$. Define the rate of decay of matrix coefficients of
$\pi$ to be
\begin{equation}
r(\pi):=\inf\left\{ r\geq2\,:\,\forall v,u\in V_{K},\;c_{v,u}^{\pi}\in L^{r}\left(G\left(F\right)\right)\right\} .\label{eq:r-def}
\end{equation}
Denote the subset of tempered representations of $G\left(F\right)$
by 
\[
\Pi^{\mathrm{temp}}\left(G\left(F\right)\right):=\left\{ \pi\in\Pi^{\mathrm{unit}}\left(G\left(F\right)\right)\;:\;r(\pi)=2\right\} .
\]
Note that we have the following chain of containments
\begin{equation}
\Pi^{\mathrm{temp}}\left(G\left(F\right)\right)\subset\Pi^{\mathrm{unit}}\left(G\left(F\right)\right)\subset\Pi\left(G\left(F\right)\right).\label{eq:rep-temp-unit-adm}
\end{equation}

Let $W_{F}$ be the local Weil group of $F$. Namely, $W_{\mathbb{C}}=\mathbb{C}^{\times}$,
$W_{\mathbb{R}}=\mathbb{C}^{\times}\sqcup\mathbb{C}^{\times}j$, where
$j^{2}=-1$ and $jzj=\bar{z}$ for any $z\in\mathbb{C}^{\times}$,
and if $F$ is non-archimedean then $W_{F}=\mathrm{Frob}_{F}^{\mathbb{Z}}\ltimes I_{F}$,
where $I_{F}$ is the inertia subgroup and $\mathrm{Frob}_{F}$ a
Frobenius element of $\mbox{Gal}\left(\overline{F}/F\right)$. Let
$|\cdot|_{F}\,:\,W_{F}\rightarrow\mathbb{R}_{>0}$ be the norm associated
to $W_{F}$ by local class field theory via $W_{F}^{\mathrm{ab}}\cong F^{*}$,
where $W_{F}^{\mathrm{ab}}=W_{F}/[W_{F},W_{F}]$ is its abelianization,
and the natural norm $|\cdot|_{F}\,:\,F^{*}\rightarrow\mathbb{R}_{>0}$.
Let $L_{F}$ be the local Langlands group of $F$. We have that $L_{F}=W_{F}$
if $F=\mathbb{R}$ or $\mathbb{C}$ and otherwise $L_{F}=W_{F}\times SL_{2}^{D}$,
where $SL_{2}^{D}=SL_{2}\left(\mathbb{C}\right)$ is the Deligne $SL_{2}$-factor
(the $D$ is meant to differentiate it from the Arthur $SL_{2}$-factor
defined below). Extend $|\cdot|_{F}\,:\,L_{F}\rightarrow\mathbb{R}_{>0}$
by making it trivial on $SL_{2}^{D}$.

Assume for simplicity that $G$ is an inner form of a split group
$G^{*}$ defined over $F$. By definition $\hat{G}=\widehat{G^{*}}$.
An $L$-parameter is a continuous homomorphism $\phi\,:\,L_{F}\rightarrow\hat{G}$
which has a semisimple image on $W_{F}$ and, if $F$ is non-archimedean,
is also analytic on $SL_{2}^{D}$. Denote the set of (relevant) local
$L$-parameters of $G$ by 
\[
\Phi\left(G\left(F\right)\right):=\left\{ \phi\,:\,L_{F}\rightarrow\hat{G}\,:\,\mbox{relevant }L\mbox{-parameter}\right\} /\hat{G},
\]
where relevant is as defined in \cite{Bor79} (this condition is superfluous
for $G$ quasi-split). Call an $L$-parameter tempered if it has bounded
image on $W_{F}$ and denote the subset of tempered $L$-parameters
of $G$ by 
\[
\Phi^{\mathrm{temp}}\left(G\left(F\right)\right):=\left\{ \phi\in\Phi\left(G\left(F\right)\right)\;:\;\phi\mbox{ tempered}\right\} .
\]
Note that if $F$ is non-archimedean, since $W_{F}=\mathrm{Frob}_{F}^{\mathbb{Z}}\ltimes I_{F}$
and $I_{F}$ is compact, then a continuous homomorphism $\phi\,:\,L_{F}\rightarrow\hat{G}$
has semisimple (resp. bounded) image on $W_{F}$ if and only if $\phi\left(\mathrm{Frob}_{F}\right)$
is semisimple (resp. is bounded, i.e. sits inside a compact subgroup
of $\hat{G}$).

A local $A$-parameter is a continuous homomorphism $\psi\,:\,L_{F}\times SL_{2}^{A}\rightarrow\hat{G}$,
where $SL_{2}^{A}:=SL_{2}\left(\mathbb{C}\right)$ is called the Arthur
$SL_{2}$-factor, such that $\psi$ has bounded image on $W_{F}$,
is also analytic on $SL_{2}^{A}$, and, if $F$ is non-archimedean,
is also analytic on $SL_{2}^{D}$. Denote the set of (relevant) local
$A$-parameters of $G$ by 
\[
\Psi\left(G\left(F\right)\right):=\left\{ \psi\,:\,L_{F}\times SL_{2}^{A}\rightarrow\hat{G}\,:\,\mbox{relevant }A\mbox{-parameter}\right\} /\hat{G},
\]
Associate to each $A$-parameter is its corresponding $L$-parameter,
\begin{equation}
\Psi\left(G\left(F\right)\right)\rightarrow\Phi\left(G\left(F\right)\right),\qquad\psi\mapsto\phi_{\psi},\qquad\phi_{\psi}\left(x\right)=\psi\left(x,\mathrm{diag}\bigl(|x|_{F}^{1/2},|x|_{F}^{-1/2}\bigr)\right).\label{eq:A-L-par}
\end{equation}
An $L$-parameter of the form $\phi_{\psi}$, for some $\psi\in\Psi\left(G\left(F\right)\right)$,
is said to be of Arthur type. Denote the set of Arthur type $L$-parameters
of $G$ by 
\begin{equation}
\Phi^{\mathrm{Ar}}\left(G\left(F\right)\right):=\left\{ \phi_{\psi}\in\Phi\left(G\left(F\right)\right)\;:\;\psi\in\Psi\left(G\left(F\right)\right)\right\} .\label{eq:Phi-Ar}
\end{equation}
By \cite[(3.7)]{Sha11}), the map in \eqref{eq:A-L-par} is injective.
Note that $\Phi^{\mathrm{temp}}\left(G\left(F\right)\right)$ embeds
into $\Psi\left(G\left(F\right)\right)$ by extending the map trivially
to $SL_{2}^{A}$, in which case the $A$-parameter and the corresponding
$L$-parameter coincides. We get the following chain of containments
\begin{equation}
\Phi^{\mathrm{temp}}\left(G\left(F\right)\right)\subset\Phi^{\mathrm{Ar}}\left(G\left(F\right)\right)\subset\Phi\left(G\left(F\right)\right).\label{eq:par-temp-Ar-adm}
\end{equation}

We remark that the chains \eqref{eq:rep-temp-unit-adm} and \eqref{eq:par-temp-Ar-adm}
are analogous in some sense but not completely. More precisely, the
Local Langlands Correspondence (LLC) for $G\left(F\right)$, aims
at parametrizing $\Pi\left(G\left(F\right)\right)$ by $\Phi\left(G\left(F\right)\right)$,
in such a way that $\Pi^{\mathrm{temp}}\left(G\left(F\right)\right)$
is parametrized by $\Phi^{\mathrm{temp}}\left(G\left(F\right)\right)$.
The LLC is expected to satisfy several properties, see for instance
\cite{Kal16} and \cite{Har22}. In fact the LLC between $\Pi\left(G\left(F\right)\right)$
and $\Phi\left(G\left(F\right)\right)$ is deduced from the LLC between
$\Pi^{\mathrm{temp}}\left(G\left(F\right)\right)$ and $\Phi^{\mathrm{temp}}\left(G\left(F\right)\right)$
(see \cite{ABPS14,SZ18}). On the other hand, $\Phi^{\mathrm{Ar}}\left(G\left(F\right)\right)$
does not parametrize the entire set $\Pi^{\mathrm{unit}}\left(G\left(F\right)\right)$
in the LLC. Instead, according to the Arthur--Ramanujan conjectures
\cite{Art89} (see also \cite{Sar05}), $\Phi^{\mathrm{Ar}}\left(G\left(F\right)\right)$
is meant to parametrize the subset of $\Pi^{\mathrm{unit}}\left(G\left(F\right)\right)$
of local components of automorphic representations .

\subsection{Unramified representations and parameters\protect\label{subsec:Unramified-rep-par}}

Let $G$ be a split connected reductive group over a non-archimedean
local field $F$, $\mathcal{O}$ the ring of integers of $F$, $\varpi$
a uniformizer of $F$ and $K=G\left(\mathcal{O}\right)$ a hyperspecial
maximal compact subgroup of $G\left(F\right)$. 

An admissible representation $\pi\in\Pi\left(G\left(F\right)\right)$
is said to be unramified if it has a non-zero $K$-invariant vector.
Denote the set of unramified representations of $G\left(F\right)$
by 
\[
\Pi^{\mathrm{ur}}\left(G\left(F\right)\right):=\left\{ \pi\in\Pi\left(G\left(F\right)\right)\,:\,\dim\pi^{K}\ne0\right\} .
\]
By the works of Borel, Casselman, Matsumoto and Satake \cite{Bor76,Cas80,Mat69,Sat63},
for any $\pi\in\Pi^{\mathrm{ur}}\left(G\left(F\right)\right)$, then
$\dim\pi^{K}=1$. Moreover, there exists an unramified character $\chi_{\pi}$
of $T(F)$, unique up to the action of $\mathcal{W}$, such that $\pi$
is the (unique) unramified subquotient of the normalized parabolic
induction $\mathrm{Ind}_{B\left(F\right)}^{G\left(F\right)}\chi_{\pi}$.
Denote the set of unramified characters of $T\left(F\right)$ by $X^{\mathrm{ur}}\left(T\left(F\right)\right)$.
Then the Borel, Casselman, Matsumoto and Satake correspondence states
that
\begin{equation}
\Pi^{\mathrm{ur}}\left(G\left(F\right)\right)\overset{\sim}{\longleftrightarrow}X^{\mathrm{ur}}\left(T\left(F\right)\right)/\mathcal{W},\qquad\pi\mapsto\chi_{\pi}.\label{eq:BCMS}
\end{equation}
Note that $X^{\mathrm{ur}}\left(T\left(F\right)\right)=\mbox{Hom}\bigl(T\left(F\right)/T\left(\mathcal{O}\right),\mathbb{C}^{\times}\bigr)$.
Moreover, the map $\lambda\mapsto\lambda(\varpi)$ induces an isomorphism
$X_{*}\left(T\right)\cong T\left(F\right)/T\left(\mathcal{O}\right)$.
Recalling the definition of the dual torus $\hat{T}=\mbox{Hom}\left(X_{*}\left(T\right),\mathbb{C}^{\times}\right)$,
we therefore get a $\mathcal{W}$-invariant isomorphism of abelian
groups $X^{\mathrm{ur}}\left(T\left(F\right)\right)\cong\hat{T}$
(see \cite[(1.12.1)]{BR91}). Combined with \eqref{eq:BCMS} we get
a bijection $\Pi^{\mathrm{ur}}\left(G\left(F\right)\right)\overset{\sim}{\longrightarrow}\hat{T}/\mathcal{W}$,
denoted $\pi\mapsto\mathrm{s}_{\pi}$ (see \cite[Proposition 6.4]{Gro98}).

An $L$-parameter $\psi\in\Phi\left(G\left(F\right)\right)$ is said
to be unramified if it trivial on the inertia subgroup $I_{F}$ and
on the Deligne $SL_{2}$-factor $SL_{2}^{D}$. Denote the set of unramified
$L$-parameters of $G$ by 
\[
\Phi^{\mathrm{ur}}\left(G\left(F\right)\right):=\left\{ \phi\in\Phi\left(G\left(F\right)\right)\,:\,\phi|_{I_{F}\times SL_{2}^{D}}\equiv1\right\} .
\]
Since $L_{F}\cong\left(\mbox{Frob}_{F}^{\mathbb{Z}}\ltimes I_{F}\right)\times SL_{2}^{D}$,
an unramified $L$-parameter is completely determined by its image
on $\mbox{Frob}_{F}$. Since any semisimple element of $\hat{G}$
is $\hat{G}$-conjugate to an element in $\hat{T}$, itself unique
up to the action of $\mathcal{W}$, we get a bijection $\Phi^{\mathrm{ur}}\left(G\left(F\right)\right)\overset{\sim}{\longrightarrow}\hat{T}/\mathcal{W}$,
denoted $\phi\mapsto\mathrm{s}_{\phi}$, where $\mathrm{s}_{\phi}$
is the $\mathcal{W}$-orbit of $\phi\left(\mathrm{Frob}_{F}\right)\in\hat{T}$.

The combination of the above two bijections is the so-called Unramified
Local Langlands Correspondence (ULLC), which gives a bijection between
unramified representations and unramified $L$-parameters of $G\left(F\right)$,
\begin{equation}
\Pi^{\mathrm{ur}}\left(G\left(F\right)\right)\overset{\sim}{\longleftrightarrow}\Phi^{\mathrm{ur}}\left(G\left(F\right)\right),\qquad\pi\leftrightarrow\phi\qquad\text{such that}\qquad\mathrm{s}_{\pi}=\mathrm{s}_{\phi}.\label{eq:ULLC}
\end{equation}
For $\phi\in\Phi^{\mathrm{ur}}\left(G\left(F\right)\right)$, denote
the corresponding representation by $\pi_{\phi}\in\Pi^{\mathrm{ur}}\left(G\left(F\right)\right)$.
Under the ULLC, the $L$-packet of an unramified $L$-parameter $\phi\in\Phi^{\mathrm{ur}}\left(G\left(F\right)\right)$,
is by definition the singleton $\Pi_{\phi}:=\left\{ \pi_{\phi}\right\} \subset\Pi^{\mathrm{ur}}\left(G\left(F\right)\right)$. 
\begin{rem}
\label{rem:ur-temp} We note that the ULLC preserves temperedness
(see for instance \cite[Section 6]{Gro98}). More precisely, for
any $\phi\in\Phi^{\mathrm{ur}}\left(G\left(F\right)\right)$, the
following are equivalent: (i) $\phi$ is tempered, (ii) $\pi_{\phi}$
is tempered, (iii) $\chi_{\pi_{\phi}}$ is unitary, and (iv) $\mathrm{s}_{\phi}$
belongs to the maximal compact subgroup of $\hat{T}$.
\end{rem}

Define the set of unramified $A$-parameters of $G\left(F\right)$:
\[
\Psi^{\mathrm{ur}}\left(G\left(F\right)\right):=\left\{ \psi\in\Psi\left(G\left(F\right)\right)\,:\,\psi|_{I_{F}\times SL_{2}^{D}}\equiv1\right\} .
\]
Note that $\psi|_{I_{F}\times SL_{2}^{D}}\cong\phi_{\psi}|_{I_{F}\times SL_{2}^{D}}$,
and therefore $\psi\in\Psi^{\mathrm{ur}}\left(G\left(F\right)\right)$
if and only if $\phi_{\psi}\in\Phi^{\mathrm{ur}}\left(G\left(F\right)\right)$,
in which case we denote 
\begin{equation}
\pi_{\psi}:=\pi_{\phi_{\psi}}\in\Pi^{\mathrm{ur}}\left(G\left(F\right)\right),\qquad\mathrm{s}_{\pi_{\psi}}=\psi\left(\mathrm{Frob}_{F},\mathrm{diag}\left(|\mathrm{Frob}_{F}|_{F}^{1/2},|\mathrm{Frob}_{F}|_{F}^{-1/2}\right)\right).\label{eq:ur-=00005Cpi-=00005Cpsi}
\end{equation}

We record a lemma about an especially simple type of unramified $A$-parameter,
namely the one corresponding to the trivial representation. 
\begin{lem}
\label{lem:triv-rep-par} \cite[Section 7]{Gro98} Let $\psi_{1}\in\Psi^{\mathrm{ur}}\left(G\left(F\right)\right)$
be such that $\psi_{1}\left(\mathrm{Frob}_{F}\right)=1$ and $\psi_{1}|_{SL_{2}^{A}}=\sigma_{\mathrm{princ}}$.
Then $\pi_{\psi_{1}}$ is the trivial representation and its $L$-parameter
is given by
\[
\phi_{1}\in\Phi^{\mathrm{ur}}\left(G\left(F\right)\right),\qquad\phi_{1}\left(\mathrm{Frob}_{F}\right)=\sigma_{\mathrm{princ}}\left(\mathrm{diag}\left(|\mathrm{Frob}_{F}|_{F}^{1/2},|\mathrm{Frob}_{F}|_{F}^{-1/2}\right)\right).
\]
\end{lem}

We now describe the one-dimensional representations of split classical
groups more generally. In the case of the symplectic group $G=Sp_{N}$,
the trivial representation is the only one-dimensional representation
of $G\left(F\right)$ since $G\left(F\right)$ is simple. For the
special orthogonal group $G=SO_{N}$, there are two $1$-dimensional
representations for $G\left(F\right)$, the trivial representation
and the spinor norm $\delta\,:\,G\left(F\right)\rightarrow F^{\times}/\left(F^{\times}\right)^{2}$.
We note that $\mathop{\mathrm{im}}\delta\cong F^{\times}/\left(F^{\times}\right)^{2}\cong\mathbb{Z}/2\mathbb{Z}$
and that $\ker\delta\cong\mbox{Spin}_{N}\left(F\right)/\left\{ \pm I\right\} $
is simple. Finally we mention that the $A$- and $L$-parameters for
the spinor norm are given by 
\[
\psi_{\delta}\in\Psi^{\mathrm{ur}}\left(G\left(F\right)\right),\qquad\psi_{\delta}\left(\mathrm{Frob}_{F}\right)=\hat{\delta}\left(\mathrm{Frob}_{F}\right),\qquad\psi_{\delta}|_{SL_{2}^{A}}=\sigma_{\mathrm{princ}},
\]
\[
\phi_{\delta}\in\Phi^{\mathrm{ur}}\left(G\left(F\right)\right),\qquad\phi_{\delta}\left(\mathrm{Frob}_{F}\right)=\hat{\delta}\left(\mathrm{Frob}_{F}\right)\cdot\sigma_{\mathrm{princ}}\left(\mathrm{diag}\left(|\mathrm{Frob}_{F}|_{F}^{1/2},|\mathrm{Frob}_{F}|_{F}^{-1/2}\right)\right),
\]
where $\hat{\delta}\,:\,L_{F}\rightarrow\hat{Z}$, for $\hat{Z}$
the center of $\hat{G}$, is such that $\hat{\delta}|_{I_{F}\times SL_{2}^{D}}\equiv1$
and $\hat{\delta}\left(\mathrm{Frob}_{F}\right)=I$ if $\varpi_{F}\in\left(F^{\times}\right)^{2}$
and $\hat{\delta}\left(\mathrm{Frob}_{F}\right)=-I$ otherwise. 
\begin{prop}
\label{prop:1-dim-rep-par} Let $G$ be a split classical group defined
over $F$. Let $\Psi^{\mathrm{ur}}\left(G\left(F\right);(\mathbf{F})\right)$
be the subset of $\psi\in\Psi^{\mathrm{ur}}\left(G\left(F\right)\right)$
with $\,\psi|_{SL_{2}^{A}}=\sigma_{\mathrm{princ}}$. Then
\[
\Psi^{\mathrm{ur}}\left(G\left(F\right);(\mathbf{F})\right)=\begin{cases}
\left\{ \psi_{1}\right\}  & \text{if }G=Sp_{N}\\
\left\{ \psi_{1},\psi_{\delta}\right\}  & \text{if }G=SO_{N}.
\end{cases}
\]
Furthermore 
\[
\left\{ \pi\in\Pi^{\mathrm{ur}}\left(G\left(F\right)\right)\,:\,\dim\pi=1\right\} =\left\{ \pi_{\psi}\,:\,\psi\in\Psi^{\mathrm{ur}}\left(G\left(F\right);(\mathbf{F})\right)\right\} .
\]
\end{prop}

\begin{proof}
Clearly $\Psi^{\mathrm{ur}}\left(G\left(F\right);(\mathbf{F})\right)$
contains $\psi_{1}$ and also $\psi_{\delta}$ when $G=SO_{N}$. In
the other direction note that any $\psi\in\Psi^{\mathrm{ur}}\left(G\left(F\right);(\mathbf{F})\right)$
is completely determined by $\psi\left(\mathrm{Frob}_{F}\right)$.
Since $\mbox{Frob}_{F}$ commute with $SL_{2}^{A}$ we get $\psi\left(\mathrm{Frob}_{F}\right)\in C_{\hat{G}}(\psi|_{SL_{2}^{A}})$,
by assumption $C_{\hat{G}}(\psi|_{SL_{2}^{A}})=C_{\hat{G}}\left(\sigma_{\mathrm{princ}}\right)$,
and by Remark~5.1.18 and Theorem~6.1.3 in \cite{CM17}, we get $C_{\hat{G}}\left(\sigma_{\mathrm{princ}}\right)=\hat{Z}$,
where for $\hat{Z}$ the center of $\hat{G}$. Hence $\Psi^{\mathrm{ur}}\left(G\left(F\right);(\mathbf{F})\right)$
is of size $|\hat{Z}|$ (since $\hat{Z}$ is finite, hence compact,
$\psi\left(\mathrm{Frob}_{F}\right)$ can be any element in $\hat{Z}$)
and the first claim now follows from the fact that $\hat{Z}=\left\{ I\right\} $
when $G$ is symplectic and $\hat{Z}=\left\{ \pm I\right\} $ when
$G$ is special orthogonal. The second claim follows from the first
and the discussion above.
\end{proof}

\subsection{Depth of representations and parameters\protect\label{subsec:Depth-rep-par}}

Let $G$ be a split connected reductive group over a non-archimedean
local field $F$, let $\mathcal{O}$ be the ring of integers of $F$
with $\varpi\in\mathcal{O}$ a uniformizer, and let $K=G\left(\mathcal{O}\right)$
be a hyperspecial maximal compact subgroup of $G\left(F\right)$. 

Let $I_{F}$ be the inertia subgroup of $L_{F}$, and let $\left\{ I_{F}^{r}\,:\,r\in\mathbb{R}_{\geq0}\right\} $
be the upper numbering filtration of it as defined in \cite[Ch.  IV,  Sec.  3]{Ser13}.
Define the depth invariant of $L$-parameters of $G\left(F\right)$
by
\[
d\,:\,\Phi\left(G\left(F\right)\right)\rightarrow\mathbb{R}_{\geq0},\qquad d(\phi)=\inf\left\{ r\in\mathbb{R}_{\geq0}\,:\,\phi\mid_{I_{F}^{r}}\equiv1\right\} .
\]
Note that any unramified $L$-parameter $\phi\in\Phi^{\mathrm{ur}}\left(G\left(F\right)\right)$
is of depth $0$, since $\phi|_{I_{F}}\equiv1$. Define the depth
invariant of $A$-parameters to be 
\[
d\,:\,\Psi\left(G\left(F\right)\right)\rightarrow\mathbb{R}_{\geq0},\qquad d(\psi)=\inf\left\{ r\in\mathbb{R}_{\geq0}\,:\,\psi\mid_{I_{F}^{r}}\equiv1\right\} .
\]
Note that the the restriction of $\psi\in\Psi\left(G\left(F\right)\right)$
to the inertia subgroup is equal to the that of its corresponding
$L$-parameter $\phi_{\psi}$, i.e. $\psi|_{I_{F}}\equiv\phi_{\psi}|_{I_{F}}$,
and therefore $d\left(\psi\right)=d\left(\phi_{\psi}\right)$. 

Let $\mathcal{B}\left(G\left(F\right)\right)$ be the Bruhat-Tits
building of $G\left(F\right)$, for any $x\in\mathcal{B}\left(G\left(F\right)\right)$
let $K_{x}$ be its corresponding parahoric subgroup and let $\left\{ K_{x,r}\,:\,r\in\mathbb{R}_{\geq0}\right\} $
be the Moy-Prasad subgroup filtration of $K_{x}$ (see \cite{MP94,MP96}).
Define the depth invariant of irreducible admissible representations
of $G\left(F\right)$ by
\[
d\,:\,\Pi\left(G\left(F\right)\right)\rightarrow\mathbb{R}_{\geq0},\qquad d(\pi)=\inf\left\{ r\in\mathbb{R}_{\geq0}\,:\,\exists x\in\mathcal{B}\left(G\left(F\right)\right),\;\dim\pi^{K_{x,r}}\ne0\right\} .
\]

Let $x_{0}\in\mathcal{B}\left(G\left(F\right)\right)$ be the hyperspecial
vertex whose corresponding parahoric is $K_{x_{0}}=K=G\left(\mathcal{O}\right)$.
Then the Moy-Prasad filtration of $K_{x_{0}}$ coincides with the
principal congruence filtration, namely, $K_{x_{0},r}=K\left(\varpi^{\left\lfloor r\right\rfloor }\right)$,
for any $r\in\mathbb{R}$, where $\left\lfloor r\right\rfloor $ denotes
the largest integer smaller then $r$, and for any $m\in\mathbb{N}$,
$K\left(\varpi^{m}\right)$ denotes the level $\varpi^{m}$ principal
congruence subgroup of $K$ by 
\[
K\left(\varpi^{m}\right):=\left\{ g\in G\left(\mathcal{O}\right)\;:\;g\equiv1\mod{\varpi}^{m}\right\} .
\]
Define the level invariant of irreducible admissible representations
of $G\left(F\right)$ by

\[
\ell\,:\,\Pi\left(G\left(F\right)\right)\rightarrow\mathbb{R}_{\geq0},\qquad\ell(\pi)=\min\{m\in\mathbb{N}_{0}\,:\,\dim\pi^{K\left(\varpi^{m}\right)}\ne0\}.
\]
Note that $\ell(\pi)=0$ if and only if $\pi\in\Pi^{\mathrm{ur}}\left(G\left(F\right)\right)$. 
\begin{lem}
\label{lem:depth-level} For any $\pi\in\Pi\left(G\left(F\right)\right)$,
\[
d(\pi)\leq\ell(\pi)\leq d(\pi)+1.
\]
\end{lem}

\begin{proof}
Since $K_{x_{0},m}=K\left(\varpi^{m}\right)$, we get that $d(\pi)\leq\ell(\pi)$.
To prove the other direction recall the following basic facts on Bruhat-Tits
buildings: (i) For any $x\in\mathcal{B}\left(G\left(F\right)\right)$
there exists $g\in G\left(F\right)$ such that $g.x\in\overline{C_{0}}$,
where $C_{0}$ is a fixed fundamental chamber in $\mathcal{B}\left(G\left(F\right)\right)$
containing $x_{0}$, (ii) the stabilizer of $C_{0}$ in $G\left(F\right)$,
denoted $K_{C_{0}}$, is the the preimage in $K=G\left(\mathcal{O}\right)$
of the standard Borel subgroup of $G\left(\mathcal{O}/\varpi\mathcal{O}\right)$,
and (iii) for any $x\in\overline{C_{0}}$ then $K_{C_{0}}\subset K_{x}$
(see \cite[§2.2, §3.7]{Tit79} and \cite[§2.2.4]{Yu09}). By fact
(i) we may assume that $x\in\overline{C_{0}}$, and by facts (ii)
and (iii) we get that $K\left(\varpi\right)\subset K_{C_{0}}\subset K_{x}$.
It follows that $K\left(\varpi^{\left\lfloor r\right\rfloor +1}\right)\subset K_{x,r}$
for any $r$, which implies the claim.
\end{proof}

\subsection{Cohomological representations and parameters\protect\label{subsec:Cohomological-rep-par}}

For this section, let $F=\mathbb{R}$, and let $G/F$ be an inner
form of a split special odd orthogonal or symplectic group $G^{*}$.
The following is mostly a summary of the relevant results from \cite{NP21}
about cohomological representations and parameters.

Let $\mathfrak{g}_{0}$ be the real Lie algebra of $G\left(\mathbb{R}\right)$
and $\mathfrak{g}=\mathfrak{g}_{0}\otimes_{\mathbb{R}}\mathbb{C}$
its complexification. Let $K\leq G\left(\mathbb{R}\right)$ be a maximal
compact subgroup, $\mathfrak{k}_{0}$ its Lie algebra and $\mathfrak{k}=\mathfrak{k}_{0}\otimes_{\mathbb{R}}\mathbb{C}$.
Let $\theta$ denote the Cartan involution such that $K$ (resp.\ $\mathfrak{k}_{0}$,
resp.\ $\mathfrak{k}$) is the $\theta$-fixed points of $G\left(\mathbb{R}\right)$
(resp.\ $\mathfrak{g}_{0}$, resp.\ $\mathfrak{g}$). Denote by
$\mathfrak{g}_{0}=\mathfrak{k}_{0}\oplus\mathfrak{p}_{0}$ and $\mathfrak{g}=\mathfrak{k}\oplus\mathfrak{p}$
the Cartan decompositions.

It is well known (see e.g. \cite[Section 3.1.1]{Tai18}) that $G\left(\mathbb{R}\right)$
satisfies the following equivalent conditions: (i) $G\left(\mathbb{R}\right)$
has discrete series representations, (ii) $G^{*}\left(\mathbb{R}\right)$
has discrete series representations, (iii) $G\left(\mathbb{R}\right)$
has a compact inner form, and (iii) $G\left(\mathbb{R}\right)$ has
a compact maximal torus, which is unique up to conjugation by $G\left(\mathbb{R}\right)$. 

Let $G_{\mathbb{C}}$ be the complex base change of $G$, and let
$T^{c}\leq B\leq G_{\mathbb{C}}$ be a pair of $\theta$-stable maximal
torus defined and compact over $\mathbb{R}$, and a $\theta$-stable
Borel subgroup defined over~$\mathbb{C}$ (see \cite[Proposition 11]{NP21}).
The dual group of $G$ is by definition the dual group of its split
inner form, i.e. $\hat{G}=\widehat{G^{*}}$. It comes equipped with
a maximal torus and a Borel subgroup $\hat{T}\leq\hat{B}\leq\hat{G}$,
and by duality $X^{*}\left(T^{c}\right)\cong X_{*}(\hat{T})$. Denote
by $\rho_{G}\in X_{*}(\hat{T})\otimes\mathbb{R}$ the half sum of
positive coroots of $\hat{G}$ w.r.t.\ $(\hat{T},\hat{B})$. Denote
by $\sigma_{\mathrm{princ}}\,:\,SL_{2}\left(\mathbb{C}\right)\rightarrow\hat{G}$
the principal $SL_{2}$ of $\hat{G}$, which maps the subgroup of
diagonal matrices $\widehat{T_{1}}\leq SL_{2}\left(\mathbb{C}\right)$
to $\hat{T}$ and maps the subgroup of upper triangular matrices $\widehat{B_{1}}\leq SL_{2}\left(\mathbb{C}\right)$
to $\hat{B}$, and denote $\omega=\sigma_{\mathrm{princ}}(\bigl(\begin{smallmatrix}0 & 1\\
-1 & 0
\end{smallmatrix}\bigr))$. Then $\sigma_{\mathrm{princ}}|_{\widehat{T_{1}}}=2\rho\in X_{*}(\hat{T})$,
the element $\omega\in N_{\hat{G}}(\hat{T})$ conjugates $\hat{B}$
to its opposite and $\omega^{2}=2\rho\left(-1\right)\in Z(\hat{G})$
(see \cite[Proposition 1]{NP21}).

Define the set of algebraic representations of $G$ to be 
\[
\Pi^{\mathrm{alg}}\left(G\right):=\left\{ E\in\Pi\left(G\left(\mathbb{C}\right)\right)\,:\,\dim E<\infty,\quad E^{\theta}\cong E\right\} .
\]
 For any $E\in\Pi^{\mathrm{alg}}\left(G\right)$, let $\lambda_{E}\in X^{*}\left(T^{c}\right)\cong X_{*}(\hat{T})$
be its highest dominant weight w.r.t.\ $B$. Its infinitesimal character
is
\[
\chi_{E}:=\rho_{G}+\lambda_{E}\in X^{*}\left(T^{c}\right)\otimes\mathbb{R}\cong X_{*}(\hat{T})\otimes\mathbb{R}.
\]
Recall from \cite[I.5]{BW00} the notion of $\left(\mathfrak{g}_{0},K\right)$-cohomology.
Given $E\in\Pi^{\mathrm{alg}}\left(G\right)$, and $\pi\in\Pi\left(G\left(\mathbb{R}\right)\right)$,
we denote
\[
H^{j}\left(\pi;E\right):=H^{j}\left(\mathfrak{g}_{0},K;\pi\otimes E\right),\qquad H^{*}\left(\pi;E\right):=\bigoplus_{j=0}^{d}H^{j}\left(\pi;E\right),
\]
where $d=\dim G-\dim K$. Define the set of cohomological representations
of $G\left(\mathbb{R}\right)$ with coefficients in $E$, by 
\[
\Pi^{\mathrm{coh}}\left(G\left(\mathbb{R}\right);E\right):=\left\{ \pi\in\Pi\left(G\left(\mathbb{R}\right)\right)\,:\,\dim H^{*}\left(\pi;E\right)\ne0\right\} .
\]
Given $\psi\in\Psi\left(G\left(\mathbb{R}\right)\right)$, define
$\chi_{\psi},\nu_{\psi}\in X_{*}(\hat{T})$ to be the coweights of
$\hat{T}$ such that $\psi|_{W_{\mathbb{C}}}$ is conjugate to $z\mapsto z^{\chi_{\psi}}\bar{z}^{\nu_{\psi}}$,
where $W_{\mathbb{C}}=\mathbb{C^{\times}}$ is the Weil group of $\mathbb{C}$.
Define the set of Adams--Johnson, or cohomological, $A$-parameters
of $G\left(\mathbb{R}\right)$ with coefficients in $E$ by 
\begin{equation}
\Psi^{\mathrm{AJ}}\left(G\left(\mathbb{R}\right);E\right):=\left\{ \psi\in\Psi\left(G\left(\mathbb{R}\right)\right)\,:\,\chi_{\psi}\equiv\chi_{E}\right\} .\label{eq:AJ-packets}
\end{equation}
Our goal in this subsection will be to describe the explicit parametrization
of cohomological representations of $G\left(\mathbb{R}\right)$ by
cohomological $A$-parameters, which follows from the works of \cite{VZ84}
and \cite{AJ87} (see Proposition~\ref{prop:AJ} and \cite{NP21}
for more details).

Denote by $\mathcal{L}$ the set of $\hat{B}$-standard Levi subgroups
of $\hat{G}$ which are $\omega$-invariant, defined up to conjugation
in $\hat{G}$. For any $E\in\Pi^{\mathrm{alg}}\left(G\right)$, denote
by $\mathcal{L}_{E}$ the subset of $\hat{L}\in\mathcal{L}$ such
that $\langle\lambda_{E},\alpha\rangle=0$ for any $\alpha\in\Delta^{*}(\hat{L},\hat{B},\hat{T})$,
the set of simple roots of $\hat{L}$. Denote by $\mathcal{Q}$ the
set of $B$-standard $\theta$-stable parabolic subalgebras $\mathfrak{q}$
of $\mathfrak{g}_{0}$ as in \cite[Definition 6]{NP21}, namely $\mathfrak{q}\subset\mathfrak{g}$
is a parabolic subalgebra of the complexification $\mathfrak{g}$
such that $\theta\left(\mathfrak{q}\right)=\mathfrak{q}$ and the
complex conjugate $\bar{\mathfrak{q}}$ is the opposite of~$\mathfrak{q}$.
Every $\mathfrak{q\in\mathcal{Q}}$ is conjugate to one of the form
$\mathfrak{q_{\lambda}}=\mathfrak{t}\oplus\bigoplus_{\langle\alpha,\lambda\rangle\geq0}\mathfrak{g_{\alpha}}$
for a weight $\lambda$ of the Lie algebra $\mathfrak{t}$ of the
torus $T^{c}\left(\mathbb{R}\right)$. For $\mathfrak{q}\in\mathcal{Q}$,
denote by $\mathfrak{q}=\mathfrak{l}+\mathfrak{u}$ its Levi decomposition,
note that $\mathfrak{l}$ is defined over $\mathbb{R}$, and let $L=\mathrm{Stab}_{G\left(\mathbb{R}\right)}\left(\mathfrak{q}\right)$
which is a connected reductive group with real Lie algebra $\mathfrak{l}_{0}=\mathfrak{l}\cap\mathfrak{g}_{0}$. 

By \cite[Proposition 12]{NP21} the map
\[
\Sigma\,:\,\mathcal{Q}\rightarrow\mathcal{L},\qquad\Sigma\left(\mathfrak{q}\right)=\hat{L},
\]
is surjective and the fiber of $\hat{L}\in\mathcal{L}$ is 
\begin{equation}
\Sigma^{-1}(\hat{L})=\mathcal{W}\bigl(K,T^{c}\left(\mathbb{R}\right)\bigr)\backslash\mathcal{W}\bigl(G\left(\mathbb{R}\right),T^{c}\left(\mathbb{R}\right)\bigr)^{\theta}/\mathcal{W}(\hat{L},\hat{T})^{\theta},\label{eq:Weyl-Group-Cosets}
\end{equation}
where $\mathcal{W}\left(H,S\right)$ is the Weyl group of the real
or complex reductive group $H$ w.r.t.\ the maximal torus $S$, and
we use the identification $\mathcal{W}\left(G\left(\mathbb{R}\right),T^{c}\left(\mathbb{R}\right)\right)\cong\mathcal{W}(\hat{G},\hat{T})$.
Furthermore, for any $E\in\Pi^{\mathrm{alg}}\left(G\right)$, $\hat{L}\in\mathcal{L}_{E}$
and $\mathfrak{q}\in\Sigma^{-1}(\hat{L})$, then $E^{\mathfrak{u}}:=E/\mathfrak{u}E$
is a one-dimensional unitary representation of $L$, and following
\cite{VZ84}, we denote by $A_{\mathfrak{q}}\left(\lambda_{E}\right)\in\Pi^{\mathrm{coh}}\left(G\left(\mathbb{R}\right)\right)$,
the cohomological induction of $E^{\mathfrak{u}}$ of degree $\dim\left(\mathfrak{u}\cap\mathfrak{k}\right)$. 

The following proposition summarizes the construction of Adams and
Johnson packets \cite{AJ87}, or $AJ$-packets for short, and their
extension of the classification result of Vogan and Zuckerman \cite{VZ84}
of cohomological representations in terms of cohomological $A$-parameters.
\begin{prop}
\cite[Theorems 5, 6, 7]{NP21} \label{prop:AJ} Let $E\in\Pi^{\mathrm{alg}}\left(G\right)$.
For any $\psi\in\Psi^{\mathrm{AJ}}\left(G\left(\mathbb{R}\right);E\right)$,
then $\hat{L}=C_{\hat{G}}\left(\phi_{\psi}|_{W_{\mathbb{C}}}\right)\in\mathcal{L}_{E}$
and $\psi|_{SL_{2}^{A}}$ is the principal $SL_{2}$ of $\hat{L}$.
Define the $AJ$-packet of $\psi$ to be 
\[
\Pi_{\psi}:=\left\{ A_{\mathfrak{q}}\left(\lambda_{E^{\vee}}\right)\,:\,\mathfrak{q}\in\Sigma^{-1}(\hat{L})\right\} .
\]
Then $A_{\mathfrak{q}}\left(\lambda_{E}\right)\ne A_{\mathfrak{q'}}\left(\lambda_{E}\right)$
for $\mathfrak{q}\ne\mathfrak{q'}\in\Sigma^{-1}(\hat{L})$, and 
\[
\Pi^{\mathrm{coh}}\left(G\left(\mathbb{R}\right);E\right)=\bigsqcup_{\psi\in\Psi^{\mathrm{AJ}}\left(G\left(\mathbb{R}\right);E\right)}\Pi_{\psi}.
\]
\end{prop}

In the next section we define the notion of Arthur packets \cite{Art13},
or $A$-packets for short, which will be our main concern in this
work. The following result of Arancibia, Moeglin, and Renard \cite{AMR18}
for split classical groups, and its extension to inner forms by Taïbi
\cite{Tai18}, shows that $AJ$-packets and $A$-packets coincides.
\begin{prop}
\cite{AMR18}\cite[Proposition 3.2.6]{Tai18} Let $E\in\Pi^{\mathrm{alg}}\left(G\right)$
and $\psi\in\Psi^{\mathrm{AJ}}\left(G\left(\mathbb{R}\right);E\right)$.
Then the $AJ$-packet $\Pi_{\psi}$ is in fact an $A$-packet. 
\end{prop}

\subsection{Automorphic representations\protect\label{subsec:Automorphic-rep}}

Let $G$ be a Gross inner form of $G^{*}$, a split special odd orthogonal
or a symplectic group, defined over a totally real number field $F$.

For any place $v$ of $F$, denote $G_{v}:=G\left(F_{v}\right)$,
and let $K_{v}\leq G_{v}$ be a maximal compact subgroup, where for
a finite place $v$, we take $K_{v}$ to be a hyperspecial parahoric
subgroup of $G\left(F_{v}\right)$. For any $v$, let $\mu_{G_{v}}$
be a measure on $G_{v}$ normalized such that $\mu_{G_{v}}\left(K_{v}\right)=1$.
Define the adelic group to be $G\left(\mathbb{A}\right):=\prod_{v}^{\mathrm{restr}}G_{v}$,
where $v$ runs over all places of $F$ and restricted means that
for almost all places the $v$-factor belongs to $K_{v}$. Equip $G\left(\mathbb{A}\right)$
with the measure $\mu=\prod_{v}\mu_{G_{v}}$. 

Denote by $\Pi\left(G\left(\mathbb{A}\right)\right)$ the set of equivalence
classes of irreducible admissible representations of $G\left(\mathbb{A}\right)$
(or more precisely, the set of equivalence classes of irreducible
admissible representations of $\left(\mathfrak{g}_{0},K_{\infty}\right)\times\prod_{v\nmid\infty}G\left(F_{v}\right)$)
where $\pi$ is admissible if $\pi=\otimes_{v}\pi_{v}$, where $\pi_{v}\in\Pi\left(G_{v}\right)$
is the $v$-local factor of $\pi$, and $\pi_{v}$ has a non-zero
$K_{v}$-fixed vector for almost all $v$. 

By the work of Borel and Harish-Chandra, $G\left(F\right)$ embeds
diagonally as a lattice in $G\left(\mathbb{A}\right)$, i.e.\ it
is a discrete subgroup and the quotient $G\left(F\right)\backslash G\left(\mathbb{A}\right)$
admits a cofinite measure $\mu$ descended from the measure $\mu_{G}$.
Consider the Hilbert space 
\[
L^{2}\left(G\left(F\right)\backslash G\left(\mathbb{A}\right)\right):=\left\{ f\,:\,G\left(F\right)\backslash G\left(\mathbb{A}\right)\rightarrow\mathbb{C}\,:\,\int_{G\left(F\right)\backslash G\left(\mathbb{A}\right)}|f(x)|^{2}d\mu(x)<\infty\right\} ,
\]
endowed with the following right regular representation of $G\left(\mathbb{A}\right)$,
\[
g.f\left(x\right)=f\left(xg\right),\qquad g\in G\left(\mathbb{A}\right),\quad f\in L^{2}\left(G\left(F\right)\backslash G\left(\mathbb{A}\right)\right),\quad x\in G\left(F\right)\backslash G\left(\mathbb{A}\right).
\]
Let $m\left(\pi\right):=\dim\mbox{Hom}_{G(\mathbb{A})}\left(\pi,L^{2}\left(G\left(F\right)\backslash G\left(\mathbb{A}\right)\right)\right)$
and call $\pi\in\Pi\left(G\left(\mathbb{A}\right)\right)$ a discrete
automorphic representation if $m\left(\pi\right)\ne0$. Then $m\left(\pi\right)$
is the multiplicity of $\pi$ in the discrete automorphic spectrum
of $G$ which we denote by
\[
L_{\mathrm{disc}}^{2}\left(G\left(F\right)\backslash G\left(\mathbb{A}\right)\right):=\bigoplus_{\pi\in\Pi\left(G\left(\mathbb{A}\right)\right)}m\left(\pi\right)\cdot\pi.
\]

We note that since $L^{2}\left(G\left(F\right)\backslash G\left(\mathbb{A}\right)\right)$
has the structure of a unitary representation, then for any discrete
automorphic representation $\pi\in\Pi\left(G\left(\mathbb{A}\right)\right)$,
$m\left(\pi\right)\ne0$, all of its local factors are unitary representations
as well. That is, $\pi_{v}\in\Pi^{\mathrm{unit}}\left(G\left(F_{v}\right)\right)$
for any place $v$, where $\Pi^{\mathrm{unit}}\left(G\left(F_{v}\right)\right)$
was defined in \eqref{eq:Pi-unit}. In fact, according to the Arthur--Ramanujan
conjectures (see \cite{Art89}), the local factors should belong to
the smaller set of Arthur-type unitary representations, i.e.\ $\pi_{v}\in\Pi^{\mathrm{Ar}}\left(G\left(F_{v}\right)\right)$
for any place $v$, where $\Pi^{\mathrm{Ar}}\left(G\left(F_{v}\right)\right)$
is the subset of representations with $L$-parameters in $\Phi^{\mathrm{Ar}}\left(G\left(F_{v}\right)\right)$
which was defined in \eqref{eq:Phi-Ar}.

Define the cohomological discrete part of the representation $L^{2}\left(G\left(F\right)\backslash G\left(\mathbb{A}\right)\right)$,
by

\[
L_{\mathrm{disc}}^{2}\left(G\left(F\right)\backslash G\left(\mathbb{A}\right)\right)^{\mathrm{coh}}=\bigoplus_{\pi\in\Pi^{\mathrm{coh}}\left(G\left(\mathbb{A}\right)\right)}m\left(\pi\right)\cdot\pi,
\]
where 
\[
\Pi^{\mathrm{coh}}\left(G\left(\mathbb{A}\right)\right)=\left\{ \pi\in\Pi\left(G\left(\mathbb{A}\right)\right)\,:\,\exists E\in\Pi^{\mathrm{alg}}\left(G\right),\quad\pi_{\infty}\in\Pi^{\mathrm{coh}}\left(G_{\infty};E\right)\right\} ,
\]
and $\Pi^{\mathrm{coh}}\left(G_{\infty};E\right)$ is the set of representations
$\pi_{\infty}=\otimes_{v\mid\infty}\pi_{v}$ such that for any $v\mid\infty$,
$\pi_{v}\in\Pi^{\mathrm{coh}}\left(G_{v};E\right)$ (defined in Subsection~\ref{subsec:Cohomological-rep-par})
if $G\left(F_{v}\right)$ is non-compact and $\pi_{v}$ is trivial
if $G\left(F_{v}\right)$ is compact. 

We note that if $G$ is a uniform Gross inner form, in particular
$G\left(F_{v}\right)$ is compact for some $v\mid\infty$, then by
the Borel--Harish-Chandra theory \cite[Theorem 5.5]{PR93} $G\left(F\right)\backslash G\left(\mathbb{A}\right)$
is compact and therefore $L^{2}\left(G\left(F\right)\backslash G\left(\mathbb{A}\right)\right)$
decomposes as a direct sum of irreducible subrepresentations. That
is, 
\[
\qquad\qquad L^{2}\left(G\left(F\right)\backslash G\left(\mathbb{A}\right)\right)=L_{\mathrm{disc}}^{2}\left(G\left(F\right)\backslash G\left(\mathbb{A}\right)\right)\qquad\qquad\text{(\ensuremath{G} a uniform inner form)}.
\]
If $G$ is a definite Gross inner form , i.e. $G_{\infty}=\prod_{v\mid\infty}G\left(F_{v}\right)$
is compact, then $\Pi^{\mathrm{coh}}\left(G\left(\mathbb{A}\right)\right)=\left\{ \pi\in\Pi\left(G\left(\mathbb{A}\right)\right)\,:\,\pi_{\infty}^{G_{\infty}}\cong\mathbb{C}\right\} $,
and hence 
\[
\qquad\qquad L_{\mathrm{disc}}^{2}\left(G\left(F\right)\backslash G\left(\mathbb{A}\right)\right)^{\mathrm{coh}}\cong L^{2}\left(G\left(F\right)\backslash G\left(\mathbb{A}\right)\right)^{G_{\infty}}\qquad\qquad\text{(\ensuremath{G} a definite inner form).}
\]
where $L^{2}\left(G\left(F\right)\backslash G\left(\mathbb{A}\right)\right)^{G_{\infty}}$
is the subspace of $G_{\infty}$-invariant vectors.

We end with the following well known consequence of the strong approximation
property, regarding the $1$-dimensional automorphic representations
(which holds more generally for any connected reductive group over
a global field).
\begin{lem}
\label{lem:dim-1} \cite[Lemma 6.2]{KST20} Let $\pi$ be a discrete
automorphic representation of $G$. Then $\dim\pi=1$ if and only
if $\dim\pi_{v}=1$ for some unramified place $v$. 
\end{lem}

\subsection{Congruence manifolds\protect\label{subsec:Congruence-manifolds}}

Let $G$ be a Gross inner form of a split special odd orthogonal or
symplectic group $G^{*}$ defined over a totally real number field
$F$.

Let $\mathcal{O}$ be the ring of integers of $F$, and for ay finite
place $v$, let $\mathcal{O}_{v}$ be the ring of integers of $F_{v}$,
let $\varpi_{v}\in\mathcal{O}_{v}$ be a uniformizer parameter of
$F_{v}$, let $p_{v}=|\mathcal{O}_{v}/\varpi_{v}\mathcal{O}_{v}|$
be the cardinality of the residue field of $F_{v}$ and let $\mathrm{ord}_{v}\,:\,F_{v}^{\times}\rightarrow\mathbb{Z}$
be the $v$-adic valuation $F_{v}$. Let $q\lhd\mathcal{O}$ be a
principal ideal, which we sometimes identify with its principal generator
$q=\prod_{v\nmid\infty}\varpi_{v}^{\mathrm{ord}_{v}(q)}$, where $\mbox{ord}_{v}(q)=\max\left\{ m\in\mathbb{N}_{0}\,:\,\varpi_{v}^{m}\mid q\right\} $.
Denote $v\nmid q$ if $\mbox{ord}_{v}(q)=0$ and $v\mid q$ otherwise,
hence $q=\prod_{v\mid q}\varpi_{v}^{\mathrm{ord}_{v}(q)}$. 

Define the level $q$ congruence subgroup $K\left(q\right)$ of $G\left(\mathbb{A}\right)$
to be
\[
K\left(q\right):=K_{\infty}K_{f}\left(q\right)\leq G\left(\mathbb{A}\right),\qquad K_{f}\left(q\right):=\prod_{v<\infty}K_{v}\left(q\right)\leq G\left(\mathbb{A}_{f}\right),\qquad K_{v}\left(q\right):=K_{v}(\varpi_{v}^{\mathrm{ord}_{v}(q)}),
\]
where $K_{\infty}\leq G_{\infty}:=\prod_{v\mid\infty}G\left(F_{v}\right)$
is a fixed maximal compact subgroup, $K_{v}:=G\left(\mathcal{O}_{v}\right)$
is a hyperspecial maximal compact subgroup for any finite prime $v$
and $K_{v}\left(\varpi_{v}^{m}\right):=\ker\left(K_{v}\rightarrow G\left(\mathcal{O}_{v}/\varpi_{v}^{m}\right)\right)$
is the level $\varpi_{v}^{m}$ congruence subgroup of $K_{v}$, for
any $m\in\mathbb{N}_{0}$. Note that if $v\nmid q$, then $K_{v}\left(q\right)=K_{v}$. 

For $G_{\infty}$ not compact, denote by $\tilde{X}=G_{\infty}/K_{\infty}$
the symmetric space associated to~$G_{\infty}$. Assume without loss
of generality that the strong approximation property holds (see \cite[Chapter 7]{PR93}),
namely $G\left(\mathbb{A}\right)=G\left(F\right)\cdot G_{\infty}K_{f}\left(q\right)$.
Then $\Gamma\left(q\right):=G\left(F\right)\bigcap K_{f}\left(q\right)=\ker\left(G\left(\mathcal{O}\right)\rightarrow G\left(\mathcal{O}/q\mathcal{O}\right)\right)\leq G_{\infty}$,
the level $q$ congruence subgroup of $G\left(\mathcal{O}\right)$,
is a lattice of $G_{\infty}$. Define the level $q$ automorphic symmetric
space of $G$, considered as a locally symmetric space covered by
$\tilde{X}$, to be
\[
X\left(q\right):=G\left(F\right)\backslash G\left(\mathbb{A}\right)/K\left(q\right)=\Gamma\left(q\right)\backslash\tilde{X}.
\]

\begin{rem}
\label{rem:str-app} If $G$ is a symplectic group then the strong
approximation property holds. If $G$ is a special orthogonal group
then the strong approximation property does not hold, in which case
we take $X\left(q\right)$ to be a finite disjoint union of locally
symmetric spaces of the form $\left(G\left(F\right)\bigcap gK_{f}\left(q\right)g^{-1}\right)\backslash\tilde{X}$,
where $g$ runs over a finite set of representatives of double cosets
in $G\left(F\right)\backslash G\left(\mathbb{A}\right)/K\left(q\right)$,
and the number of representatives grows like $\asymp1$ w.r.t. $q$
(see Lemma~\ref{lem:str-app-SO}). Since we are only interested in
proving asymptotic bounds w.r.t.\ $q$ we may assume that strong
approximation holds without loss of generality.
\end{rem}

Let $E\in\Pi^{\mathrm{alg}}\left(G\right)$, considered as a coefficient
system on $X\left(q\right)$, and denote the $L^{2}$-cohomology of
$X\left(q\right)$ with coefficients in $E$, by 
\[
H_{(2)}^{*}\left(X\left(q\right);E\right)=H_{(2)}^{*}\left(G\left(F\right)\backslash G\left(\mathbb{A}\right)/K\left(q\right);E\right).
\]
By Matsushima's decomposition~\eqref{eq:Matsushima}, we get 
\[
H_{(2)}^{*}\left(X\left(q\right);E\right)\cong\bigoplus_{\pi}m\left(\pi\right)\cdot\Biggl(H^{*}\left(\pi_{\infty};E\right)\otimes\bigotimes_{p\nmid q}\pi_{p}^{K_{p}}\otimes\bigotimes_{p\mid q}\pi_{p}^{K_{p}(\varpi_{p}^{\mathrm{ord}_{p}(q)})}\Biggr),
\]
where the direct sum runs over $\pi=\pi_{\infty}\otimes\bigotimes_{p\nmid\infty}\pi_{p}\in\Pi\left(G\left(\mathbb{A}\right)\right)$
and where we have used the explicit expressions for $K_{v}(q)$ above.

We note that if $G$ is uniform, i.e.\ $F\ne\mathbb{Q}$ is a totally
real number field and there exists a unique $v_{1}\mid\infty$ such
that $G\left(F_{v_{1}}\right)$ is non-compact and $G\left(F_{v}\right)$
is compact for any other $v_{1}\ne v\mid\infty$, then by Borel--Harish-Chandra
theory \cite[Theorem 5.5]{PR93}, $X\left(q\right)$ is a compact
manifold, in which case the $L^{2}$-cohomology coincides with the
usual Betti cohomology 
\[
H_{(2)}^{*}\left(X\left(q\right);E\right)\cong H^{*}\left(X\left(q\right);E\right).
\]

The case of $G$ definite will be discussed in Section~\ref{sec:Gross-forms}.

\subsection{Some asymptotic counts\protect\label{subsec:Asymptotic}}

As in the previous subsection, let $q$ be a principal ideal of $\mathcal{O}$,
which we identify with its generator $\prod_{v\mid q}\varpi_{v}^{\mathrm{ord}_{v}(q)}$.
Define the norm of $q$ to be $|q|=\prod_{v\mid q}p_{v}^{\mathrm{ord}_{v}(q)}$.
Denote by $\omega\left(q\right)$ the number of primes dividing $q$. 
\begin{notation}
We now slightly extend the asymptotic notations from the introduction
to include functions on principal ideals of $\mathcal{O}$. Let $f_{1}=f_{1}\left(q\right)$
and $f_{2}=f_{2}\left(q\right)$ be two positive real valued functions
on $q$. Denote $f_{1}\ll f_{2}$, if for any $\epsilon>0$, there
exists $C_{\epsilon}\in\mathbb{R}_{>0}$, such that $f_{1}(q)\leq C_{\epsilon}\cdot|q|^{\epsilon}\cdot f_{2}(q)$
for any $q$. Denote $f_{1}\asymp f_{2}$ if $f_{1}\ll f_{2}$ and
$f_{2}\ll f_{1}$. We again caution the reader that the $\epsilon$
that is written explicitly, for example in \cite{SX91}, is implicit
here in our notation $\ll$. 
\end{notation}

Let us record a few technical lemmas which we shall use freely later
in our paper. The proofs of the lemmas are classical and well known;
we record them here for completeness. 
\begin{lem}
\label{lem:prime-omega} There exists a constant $c=c_{F}>0$, such
that for any $q$, the number of primes $\omega(q)$ dividing $q$
satisfies
\begin{equation}
\omega\left(q\right)\leq\frac{c\log|q|}{\log\log|q|}.\label{eq:prime-omega}
\end{equation}
In particular, for any $\epsilon>0$, if $\log\log|q|\geq\frac{c}{\epsilon}$,
then $2^{\omega\left(q\right)}\leq|q|^{\epsilon}$.
\end{lem}

\begin{proof}
Let $\pi\left(x\right)$ be the number of primes $v$ of norm $p_{v}\leq x$,
and let $\omega\left(q;x\right)$ be the number of primes dividing
$q$ with $p_{v}>x$. Then clearly, $\omega\left(q\right)\leq\pi\left(x\right)+\omega\left(q;x\right)$
for any $x$. Note that $\omega\left(q;x\right)\leq\frac{\log|q|}{\log x}$,
since $x^{\omega\left(q;x\right)}\leq\prod_{v\mid q,\,p_{v}>x}p_{v}\leq\prod_{v\mid q}p_{v}\leq|q|$.
By the prime number theorem for number fields (or even a Chebychev
style upper bound), there exists $c'=c'_{F}>1$, such that $\pi\left(x\right)\leq c'\frac{x}{\log x}$
for any $x$. Taking $x=\log|q|$ and $c=c'+1$ we get the claim.
\end{proof}
\begin{lem}
\label{lem:prime-smallest} There exists a constant $d=d_{F}>0$,
such that for any $q$,
\[
\min_{u\nmid q}p_{u}\leq d\cdot\log|q|.
\]
In particular, for any $\epsilon>0$, there exists $c>0$, such that
if $u$ is the smallest finite place coprime to $q$, then $p_{u}^{2}\leq c|q|^{\epsilon}$.
\end{lem}

\begin{proof}
Let $u$ be the smallest finite place coprime to $q$. For $q=\prod_{v\mid q}\varpi_{v}^{n_{v}}$
denote $\tilde{q}=\prod_{v\mid q,v<u}\varpi_{v}$, where we ordered
the finite places according to their norms. Then $u$ is also the
smallest finite place coprime to $\tilde{q}$ and $|\tilde{q}|\leq|q|$,
hence it suffices to assume $q=\tilde{q}=\prod_{v<u}\varpi_{v}$ (note
that by definition of $u$ there is no place $w<u$ such that $w\nmid q$).
Let $\theta\left(x\right)=\sum_{p_{v}<x}\log p_{v}$ be the Chebychev
function of primes of $F$. By the prime number theorem for number
fields $\lim_{x\rightarrow\infty}\frac{\theta\left(x\right)}{x}=1$.
In particular, there exists $d>1$ such that $\theta\left(x\right)\geq d^{-1}x$,
for any $x$. Note that $\log|q|=\theta\left(p_{u}\right)$ and therefore
$d^{-1}p_{u}\leq\log|q|$, which proves the claim.
\end{proof}
\begin{lem}
\label{lem:asymptotic} Let $f\,:\,\mathcal{O}\rightarrow\mathbb{R}_{>0}$
a multiplicative function, i.e. $f\left(\prod_{v}\varpi_{v}^{n_{v}}\right)=\prod_{v}f\left(\varpi_{v}^{n_{v}}\right)$.
Let $d>0$, and assume that for any $\epsilon>0$, there exists $c_{\epsilon}>0$,
such that $f\left(\varpi_{v}^{n}\right)\leq c_{\epsilon}p_{v}^{\left(d+\epsilon\right)n}$,
for any prime $v$ and any $n\in\mathbb{N}_{0}$. Then $f\left(q\right)\ll|q|^{d}$.
\end{lem}

\begin{proof}
By the assumptions we get that
\[
f\left(q\right)=\prod_{v\mid q}f\left(\varpi_{v}^{n_{v}}\right)\leq\prod_{v\mid q}c_{\epsilon}p_{v}^{\left(d+\epsilon\right)n_{v}}=c_{\epsilon}^{\omega\left(q\right)}|q|^{d+\epsilon}.
\]
By Lemma~\ref{lem:prime-omega}, for $|q|$ large enough we get that
$c_{\epsilon}^{\omega\left(q\right)}\leq|q|^{\epsilon}$, which completes
the proof.
\end{proof}
The following lemma will be used repeatedly in our paper.
\begin{lem}
\label{lem:coh-vol-dim} Fix $E\in\Pi^{\mathrm{alg}}\left(G\right)$
and let $|q|\rightarrow\infty$. Then 
\begin{equation}
\dim H_{(2)}^{*}\left(X\left(q\right);E\right)\asymp\vol\left(X\left(q\right)\right)\asymp\left[K_{f}\left(1\right):K_{f}\left(q\right)\right]\asymp|G\left(\mathcal{O}/q\mathcal{O}\right)|\asymp|q|^{\dim G}.\label{eq:coh-vol-dim}
\end{equation}
\end{lem}

\begin{proof}
The first estimate $\dim H_{(2)}^{*}\left(X\left(q\right);E\right)\asymp\vol\left(X\left(q\right)\right)$
is well known. The second estimate $\vol\left(X\left(q\right)\right)\asymp\left[K_{f}\left(1\right):K_{f}\left(q\right)\right]$,
follows from the fact that $X\left(q\right)$ is a $K\left(1\right)/K\left(q\right)\cong K_{f}\left(1\right)/K_{f}\left(q\right)$
cover of $X\left(1\right)$. The third estimate $\left[K_{f}\left(1\right):K_{f}\left(q\right)\right]\asymp|G\left(\mathcal{O}/q\mathcal{O}\right)|$
follows from the strong approximation property of $G^{sc}$, the simply
connected cover of the derived subgroup of $G$, i.e.\ $G^{sc}=\mathrm{Spin}_{2n+1}$
if $G=SO_{2n+1}$ and $G^{sc}=G$ if $G=Sp_{2n}$, for which $K^{sc}\left(1\right)/K^{sc}\left(q\right)\cong G^{sc}\left(\mathcal{O}/q\mathcal{O}\right)$.
The last estimate $|G\left(\mathcal{O}/q\mathcal{O}\right)|\asymp|q|^{\dim G}$
is well known for any split group.
\end{proof}
Lastly, we record the following Lemma which is needed in the case
where the strong approximation property fails. 
\begin{lem}
\label{lem:str-app-SO} Let $G$ be a special orthogonal group, and
let $v=\infty$ if $G_{\infty}$ is non-compact and $v=\mathfrak{p}$
if $G_{\infty}$ is compact and $G_{\mathfrak{p}}$ is not. Then 
\[
h:=\left|G\left(F\right)\backslash G\left(\mathbb{A}\right)/G_{v}K\left(q\right)\right|\asymp1.
\]
Furthermore, if $g_{1},\ldots,g_{h}$ are representatives of $G\left(F\right)\backslash G\left(\mathbb{A}\right)/G_{v}K\left(q\right)$,
denote $\Gamma^{i}\left(q\right)=G\left(F\right)\bigcap g_{i}K^{v}\left(q\right)g_{i}^{-1}\leq G_{v}$
and $X^{i}\left(q\right)=\Gamma^{i}\left(q\right)\backslash\tilde{X}$.
Then for any $i=1,\ldots,h$,
\[
\vol\left(X^{i}\left(q\right)\right)\asymp\vol\left(X\left(q\right)\right).
\]
\end{lem}

\begin{proof}
We shall prove the claim for $v=\infty$, the proof for the $p$-adic
case is similar. By \cite[Page 189]{Kne65}, $\overline{G\left(F\right)\cdot G_{\infty}}$,
and hence $G_{0}=G\left(F\right)\cdot G_{\infty}\cdot K\left(q\right)$,
contains the derived subgroup $G\left(\mathbb{A}\right)^{\mathrm{der}}$
of $G\left(\mathbb{A}\right)$, and therefore $G_{0}$ is a subgroup
of $G\left(\mathbb{A}\right)$ of index $h=\left[G\left(\mathbb{A}\right):G_{0}\right]$.
For any finite place $\ell$, the derived subgroup $G_{\ell}^{\mathrm{der}}$
of $G_{\ell}$ is the kernel of the spinor norm $N\colon G\to\left\{ \pm1\right\} $,
hence $\left[G_{\ell}:G_{\ell}^{\mathrm{der}}\right]=2$, and for
all but finitely many places $\ell$, the spinor norm is non-trivial
on $K_{\ell}$, hence $\left[G_{\ell}:G_{\ell}^{\mathrm{der}}K_{\ell}\right]=1$.
Combining the previous claims together with Lemma~\ref{lem:prime-omega},
we get
\[
h\leq\left[G\left(\mathbb{A}\right):G\left(\mathbb{A}\right)^{\mathrm{der}}G_{\infty}K\left(q\right)\right]=\prod_{\ell}\left[G_{\ell}:G_{\ell}^{\mathrm{der}}\cdot K_{\ell}\left(q\right)\right]\asymp\prod_{\ell\mid q}\left[G_{\ell}:G_{\ell}^{\mathrm{der}}\right]=2^{\omega\left(q\right)}\asymp1.
\]
Since $X\left(q\right)=\bigsqcup_{i=1}^{h}X^{i}\left(q\right)$ we
get that $\vol\left(X\left(q\right)\right)=\sum_{i=1}^{h}\vol\left(X^{i}\left(q\right)\right)\geq\vol\left(X^{i}\left(q\right)\right)$.
On the other hand, recalling that $G_{0}=G\left(F\right)G_{\infty}K\left(q\right)$
contains the derived subgroup of $G\left(\mathbb{A}\right)$, we have
\[
X^{i}\left(q\right)\cong\Gamma^{i}\left(q\right)\backslash G_{\infty}/K_{\infty}\cong G\left(F\right)\backslash G\left(F\right)gG_{\infty}K_{f}\left(q\right)/K\left(q\right)\cong G\left(F\right)\backslash gG_{0}/K\left(q\right).
\]
Denote $K_{0}:=K_{f}\left(1\right)\cap G_{0}$. Note that $X^{i}\left(q\right)$
is a $K_{0}/K_{f}\left(q\right)$ cover of $\left(G\left(F\right)\cap K_{0}\right)\backslash\tilde{X}$,
hence $\vol\left(X^{i}\left(q\right)\right)\asymp\left[K_{0}:K_{f}\left(q\right)\right]$.
By Lemma~\ref{lem:coh-vol-dim}, we have $\vol\left(X\left(q\right)\right)\asymp\left[K_{f}\left(1\right):K_{f}\left(q\right)\right]$,
hence 
\[
\frac{\vol\left(X^{i}\left(q\right)\right)}{\vol\left(X\left(q\right)\right)}\asymp\frac{\left[K_{0}:K_{f}\left(q\right)\right]}{\left[K_{f}\left(1\right):K_{f}\left(q\right)\right]}=\left[K_{f}\left(1\right):K_{0}\right]^{-1}=\prod_{\ell}\left[K_{\ell}:K_{0,\ell}\right]^{-1}.
\]
Since $G_{\ell}^{\mathrm{der}}=\ker N$, and $G_{\ell}^{\mathrm{der}}\subseteq G_{0}$,
we get that $\left[K_{\ell}:K_{0,\ell}\right]\in\left\{ 1,2\right\} $,
and for $K_{0,\ell}=K_{\ell}$ for all $\ell\nmid q$ except a finite
set of ramified places (which depends only on $G$ and not on $q$),
which gives $\prod_{\ell}\left[K_{\ell}:K_{0,\ell}\right]^{-1}\asymp2^{-\omega\left(q\right)}\asymp1$,
as needed. 
\end{proof}

\section{Arthur's endoscopic classification \protect\label{sec:Arthur-classification}}

In this Section we describe Arthur's Endoscopic Classification of
Representations of split classical groups over local and global fields,
and its extension to Gross inner forms by Taïbi. Throughout this section
$G$ is a split odd special orthogonal or symplectic group defined
over a global or local field $F$ of characteristic $0$.

\subsection{Local Arthur classification\protect\label{subsec:Local-Arthur-classification}}

Let $F$ be a local field of characteristic $0$. In \cite{Art13},
Arthur associates to each local $A$-parameter $\psi$ the following
data: 
\begin{itemize}
\item A finite local $A$-packet $\Pi_{\psi}\subset\Pi\left(G\left(F\right)\right)$
(multiplicity-free by Mœglin's work \cite{Moe11}, see also Theorem~\ref{thm:Moeglin-A-size}
below). 
\item A finite abelian $2$-group $\mathcal{S}_{\psi}:=S_{\psi}/S_{\psi}^{0}\hat{Z}$,
where $S_{\psi}:=C_{\hat{G}}\left(\psi\right)$ is the centralizer
of the image $\psi$ in $\hat{G}$, $S_{\psi}^{0}$ is its identity
component, and $\hat{Z}$ is the center of $\hat{G}$.
\item A sign character $\epsilon_{\psi}\in\widehat{\mathcal{S}_{\psi}}$,
where $\widehat{\mathcal{S}_{\psi}}:=\mbox{Hom}\left(\mathcal{S}_{\psi},\left\{ \pm1\right\} \right)$
is the dual group of $\mathcal{S}_{\psi}$.
\item A map $\langle,\rangle_{\psi}\,:\,\Pi_{\psi}\rightarrow\widehat{\mathcal{S}_{\psi}}$,
denoted by $\pi\mapsto\langle\cdot,\pi\rangle_{\psi}$.
\item If $F$ is $p$-adic and $\psi\in\Psi^{\mathrm{ur}}\left(G\right)$
then $\pi_{\psi}\in\Pi_{\psi}\bigcap\Pi^{\mathrm{ur}}\left(G\left(F\right)\right)$,
as defined in \eqref{eq:ur-=00005Cpi-=00005Cpsi}, and it satisfies
$\langle\cdot,\pi_{\psi}\rangle_{\psi}\equiv1$.
\end{itemize}
The construction of the $A$-packet $\Pi_{\psi}$, sign character
$\epsilon_{\psi}\in\widehat{\mathcal{S}_{\psi}}$ and map $\langle,\rangle_{\psi}\,:\,\Pi_{\psi}\rightarrow\widehat{\mathcal{S}_{\psi}}$,
are characterized canonically by \cite[Theorem 2.2.1]{Art13}, using
endoscopic and twisted endoscopic character relations.

The following theorem constitutes the local Langlands correspondence
for (quasi-)split classical groups as proved by Arthur \cite{Art13}. 
\begin{thm}
\cite[Theorem 1.5.1]{Art13} \label{thm:Arthur=0000201.5.1} If $\psi=\phi\in\Phi^{\mathrm{temp}}\left(G\left(F\right)\right)$,
then $\Pi_{\phi}\subset\text{\ensuremath{\Pi}}^{\mathrm{temp}}\left(G\left(F\right)\right)$,
$\langle,\rangle_{\phi}\,:\,\Pi_{\phi}\rightarrow\widehat{\mathcal{S}_{\phi}}$
is a bijection if $F$ is non-archimedean and an injection otherwise,
and $\text{\ensuremath{\Pi}}^{\mathrm{temp}}\left(G\left(F\right)\right)$
is the disjoint union of the tempered $L$-packets $\Pi_{\phi}$,
for $\phi\in\Phi^{\mathrm{temp}}\left(G\left(F\right)\right)$.
\end{thm}

Theorem~\ref{thm:Arthur=0000201.5.1} gives a correspondence between
tempered representations and tempered $L$-parameters of $G\left(F\right)$,
which we call the tempered LLC. Then by using the Langlands classification
via its quotient theorem \cite{Lan89,Kon03,Sil78}, and the analogous
classification for $L$-parameters due to Silberger and Zink \cite{SZ18}
(see also \cite{ABPS14}), one gets the full LLC. In particular, for
any $L$-parameter $\phi\in\Phi\left(G\left(F\right)\right)$ there
is a unique corresponding $L$-packet $\Pi_{\phi}\subset\Pi\left(G\left(F\right)\right)$.
\begin{prop}
\cite[Proposition 7.4.1]{Art13} \label{prop:Arthur=0000207.4.1}
If $\psi\in\Psi\left(G\left(F\right)\right)$ and $\phi_{\psi}\in\Phi\left(G\left(F\right)\right)$
its corresponding $L$-parameter, then $\Pi_{\phi_{\psi}}\subset\Pi_{\psi}$.
\end{prop}

The structure of the $L$-packet $\Pi_{\phi_{\psi}}$ was described
by Shahidi in \cite[Section 3]{Sha11}, relying on the tempered LLC. 
\begin{prop}
\cite[Proposition 3.32]{Sha11} \label{prop:Shahidi} For $\psi\in\Psi\left(G\left(F\right)\right)$,
let $\phi_{+}\in\Phi^{\mathrm{ur}}\left(G\left(F\right)\right)$,
$\phi_{+}\left(w\right)=\psi|_{SL_{2}^{A}}\left(\mathrm{diag}\left(|w|^{1/2},|w|^{-1/2}\right)\right)$,
let $\nu\in X^{\mathrm{ur}}\left(T\right)$ be the unramified character
associated to $\phi_{+}$ by \eqref{eq:ULLC} and \eqref{eq:BCMS},
let $M\leq G$ be the standard Levi subgroup whose dual is $\widehat{M}=C_{\hat{G}}\left(\phi_{+}\right)$,
let $P$ be the standard parabolic subgroup with Levi subgroup $M$,
let $\phi\in\Phi^{\mathrm{temp}}\left(M\left(F\right)\right)$, $\phi=\psi|_{L_{F}}$,
and let $\Pi_{\phi}^{M}\subset\text{\ensuremath{\Pi}}^{\mathrm{temp}}\left(M\left(F\right)\right)$
be the tempered $L$-packet of $\phi$ in $M\left(F\right)$, and
let $j\left(P,\tau,\nu\right)$ be the Langlands quotient of the (normalized)
parabolic induction $\mathrm{Ind}_{P}^{G}\left(\tau\otimes\nu\right)$,
for any $\tau\in\Pi_{\phi}^{M}$. Then 
\[
\Pi_{\phi_{\psi}}=\left\{ j\left(P,\tau,\nu\right)\,:\,\tau\in\Pi_{\phi}^{M}\right\} .
\]
Note that $P$ and $\nu$ depend only on $\psi|_{SL_{2}^{A}}$, the
Arthur $SL_{2}$-type of $\psi$. 
\end{prop}

It remains to determine the complement of the $L$-packet $\Pi_{\phi_{\psi}}$
inside the $A$-packet $\Pi_{\psi}$. For this, we turn to Moeglin's
explicit construction of the local $A$-packets \cite{Moe06,Moe06-2,Moe09,Moe09-2,Moe10,Moe11}.
Here are two useful consequences of Moeglin's explicit construction. 
\begin{thm}
\cite{Moe11} \label{thm:Moeglin-A-size} If $\psi\in\Psi\left(G\left(F\right)\right)$,
then $\Pi_{\psi}$ is multiplicity free and $|\Pi_{\psi}|\leq C$,
where $C$ depends only on $\mathrm{rank}\left(G\right)$.
\end{thm}

\begin{prop}
\cite[Propositions 4.4 and 6.4]{Moe09} \label{prop:Moeglin-A-ur}
If $\psi\in\Psi\left(G\left(F\right)\right)$, then $\Pi_{\psi}\bigcap\Pi^{\mathrm{ur}}\left(G\right)=\Pi_{\phi_{\psi}}\bigcap\Pi^{\mathrm{ur}}\left(G\right)$.
In particular, if $\psi\not\in\Psi^{\mathrm{ur}}\left(G\right)$ then
$\Pi_{\psi}\bigcap\Pi^{\mathrm{ur}}\left(G\right)=\emptyset$, and
if $\psi\in\Psi^{\mathrm{ur}}\left(G\right)$ then $\Pi_{\psi}\bigcap\Pi^{\mathrm{ur}}\left(G\right)=\left\{ \pi_{\psi}\right\} $,
where $\pi_{\psi}=\pi_{\phi_{\psi}}$ is as defined in \eqref{eq:ur-=00005Cpi-=00005Cpsi}.
\end{prop}

See Subsection~\ref{subsec:Schmidt-A-packet} for more information
on the local $A$-packets of the group $SO_{5}\left(F\right)$.

\subsection{Global Arthur classification\protect\label{subsec:Global-Arthur-classification}}

Let $F$ be a global field of characteristic $0$. In \cite{Art13},
Arthur defined the notion of a global $A$-parameters of quasi-split
classical groups, which we now briefly review (see \cite{Art13} for
the precise definitions). 

First let us consider the case of general linear groups $GL_{N}$.
Denote by $\Psi_{c}\left(GL_{n}\right)$ the set of cuspidal automorphic
representations $\pi$ of $GL_{n}$ (with unitary central character).
For any $m\in\mathbb{N}$, denote by $\nu(m)$ the unique $m$-dimensional
irreducible representation of $SL_{2}^{A}=SL_{2}\left(\mathbb{C}\right)$.
Define the set of global $A$-parameters of $GL_{N}$ to be 
\[
\Psi\left(GL_{N}\right):=\left\{ \psi_{N}=\boxplus_{i=1}^{r}\left(\mu_{i}\boxtimes\nu(m_{i})\right)\,:\,\sum_{i=1}^{r}n_{i}m_{i}=N,\;\mu_{i}\in\Psi_{c}\left(GL_{n_{i}}\right)\right\} .
\]
where $\psi_{N}=\boxplus_{i=1}^{r}\left(\mu_{i}\boxtimes\nu(m_{i})\right)$
is an unordered formal sum of ordered formal products. For any $\mu\in\Psi_{c}\left(GL_{n}\right)$,
denote its dual by $\mu^{\vee}:=\mu\circ\theta$, where $\theta(g)=\,^{t}g^{-1}$.
For any $\mu\in\Psi_{c}\left(GL_{n}\right)$, such that $\mu^{\vee}\cong\mu$,
in \cite[Theorem 1.4.1]{Art13} Arthur associates a simple quasi-split
classical group $G_{\mu}$ whose dual group has standard representation
of dimension $n_{i}$. Denote by $\mathcal{L}_{\mu}$ the $L$-group
of $G_{\mu}$ and by $\tilde{\mu}\,:\,\mathcal{L}_{\mu}\rightarrow GL_{n_{i}}\left(\mathbb{C}\right)$
its standard representation. For any $\psi_{N}=\boxplus_{i=1}^{r}\left(\mu_{i}\boxtimes\nu(m_{i})\right)\in\Psi\left(GL_{N}\right)$,
such that $\mu_{i}^{\vee}\cong\mu_{i}$ for any $i$, in \cite[(1.4.4)]{Art13}
Arthur defines $\mathcal{L}_{\psi_{N}}$ to be the fiber product of
the $\mathcal{L}_{\mu_{i}}$ over the Galois group of $F$, and in
\cite[(1.4.5)]{Art13} he defines the $L$-homomorphism 
\[
\tilde{\psi}_{N}\,:\,\mathcal{L}_{\psi_{N}}\times SL_{2}^{A}\rightarrow GL_{N}\left(\mathbb{C}\right),\qquad\tilde{\psi}_{N}=\bigoplus_{i}\left(\tilde{\mu}_{i}\otimes\nu(m_{i})\right).
\]

Consider now the case of split classical groups $G$. Let $\mathrm{Std}_{\hat{G}}\,:\,\hat{G}\rightarrow GL_{N}\left(\mathbb{C}\right)$
be the standard representation of $\hat{G}$. Define the set of elliptic
or discrete $A$-parameters of $G$ to be 
\[
\Psi_{2}\left(G\right):=\left\{ \psi=(\psi_{N},\tilde{\psi})\,:\,\begin{array}{l}
\psi_{N}=\boxplus_{i=1}^{r}\left(\mu_{i}\boxtimes\nu(m_{i})\right)\in\Psi\left(GL_{N}\right),\\
\mu_{i}^{\vee}\cong\mu_{i},\quad\left(\mu_{i},m_{i}\right)\ne\left(\mu_{j},m_{j}\right),\;\forall i\ne j\\
\tilde{\psi}\,:\,\mathcal{L}_{\psi_{N}}\times SL_{2}^{A}\rightarrow\hat{G},\quad\mathrm{Std}_{\hat{G}}\circ\tilde{\psi}=\tilde{\psi}_{N}
\end{array}\right\} ,
\]
where the $L$-homomorphism $\tilde{\psi}\,:\,\mathcal{L}_{\psi_{N}}\times SL_{2}^{A}\rightarrow\hat{G}$
is defined up to conjugation by $\hat{G}$. 
\begin{rem}
\label{rem:Arthur-fiber} By \cite[pp. 31-32]{Art13}, the projection
$\Psi_{2}\left(G\right)\rightarrow\Psi\left(GL_{N}\right)$, $\psi\mapsto\psi_{N}$,
is injective for $G=Sp_{2n}$ or $SO_{2n+1}$, and it has a fiber
of size at most $2$ for $G=SO_{2n}$.
\end{rem}

Let $\mathcal{A}\left(GL_{N}\right)$ be the set of automorphic representations
of $GL_{N}$. By the theory of Eisenstein series and their residues
\cite{Lan76,MW89}, there is a natural bijection $\Psi\left(GL_{N}\right)\rightarrow\mathcal{A}\left(GL_{N}\right)$,
$\psi_{N}\mapsto\pi_{\psi_{N}}$. For any place $v$, there is a localization
map $\mathcal{A}\left(GL_{N}\right)\rightarrow\Pi\left(GL_{N}\left(F_{v}\right)\right)$,
$\pi\mapsto\pi_{v}$. By the local Langlands correspondence for $GL_{N}$
proved in \cite{Lan89,Hen00,HT01}, there is a natural bijection $\Pi\left(GL_{N}\left(F_{v}\right)\right)\rightarrow\Phi\left(GL_{N}\left(F_{v}\right)\right)$,
$\pi_{v}\mapsto\phi^{\pi_{v}}$, which among other properties preserves
temperedness. Denote the composition of these three maps by 
\begin{equation}
\Psi\left(GL_{N}\right)\rightarrow\Phi\left(GL_{N}\left(F_{v}\right)\right),\qquad\psi_{N}\mapsto\psi_{N,v}:=\phi^{\pi_{\psi_{N},v}}.\label{eq:LLC-GLN}
\end{equation}

In this paper we consider only cohomological automorphic representations.
Thanks to recent joint works by many experts in the Langlands program,
the Generalized Ramanujan Conjecture holds in this case (see Theorem~\ref{thm:GRPC}
and Proposition~\ref{prop:GRPC-classical} below). This allows us
to avoid the introduction of the supersets $\Psi^{+}\left(G_{v}\right)\supset\Psi\left(G_{v}\right)$.
Denote by $\Psi_{2}^{\mathrm{AJ}}\left(G\right)$ the subset of cohomological
discrete $A$-parameters $\psi$ of $G$, i.e. for any $v\mid\infty$
there exists $E_{v}\in\Pi^{\mathrm{alg}}\left(GL_{N}\right)$ such
that $\psi_{N,v}\in\Psi^{\mathrm{AJ}}\left(GL_{N}\left(F_{v}\right);E_{v}\right)$.

To any cohomological discrete (\emph{global}) $A$-parameter $\psi=(\psi_{N},\tilde{\psi})\in\Psi_{2}^{\mathrm{AJ}}\left(G\right)$,
in \cite{Art13} Arthur associates the following datum: 
\begin{itemize}
\item A local $A$-parameter $\psi_{v}\in\Psi\left(G_{v}\right)$ for any
$v$, defined in \cite[Theorem 1.4.2, (1.4.14), (1.4.16)]{Art13},
satisfying $\psi_{N,v}=\mathrm{Std}_{\hat{G}}\circ\psi_{v}$ as local
$L$-parameters of $GL_{N}\left(F_{v}\right)$. 
\item A global $A$-packet $\Pi_{\psi}=\left\{ \pi=\otimes_{v}\pi_{v}\in\prod_{v}\Pi_{\psi_{v}}\;:\;\pi_{v}=\pi_{\psi_{v}}\mbox{ for almost all }v\right\} $,
where $\Pi_{\psi_{v}}\subset\Pi\left(G\left(F_{v}\right)\right)$
is the local $A$-packet associated to the local $A$-parameter $\psi_{v}$.
\item A finite abelian $2$-group $\mathcal{S}_{\psi}=S_{\psi}/S_{\psi}^{0}\hat{Z}$,
where $S_{\psi}=C_{\hat{G}}(\tilde{\psi})$ is the centralizer of
the image of $\tilde{\psi}$ in $\hat{G}$, $S_{\psi}^{0}$ is its
identity component, and $\hat{Z}$ is the center of $\hat{G}$. By
\cite[(1.4.16)]{Art13}, for any $v$, then $S_{\psi}\subset S_{\psi_{v}}$
which induces $\mathcal{S}_{\psi}\to\mathcal{S}_{\psi_{v}}$.
\item A global sign character $\epsilon_{\psi}\in\widehat{\mathcal{S}_{\psi}}:=\mbox{Hom}\left(\mathcal{S}_{\psi},\left\{ \pm1\right\} \right)$,
defined in \cite[(1.5.6), (1.5.7)]{Art13}. 
\item A map $\langle,\rangle_{\psi}\,:\,\Pi_{\psi}\rightarrow\widehat{\mathcal{S}_{\psi}}$,
$\pi\mapsto\langle\cdot,\pi\rangle_{\psi}:=\prod_{v}\langle\cdot,\pi_{v}\rangle_{\psi_{v}}$.
Note that $\pi_{v}=\pi_{\psi_{v}}$ and $\langle\cdot,\pi_{\psi_{v}}\rangle_{\psi_{v}}\equiv1$,
for almost all $v$. 
\end{itemize}
(Analogous notions were defined also for non-cohomological global
$A$-parameters $\psi\in\Psi_{2}\left(G\right)$ see \cite[Section 1]{Art13}
for the precise definition of $\psi_{v}$ and $\Pi_{\psi_{v}}$, for
any $v$). Given a discrete $A$-parameter $\psi\in\Psi_{2}\left(G\right)$,
consider the following subset of the global $A$-packet 
\[
\Pi_{\psi}(\epsilon_{\psi}):=\left\{ \pi\in\Pi_{\psi}\;:\;\langle\cdot,\pi\rangle_{\psi}=\epsilon_{\psi}\right\} .
\]

The following decomposition of the discrete automorphic spectrum,
called the Endoscopic Classification of Representations (ECR), for
all quasi-split classical groups $G$, was proved by Arthur in \cite{Art13}.
\begin{thm}
\cite[Theorem 1.5.2]{Art13} \label{thm:Arthur=0000201.5.2} Let $G$
be a split odd special orthogonal or symplectic group over a number
field. Then the discrete automorphic spectrum of $G$ decomposes as
follows
\[
L_{\mathrm{disc}}^{2}\left(G\left(F\right)\backslash G\left(\mathbb{A}\right)\right)\cong\bigoplus_{\psi\in\Psi_{2}\left(G\right)}\bigoplus_{\pi\in\Pi_{\psi}(\epsilon_{\psi})}\pi.
\]
In particular, for any $\pi\in\Pi\left(G\left(\mathbb{A}\right)\right)$,
then $m\left(\pi\right)=1$ if $\pi\in\Pi_{\psi}(\epsilon_{\psi})$
for a unique $\psi\in\Psi_{2}\left(G\right)$, and $m\left(\pi\right)=0$
otherwise.
\end{thm}

The above ECR was extended by Taïbi in \cite{Tai18} for certain inner
forms of quasi-split classical groups, which includes the Gross inner
forms described in Subsection~\ref{subsec:Gross-forms}, and only
for the cohomological discrete automorphic spectrum. 

Let $G$ be a Gross inner form of the split classical group $G^{*}$,
defined over a totally real number field $F$. For any $v\nmid\infty$,
since $G\left(F_{v}\right)\cong G^{*}\left(F_{v}\right)$, then the
local $A$-parameters and $A$-packets of $G\left(F_{v}\right)$ are
by definition those of the split group $G^{*}\left(F_{v}\right)$.
For any $v\mid\infty$, then $F_{v}=\mathbb{R}$, the definition of
local $A$-parameters of $G\left(F_{v}\right)$ is as before, and
for any $E\in\Pi^{\mathrm{alg}}\left(G\right)$, the local $A$-packets
of cohomological $A$-parameters $\psi\in\Psi^{\mathrm{AJ}}\left(G\left(F_{v}\right);E\right)$,
was defined in Subsection~\ref{subsec:Cohomological-rep-par} to
be the $AJ$-packet. Define the set of cohomological\emph{ }discrete
$A$-parameters of $G$ to be
\[
\Psi_{2}^{\mathrm{AJ}}\left(G\right):=\left\{ \psi\in\Psi_{2}\left(G^{*}\right)\,:\,\forall v\mid\infty,\quad\psi_{v}\in\bigcup_{E\in\Pi^{\mathrm{alg}}\left(G\right)}\Psi^{\mathrm{AJ}}\left(G\left(F_{v}\right);E\right)\right\} .
\]

\begin{thm}
\cite{Tai18} \label{thm:Arthur-Taibi} Let $G$ be a Gross inner
form of a split odd special orthogonal or symplectic group over a
totally real number field. Then the cohomological discrete automorphic
spectrum of $G$ decomposes as follows
\[
L_{\mathrm{disc}}^{2}\left(G\left(F\right)\backslash G\left(\mathbb{A}\right)\right)^{\mathrm{coh}}\cong\bigoplus_{\psi\in\Psi_{2}^{\mathrm{AJ}}\left(G\right)}\bigoplus_{\pi\in\Pi_{\psi}(\epsilon_{\psi})}\pi.
\]
In particular, for any $\pi\in\Pi^{\mathrm{coh}}\left(G\left(\mathbb{A}\right)\right)$,
then $m\left(\pi\right)=1$ if $\pi\in\Pi_{\psi}(\epsilon_{\psi})$
for a unique $\psi\in\Psi_{2}^{\mathrm{AJ}}\left(G\right)$, and $m\left(\pi\right)=0$
otherwise.
\end{thm}

\subsection{Arthur shapes and $SL_{2}$-types\protect\label{subsec:Arthur-shapes-types}}

We now define the notion of $A$-shapes which was introduced by Marshall
and Shin in \cite{MS19}.
\begin{defn}
For each discrete $A$-parameter $\psi\in\Psi_{2}\left(G\right)$
there is an associated $A$-shape which is a list of 2-tuples $\varsigma(\psi)=\left(\left(n_{i},m_{i}\right)_{i}\right)$
with integers $m_{i}$ and $n_{i}$ given by the associated $\psi_{N}=\boxplus_{i}\left(\mu_{i}\boxtimes\nu(m_{i})\right)$.
Denote the set of $A$-shapes of $G$ by 
\[
\mathcal{M}\left(G\right)=\left\{ \varsigma(\psi)\,:\,\psi\in\Psi_{2}\left(G\right)\right\} ,
\]
and for any $A$-shape $\varsigma\in\mathcal{M}\left(G\right)$, denote
the set of all discrete $A$-parameters with $A$-shape $\varsigma$
by 
\[
\Psi_{2}\left(G,\varsigma\right)=\left\{ \psi\in\Psi_{2}\left(G\right)\,:\,\varsigma(\psi)=\varsigma\right\} .
\]
\end{defn}

One can also associate to each global and local $A$-parameter as
well as to each $A$-shape their Arthur $SL_{2}$ type. 
\begin{defn}
Define the set of all complex homomorphisms from $SL_{2}^{A}=SL_{2}\left(\mathbb{C}\right)$
to $\hat{G}$, up to conjugation by $\hat{G}$, by 
\[
\mathcal{D}\left(G\right)=\mbox{Hom}\left(SL_{2}^{A},\hat{G}\right)/\hat{G}.
\]
Denote by $\sigma_{\mathrm{triv}}\in\mathcal{D}\left(G\right)$ the
trivial $SL_{2}$ of $\hat{G}$ and by $\sigma_{\mathrm{princ}}\in\mathcal{D}\left(G\right)$
the principal $SL_{2}$ of $\hat{G}$.

For a global $A$-parameter $\psi\in\Psi_{2}\left(G\right)$ (resp.
a local $A$-parameter $\psi_{v}\in\Psi\left(G_{v}\right)$), define
its Arthur $SL_{2}$-type to be $\tilde{\psi}|_{SL_{2}^{A}}\in\mathcal{D}\left(G\right)$
(resp. $\psi_{v}|_{SL_{2}^{A}}\in\mathcal{D}\left(G\right)$ ). For
an $A$-shape $\varsigma=\left(\left(n_{i},m_{i}\right)_{i}\right)\in\mathcal{M}\left(G\right)$,
denote its Arthur $SL_{2}$-type to be $\sigma_{\varsigma}=\bigoplus_{i=1}^{r}\nu(m_{i})^{n_{i}}$.
Note that an $A$-shape carries slightly more information than the
Arthur $SL_{2}$-type. 
\end{defn}

By theorems of Jacobson--Morozov and Kostant (see Theorems~3.3.1
and~3.4.10 of \cite{CM17}) the set $\mathcal{D}\left(G\right)$
is isomorphic to the set of complex nilpotent orbits in the Lie algebra
of $\hat{G}$. For a classical group the set of nilpotent orbits is
in bijection with a certain subset of integer partitions of $N$ (where
$N$ is the dimension of the standard representation of $\hat{G}$
as above) determined by the Cartan type. In fact, the partition corresponding
to an Arthur $SL_{2}$-type $\sigma_{\varsigma}=\bigoplus_{i=1}^{r}\nu(m_{i})^{n_{i}}$
is $(m_{1}^{n_{1}},\ldots,m_{r}^{n_{r}})$ up to reordering the components
and merging identical bases $m_{i}$. More generally, for any simple
group, nilpotent orbits can be parametrized by certain assignments
of integer weights in $\{0,1,2\}$ to each simple root $\alpha_{i}^{\vee}$
of $\hat{G}$, usually presented as a weighted Dynkin diagram where
each node is labeled by the associated weight. The integer weights
$(w_{\sigma})_{i}$ associated to some $\sigma\in\mathcal{D}\left(G\right)$
are given by the pairing of $\alpha_{i}$ with an appropriately chosen
cocharacter in the conjugacy class $x\mapsto\sigma(\mathrm{diag}(x,x^{-1}))$
as described in Section~3.5 of \cite{CM17}. In particular, the trivial
orbit has weights $(0,\ldots,0)$ and the principal (or regular) orbit
has weights $(2,\ldots,2)$. 

There is a partial ordering of orbits based on inclusion after the
Zariski closure operation. For classical groups this corresponds to
the standard partial ordering for partitions. There is a unique maximal
orbit with respect to this partial order and this is the principal
orbit. There is also a unique, next-to-maximal orbit which is called
the subregular orbit, as well as a unique minimal orbit (besides the
trivial orbit). The orbits for $\hat{G}=Sp_{4}\left(\mathbb{C}\right)$,
which will appear frequently in this paper, are precisely these four
orbits and their different parametrization are shown in Table~\ref{tab:sp4-orbits}.

Note that by Proposition \ref{prop:AJ}, the Arthur $SL_{2}$-type
of a cohomological $A$-parameter is principal in some Levi subgroup
of $\hat{G}$. Call such Arthur $SL_{2}$-types or the corresponding
nilpotent orbits \emph{Principal in a Levi}, or PL for short. For
classical groups, the PL nilpotent orbits were characterized in \cite[Proposition 6.2.2]{GS15}
and the conjugacy classes of Levi subgroups where these are principal
were constructed in the proof. For $\hat{G}=Sp_{4}\left(\mathbb{C}\right)$
all nilpotent orbits are PL and we present corresponding standard
Levi subgroups $\hat{G}$, $\hat{S}$, $\hat{M}$ and $\hat{T}$ in
Table~\ref{tab:sp4-orbits} up to conjugacy. For higher rank groups,
not all orbits are PL; for example for $\hat{G}=Sp_{6}\left(\mathbb{C}\right)$
the orbit $(4,2)$ corresponding to the Dynkin weight $(2,0,2)$ is
not PL. In this paper we are only dealing with cohomological representations
and $A$-parameters, hence we shall only concern ourselves with PL
Arthur $SL_{2}$-types. 

\begin{table}[tbh]
\caption{\protect\label{tab:sp4-orbits}Complex nilpotent orbits for $\hat{G}=Sp_{4}\left(\mathbb{C}\right)$
written as partitions and weighted Dynkin diagrams. The conjugacy
classes of Levi subgroups on which the orbits are principal are also
shown with the common names of the associated standard parabolic subgroups
in parenthesis. }

\centering{}%
\begin{tabular}{lllcll}
\toprule 
Orbit & Notation & Partition & %
\begin{tabular}{c}
Weighted Dynkin\tabularnewline
diagram\tabularnewline
\end{tabular}  & \multicolumn{2}{l}{Levi subgroup}\tabularnewline
\midrule
principal & $\sigma_{\mathrm{princ}}$ & $(4)$ & \begin{tikzpicture}[baseline=-3pt,thick, scale=0.65, every node/.style={scale=0.65}]   \draw[double, double distance=0.5mm] (0,0) node[draw, circle, fill=white, label=above:2] {} -- (1,0) node[draw, circle, fill=white, label=above:2] {};   \draw[ultra thick, white] (0.6,-0.2) -- (0.4,0) -- (0.6,0.2);     \draw (0.6,-0.2) -- (0.4,0) -- (0.6,0.2); \end{tikzpicture} & $\hat{G}$ & \tabularnewline
subregular & $\sigma_{\mathrm{subreg}}$ & $(2^{2})$ & \begin{tikzpicture}[baseline=-3pt,thick, scale=0.65, every node/.style={scale=0.65}]   \draw[double, double distance=0.5mm] (0,0) node[draw, circle, fill=white, label=above:0] {} -- (1,0) node[draw, circle, fill=white, label=above:2] {};   \draw[ultra thick, white] (0.6,-0.2) -- (0.4,0) -- (0.6,0.2);     \draw (0.6,-0.2) -- (0.4,0) -- (0.6,0.2); \end{tikzpicture} & $\hat{S}\cong GL_{2}$ & (Siegel)\tabularnewline
minimal & $\sigma_{\mathrm{min}}$ & $(2,1^{2})$ & \begin{tikzpicture}[baseline=-3pt,thick, scale=0.65, every node/.style={scale=0.65}]   \draw[double, double distance=0.5mm] (0,0) node[draw, circle, fill=white, label=above:1] {} -- (1,0) node[draw, circle, fill=white, label=above:0] {};   \draw[ultra thick, white] (0.6,-0.2) -- (0.4,0) -- (0.6,0.2);     \draw (0.6,-0.2) -- (0.4,0) -- (0.6,0.2); \end{tikzpicture} & $\hat{M}\cong Sp_{2}\times GL_{1}$ & (Klingen)\tabularnewline
trivial & $\sigma_{\mathrm{triv}}$ & $(1^{4})$ & \begin{tikzpicture}[baseline=-3pt,thick, scale=0.65, every node/.style={scale=0.65}]   \draw[double, double distance=0.5mm] (0,0) node[draw, circle, fill=white, label=above:0] {} -- (1,0) node[draw, circle, fill=white, label=above:0] {};   \draw[ultra thick, white] (0.6,-0.2) -- (0.4,0) -- (0.6,0.2);     \draw (0.6,-0.2) -- (0.4,0) -- (0.6,0.2); \end{tikzpicture} & $\hat{T}$ & (Borel)\tabularnewline
\bottomrule
\end{tabular}
\end{table}

\begin{lem}
\label{lem:uniform-shape} Let $\psi\in\Psi_{2}\left(G\right)$. Then
the Arthur $SL_{2}$-type of $\psi$, $\psi_{v}$ for any $v$, and
$\varsigma(\psi)$, coincides.
\end{lem}

\begin{proof}
Follows directly from the definitions, the construction of $\tilde{\psi}$
for the $A$-parameter $\psi=(\psi_{N},\tilde{\psi})$ and the fact
that $\psi_{N,v}=\mathrm{Std}_{\hat{G}}\circ\psi_{v}$ for any place
$v$ (see \cite[(1.4.5)]{Art13}).
\end{proof}
By Lemma~\ref{lem:uniform-shape}, the $A$-shape of $\psi$ uniquely
determines the Arthur $SL_{2}$-type of each local factor. The converse
is not true in general. For example, both $A$-shapes $\left(\left(2,2\right)\right)$
and $\left(\left(1,2\right),\left(1,2\right)\right)$ of $G=SO_{5}$
have the same Arthur $SL_{2}$-type corresponding to the nilpotent
orbit with partition $(2^{2})$.

Following \cite{NP21}, we observe that for $G$ either a special
odd orthogonal or a symplectic group with standard representation
$\mathrm{Std}_{\hat{G}}\,:\,\hat{G}\rightarrow GL_{N}\left(\mathbb{C}\right)$
for its dual group, then the principal $SL_{2}$ of $\hat{G}$ is
the restriction of the principal $SL_{2}$ of $GL_{N}\left(\mathbb{C}\right)$,
i.e. $\sigma_{\mathrm{princ}}=\nu(N)$. In other words, the unique
principal nilpotent orbit is given by the partition $(N)$. Therefore,
the only $A$-shape $\varsigma\in\mathcal{M}\left(G\right)$ with
a $\hat{G}$-principal Arthur $SL_{2}$-type $\sigma_{\varsigma}=\sigma_{\mathrm{princ}}$,
is the $A$-shape $\left(\left(1,N\right)\right)\in\mathcal{M}\left(G\right)$
(which was denoted by $(\mathbf{F})$ in Subsection~\ref{subsec:Unramified-rep-par}
above and discussed in more detail in Section~\ref{sec:Schmidt}
below). On the other hand, there could be several $A$-shapes $\varsigma\in\mathcal{M}\left(G\right)$
whose Arthur $SL_{2}$-type is $\sigma_{\varsigma}=\sigma_{\mathrm{triv}}$,
the trivial $SL_{2}$ of $\hat{G}$ corresponding to the partition
$(1,\ldots,1)$. These $A$-shapes are called generic and are defined
as follows.
\begin{defn}
Define the subset of generic $A$-shapes of $G$ to be 
\[
\mathcal{M}^{g}\left(G\right)=\left\{ \left(\left(n_{i},m_{i}\right)_{i}\right)\in\mathcal{M}\left(G\right)\,:\,m_{i}=1,\;\forall i\right\} =\left\{ \varsigma(\psi)\,:\,\psi\in\Psi_{2}\left(G\right),\quad\tilde{\psi}_{G}|_{SL_{2}^{A}}\equiv1\right\} ,
\]
and define the subsets of generic $A$-parameters of $G$ to be 
\[
\Psi_{2}^{g}\left(G\right)=\left\{ \psi\in\Psi_{2}\left(G\right)\,:\,\varsigma(\psi)\in\mathcal{M}^{g}\left(G\right)\right\} .
\]
\end{defn}

We note that by Arthur's conjectures \cite{Art89}, generic $A$-parameters
are conjectured to be tempered locally everywhere, namely they should
satisfy the (correct) generalization of the Ramanujan conjecture.
See \cite{Sha11} for more details and Proposition~\ref{prop:GRPC-classical}
below.
\begin{defn}
\label{def:shape-group} For an $A$-shape $\varsigma=\left(\left(n_{i},m_{i}\right)_{i}\right)$,
denote $GL_{N}^{\{\varsigma\}}:=\prod_{i}GL_{n_{i}}$ and 
\[
\iota_{\varsigma}\,:\,GL_{N}^{\{\varsigma\}}\left(\mathbb{C}\right)\rightarrow GL_{N}\left(\mathbb{C}\right),\qquad\iota_{\varsigma}\left(g_{1},\ldots,g_{r}\right)=\mathrm{diag}\left(\stackrel{m_{1}}{\overbrace{g_{1},\ldots,g_{1}}},\ldots,\stackrel{m_{r}}{\overbrace{g_{r},\ldots,g_{r}}}\right).
\]
Assume now $\varsigma\in\mathcal{M}\left(G\right)$ and identify $\hat{G}$
as a subgroup of $GL_{N}\left(\mathbb{C}\right)$ via the standard
representation. Denote $G^{\{\varsigma\}}:=\prod_{i}G_{i}^{\{\varsigma\}}$,
where $G_{i}^{\{\varsigma\}}$ is the (almost-simple) quasi-split
classical group whose dual is 
\[
\widehat{G_{i}^{\{\varsigma\}}}=\left\{ g_{i}\in GL_{n_{i}}\left(\mathbb{C}\right)\,:\,\iota_{\varsigma}\left(1,\ldots,1,g_{i},1\ldots,1\right)\in\hat{G}\right\} .
\]
In particular the standard representation of $\widehat{G_{i}^{\{\varsigma\}}}$
is of dimension $n_{i}$. Denote $\Psi(GL_{N}^{\{\varsigma\}}):=\prod_{i}\Psi\left(GL_{n_{i}}\right)$,
$\Psi_{2}(G^{\{\varsigma\}}):=\prod_{i}\Psi_{2}(G_{i}^{\{\varsigma\}})$
and $\Psi_{2}^{g}(G^{\{\varsigma\}}):=\prod_{i}\Psi_{2}^{g}(G_{i}^{\{\varsigma\}})$.
\end{defn}

Note that $G^{\{\varsigma\}}=G$ for the $A$-shape $\varsigma=\left(\left(N,1\right)\right)$
and that $\widehat{G^{\{\varsigma\}}}=\hat{Z}$, the center of $\hat{G}$,
for the $A$-shape $\varsigma=\left(\left(1,N\right)\right)$. Also
note that $\iota_{\varsigma}\left(GL_{N}^{\{\varsigma\}}\left(\mathbb{C}\right)\right)$
commutes with the image of $\sigma_{\varsigma}=\bigoplus_{i}\nu(m_{i})^{n_{i}}$,
the Arthur $SL_{2}$-type of $\varsigma$. 
\begin{prop}
\label{prop:gen-func-global} Let $\varsigma=\left(\left(n_{1},m_{1}\right),\ldots,\left(n_{r},m_{r}\right)\right)\in\mathcal{M}\left(G\right)$.
Then there is a natural map 
\[
\Psi_{2}\left(G,\varsigma\right)\rightarrow\Psi_{2}^{g}(G^{\{\varsigma\}}),\qquad\psi\mapsto\psi^{\{\varsigma\}},
\]
which is surjective and whose fibers are of size at most $2^{r}$.
\end{prop}

\begin{proof}
For any $\psi=(\psi_{N},\tilde{\psi})\in\Psi_{2}\left(G,\varsigma\right)$,
with $\psi_{N}=\boxplus_{i}\left(\mu_{i}\boxtimes\nu(m_{i})\right)$,
define $\psi^{\{\varsigma\}}=(\psi_{N}^{\{\varsigma\}},\widetilde{\psi^{\{\varsigma\}}})\in\Psi_{2}^{g}(G^{\{\varsigma\}})$,
where $\psi_{N}^{\{\varsigma\}}:=\boxplus_{i}\left(\mu_{i}\boxtimes\nu(1)\right)\in\Psi(GL_{N}^{\{\varsigma\}})$,
and $\widetilde{\psi^{\{\varsigma\}}}$ is defined as follows: By
\cite[Theorem 1.4.1, (1.4.4), (1.4.5)]{Art13}, we get that $\mathcal{L}_{\psi_{N}^{\{\varsigma\}}}=\mathcal{L}_{\psi_{N}}$
and $\widetilde{\psi_{N}^{\{\varsigma\}}}\,:\,\mathcal{L}_{\psi_{N}}\times SL_{2}^{A}\rightarrow GL_{N}^{\{\varsigma\}}\left(\mathbb{C}\right)$
is defined by $\iota_{\varsigma}\circ\widetilde{\psi_{N}^{\{\varsigma\}}}|_{\mathcal{L}_{\psi_{N}}}=\tilde{\psi}_{N}|_{\mathcal{L}_{\psi_{N}}}=\bigoplus_{i}\tilde{\mu}_{i}^{m_{i}}$
and $\widetilde{\psi_{N}^{\{\varsigma\}}}|_{SL_{2}^{A}}\equiv1$.
Since $\widehat{G^{\{\varsigma\}}}=\widehat{GL_{N}^{\{\varsigma\}}}\bigcap\iota_{\varsigma}^{-1}(\hat{G})$
and $\tilde{\psi}\,:\,\mathcal{L}_{\psi_{N}}\times SL_{2}^{A}\rightarrow\hat{G}$,
then we can define $\widetilde{\psi^{\{\varsigma\}}}\,:\,\mathcal{L}_{\psi_{N}}\times SL_{2}^{A}\rightarrow\widehat{G^{\{\varsigma\}}}$
by $\iota_{\varsigma}\circ\widetilde{\psi^{\{\varsigma\}}}|_{\mathcal{L}_{\psi_{N}}}\equiv\tilde{\psi}|_{\mathcal{L}_{\psi_{N}}}$
and $\widetilde{\psi^{\varsigma}}|_{SL_{2}^{A}}\equiv1$. 

To prove surjectivity of the map, take any generic $A$-parameter
$\varphi=(\varphi_{N},\tilde{\varphi})\in\Psi_{2}^{g}(G^{\{\varsigma\}})$,
$\varphi_{N}=\boxplus_{i}\left(\mu_{i}\boxtimes\nu(1)\right)\in\Psi(GL_{N}^{\{\varsigma\}})$
and $\tilde{\varphi}\,:\,\mathcal{L}_{\varphi}\times SL_{2}^{A}\rightarrow\widehat{G^{\{\varsigma\}}}$
such that $\tilde{\varphi}|_{SL_{2}^{A}}\equiv1$. Define $\psi:=(\psi_{N},\tilde{\psi})\in\Psi\left(G\right)$,
by $\psi_{N}:=\boxplus_{i}\left(\mu_{i}\boxtimes\nu(m_{i})\right)\in\Psi\left(GL_{N}\right)$
and $\tilde{\psi}\,:\,\mathcal{L}_{\psi}\times SL_{2}^{A}\rightarrow\hat{G}$,
by $\tilde{\psi}|_{\mathcal{L}_{\psi}}:=\tilde{\varphi}|_{\mathcal{L}_{\psi}}$
and $\tilde{\psi}|_{SL_{2}^{A}}:=\bigoplus_{i=1}^{r}\nu(m_{i})^{n_{i}}$.
Then $\psi\in\Psi\left(G,\varsigma\right)$ and it is easy to see
that $\psi^{\{\varsigma\}}=\varphi$. 

The bound on the size of the fibers follows from the fact that for
a fixed $\varphi=(\varphi_{N},\tilde{\varphi})\in\Psi^{g}(G^{\{\varsigma\}})$,
if $\psi=(\psi_{N},\tilde{\psi})\in\Psi\left(G\right)$ is such that
$\psi^{\{\varsigma\}}=\varphi$, then $\psi_{N}$ is completely determined
by $\varphi_{N}$. The claim now follows from Remark \ref{rem:Arthur-fiber}.
\end{proof}
\begin{prop}
\label{prop:gen-func-local} Let $\varsigma\in\mathcal{M}\left(G\right)$,
let $\psi\in\Psi_{2}\left(G,\varsigma\right)$ and let $\psi^{\{\varsigma\}}\in\Psi_{2}^{g}(G^{\{\varsigma\}})$,
as in Proposition~\ref{prop:gen-func-global}. Then 
\[
\psi_{v}|_{L_{F_{v}}}\equiv\iota_{\varsigma}\circ\psi_{v}^{\{\varsigma\}}|_{L_{F_{v}}},\qquad\forall v,
\]
where $L_{F_{v}}$ is the local Langlands group of the local field
$F_{v}$.
\end{prop}

\begin{proof}
For any place $v$, since $\psi_{N,v}=\mathrm{Std}_{\hat{G}}\circ\psi_{v}$
and $\psi_{N,v}^{\{\varsigma\}}=\mathrm{Std}_{\widehat{G^{\{\varsigma\}}}}\circ\psi_{v}^{\{\varsigma\}}$,
and since $\mathrm{Std}_{\hat{G}}\circ\iota_{\varsigma}|_{\widehat{G^{\{\varsigma\}}}}=\iota_{\varsigma}\circ\mathrm{Std}_{\widehat{G^{\{\varsigma\}}}}$,
it suffices to prove $\psi_{N,v}|_{L_{F_{v}}}\equiv\iota_{\varsigma}\circ\psi_{N,v}^{\{\varsigma\}}|_{L_{F_{v}}}$.
Let $\psi_{N}=\boxplus_{i}\left(\mu_{i}\boxtimes\nu(m_{i})\right)$
and for any $i$, let $\mu_{i,v}\in\Phi\left(GL_{n_{i}}\left(F_{v}\right)\right)$
be as in \eqref{eq:LLC-GLN}. Then $\psi_{N,v}|_{L_{F_{v}}}=\oplus_{i}\mu_{i,v}^{\oplus m_{i}}|_{L_{F_{v}}}=\iota_{\varsigma}\circ\psi_{N,v}^{\{\varsigma\}}|_{L_{F_{v}}}$,
where $\mu_{i,v}^{\oplus m_{i}}:=\underset{m_{i}}{\underbrace{\mu_{i,v}\oplus\ldots\oplus\mu_{i,v}}}$.
\end{proof}

\section{Explicit classification for $SO_{5}$ \protect\label{sec:Schmidt}}

In this section we summarize Schmidt's \cite{Sch18,Sch20} explicit
description of the Arthur classification for $SO_{5}$, and describe
explicitly the local and global $A$-parameters and $A$-packets of
$G=SO_{5}\cong PGSp_{4}$, over a global field $F$. 

Following Schmidt's notation, the elements of the set $\mathcal{M}\left(SO_{5}\right)$
of $A$-shapes of $SO_{5}$ are named as 
\[
\mathcal{M}\left(SO_{5}\right)=\lbrace\quad(\mathbf{G})=\left(\left(4,1\right)\right)\quad,\quad(\mathbf{Y})=\left(\left(2,1\right),\left(2,1\right)\right)\quad,\quad(\mathbf{F})=\left(\left(1,4\right)\right)\quad,
\]
\[
\quad(\mathbf{B})=\left(\left(1,2\right),\left(1,2\right)\right)\quad,\quad(\mathbf{Q})=\left(\left(2,2\right)\right)\quad,\quad(\mathbf{P})=\left(\left(2,1\right),\left(1,2\right)\right)\quad\rbrace.
\]
where $(\mathbf{G})$ stands for General, $(\mathbf{Y})$ for Yoshida,
$(\mathbf{F})$ for Finite, $(\mathbf{B})$ for Howe--Piatetski-Shapiro,
$(\mathbf{Q})$ for Soudry and $(\mathbf{P})$ for Saito--Korukawa. 

As described in Section~\ref{subsec:Arthur-shapes-types}, the Arthur
$SL_{2}$-types of these $A$-shapes are (compare with Table~\ref{tab:sp4-orbits}):
\begin{equation}
\begin{alignedat}{2}\sigma_{(\mathbf{F})} & =\nu(4) &  & =\sigma_{\mathrm{princ}}\\
\sigma_{(\mathbf{B})}=\sigma_{(\mathbf{Q})} & =\nu(2)^{2} &  & =\sigma_{\mathrm{subreg}}\\
\sigma_{(\mathbf{P})} & =\nu(2)\oplus\nu(1)^{2} &  & =\sigma_{\mathrm{min}}\\
\sigma_{(\mathbf{G})}=\sigma_{(\mathbf{Y})} & =\nu(1)^{4} &  & =\sigma_{\mathrm{triv}.}
\end{alignedat}
\label{eq:SO5-Arthur-SL2}
\end{equation}
Therefore the subset of generic $A$-shapes of $SO_{5}$ is
\[
\mathcal{M}^{g}\left(G\right)=\left\{ \quad(\mathbf{G})=\left(\left(4,1\right)\right)\quad,\quad(\mathbf{Y})=\left(\left(2,1\right),\left(2,1\right)\right)\quad\right\} .
\]
Note that for $(\mathbf{P})$ and $(\mathbf{F})$ the Arthur $SL_{2}$-type
determines the $A$-shape, but not for $(\mathbf{G})$, $(\mathbf{Y})$,
$(\mathbf{B})$ and $(\mathbf{Q})$. 

\subsection{Global $A$-parameters}

For any $A$-shape $\varsigma\in\mathcal{M}\left(SO_{5}\right)$ we
describe the structure of the elliptic $A$-parameter $\psi=(\psi_{N},\tilde{\psi}_{G})$
of $A$-shape $\varsigma$, in particular we describe explicitly the
complex group $\mathcal{L}_{\psi}$, the complex homomorphism $\tilde{\psi}\,:\,\mathcal{L}_{\psi}\times SL_{2}^{A}\rightarrow Sp_{4}\left(\mathbb{C}\right)$,
and the finite $2$-group $\mathcal{S}_{\psi}=S_{\psi}/S_{\psi}^{0}\cdot Z$,
where $S_{\psi}=C_{\hat{G}}(\tilde{\psi})$ and $\hat{Z}=\left\{ \pm I\right\} \leq Sp_{4}\left(\mathbb{C}\right)$.
\begin{description}
\item [{(G)}] If $\psi\in\Psi_{2}\left(SO_{5},(\mathbf{G})\right)$ then
$\psi_{N}=\mu\boxtimes\nu(1)$, where $\mu$ is a self-dual cuspidal
automorphic representation of $GL_{4}/F$ of symplectic type. Call
this a general type Arthur parameter. Then 
\[
\mathcal{L}_{\psi}=Sp_{4}\left(\mathbb{C}\right),\quad\tilde{\psi}|_{\mathcal{L}_{\psi}}\left(g\right)=g,\quad\tilde{\psi}|_{SL_{2}^{A}}\equiv1,
\]
and $\mathcal{S}_{\psi}=\left\{ I\right\} $.
\item [{(Y)}] If $\psi\in\Psi_{2}\left(SO_{5},(\mathbf{Y})\right)$ then
$\psi_{N}=\left(\mu_{1}\boxtimes\nu(1)\right)\boxplus\left(\mu_{2}\boxtimes\nu(1)\right)$,
where $\mu_{1},\mu_{2}$ are two distinct cuspidal automorphic representations
of $GL_{2}/F$ with trivial central character. Call this a Yoshida
type Arthur parameter. Then
\[
\mathcal{L}_{\psi}=SL_{2}\left(\mathbb{C}\right)^{2},\quad\tilde{\psi}|_{\mathcal{L}_{\psi}}\left(g,g'\right)=\left(\begin{array}{cccc}
g_{11} &  & g_{12}\\
 & g'_{11} &  & g'_{12}\\
g_{21} &  & g_{22}\\
 & g'_{21} &  & g'_{22}
\end{array}\right),\quad\tilde{\psi}|_{SL_{2}^{A}}\equiv1,
\]
 and $\mathcal{S}_{\psi}=\left\{ I,\mathrm{diag}\left(-1,1,-1,1\right)\right\} \cong\mathbb{Z}/2\mathbb{Z}$.
\item [{(F)}] If $\psi\in\Psi_{2}\left(SO_{5},(\mathbf{F})\right)$ then
$\psi_{N}=\chi\boxtimes\nu(4)$, where $\chi$ is a quadratic Hecke
character of $GL_{1}/F$. Call this a finite type Arthur parameter.
Then 
\[
\mathcal{L}_{\psi}=\left\{ \pm1\right\} ,\quad\tilde{\psi}|_{\mathcal{L}_{\psi}}\left(\pm1\right)=\pm1,\quad\tilde{\psi}|_{SL_{2}^{A}}\equiv\nu(4),
\]
 and $\mathcal{S}_{\psi}=\left\{ I\right\} $.
\item [{(B)}] If $\psi\in\Psi_{2}\left(SO_{5},(\mathbf{B})\right)$ then
$\psi_{N}=\left(\chi_{1}\boxtimes\nu(2)\right)\boxplus\left(\chi_{2}\boxtimes\nu(2)\right)$,
where $\chi_{1},\chi_{2}$ are two distinct quadratic Hecke characters
of $GL_{1}/F$. Call this a Howe--Piatetski-Shapiro type Arthur parameter.
Then 
\[
\mathcal{L}_{\psi}=\left\{ \pm1\right\} ^{2},\quad\tilde{\psi}|_{\mathcal{L}_{\psi}}\left(x,y\right)=\left(\begin{array}{cccc}
x\\
 & y\\
 &  & x\\
 &  &  & y
\end{array}\right),\quad\tilde{\psi}|_{SL_{2}^{A}}\left(h\right)=\left(\begin{array}{cccc}
1\\
 & h_{11} &  & h_{12}\\
 &  & 1\\
 & h_{21} &  & h_{22}
\end{array}\right),
\]
and $\mathcal{S}_{\psi}=\left\{ I,\mathrm{diag}\left(-1,1,-1,1\right)\right\} \cong\mathbb{Z}/2\mathbb{Z}$.
\item [{(Q)}] If $\psi\in\Psi_{2}\left(SO_{5},(\mathbf{Q})\right)$ then
$\psi_{N}=\mu\boxtimes\nu(2)$, where $\mu$ is a self-dual cuspidal
automorphic representation of $GL_{2}/F$ with a quadratic non-trivial
central character $\omega_{\mu}$, such that $\mu$ is an automorphic
induction of an Hecke character $\theta$ of $GL_{1}/K$, where $K$
is the quadratic extension of $F$ associated to $\omega_{\mu}$ by
class field theory. Call this a Soudry type Arthur parameter. Then
\[
\mathcal{L}_{\psi}=O_{2}\left(\mathbb{C}\right),\quad\tilde{\psi}|_{\mathcal{L}_{\psi}}\left(g\right)=\left(\begin{array}{cc}
g\\
 & \theta(g)
\end{array}\right),\quad\tilde{\psi}|_{SL_{2}^{A}}\left(h\right)=\left(\begin{array}{cccc}
h_{11} &  &  & h_{12}\\
 & h_{11} & h_{12}\\
 & h_{21} & h_{22}\\
h_{21} &  &  & h_{22}
\end{array}\right),
\]
 and $\mathcal{S}_{\psi}=\left\{ I\right\} $.
\item [{(P)}] If $\psi\in\Psi_{2}\left(SO_{5},(\mathbf{P})\right)$ then
$\psi_{N}=\left(\mu\boxtimes\nu(1)\right)\boxplus\left(\chi\boxtimes\nu(2)\right)$,
where $\mu$ is a cuspidal automorphic representation of $GL_{2}/F$
with trivial central character and $\chi$ is a quadratic Hecke character.
Call this a Saito--Korukawa type Arthur parameter. Then 
\[
\mathcal{L}_{\psi}=SL_{2}\left(\mathbb{C}\right)\times\left\{ \pm1\right\} ,\quad\tilde{\psi}|_{\mathcal{L}_{\psi}}\left(g,\pm1\right)=\left(\begin{array}{cccc}
g_{11} &  & g_{12}\\
 & \pm1\\
g_{21} &  & g_{22}\\
 &  &  & \pm1
\end{array}\right),
\]
\[
\tilde{\psi}|_{SL_{2}^{A}}\left(h\right)=\left(\begin{array}{cccc}
1\\
 & h_{11} &  & h_{12}\\
 &  & 1\\
 & h_{21} &  & h_{22}
\end{array}\right),
\]
and $\mathcal{S}_{\psi}=\left\{ I,\mathrm{diag}\left(-1,1,-1,1\right)\right\} \cong\mathbb{Z}/2\mathbb{Z}$.
\end{description}

\subsection{Local $A$-parameters}

For a discrete global $A$-parameter $\psi\in\Psi_{2}\left(G\right)$
and any place $v$ of $F$, we describe the corresponding local $A$-parameter
$\psi_{v}\in\Psi\left(G\left(F_{v}\right)\right)$, the local $L$-parameter
$\phi_{\psi_{v}}\in\Phi\left(G\left(F_{v}\right)\right)$ and a distinguished
member of the corresponding local $L$- and $A$- packets $\pi_{\psi_{v}}\in\Pi_{\phi_{\psi_{v}}}\subset\Pi_{\psi_{v}}$.
When $\psi_{v}$ is unramified, then $\pi_{\psi_{v}}$ agrees with
\eqref{eq:ur-=00005Cpi-=00005Cpsi}.

Denote by $\hat{Z}=\left\{ \pm1\right\} $ the center of $Sp_{4}\left(\mathbb{C}\right)$.
Denote by $B$, $Q$ and $P$ the Borel, Klingen and Siegel parabolic
subgroups of $G=SO_{5}\cong PGSp_{4}$. Note that the Klingen $Q$
and Siegel $P$ parabolic subgroups are dual to the parabolic subgroups
of the dual group $\hat{G}=Sp_{4}\left(\mathbb{C}\right)$, whose
Levi subgroups are the Siegel $\hat{S}$ and Klingen $\hat{M}$, respectively.
\begin{description}
\item [{(G)}] Let $\psi=\mu\boxtimes\nu(1)\in\Psi_{2}\left(SO_{5},(\mathbf{G})\right)$.
Let $\mu_{v}\in\Pi\left(GL_{4}\left(F_{v}\right)\right)$ be the $v$
local factor of $\mu$, which is also self-dual of symplectic type,
and by the Local Langlands Correspondence of $GL_{n}$ \cite{HT01,Hen00,Sch13},
denote its corresponding $L$-parameter by $\tilde{\mu}_{v}\,:\,L_{F_{v}}\rightarrow Sp_{4}\left(\mathbb{C}\right)=\mathcal{L}_{\psi}$.
Then the local $A$-parameter is
\[
\psi_{v}=\tilde{\psi}\circ\left(\tilde{\mu}_{v}\otimes\nu(1)\right)\,:\,L_{F_{v}}\times SL_{2}\left(\mathbb{C}\right)\rightarrow\mathcal{L}_{\psi}\times SL_{2}\left(\mathbb{C}\right)\rightarrow Sp_{4}\left(\mathbb{C}\right).
\]
Since $\psi_{v}$ is trivial on $SL_{2}\left(\mathbb{C}\right)$,
we get that $\psi_{v}=\phi_{\psi_{v}}$, and by \cite[Theorem 1.1]{Sch18},
there exists a generic $\pi_{\psi}\in\Pi_{\psi}$ with $\langle\cdot,\pi_{\psi}\rangle_{\psi}\equiv1$.
\item [{(Y)}] Let $\psi=\left(\mu_{1}\boxtimes\nu(1)\right)\boxplus\left(\mu_{2}\boxtimes\nu(1)\right)\in\Psi_{2}\left(SO_{5},(\mathbf{Y})\right)$.
Let $\mu_{1,v},\mu_{2,v}\in\Pi\left(GL_{2}\left(F_{v}\right)\right)$
be the $v$ local factors of $\mu_{1},\mu_{2}$, and by the Local
Langlands Correspondence of $GL_{n}$ \cite{HT01,Hen00,Sch13}, denote
their corresponding $L$-parameters by $\tilde{\mu}_{1,v},\tilde{\mu}_{2,v}\,:\,L_{F_{v}}\rightarrow SL_{2}\left(\mathbb{C}\right)$,
hence $\tilde{\mu}_{1,v}\times\tilde{\mu}_{2,v}\,:\,L_{F_{v}}\rightarrow SL_{2}\left(\mathbb{C}\right)^{2}=\mathcal{L}_{\psi}$.
Then the local $A$-parameter is
\[
\psi_{v}=\tilde{\psi}\circ\left(\left(\tilde{\mu}_{1,v}\times\tilde{\mu}_{2,v}\right)\otimes\nu(1)\right)\,:\,L_{F_{v}}\times SL_{2}\left(\mathbb{C}\right)\rightarrow\mathcal{L}_{\psi}\times SL_{2}\left(\mathbb{C}\right)\rightarrow Sp_{4}\left(\mathbb{C}\right).
\]
Since $\psi_{v}$ is trivial on $SL_{2}\left(\mathbb{C}\right)$,
we get that $\psi_{v}=\phi_{\psi_{v}}$, and by \cite[Theorem 1.1]{Sch18},
there exists a generic $\pi_{\psi}\in\Pi_{\psi}$ with $\langle\cdot,\pi_{\psi}\rangle_{\psi}\equiv1$.
\item [{(F)}] Let $\psi=\chi\boxtimes\nu(4)\in\Psi_{2}\left(SO_{5},(\mathbf{F})\right)$.
Let $\chi_{v}$ be the $v$ local factor of $\chi$, and by local
class field theory denote its corresponding $L$-parameter by $\tilde{\chi}_{v}\,:\,L_{F_{v}}\rightarrow\left\{ \pm1\right\} =\mathcal{L_{\psi}}$.
Then the local $A$-parameter is 
\[
\psi_{v}=\tilde{\psi}\circ\left(\tilde{\chi}_{v}\otimes\nu(2)\right)\,:\,L_{F_{v}}\times SL_{2}\left(\mathbb{C}\right)\rightarrow\mathcal{L}_{\psi}\times SL_{2}\left(\mathbb{C}\right)\rightarrow Sp_{4}\left(\mathbb{C}\right).
\]
The corresponding $L$-parameter $\phi_{\psi_{v}}\in\Phi\left(G\left(F_{v}\right)\right)$
is 
\[
\phi_{\psi_{v}}(w)=\mathrm{diag}\left(\chi_{v}(w)|w|_{v}^{3/2},\chi_{v}(w)|w|_{v}^{1/2},\chi_{v}(w)|w|_{v}^{-1/2},\chi_{v}(w)|w|_{v}^{-3/2}\right)
\]
and there is a distinguished member in the packet which is the one-dimensional
representation $\pi_{\psi_{v}}=\chi_{v}1_{G\left(F_{v}\right)}\in\Pi_{\phi_{\psi_{v}}}$.
\item [{(B)}] Let $\psi=\left(\chi_{1}\boxtimes\nu(2)\right)\boxplus\left(\chi_{2}\boxtimes\nu(2)\right)\in\Psi_{2}\left(SO_{5},(\mathbf{B})\right)$.
Let $\chi_{1,v},\chi_{2,v}$ be the $v$ local factors of $\chi_{1},\chi_{2}$,
and by local class field theory denote their corresponding $L$-parameters
by $\tilde{\chi}_{1,v},\tilde{\chi}_{2,v}\,:\,L_{F_{v}}\rightarrow\hat{Z}$,
hence $\tilde{\chi}_{1,v}\times\tilde{\chi}_{2,v}\,:\,L_{F_{v}}\rightarrow\left\{ \pm1\right\} =\mathcal{L}_{\psi}$.
Then the local $A$-parameter is 
\[
\psi_{v}=\tilde{\psi}\circ\left(\left(\tilde{\chi}_{1,v}\times\tilde{\chi}_{2,v}\right)\otimes\nu(1)\right)\,:\,L_{F_{v}}\times SL_{2}\left(\mathbb{C}\right)\rightarrow\mathcal{L}_{\psi}\times SL_{2}\left(\mathbb{C}\right)\rightarrow Sp_{4}\left(\mathbb{C}\right).
\]
The corresponding $L$-parameter $\phi_{\psi_{v}}\in\Phi\left(G\left(F_{v}\right)\right)$
is 
\[
\phi_{\psi_{v}}(w)=\mathrm{diag}\left(\chi_{1,v}(w)|w|_{v}^{1/2},\chi_{2,v}(w)|w|_{v}^{1/2},\chi_{1,v}(w)|w|_{v}^{-1/2},\chi_{2,v}(w)|w|_{v}^{-1/2}\right)
\]
and there is a distinguished member in the packet which is the following
Langlands quotient of a parabolically induced representation from
the Borel parabolic subgroup $B$,
\[
\pi_{\psi_{v}}=j\left(B,\:\chi_{1,v}\chi_{2,v}\times\chi_{1,v}\chi_{2,v}\rtimes\chi_{2,v},\:|\cdot|_{v}\times1\rtimes|\cdot|_{v}^{-1/2}\right)\in\Pi_{\phi_{\psi_{v}}}.
\]
\item [{(Q)}] Let $\psi=\mu\boxtimes\nu(2)\in\Psi_{2}\left(SO_{5},(\mathbf{Q})\right)$.
Let $\mu_{v}\in\Pi\left(GL_{2}\left(F_{v}\right)\right)$ be the $v$
local factor of $\mu$, denote its corresponding $L$-parameter by
$\tilde{\mu}_{v}\,:\,L_{F_{v}}\rightarrow GL_{2}\left(\mathbb{C}\right)$
and denote its central character by $\omega_{\mu_{v}}$. Since $\mu$
is an automorphic induction from a Hecke character of a quadratic
field extension of $F$ associated to $\omega_{\mu}$, we get that
the $L$-parameter of $\mu_{v}$ is given by $\tilde{\mu}_{v}\,:\,L_{F_{v}}\rightarrow O_{2}\left(\mathbb{C}\right)=\mathcal{L}_{\psi}$
(see \cite[Section 1]{Sch20}). Then the local $A$-parameter is
\[
\psi_{v}=\tilde{\psi}\circ\left(\tilde{\mu}_{v}\otimes\nu(1)\right)\,:\,L_{F_{v}}\times SL_{2}\left(\mathbb{C}\right)\rightarrow\mathcal{L}_{\psi}\times SL_{2}\left(\mathbb{C}\right)\rightarrow Sp_{4}\left(\mathbb{C}\right).
\]
The corresponding $L$-parameter $\phi_{\psi_{v}}\in\Phi\left(G\left(F_{v}\right)\right)$
is 
\[
\phi_{\psi_{v}}(w)=\mathrm{diag}\left(|w|_{v}^{1/2}\tilde{\mu}_{v}(w),|w|_{v}^{-1/2}\theta\left(\tilde{\mu}_{v}(w)\right)\right)
\]
and there is a distinguished member in the packet which is the following
Langlands quotient of a parabolically induced representation from
the Klingen parabolic subgroup $Q$
\[
\pi_{\psi_{v}}=j\left(Q,\:\omega_{\mu_{v}}\rtimes\mu_{v},\:|\cdot|_{v}\rtimes|\cdot|_{v}^{-1/2}\right)\in\Pi_{\phi_{\psi_{v}}}.
\]
\item [{(P)}] Let $\psi=\left(\mu\boxtimes\nu(1)\right)\boxplus\left(\chi\boxtimes\nu(2)\right)\in\Psi_{2}\left(SO_{5},(\mathbf{P})\right)$.
Let $\mu_{v}\in\Pi\left(PGL_{2}\left(F_{v}\right)\right)$ and $\chi_{v}\,:\,F_{v}^{*}\rightarrow\left\{ \pm1\right\} $
be the $v$ local factors of $\mu$ and $\chi$, denote their corresponding
$L$-parameters by $\tilde{\mu}_{v}\,:\,L_{F_{v}}\rightarrow SL_{2}\left(\mathbb{C}\right)$
and $\tilde{\chi}_{v}\,:\,L_{F_{v}}\rightarrow\left\{ \pm1\right\} $,
hence $\tilde{\mu}_{v}\times\tilde{\chi}_{v}\,:\,L_{F_{v}}\rightarrow\mathcal{L}_{\psi}$.
Then the local $A$-parameter is
\[
\psi_{v}=\tilde{\psi}\circ\left(\left(\tilde{\mu}_{v}\times\tilde{\chi}_{v}\right)\otimes\nu(1)\right)\,:\,L_{F_{v}}\times SL_{2}\left(\mathbb{C}\right)\rightarrow\mathcal{L}_{\psi}\times SL_{2}\left(\mathbb{C}\right)\rightarrow Sp_{4}\left(\mathbb{C}\right).
\]
The corresponding $L$-parameter $\phi_{\psi_{v}}\in\Phi\left(G\left(F_{v}\right)\right)$
is 
\[
\phi_{\psi_{v}}(w)=\left(\begin{array}{cccc}
\mu_{v}(w)_{11} &  & \mu_{v}(w)_{12}\\
 & \chi_{v}(w)|w|^{1/2}\\
\mu_{v}(w)_{21} &  & \mu_{v}(w)_{22}\\
 &  &  & \chi_{v}(w)|w|^{-1/2}
\end{array}\right)
\]
and there is a distinguished member in the packet which is the following
Langlands quotient of a parabolically induced representation from
the Siegel parabolic subgroup $P$,
\[
\pi_{\psi_{v}}=j\left(P,\:\chi_{v}\mu_{v}\rtimes\chi_{v},\:|\cdot|_{v}^{1/2}\rtimes|\cdot|_{v}^{-1/2}\right)\in\Pi_{\phi_{\psi_{v}}}.
\]
\end{description}

\subsection{Local $A$-packets\protect\label{subsec:Schmidt-A-packet}}

In \cite{Sch18,Sch20}, Schmidt describe explicitly the local $A$-packets
of $G=SO_{5}$ (see in particular \cite[Tables 1, 2, 3]{Sch20}).
The following proposition summarizes certain key properties that we
shall need.
\begin{prop}
\cite{Sch18,Sch20} \label{prop:Schmidt} Let $G=SO_{5}$, $\varsigma\in\mathcal{M}\left(G\right)$
and $\psi\in\Psi_{2}\left(G,\varsigma\right)$. Then for any place
$v$,
\[
\Pi_{\psi_{v}}=\Pi_{\phi_{\psi_{v}}},\qquad\forall\varsigma\in\left\{ (\mathbf{G}),(\mathbf{Y}),(\mathbf{F})\right\} ,
\]
and 
\[
\Pi_{\phi_{\psi_{v}}}\subset\Pi_{\psi_{v}}\subset(\Pi_{\phi_{\psi_{v}}}\cup\Pi_{\phi_{\psi_{v}}^{*}}),\qquad\forall\varsigma\in\left\{ (\mathbf{B}),(\mathbf{Q}),(\mathbf{P})\right\} ,
\]
where $\phi_{\psi_{v}}^{*}\in\Phi\left(G\left(F_{v}\right)\right)$
is an $L$-parameter of $G$ satisfying 
\[
\phi_{\psi_{v}}^{*}|_{I_{F_{v}}}\equiv\phi_{\psi_{v}}|_{I_{F_{v}}}\qquad\mbox{and}\qquad\phi_{\psi_{v}}^{*}|_{SL_{2}^{D}}\not\equiv1.
\]
\end{prop}

\begin{proof}
The claim about the $A$-packets of $A$-shape $\varsigma\in\left\{ (\mathbf{G}),(\mathbf{Y}),(\mathbf{F})\right\} $
appears in \cite[Section 1]{Sch18}, while the claim about the $A$-packets
of $A$-shape $\varsigma\in\left\{ (\mathbf{B}),(\mathbf{Q}),(\mathbf{P})\right\} $
follows from the description of the $L$-parameters of the members
of the $A$-packets in \cite[Tables 1, 2, 3]{Sch20}. Note that
in the latter case, $\phi_{\psi}^{*}$ differs from $\phi_{\psi}$
by how it acts on the Deligne $SL_{2}$, and that they agree on the
inertia group $I_{F}$.
\end{proof}
We shall need a bit more information on the members of $A$-packet
of the $A$-shape $(\mathbf{B})$ and $(\mathbf{Q})$. Note that the
for each local $A$-packet $\Pi_{\psi_{p}}$, the distinguished member
$\pi_{\psi_{v}}\in\Pi_{\psi_{v}}$ was described explicitly in the
previous subsection as a subquotient of a parabolically induced representation
from the parabolic subgroup associated to the type $\varsigma\left(\psi\right)$. 
\begin{prop}
\cite{Sch20} \label{prop:Schmidt-BQ} Let $G=SO_{5}$, $\varsigma\in\left\{ (\mathbf{B}),(\mathbf{Q})\right\} $
and $\psi\in\Psi_{2}\left(G,\varsigma\right)$. Then for any finite
place $v$, either 
\[
\Pi_{\psi_{v}}=\left\{ \pi_{\psi_{v}}\right\} \qquad\mbox{or}\qquad\Pi_{\psi_{v}}=\left\{ \pi_{\psi_{v}},\pi_{\psi_{v}}^{*}\right\} ,
\]
 in the latter case $\pi_{\psi_{v}}^{*}$ is a subquotient of a parabolically
induced representation from the Borel subgroup $B$ (for both shapes).
\end{prop}

\begin{proof}
Follows from the explicit description of the local $A$-packets of
Schmidt, see \cite[Section 2]{Sch20} for type $(\mathbf{B})$ and
\cite[Section 4]{Sch20} for type $(\mathbf{Q})$. 
\end{proof}

\section{Depth, cohomology, endoscopy and the CSXDH \protect\label{sec:Depth-cohomology}}

In this Section we collect depth and cohomology preservation results
between $A$-parameters and their $A$-packets, and present a reformulation
of the CSXDH in terms of the $A$-shapes of $A$-parameters. Let $G$
be a Gross inner form of a split special odd orthogonal or symplectic
group $G^{*}$ defined over $F$ a global or local field of characteristic
$0$. If $F$ is a non-archimedean local field assume $G=G^{*}$ is
split over $F$, if $F$ is an archimedean local field assume $F=\mathbb{R}$,
and if $F$ is a global field assume $F$ is totally real and $G$
is a Gross inner form of a split classical group. For most (but not
all) of this section $G^{*}=SO_{5}$.

\subsection{Non-archimedean parameters: depth\protect\label{subsec:depth}}

Let $F$ be a non-archimedean local field with residue characteristic
$\ell_{F}$. Let $G$ be a split classical group defined over $F$. 

Recall that the LLC is a conjectural parametrization of $\Pi\left(G\left(F\right)\right)$
by $\Phi\left(G\left(F\right)\right)$, namely, for each $L$-parameter
$\phi\in\Phi\left(G\left(F\right)\right)$ the LLC associates a finite
$L$-packet $\Pi_{\phi}\subset\Pi\left(G\left(F\right)\right)$, which
should satisfy certain expected properties (see for instance \cite{Har22,Kal22}).
In Section~\ref{subsec:Depth-rep-par} we defined notions of depth
for both representations and $L$-parameters. One of the expected
properties the LLC should satisfy (assuming $\ell_{F}$ is sufficiently
large) is the following conjectural depth preservation property, 
\begin{equation}
d(\phi)=d(\pi),\qquad\forall\phi\in\Phi\left(G\left(F\right)\right),\quad\forall\pi\in\Pi_{\phi}.\label{eq:depth-pres}
\end{equation}
In \cite{ABPS16}, the authors studied inner forms of $GL_{n}$ and
$SL_{n}$, and gave both positive and negative results towards \eqref{eq:depth-pres}.
For our application of the SXDH we are interested in the following
weaker version of the depth preservation in the LLC: There exists
a constant $C>0$, which depends only on $G$, such that for $F$
with sufficiently large $\ell_{F}$, the following conjectural weak
depth preservation property holds 
\begin{equation}
\left|d(\phi)-d(\pi)\right|\leq C,\qquad\forall\phi\in\Phi\left(G\left(F\right)\right),\quad\forall\pi\in\Pi_{\phi}.\label{eq:depth-pres-weak}
\end{equation}

The following proposition due to Ganapathy and Varma \cite{GV17},
proves the weak depth preservation property between $L$-parameters
and their $L$-packets of $G\left(F\right)$, namely \eqref{eq:depth-pres-weak},
assuming the residue characteristic is sufficiently large.
\begin{prop}
\label{prop:GV}\cite[(1.0.2), (1.0.4)]{GV17} Assume $\ell_{F}$
$>4\cdot\mathrm{rank}\left(G\right)$. Then
\[
d(\pi)\leq d(\phi)\leq d(\pi)+1,\qquad\forall\phi\in\Phi\left(G\left(F\right)\right),\quad\forall\pi\in\Pi_{\phi}.
\]
\end{prop}

In \cite{Oi22}, Oi strengthened this result and showed the following.
\begin{prop}
\cite{Oi22} Assume $\ell_{F}>4\cdot\mathrm{rank}\left(G\right)$.
Then 
\[
d(\phi)=\max\left\{ d(\pi)\,:\,\pi\in\Pi_{\phi}\right\} ,\qquad\forall\phi\in\Phi\left(G\left(F\right)\right).
\]
In particular, if $\Pi_{\phi}=\left\{ \pi\right\} $ then $d(\phi)=d(\pi)$,
and \eqref{eq:depth-pres} holds in this case.
\end{prop}

As a consequence of the above weak depth preservation property for
$L$-parameters and $L$-packets of Ganapathy and Varma \cite{GV17},
combined with the explicit description of the $A$-packets given by
Schmidt \cite{Sch18,Sch20}, we get the following weak depth preservation
property for $A$-parameters and $A$-packets for $SO_{5}\left(F\right)$.
\begin{prop}
\label{prop:GV-A-SO5} Let $G=SO_{5}$ and assume $\ell_{F}\geq11$.
Then 
\[
d(\pi)\leq d(\psi)\leq d(\pi)+1,\qquad\forall\psi\in\Psi\left(G\left(F\right)\right),\quad\forall\pi\in\Pi_{\psi}.
\]
\end{prop}

\begin{proof}
Recall that since $\psi|_{I_{F}}\equiv\phi_{\psi}|_{I_{F}}$ it follows
from the definitions of $d(\psi)$ and $d(\phi_{\psi})$ that $d(\psi)=d(\phi_{\psi})$.
By Proposition~\ref{prop:Schmidt}, we get that either $\Pi_{\psi}=\Pi_{\phi_{\psi}}$
or $\Pi_{\phi_{\psi}}\subset\Pi_{\psi}\subset(\Pi_{\phi_{\psi}}\cup\Pi_{\phi_{\psi}^{*}})$
with $\phi_{\psi}|_{I_{F}}\equiv\phi_{\psi}^{*}|_{I_{F}}$, and therefore
\[
d(\psi)=d(\phi_{\psi})=d(\phi_{\psi}^{*}).
\]
Hence for any $\pi\in\Pi_{\psi}$, we get $\pi\in\Pi_{\phi}$, for
either $\phi=\phi_{\psi}$ or $\phi_{\psi}^{*}$, and by Proposition~\ref{prop:GV},
\[
d(\pi)\leq d(\phi)\leq d(\pi)+1,
\]
which completes the proof.
\end{proof}

\subsection{Archimedean parameters: cohomology }

Let $F=\mathbb{R}$ be the field of real numbers. Let $G$ be an inner
form of $G^{*}=SO_{5}$. Let $E\in\Pi^{\mathrm{alg}}\left(G\right)$
and $\psi\in\Psi^{\mathrm{AJ}}\left(G\left(\mathbb{R}\right);E\right)$.
Following \cite{AJ87,Art89,NP21}, we describe explicitly the representations
in the packet $\Pi_{\psi}$, and their degrees of $\left(\mathfrak{g}_{0},K\right)$-cohomology,
in terms of the Arthur $SL_{2}$-type of $\psi$. 

We recall some setup from Section~\ref{subsec:Cohomological-rep-par}.
Let $K\leq G\left(\mathbb{R}\right)$ be a maximal compact subgroup
and let $\tilde{X}=G\left(\mathbb{R}\right)/K$ be the corresponding
symmetric space of dimension $d$. Denote by $\theta$ the Cartan
involution such that $K$ is the fix points of $G\left(\mathbb{R}\right)$
under $\theta$, by $G_{\mathbb{C}}$ the complex base change of $G$,
and let $T^{c}\leq B\leq G_{\mathbb{C}}$ be a $\theta$-stable maximal
torus defined and compact over $\mathbb{R}$ and a $\theta$-stable
Borel subgroup defined over $\mathbb{C}$. Let $\mathfrak{g}_{0}$
be the real Lie algebra of $G(\mathbb{R})$ with complexification
$\mathfrak{g}=\mathfrak{g}_{0}\otimes_{\mathbb{R}}\mathbb{C}$, and
let $E\in\Pi^{\mathrm{alg}}\left(G\right)$ with highest weight~$\lambda_{E}$.
For any $\psi\in\Psi^{\mathrm{AJ}}\left(G\left(\mathbb{R}\right);E\right)$,
define the $\left(\mathfrak{g}_{0},K\right)$-cohomology with coefficients
in $E$, by 
\[
H^{*}\left(\psi;E\right):=\bigoplus_{i=0}^{d}H^{i}\left(\psi;E\right),\qquad H^{i}\left(\psi;E\right):=\bigoplus_{\pi\in\Pi_{\psi}}H^{i}\left(\pi;E\right).
\]
Recall from Subsection~\ref{subsec:Cohomological-rep-par}, that
$H^{i}\left(\pi;E\right)=H^{i}\left(\mathfrak{g}_{0},K;\pi\otimes E^{\vee}\right)$.
When $E=\mathbb{C}$ is the trivial representation, we have $\lambda_{E}=0$,
and we shall use the abbreviations $H^{*}\left(\pi\right)=H^{*}\left(\pi;\mathbb{C}\right)$
and $H^{*}\left(\psi\right)=H^{*}\left(\psi;\mathbb{C}\right)$. 

We begin by recalling two results that will prove useful in explicit
computations.
\begin{prop}
\cite[Theorem 14]{NP21} \label{prop:NP-uniform} For any $E\in\Pi^{\mathrm{alg}}\left(G\right)$
and any $\psi\in\Psi^{\mathrm{AJ}}\left(G\left(\mathbb{R}\right);E\right)$,
\[
\dim H^{*}\left(\psi;E\right)=2^{\mathrm{rank}\left(G\right)-\mathrm{rank}\left(K\right)}|\mathcal{W}\left(K,T^{c}\left(\mathbb{R}\right)\right)\backslash\mathcal{W}\left(G\left(\mathbb{R}\right),T^{c}\left(\mathbb{R}\right)\right)^{\theta}|.
\]
\end{prop}

Here, $\mathcal{W}\left(H,T\right)$ is the Weyl group of the real
reductive group $H$ with respect to the maximal torus $T$.

For a $B$-standard $\theta$-stable parabolic subalgebra $\mathfrak{q}=\mathfrak{l}\oplus\mathfrak{u}\in\mathcal{Q}$,
recall the definition of the cohomological representation $A_{\mathfrak{q}}(\lambda_{E})\in\Pi^{\mathrm{coh}}\left(G\left(\mathbb{R}\right);E\right)$
as the cohomological induction of $E/\mathfrak{u}E$ of degree $\dim(\mathfrak{u}\cap\mathfrak{k})$,
and which are the constituents of the $AJ$-packets.
\begin{prop}
\cite{VZ84,BW00}\label{prop:Compute=000020coho} Let $\mathfrak{q}=\mathfrak{l}\oplus\mathfrak{u}$
be a $\theta$-stable parabolic subalgebra and let $\mathfrak{g}=\mathfrak{k}\oplus\mathfrak{p}$
be the Cartan decomposition. 
\begin{enumerate}
\item The smallest degree i such that $H^{i}\left(\mathfrak{g}_{0},K;A_{\mathfrak{q}}\left(\lambda_{E}\right)\otimes E\right)\neq0$
is $i=R(\mathfrak{q})=\dim\mathfrak{u\cap p}.$
\item The cohomology satisfies Poincaré duality with respect to $d=\dim\mathfrak{p},$
i.e. 
\[
H^{i}\left(\mathfrak{g}_{0},K;A_{\mathfrak{q}}\left(\lambda_{E}\right)\otimes E\right)=H^{d-i}\left(\mathfrak{g}_{0},K;A_{\mathfrak{q}}\left(\lambda_{E}\right)\otimes E\right).
\]
\end{enumerate}
\end{prop}

In our situation, the Lie group $G\left(\mathbb{R}\right)$ is isomorphic
to either the split form $SO\left(3,2\right)\cong G^{*}\left(\mathbb{R}\right)$,
the hyperbolic form $SO\left(1,4\right)$, or the compact form $SO\left(5\right)=SO\left(5,0\right)$.
Then $K=S\left(O\left(3\right)\times O\left(2\right)\right)$ in the
split case, $K=S\left(O\left(1\right)\times O\left(4\right)\right)\cong O\left(4\right)$
in the hyperbolic case, $K=SO\left(5\right)=G\left(\mathbb{R}\right)$
in the compact case, and
\[
d:=\dim G\left(\mathbb{R}\right)-\dim K=\begin{cases}
6 & \text{split}\\
4 & \text{hyperbolic}\\
0 & \text{compact.}
\end{cases}
\]

We now give an explicit description of the cohomological $A$-packets.
The subgroups $T^{c}\leq B\leq G_{\mathbb{C}}$ were introduced above.
Let $\hat{T}\leq\hat{B}\leq\hat{G}$ be the dual maximal torus and
Borel subgroup of the dual group $\hat{G}$. 

Given a parameter $\psi\in\Psi^{\mathrm{AJ}}\left(G\left(\mathbb{R}\right);E\right)$,
the Arthur $SL_{2}$-type of $\psi$ determines a standard Levi subgroup
$\hat{L}\subset\hat{G}$, for which it is principal as explained in
Section~\ref{subsec:Arthur-shapes-types}. Proposition~\ref{prop:AJ}
then states that the packet $\Pi_{\psi}$ is in bijection with the
set of $G\left(\mathbb{R}\right)$-conjugacy classes of $\theta$-stable
parabolic subalgebras $\mathfrak{q}=\mathfrak{l}\oplus\mathfrak{u}$
such that the dual of $L=\mathrm{Stab}_{G\left(\mathbb{R}\right)}(\mathfrak{q})$
is identified with $\hat{L}$ under the data defining $\hat{G}$.
Moreover, as explained in Section~\ref{subsec:Cohomological-rep-par},
this set admits a combinatorial description: it is in bijection with
\begin{equation}
\mathcal{W}\left(K,T^{c}\left(\mathbb{R}\right)\right)\backslash\mathcal{W}\left(G\left(\mathbb{R}\right),T^{c}\left(\mathbb{R}\right)\right)/\mathcal{W}(\hat{L},\hat{T}),\label{eq:Weyl-Group-Cosets-2}
\end{equation}
where $\mathcal{W}\left(G\left(\mathbb{R}\right),T^{c}\left(\mathbb{R}\right)\right)$
is identified with $\mathcal{W}(\hat{G},\hat{T})$. The attentive
reader will notice that a requirement of $\theta$-equivariance is
missing in comparison with \eqref{eq:Weyl-Group-Cosets}; it is automatically
satisfied since $T^{c}$ is compact over $\mathbb{R}$. 

Recall from \eqref{eq:SO5-Arthur-SL2} that the $A$-parameters $\psi\in\Psi\left(G\left(\mathbb{R}\right)\right)$
can admit four different possible Arthur $SL_{2}$-types, and that
the $A$-shape determines the Arthur $SL_{2}$-type. We reproduce
the information from \eqref{eq:SO5-Arthur-SL2} and Table \ref{tab:sp4-orbits},
namely the correspondence between partitions and nilpotent orbits,
the associated $A$-shapes and their corresponding Levi subgroups
denoted $\hat{G},$ $\hat{S}$, $\hat{M}$, and $\hat{T}$ in Table~\ref{tab:sp4-orbits-copy}. 

\begin{table}[tbh]
\caption{\protect\label{tab:sp4-orbits-copy}Complex nilpotent orbits for $\hat{G}=Sp_{4}\left(\mathbb{C}\right)$
their associated $A$-shapes and the conjugacy classes of Levi subgroups
on which the orbits are principal.}

\centering{}%
\begin{tabular}{llllll}
\toprule 
\multicolumn{2}{l}{Orbit} & Shapes &  & \multicolumn{2}{l}{Levi subgroup}\tabularnewline
\midrule
$\sigma_{\mathrm{princ}}$ & $=\nu(4)$ & $(\mathbf{F})$ &  & $\hat{G}$ & \tabularnewline
$\sigma_{\mathrm{subreg}}$ & $=\nu(2)^{2}$ & $(\mathbf{B})$, $(\mathbf{Q})$ &  & $\hat{S}\cong GL_{2}$ & (Siegel)\tabularnewline
$\sigma_{\mathrm{min}}$ & $=\nu(2)\oplus\nu(1)^{2}$ & $(\mathbf{P})$ &  & $\hat{M}\cong Sp_{2}\times GL_{1}$ & (Klingen)\tabularnewline
$\sigma_{\mathrm{triv}}$ & $=\nu(1)^{4}$ & $(\mathbf{G})$, $(\mathbf{Y})$ &  & $\hat{T}$ & (Borel)\tabularnewline
\bottomrule
\end{tabular}
\end{table}

In the notation of Subsection~\ref{subsec:Cohomological-rep-par},
we have $\mathcal{L}=\{\hat{T},\hat{M},\hat{S},\hat{G}$\}. For each
$E\in\Pi^{\mathrm{alg}}\left(G\right)$ with highest weight $\lambda_{E}\in X_{*}(\hat{T})$,
we also have 
\begin{equation}
\mathcal{L}_{E}=\left\{ \hat{L}\in\mathcal{L}\,:\,\langle\lambda_{E},\alpha\rangle=0\quad\forall\alpha\in\Delta^{*}(\hat{L},\hat{B},\hat{T})\right\} .\label{eq:Admissible-Levis}
\end{equation}
By Proposition~\ref{prop:AJ} the set $\Psi^{\mathrm{AJ}}\left(G\left(\mathbb{R}\right);E\right)$
is indexed by $\mathcal{L}_{E}$. In particular:
\begin{lem}
Let $e_{1},e_{2}\in X_{*}(\hat{T})$ be the standard basis. Let $E\in\Pi^{\mathrm{alg}}\left(G\right)$.
Then:
\begin{enumerate}
\item $\hat{T}\in\mathcal{L}_{E}$ for all $E$. 
\item $\hat{M}\in\mathcal{L}_{E}$ if and only if $\lambda_{E}=ne_{1}$
for some $n\in\mathbb{N}$. 
\item $\hat{S}\in\mathcal{L}_{E}$ if and only if $\lambda_{E}=n\left(e_{1}+e_{2}\right)$
for some $n\in\mathbb{N}$. 
\item $\hat{G}\in\mathcal{L}_{E}$ if and only if $E$ is the trivial representation. 
\end{enumerate}
\end{lem}

\begin{proof}
From \eqref{eq:Admissible-Levis} we have that $\hat{L}\in\mathcal{L}_{E}$
if and only if $\langle\lambda_{E},\alpha\rangle=0$ for all positive
roots $\alpha$ of $\hat{T}$ in $\hat{L}$. Standard Levi subgroups
are determined by subsets of the simple roots $\Delta^{*}(\hat{G},\hat{B},\hat{T})$:
the corresponding root spaces are contained in $\hat{L}$. Let $e_{1}^{\vee},e_{2}^{\vee}\in X^{*}(\hat{T})$
be the standard dual basis, and note that in this case, $\Delta^{*}(\hat{G},\hat{B},\hat{T})=\left\{ e_{2}^{\vee},\frac{-1}{2}e_{1}^{\vee}+\frac{1}{2}e_{2}^{\vee}\right\} $
and the correspondence is 
\[
\hat{T}\leftrightarrow\emptyset,\quad\hat{M}\leftrightarrow\left\{ e_{2}^{\vee}\right\} \quad\hat{S}\leftrightarrow\left\{ \frac{-1}{2}e_{1}^{\vee}+\frac{1}{2}e_{2}^{\vee}\right\} \quad\hat{G}\leftrightarrow\Delta^{*}(\hat{G},\hat{B},\hat{T}).
\]
 The result follows immediately. 
\end{proof}
\begin{prop}
\label{prop:coh-deg-par} Let $G$ be an inner form of $SO_{5}$ defined
over $\mathbb{R}$. Let $E\in\Pi^{\mathrm{alg}}\left(G\left(\mathbb{R}\right)\right)$,
let $\hat{L}\in\mathcal{L}_{E}$ and $\sigma\in\mathcal{D}\left(G\right)$
be such that $\sigma$ is principal in $\hat{L}$. Then for all parameters
$\psi\in\Psi^{\mathrm{AJ}}\left(G\left(\mathbb{R}\right);E\right)$
of Arthur $SL_{2}$-type $\sigma$, the $A$-packet $\Pi_{\psi}$
and the corresponding cohomology degrees are as follows:
\begin{enumerate}
\item If $G\left(\mathbb{R}\right)$ is the split form $SO\left(3,2\right)$
then the $A$-packet are as follows: 
\begin{enumerate}
\item If $\sigma=\sigma_{\mathrm{triv}}$, then $\Pi_{\psi}$ consists of
four discrete series representations, each contributing one dimension
of cohomology in degree $3$. 
\item If $\sigma=\sigma_{\mathrm{min}}$, then $\Pi_{\psi}$ consists of
three representations: Two discrete series, each contributing one
dimension of cohomology in degree $3$, and one non-tempered representation
which contributes one dimension in degrees $2$ and $4$.
\item If $\sigma=\sigma_{\mathrm{subreg}}$, then $\Pi_{\psi}$ consists
of two non-tempered representations, each contributing one dimension
of cohomology in degrees $2$ and $4$. 
\item If $\sigma=\sigma_{\mathrm{princ}}$, then $\Pi_{\psi}$ contains
only the finite-dimensional representation $E$, which is necessarily
trivial. It contributes one dimension of cohomology in each of the
degrees $0,2,4$ and $6$. 
\end{enumerate}
\item If $G\left(\mathbb{R}\right)$ is the hyperbolic form $SO\left(1,4\right)$,
the $A$-packets are as follows: 
\begin{enumerate}
\item If $\sigma=\sigma_{\mathrm{triv}}$, then $\Pi_{\psi}$ consists of
two discrete series, each contributing one dimension of cohomology
in degree $2$.
\item If $\sigma=\sigma_{\mathrm{min}}$, then $\Pi_{\psi}$ consists of
a unique non-tempered representation, which contributes one dimension
of cohomology in each of the degrees $1$ and $3$. 
\item If $\sigma=\sigma_{\mathrm{subreg}}$, then $\Pi_{\psi}$ coincides
with the $\nu(1)^{4}$ packet associated to the same $E$.
\item If $\sigma=\sigma_{\mathrm{princ}}$, then $\Pi_{\psi}$ contains
only the finite-dimensional representation $E$, which is necessarily
trivial. It contributes one dimension of cohomology in each of the
degrees $0$ and $4$. 
\end{enumerate}
\item If $G\left(\mathbb{R}\right)$ is the compact form $SO\left(5\right)$,
all $A$-packets are identical and consist of the representation $E$,
which has one dimension of cohomology in degree $0$.
\end{enumerate}
\end{prop}

\begin{proof}
First we compute the cardinalities of the $A$-packets. We write $\mathcal{W}\left(G\left(\mathbb{R}\right),T^{c}\left(\mathbb{R}\right)\right)=\langle\sigma_{1},\sigma_{2}\mid\sigma_{1}^{2}=\sigma_{2}^{2}=(\sigma_{1}\sigma_{2})^{2}=1\rangle$,
where $\sigma_{1}$ and $\sigma_{2}$ are respectively reflections
across the short and long roots in the root system of $SO_{5}$. We
caution the reader that the notions of ``short'' and ``long''
roots are swapped in the duality between the root systems of $SO_{5}$
and $Sp_{4}$. The subgroup $\mathcal{W}\left(K,T^{c}\left(\mathbb{R}\right)\right)$
depends on the inner form of $G$: we have 
\begin{equation}
SO\left(5\right)\leftrightarrow\text{full group},\qquad S\left(O\left(1\right)\times O\left(4\right)\right)\leftrightarrow\langle\sigma_{1},\sigma_{2}\sigma_{1}\sigma_{2}\rangle,\qquad S\left(O\left(3\right)\times O\left(2\right)\right)\leftrightarrow\langle\sigma_{1}\rangle.\label{eq:compact-Weyl-groups}
\end{equation}
As for the choice of Levi subgroups $\hat{L}$, the corresponding
subgroups of the Weyl group are: 
\[
\text{\ensuremath{\hat{G}}}\leftrightarrow\text{full group,}\quad\text{\ensuremath{\hat{S}}}\leftrightarrow\text{\ensuremath{\langle\sigma_{2}\rangle},\ensuremath{\quad\hat{M}\leftrightarrow}\ensuremath{\langle\sigma_{1}\rangle,}\ensuremath{\quad\text{\ensuremath{\hat{T}}}\leftrightarrow\text{id}}}.
\]
The cardinalities are then readily computed from \eqref{eq:Weyl-Group-Cosets-2}.

For the cohomology, we will spell out two cases: the minimal Arthur
$SL_{2}$-type $\sigma_{\mathrm{min}}=\nu(2)\oplus\nu(1)^{2}$ for
the split form $SO\left(3,2\right)$, and the subregular Arthur $SL_{2}$-type
$\sigma_{\mathrm{subreg}}=\nu(2)^{2}$ for the hyperbolic form $SO\left(1,4\right)$.
All the other cases are similar and simpler. 

We begin with the minimal case $\sigma=\sigma_{\mathrm{min}}$ for
the split form $G\left(\mathbb{R}\right)=SO\left(3,2\right)$. Recall
that we have fixed a compact Cartan subgroup $T^{c}\left(\mathbb{R}\right)$
with complexified Lie algebra $\mathfrak{t}$ and a dual torus $\hat{T}\subset\hat{M}$.
The subgroup $\hat{M}$ is isomorphic to $GL_{1}\times SL_{2}\subset Sp_{4}=\hat{G}$,
it contains a long root of $Sp_{4}$. Consequently, in the notation
of Subsection~\ref{subsec:Cohomological-rep-par}, the Levi components
of the $\theta$-stable parabolic subalgebras $\mathfrak{q}\in\Sigma^{-1}(\hat{M})$
must contain the root space associated to a short root of $SO_{5}$
for the torus $T^{c}$. 

Let us recall the structure of the complexified Lie algebra $\mathfrak{g}$
of $G\left(\mathbb{R}\right)$: it is $10$-dimensional with a $2$-dimensional
torus $\mathfrak{t}$ and admits a Cartan decomposition 
\[
\mathfrak{g}=\mathfrak{k}\oplus\mathfrak{p},\qquad\dim\mathfrak{k}=4,\quad\dim\mathfrak{p}=6
\]
such that $\mathfrak{t}\subset\mathfrak{k}\simeq\mathfrak{so}(3)\times\mathfrak{so}(2)$.
Now write $\mathfrak{t}\simeq\mathbb{C}^{2}$ in coordinates $(a_{1},a_{2})$
with weights $e_{1},e_{2}$ such that $e_{i}(a_{j})=\delta_{i=j}$.
In this context, the short roots are $\pm e_{i}$ and the long roots
are $\pm(e_{1}\pm e_{2})$. We fix the basis so that $(a_{1},0)\subset\mathfrak{so}(3)$
and $(0,a_{2})\subset\mathfrak{so}(2)$. It follows that in the root
space decomposition 
\[
\mathfrak{g}=\mathfrak{t}\oplus\bigoplus_{\alpha}\mathfrak{g}_{\alpha},
\]
 we have $\mathfrak{g}_{\pm e_{1}}\subset\mathfrak{k},$ and all the
other root spaces are subspaces of $\mathfrak{p}$. 

Now, we have established that our Lie algebras $\mathfrak{q=\mathfrak{l}}\oplus\mathfrak{u}\in\Sigma^{-1}(\hat{M})$
must be such that $\mathfrak{l}$ contains a short root space and
its opposite. From \cite[p. 277]{AJ87} every $\theta$-stable parabolic
subalgebra is of the form $\mathfrak{q}_{\lambda}=\mathfrak{t}\oplus\bigoplus_{\langle\alpha,\lambda\rangle\geq0}\mathfrak{g_{\alpha}}$
for some weight $\lambda\in\mathfrak{t}^{*}$. It follows that in
this case, the four possible $\theta$-stable parabolic subalgebras
are $\mathfrak{q}_{\pm e_{i}}$ for $i=1,2$. We know by the previous
argument that $|\Pi_{\psi}|=3$, which means that two of these are
conjugate under~$K$, or, equivalently, under $\mathcal{W}\left(K,T^{c}\left(\mathbb{R}\right)\right)$.
Since $\mathcal{W}\left(K,T^{c}\left(\mathbb{R}\right)\right)\simeq\mathbb{Z}/2\mathbb{Z}$
sends $e_{1}$ to $-e_{1}$, we conclude that $\mathfrak{q}_{e_{1}}\sim\mathfrak{q}_{-e_{1}}$,
and that the three distinct conjugacy classes of $\theta$-stable
parabolic subalgebras are $\mathfrak{q}_{e_{1}},\mathfrak{q}_{e_{2}},\mathfrak{q}_{-e_{2}}$.
From this explicit description and using Propositions~\ref{prop:NP-uniform}
and~\ref{prop:Compute=000020coho}, we compute that $A_{\mathfrak{q}_{e_{1}}}(\lambda_{E})$
has cohomology in degrees 2 and 4, and that the other two representations
have cohomology in degree 3. 

Now consider the subregular case $\sigma=\sigma_{\mathrm{subreg}}$
for the hyperbolic form $G\left(\mathbb{R}\right)=SO\left(1,4\right)$.
Again we have a compact maximal torus $T^{c}\left(\mathbb{R}\right)$
with complexified Lie algebra $\mathfrak{t}$. In this case $\hat{S}$
contains a short root, so the Levi subgroups for the elements of $\Sigma^{-1}(\hat{S})$
must contain a long one. We keep the same coordinates for the torus,
but different roots will be compact. Indeed, the Lie algebra $\mathfrak{k}$
of the maximal subgroup is the 6-dimensional $\mathfrak{so}(4)$ and
it contains the root spaces associated to $\pm(e_{i}\pm e_{j}$);
the short roots $\pm e_{i}$ are non-compact. The four possible Lie
algebras whose Levi subgroups contains a long root are then $\mathfrak{q}_{\pm(e_{i}\pm e_{j})}$.
We know from our computations of cardinalities that the $A$-packet
contains only two representations, which means that these four $\theta$-stable
parabolic Lie algebras account only for two $\mathcal{W}\left(K,T^{c}\left(\mathbb{R}\right)\right)$
orbits. We established in \eqref{eq:compact-Weyl-groups} that $\mathcal{W}\left(K,T^{c}\left(\mathbb{R}\right)\right)\simeq(\mathbb{Z}/2\mathbb{Z})^{2}$
permutes the roots $\pm(e_{1}+e_{2})$ and $\pm(e_{1}-e_{2})$ respectively.
It follows that $\mathfrak{q}_{e_{1}+e_{2}}$ and $\mathfrak{q}_{e_{1}-e_{2}}$
are two representatives for $\Sigma^{-1}(\hat{S})$. Since their Levi
components are entirely contained in $\mathfrak{k}$, the associated
representations are discrete series \cite[p. 58]{VZ84}. Since the
discrete series $A$-packet $\Pi_{\nu(1)^{4}}$ contains all the discrete
series with the same infinitesimal character as $E$, we conclude
that $\Pi_{\nu(1)^{4}}=\Pi_{\nu(2)^{2}}$.
\end{proof}

\subsection{Global parameters: depth, cohomology and endoscopy\protect\label{subsec:=000020Depth-cohom-endos}}

Let $G$ be a Gross inner form of $G^{*}$, a split special odd orthogonal
or symplectic group, defined over $F$, a totally real number field.
\begin{defn}
\label{def:h(pi,q,E)} Let $E\in\Pi^{\mathrm{alg}}\left(G\right)$
and $q\lhd\mathcal{O}$. Define for any $\pi\in\Pi\left(G\left(\mathbb{A}\right)\right)$,
\[
H^{*}\left(\pi;q,E\right):=H^{*}\left(\pi_{\infty};E\right)\otimes\bigotimes_{v\mid q}\pi_{v}^{K_{v}\left(q\right)}\otimes\bigotimes_{v\nmid q\infty}\pi_{v}^{K_{v}},\qquad h\left(\pi;q,E\right):=\dim H^{*}\left(\pi;q,E\right),
\]
and define for any discrete $A$-parameter $\psi\in\Psi_{2}\left(G\right)$,
\[
H^{*}\left(\psi;q,E\right)=\bigoplus_{\pi\in\Pi_{\psi}(\epsilon_{\psi})}H^{*}\left(\pi;q,E\right),\qquad h\left(\psi;q,E\right):=\sum_{\pi\in\Pi_{\psi}(\epsilon_{\psi})}h\left(\pi;q,E\right).
\]
Define the following subset of global $A$-parameters
\[
\Psi_{2}^{\mathrm{AJ}}\left(G;q,E\right):=\left\{ \psi\in\Psi_{2}\left(G\right)\,:\,\begin{array}{lc}
\psi_{v}\in\Psi^{\mathrm{AJ}}\left(G_{v};E\right) & \forall v\mid\infty\\
d(\psi_{v})\leq\mathrm{ord}_{v}(q)+1 & \forall v\mid q\\
\psi_{v}\in\Psi^{\mathrm{ur}}\left(G_{v}\right) & \forall v\nmid q\infty
\end{array}\right\} .
\]
When $E=\mathbb{C}$, the trivial representation, we abbreviate $H^{*}\left(\pi;q\right):=H^{*}\left(\pi;q,\mathbb{C}\right)$,
$h\left(\pi;q\right)=h\left(\pi;q,\mathbb{C}\right)$, $H^{*}\left(\psi;q\right):=H^{*}\left(\psi;q,\mathbb{C}\right)$,
$h\left(\psi;q\right):=h\left(\psi;q,\mathbb{C}\right)$ and $\Psi_{2}^{\mathrm{AJ}}\left(G;q\right):=\Psi_{2}^{\mathrm{AJ}}\left(G;q,\mathbb{C}\right)$.
\end{defn}

The following proposition summarizes the cohomology and depth preservation
properties between $A$-parameters and the members of their $A$-packet.
\begin{prop}
\label{prop:coh-dep-A-par-pac} Fix $E\in\Pi^{\mathrm{alg}}\left(G\right)$.
For $q=\prod_{v\mid q}\varpi_{v}^{\mathrm{ord}_{v}(q)}$ denote $s(q)=\prod_{v\mid q}\varpi_{v}$.
Let $\psi\in\Psi_{2}\left(G\right)$ and let $\pi\in\Pi_{\psi}$.
Then:
\begin{enumerate}
\item If $h\left(\pi;q,E\right)\ne0$, then $\psi\in\Psi_{2}^{\mathrm{AJ}}\left(G;q,E\right)$.
\item If $\psi\in\Psi_{2}^{\mathrm{AJ}}\left(G;q,E\right)$ and $\pi_{v}^{K_{v}}\ne0$
for any $v\nmid q\infty$, then $h\left(\pi;s(q)^{2}q,E\right)\ne0$.
\end{enumerate}
\end{prop}

\begin{proof}
For $v\mid\infty$, by Proposition~\ref{prop:AJ}, $H^{*}\left(\pi_{v};E\right)\ne0$
if and only if $\psi_{v}\in\Psi^{\mathrm{AJ}}\left(G_{v};E\right)$.
For $v\mid q$, by Proposition~\ref{prop:GV-A-SO5} and Lemma~\ref{lem:depth-level},
$\pi_{v}^{K_{v}(\varpi_{v}^{\mathrm{ord}_{v}(q)})}\ne\left\{ 0\right\} $
implies $d(\psi_{v})\leq\mbox{ord}_{v}(q)+1$, and $d(\psi_{v})\leq\mathrm{ord}_{v}(q)+1$
implies $\pi_{v}^{K_{v}(\varpi_{v}^{\mathrm{ord}_{v}(q)+2})}\ne\left\{ 0\right\} $.
For $v\nmid q\infty$, $K_{v}\left(q\right)=K_{v}$, and by Proposition~\ref{prop:Moeglin-A-ur},
$\pi^{K_{v}}\ne\left\{ 0\right\} $ implies $\psi_{v}\in\Psi^{\mathrm{ur}}\left(G_{v}\right)$.
Combining it all we get both claims.
\end{proof}
We call the following statement the Arthur--Matsushima decomposition,
whose proof combines Matsushima's formula together with Arthur's classification.
\begin{thm}
\label{thm:Arthur-Matsushima} For any $E\in\Pi^{\mathrm{alg}}\left(G\right)$
and $q\lhd\mathcal{O}$, 
\begin{equation}
H_{(2)}^{*}\left(X\left(q\right);E\right)\cong\bigoplus_{\psi\in\Psi_{2}^{\mathrm{AJ}}\left(G;q,E\right)}\bigoplus_{\pi\in\Pi_{\psi}(\epsilon_{\psi})}H^{*}\left(\pi;q,E\right).\label{eq:Arthur-Matsushima}
\end{equation}
\end{thm}

\begin{proof}
By Matsushima's formula~\eqref{eq:Matsushima}, 
\[
H_{(2)}^{*}\left(X\left(q\right);E\right)\cong\bigoplus_{\pi\in\Pi\left(G\left(\mathbb{A}\right)\right)}m\left(\pi\right)\cdot H^{*}\left(\pi;q,E\right).
\]
By Arthur classification Theorem~\ref{thm:Arthur=0000201.5.2}, $m\left(\pi\right)=1$
if $\pi\in\Pi_{\psi}(\epsilon_{\psi})$ for a unique $\psi\in\Psi_{2}\left(G\right)$,
and $m\left(\pi\right)=0$ otherwise. Therefore 
\[
H_{(2)}^{*}\left(X\left(q\right);E\right)\cong\bigoplus_{\psi\in\Psi_{2}\left(G\right)}\bigoplus_{\pi\in\Pi_{\psi}(\epsilon_{\psi})}H^{*}\left(\pi;q,E\right).
\]
By the first claim of Proposition~\ref{prop:coh-dep-A-par-pac},
\[
H^{*}\left(\pi;q,E\right)\ne0\qquad\Rightarrow\qquad\psi\in\Psi_{2}^{\mathrm{AJ}}\left(G;q,E\right),
\]
which completes the proof.
\end{proof}
\begin{defn}
\label{def:h-shape} We decompose $\Psi_{2}^{\mathrm{AJ}}\left(G;q,E\right)$
using the $A$-shapes, namely, for any $\varsigma\in\mathcal{M}\left(G\right)$,
denote 
\[
\Psi_{2}^{\mathrm{AJ}}\left(G,\varsigma;q,E\right):=\left\{ \psi\in\Psi_{2}^{\mathrm{AJ}}\left(G;q,E\right)\,:\,\varsigma(\psi)=\varsigma\right\} ,
\]
and 
\[
h\left(G,\varsigma;q,E\right):=\sum_{\psi\in\Psi_{2}^{\mathrm{AJ}}\left(G,\varsigma;q,E\right)}h\left(\psi;q,E\right)=\sum_{\psi\in\Psi_{2}^{\mathrm{AJ}}\left(G,\varsigma;q,E\right)}\sum_{\pi\in\Pi_{\psi}(\epsilon_{\psi})}h\left(\pi;q,E\right).
\]
When $E=\mathbb{C}$, the trivial representation, we abbreviate $\Psi_{2}^{\mathrm{AJ}}\left(G,\varsigma;q\right):=\Psi_{2}^{\mathrm{AJ}}\left(G,\varsigma;q,\mathbb{C}\right)$
and $h\left(G,\varsigma;q\right):=h\left(G,\varsigma;q,\mathbb{C}\right)$.
\end{defn}

\subsection{Restatement of the CSXDH\protect\label{subsec:Restatement-CSXDH}}

Our goal is to provide upper bounds on $h\left(G,\varsigma;q,E\right)$,
for fixed $E$ and $|q|\rightarrow\infty$, in terms of the invariant
$r(\varsigma)$ defined as follows.
\begin{defn}
\label{def:r-shape} Let $\varsigma\in\mathcal{M}\left(G\right)$.
Then its \emph{worst-case rate of decay of matrix coefficients} is
\[
r\left(\varsigma\right):=\sup_{\psi\in\Psi_{2}^{\mathrm{AJ}}\left(G,\varsigma\right)}\sup_{\pi\in\Pi_{\psi}}\sup_{v}r\left(\pi_{v}\right),
\]
where $r(\pi_{v})$ was defined in \eqref{eq:r-def}, and $v$ runs
over all places except for the finitely many finite places $v$ for
which $\pi_{v}$ is ramified. 
\end{defn}

We now state a stronger version of the Cohomological Sarnak--Xue
Density Hypothesis, in terms of $A$-shapes.
\begin{conjecture}[CSXDH-shape]
\label{conj:CSXDH-shape} Fix $E\in\Pi^{\mathrm{alg}}\left(G\right)$.
Then for any $\varsigma\in\mathcal{M}\left(G\right)$, 
\begin{equation}
h\left(G,\varsigma;q,E\right)\ll\vol\left(X\left(q\right)\right)^{\frac{2}{r(\varsigma)}}.\label{eq:CSXDH-shape}
\end{equation}
\end{conjecture}

We expect Conjecture~\ref{conj:CSXDH-shape} to hold for any connected
reductive group over a global field. However we note that its formulation
requires a classification of the automorphic representations of that
group in terms of $A$-parameters, or at least associating an $A$-shape
to each such automorphic representation.
\begin{prop}
\label{prop:CSXDH-shape->=00005Cinfty} Conjecture~\ref{conj:CSXDH-shape}
implies Conjecture~\ref{conj:CSXDH-=00005Cinfty} (CSXDH).
\end{prop}

\begin{proof}
Let $\pi_{\infty}\in\Pi^{\mathrm{coh}}\left(G_{\infty};E\right)$.
By the strong approximation property $\Gamma\left(q\right)\backslash G_{\infty}$
embeds as a $G_{\infty}$-set in $G\left(F\right)\backslash G\left(\mathbb{A}\right)/K_{f}\left(q\right)$.
Hence
\begin{multline*}
m\left(\pi_{\infty};q\right):=\dim\mbox{Hom}_{G(\mathbb{R})}\left(\pi_{\infty},L^{2}\left(\Gamma\left(q\right)\backslash G_{\infty}\right)\right)\\
\leq\dim\mbox{Hom}_{G(\mathbb{R})}\left(\pi_{\infty},L^{2}\left(G\left(F\right)\backslash G\left(\mathbb{A}\right)\right)^{K_{f}(q)}\right)=\sum_{\pi_{f}}m\left(\pi_{\infty}\otimes\pi_{f}\right)\dim\pi_{f}^{K_{f}(q)}.
\end{multline*}
Since $h\left(\pi_{\infty}\otimes\pi_{f};q,E\right)=\dim H\left(\pi_{\infty};E\right)\cdot\dim\pi_{f}^{K_{f}(q)}$
and $\dim H\left(\pi_{\infty};E\right)\geq1$, the above expression
is bounded by
\[
\leq\sum_{\pi_{f}}m\left(\pi_{\infty}\otimes\pi_{f}\right)h\left(\pi_{\infty}\otimes\pi_{f};q,E\right)=:\tilde{m}\left(\pi_{\infty};q\right).
\]
By Theorem~\ref{thm:Arthur=0000201.5.2}, if $\pi_{f}\in\Pi\left(G\left(\mathbb{A}_{f}\right)\right)$
is such that $m\left(\pi_{\infty}\otimes\pi_{f}\right)\ne0$ then
$m\left(\pi_{\infty}\otimes\pi_{f}\right)=1$ and there exists a unique
$\psi\in\Psi_{2}\left(G\right)$ such that $\pi_{\infty}\otimes\pi_{f}\in\Pi_{\psi}(\epsilon_{\psi})$.
By Proposition~\ref{prop:coh-dep-A-par-pac}, if $\pi\in\Pi_{\psi}$
is such that $h\left(\pi;q,E\right)\ne0$ then $\psi\in\Psi_{2}^{\mathrm{AJ}}\left(G;q,E\right)$.
Denote $\mathcal{M}\left(G;\pi_{\infty}\right):=\left\{ \varsigma\in\mathcal{M}\left(G\right)\,:\,\exists\psi\in\Psi_{2}\left(G\right)\text{ with }\varsigma=\varsigma(\psi),\,\pi_{\infty}\in\Pi_{\psi_{\infty}}\right\} $
and note that if $\pi_{\infty}\otimes\pi_{f}\in\Pi_{\psi}$ then $\varsigma\left(\psi\right)\in\mathcal{M}\left(G;\pi_{\infty}\right)$.
Therefore
\[
\tilde{m}\left(\pi_{\infty};q\right)\leq\sum_{\varsigma\in\mathcal{M}\left(G;\pi_{\infty}\right)}\sum_{\psi\in\Psi_{2}^{\mathrm{AJ}}\left(G,\varsigma;q,E\right)}\sum_{\pi\in\Pi_{\psi}(\epsilon_{\psi})}h\left(\pi;q,E\right)=\sum_{\varsigma\in\mathcal{M}\left(G,\pi_{\infty}\right)}h\left(G,\varsigma;q,E\right).
\]
By definition $r\left(\varsigma\right)\geq r\left(\pi_{\infty}\right)$
for any $\varsigma\in\mathcal{M}\left(G;\pi_{\infty}\right)$ and
note that $|\mathcal{M}\left(G;\pi_{\infty}\right)|$ is bounded since
$|\mathcal{M}(G)|$ is bounded. By Conjecture~\ref{conj:CSXDH-shape},
we get 
\[
m\left(\pi_{\infty};q\right)\leq\sum_{\varsigma\in\mathcal{M}\left(G,\pi_{\infty}\right)}h\left(G,\varsigma;q,E\right)\ll\sum_{\varsigma\in\mathcal{M}\left(G,\pi_{\infty}\right)}\vol\left(X\left(q\right)\right)^{\frac{2}{r(\varsigma)}}\ll\vol\left(X\left(q\right)\right)^{\frac{2}{r(\pi_{\infty})}}.
\]
\end{proof}
Conjecture~\ref{conj:CSXDH-shape} is a priori stronger than Conjecture~\ref{conj:CSXDH-=00005Cinfty}
since it considers the worst-case rate of decay of matrix coefficients
among all local factors of the automorphic representations, and not
just the infinite factor. However, assuming some well-believed purity
conjectures which would follow from the Arthur--Ramanujan conjectures
\cite{Art89,Clo02}, the two conjectures should in fact be equivalent.
In fact, in the setting of cohomological representations of classical
groups, the Arthur--Ramanujan conjectures are within reach thanks
to recent advances in the Langlands program, combined with Deligne's
resolution of the Weil conjectures, see for example Theorem~\ref{thm:GRPC}
below.

The main result of our paper is the following.
\begin{thm}
\label{thm:CSXDH-SO5-shape} Let $G$ be a Gross inner form of the
split group $SO_{5}$ defined over a totally real number field. Then
Conjecture~\ref{conj:CSXDH-shape} holds.
\end{thm}

Note that Theorem~\ref{thm:CSXDH-SO5-shape} combined with Proposition~\ref{prop:CSXDH-shape->=00005Cinfty}
implies Theorem~\ref{thm:CSXDH-SO5-=00005Cinfty} from the introduction.
The next two sections will give the proof of Theorem~\ref{thm:CSXDH-SO5-shape}.
More precisely, in Section~\ref{sec:Bounds-on-rate} we give bounds
on the rate of the $A$-shapes of $SO_{5}$, while in Section~\ref{sec:Bounds-on-cohomology}
we give bounds on the cohomological dimension of the $A$-shapes:
combining these two bounds, we get \eqref{eq:CSXDH-shape}. 

\section{Bounds on rate of decay of matrix coefficients $r(\varsigma)$ \protect\label{sec:Bounds-on-rate}}

In this section we give bounds and exact estimates on the worst case
rate of decay of matrix coefficients of an $A$-shape of a Gross inner
form of a split classical group. In order to do this we invoke known
results towards the Generalized Ramanujan--Petersson Conjecture for
cohomological, self-dual, automorphic representations of general linear
groups. Many results in this section applies to any classical (and
even for any connected reductive) groups. The main result of this
section is the following:
\begin{thm}
\label{thm:r-shape} Let $G$ be a Gross inner form of the split classical
group $SO_{5}$ defined over a totally real number field $F$. Then
\[
r(\mathbf{G})=r(\mathbf{Y})=2,\quad r(\mathbf{F})=\infty,\quad r(\mathbf{B})=r(\mathbf{Q})=4,\quad r(\mathbf{P})=3.
\]
\end{thm}

The proof follows below from Corollary~\ref{cor:r-GY} for the $A$-shapes
$(\mathbf{G})$ and $(\mathbf{Y})$, from Corollary~\ref{cor:r-F}
for the $A$-shape $(\mathbf{F})$, from Corollary~\ref{cor:r-exact-BQ}
for the $A$-shapes $(\mathbf{B})$ and $(\mathbf{Q})$, and from
Corollary~\ref{cor:r-exact-P} for the $A$-shape $(\mathbf{P})$. 

\subsection{Bounds for $\left(\mathbf{G}\right)$, $\left(\mathbf{Y}\right)$
and $\left(\mathbf{F}\right)$}

We begin with the generic $A$-shapes $(\mathbf{G})$ and $(\mathbf{Y})$.
Our main tool is the Generalized Ramanujan--Petersson Conjecture
(GRPC). 
\begin{thm}[GRPC]
\cite{Clo13} \label{thm:GRPC} If $\pi$ is a (unitary) cohomological,
self-dual, cuspidal, automorphic representation of $GL_{n}/F$, where
$n\geq2$ and $F$ a totally real field, then $\pi_{v}$ is tempered
at any unramified place $v$.
\end{thm}

This result was first proved by Eichler \cite{Eic54} for $GL_{2}/\mathbb{Q}$
and $\pi_{\infty}\cong D_{2}$, the weight $2$ discrete series representation.
It was later extended by Deligne \cite{Del80} to all cohomological
cuspidal automorphic representations of $GL_{2}/\mathbb{Q}$. Then,
thanks to the contribution of many experts working on the Langlands
program, including Kottowitz and Ngo \cite{Kot92,Ngo10}, the result
was extended to $GL_{n}/F$, for any $n\geq2$ and any totally real
number field $F$. For the full history of this result see the surveys
\cite{Sar05,Shi20} and the references therein. 

By combining Theorem~\ref{thm:GRPC} together with Arthur's endoscopic
classification for split classical groups (Theorem~\ref{thm:Arthur=0000201.5.2}),
and its extension by Taïbi to Gross inner forms (Theorem~\ref{thm:Arthur-Taibi}),
we get that the GRPC holds in the following manner for cohomological,
generic, automorphic representations of classical groups. 
\begin{prop}
\label{prop:GRPC-classical} Let $G$ be a Gross inner form of a split
classical group defined over a totally real field $F$. If $\pi$
is a cohomological automorphic representation of $G$, whose $A$-parameter
has a generic $A$-shape, then $\pi_{v}$ is tempered at any unramified
place $v$. In particular, $r\left(\varsigma\right)=2$ for any generic
$A$-shape $\varsigma\in\mathcal{M}\left(G\right)$ . 
\end{prop}

\begin{proof}
Let $\pi$ be a cohomological automorphic representation of $G$ whose
$A$-parameter $\psi\in\Psi_{2}^{\mathrm{AJ}}\left(G\right)$, $\pi\in\Pi_{\psi}$,
has generic $A$-shape $\varsigma(\psi)=\varsigma$. Since $\varsigma$
is generic we get that $\pi^{\psi_{N}}$, the automorphic representation
of $GL_{N}$ corresponding to the formal object $\psi_{N}$, is the
normalized parabolic induction $\mathrm{Ind}_{M\left(\mathbb{A}\right)}^{GL_{N}\left(\mathbb{A}\right)}\left(\mu_{1}\otimes\ldots\otimes\mu_{k}\right)$,
where $M\cong\prod_{i}GL_{n_{i}}$ is a block diagonal Levi subgroup
of $GL_{N}$, $N=\sum_{i}n_{i}$, and $\mu_{i}$ is a cohomological
self-dual cuspidal automorphic representation of $GL_{n_{i}}/F$.
Note that the set of ramified places of $\psi_{N}$ is equal the union
of ramified places of the $\mu_{i}$, $i=1,\ldots,k$. By Theorem~\ref{thm:GRPC}
we get that $\mu_{i,v}$ is tempered for any unramified place $v$
of $\psi_{N}$. By the Local Langlands Correspondence (LLC) for $GL_{N}$
of \cite{HT01,Hen00}, see also \cite[Theorem 1.3.1]{Art13}, tempered
representations correspond to tempered parameters, hence the $L$-parameter
$\phi_{i,v}$ of $\mu_{i,v}$ is tempered, i.e. with bounded image
on $W_{F_{v}}$, for any $i=1,\ldots,k$ (in fact since $v$ is an
unramified place, this claim follows from Remark~\ref{rem:ur-temp}).
Also by the LLC of $GL_{N}$, parabolic induction of representations
corresponds to direct sum of $L$-parameters, namely the $L$-parameter
of $\psi_{N,v}$ is $\bigoplus_{i}\phi_{i,v}\,:\,L_{F_{v}}\rightarrow\hat{M}\rightarrow GL_{N}\left(\mathbb{C}\right)$.
Therefore the $L$-parameter of $\psi_{N,v}$ and also $\psi_{v}$
have bounded image on $W_{F_{v}}$, and hence are also tempered. By
Theorem~\ref{thm:Arthur=0000201.5.1}, we have $\psi_{v}=\phi_{\psi_{v}}$
and $\pi_{v}\in\Pi_{\psi_{v}}$ is tempered, which proves the claim
for the unramified places. The claim about the archimedean places
will be proved in Corollary~\ref{cor:r-coh-A-triv} below. 
\end{proof}
\begin{cor}
\label{cor:r-GY} Let $G$ be a Gross inner form of $SO_{5}$. Then
$r(\mathbf{G})=r(\mathbf{Y})=2$. 
\end{cor}

\begin{proof}
Follows from Proposition~\ref{prop:GRPC-classical} and the fact
that $(\mathbf{G})$ and $(\mathbf{Y})$ are generic.
\end{proof}
We now turn to the case of $A$-shape $(\mathbf{F})$, generalized
to all classical groups. The following claim is undoubtedly well-known
to the experts, however we could not find it in the literature, so
we provide a short proof for it.
\begin{prop}
\label{prop:aut-dim-1} Let $G$ be a Gross inner form of a split
classical group. Let $(\mathbf{F})=\left(\left(1,N\right)\right)\in\mathcal{M}\left(G\right)$,
where $N$ is the dimension of the standard representation of $\hat{G}$.
Then a discrete automorphic representation $\pi$ of $G$ has an $A$-parameter
with $A$-shape $(\mathbf{F})$ if and only if $\dim\pi=1$.
\end{prop}

\begin{proof}
By Lemma~\ref{lem:dim-1} it suffices to show that for some place
$v$ at which $\pi$ is unramified, the $A$-parameter $\psi_{v}$
has $A$-shape $(\mathbf{F})$ if and only if $\dim\pi_{v}=1$. Since
the Arthur $SL_{2}$-type of $(\mathbf{F})$ is $\sigma_{\mathrm{princ}}$,
the principal $SL_{2}$ of $\hat{G}$, by Lemma~\ref{lem:uniform-shape}
we get that $\psi_{v}|_{SL_{2}^{A}}\equiv\sigma_{\mathrm{princ}}$.
Since $\pi_{v}$ is unramified, by Proposition~\ref{prop:Moeglin-A-ur}
we get $\pi_{v}=\pi_{\psi_{v}}$. The claim now follows from Proposition~\ref{prop:1-dim-rep-par}. 
\end{proof}
\begin{cor}
\label{cor:r-F} Let $G$ be a Gross inner form of a split classical
group. Then $r(\mathbf{F})=\infty$. 
\end{cor}

\begin{proof}
Let $\pi$ be an automorphic representation of $A$-shape $(\mathbf{F})$.
By Proposition~\ref{prop:aut-dim-1}, $\pi$ and all of its local
factors are $1$-dimensional representations, and more precisely,
by the discussion before Proposition~\ref{prop:1-dim-rep-par} we
get that $\pi_{v}$ is either trivial or a quadratic character. In
either case we get for any $r<\infty$, that $\int_{G_{v}}|\pi_{v}(g)|^{r}dg=\int_{G_{v}}dg=\vol G_{v}=\infty$.
\end{proof}

\subsection{Bounds for $\left(\mathbf{B}\right)$, $\left(\mathbf{Q}\right)$
and $\left(\mathbf{P}\right)$}

Lastly, we consider the remaining $A$-shapes $\varsigma\in\left\{ (\mathbf{B}),(\mathbf{Q}),(\mathbf{P})\right\} $,
where we will first show a sufficient bound for $r(\varsigma)$ to
prove our main result Theorem~\ref{thm:CSXDH-SO5-shape} for these
$A$-shapes in Proposition~\ref{cor:r-weak-BQP}, and then give the
sharp values for $r(\varsigma)$ in Theorem~\ref{cor:r-exact-P}.

The following special case of a the result of Oh \cite{Oh02}, is
a quantitative version of Kazhdan's property (T), which gives an explicit
upper bound on the rate of decay of matrix coefficients of all infinite-dimensional
irreducible unitary representations of $G_{v}=SO_{5}\left(F_{v}\right)$,
which is uniform for all places and all representations. 
\begin{thm}
\cite{Oh02} \label{thm:Oh} For any $v$, and any $\sigma\in\Pi^{\mathrm{unit}}\left(G_{v}\right)$
with $\dim\sigma\ne1$, then $r\left(\sigma\right)\leq4$.
\end{thm}

Note that Theorem~\ref{thm:Oh} implies the following bounds on $r\left(\varsigma\right)$
for $\varsigma\in\left\{ (\mathbf{B}),(\mathbf{Q}),(\mathbf{P})\right\} $,
which is not sharp for $\varsigma=(\mathbf{P})$ according to Theorem~\ref{thm:r-shape}.
\begin{cor}
\label{cor:r-weak-BQP} Let $G$ be a Gross inner form of $SO_{5}$.
Then $r(\mathbf{B}),r(\mathbf{Q}),r(\mathbf{P})\leq4$. 
\end{cor}

\begin{proof}
Let $\varsigma\in\left\{ (\mathbf{B}),(\mathbf{Q}),(\mathbf{P})\right\} $,
$\psi\in\Psi_{2}^{\mathrm{AJ}}\left(G,\varsigma\right)$, $\pi\in\Pi_{\psi}(\epsilon_{\psi})$
and $v$ be a place. By Theorem~\ref{thm:Arthur=0000201.5.2}, $\pi$
is automorphic, hence its local factors are unitary representations
$\pi_{v}\in\Pi^{\mathrm{unit}}\left(G_{v}\right)$. By Proposition~\ref{prop:aut-dim-1},
since $\varsigma\ne(\mathbf{F})$, then $\dim\pi_{v}\ne1$. By Theorem~\ref{thm:Oh},
we get $r\left(\pi_{v}\right)\leq4$.
\end{proof}
As a consequence of Corollary~\ref{cor:r-weak-BQP}, Conjecture~\ref{conj:CSXDH-shape}
for $G$ an inner form of $SO_{5}$ and for the $A$-shapes $\varsigma\in\left\{ (\mathbf{B}),(\mathbf{Q}),(\mathbf{P})\right\} $,
will follow from the bound 
\begin{equation}
h\left(G,\varsigma;q,E\right)\ll\vol\left(X\left(q\right)\right)^{\frac{1}{2}}.\label{eq:goal-upper-bound}
\end{equation}
The bound \eqref{eq:goal-upper-bound} will be achieved in Section
\ref{sec:Bounds-on-cohomology}. Combined with Corollary \ref{cor:r-weak-BQP},
this will be sufficient to prove Theorem~\ref{thm:CSXDH-SO5-shape}.
However, more is true and in the rest of this section we shall compute
$r(\mathbf{P})$ exactly and prove that the above bounds for $r(\mathbf{B})$
and $r(\mathbf{Q})$ are sharp. More generally, we will give a formula
for $r(\varsigma)$, for any $G$ and any $\varsigma\in\mathcal{M}\left(G\right)$,
in terms of the root data of the group and the Arthur $SL_{2}$-type
of $\varsigma$ (see Definition~\ref{def:r-formula}). 

\subsection{Bounds via the Arthur $SL_{2}$-type}

We begin by setting uniform notations for both the archimedean and
non-archimedean settings. Let $G$ be a connected reductive group
defined over a local field $F_{v}$. Fix $A_{0}\leq P_{0}\leq G$
a pair of a maximal $F_{v}$-split torus and a minimal $F_{v}$-parabolic
subgroup whose Levi subgroup is the centralizer of the torus. Denote
$X^{*}:=\mathrm{Hom}\left(A_{0},GL_{1}\right)$, $X^{*}:=\mathrm{Hom}\left(GL_{1},A_{0}\right)$,
$\mathfrak{a}_{0}:=X_{*}\otimes_{\mathbb{Z}}\mathbb{R}$, $\mathfrak{a}_{0}^{*}:=X^{*}\otimes_{\mathbb{Z}}\mathbb{R}$
and let $\langle,\rangle\,:\,\mathfrak{a}_{0}^{*}\times\mathfrak{a}_{0}\rightarrow\mathbb{R}$
be the natural pairing. Let $\Delta^{*}=\left\{ \alpha_{i}\right\} _{i=1}^{n}\subset X^{*}$
and $\Delta_{*}=\left\{ \alpha_{i}^{\vee}\right\} _{i=1}^{n}\subset X_{*}$
be the sets of simple roots and coroots w.r.t. $P_{0}$, let $\left\{ \omega_{i}\right\} _{i=1}^{n}\subset\mathfrak{a}_{0}^{*}$
and $\left\{ \omega_{i}^{\vee}\right\} _{i=1}^{n}\subset\mathfrak{a}_{0}$
be the sets of fundamental weights and coweights defined by $\langle\omega_{i},\alpha_{j}^{\vee}\rangle=\delta_{ij}$
and $\langle\alpha_{i},\omega_{j}^{\vee}\rangle=\delta_{ij}$, and
let $\Lambda^{+}=\mbox{span}_{\mathbb{Z}_{\geq0}}\left\{ \omega_{i}\right\} \subset\mathfrak{a}_{0}^{*}$
and $\Lambda_{+}=\mbox{span}_{\mathbb{Z}_{\geq0}}\left\{ \omega_{i}^{\vee}\right\} \subset\mathfrak{a}_{0}$
be the cones of dominant characters and cocharacters. Define the following
partial order on $\mathfrak{a}_{0}^{*}$,
\begin{equation}
\nu\preccurlyeq\mu\qquad\Leftrightarrow\qquad\langle\nu-\mu,\lambda\rangle\leq0,\qquad\forall\lambda\in\Lambda_{+}.\label{eq:partial-order}
\end{equation}
Note that for $\mu=\sum_{i}\alpha_{i}\otimes r_{i}\in\mathfrak{a}_{0}^{*}$,
then $r_{i}=\langle\mu,\omega_{i}^{\vee}\rangle$ and therefore $0\preccurlyeq\mu$
if and only if $r_{i}\geq0$ for all $i$. Also denote $\nu\prec\mu$
if $\nu\preccurlyeq\mu$ and $\nu\ne\mu$. The Weyl vector of $G$
is equal to half the sum of positive roots w.r.t. $P_{0}$, or equivalently
the sum of fundamental weights, 
\begin{equation}
\rho:=\sum_{i=1}^{n}\omega_{i}\in\mathfrak{a}_{0}^{*}.\label{eq:Weyl-vector}
\end{equation}
Finally, let $C:=(\langle\alpha_{i},\alpha_{j}^{\vee}\rangle)_{i,j}\in M_{n}\left(\mathbb{Z}\right)$
be the Cartan matrix of $G$ (using the Bourbaki conventions), and
let $\bar{C}=(C^{t})^{-1}$ be its inverse transpose. 
\begin{defn}
\label{def:r-formula} For any $\nu\in\mathfrak{a}_{0}^{*}$, define
its weighted Dynkin vector to be $w_{\nu}:=\left(\langle\nu,\alpha_{i}^{\vee}\rangle\right)_{i}$,
and define its rate invariant to be
\[
r\left(\nu\right):=2\max_{1\leq i\leq n}\left(1-\frac{\left(\bar{C}w_{\nu}\right)_{i}}{\left(\bar{C}\mathbf{1}\right)_{i}}\right)^{-1}.
\]
For any Arthur $SL_{2}$-type $\sigma\in\mathcal{D}\left(G\right)$,
assume w.l.o.g. that $\sigma\left(T_{1}\right)\subset\hat{T}$, where
$T_{1}=\left\{ \mbox{diag}\left(x,x^{-1}\right)\right\} \cong\mathbb{C}^{\times}$,
hence $\sigma|_{T_{1}}\in X_{*}(\hat{T})\cong X^{*}\left(T\right)$,
and furthermore assume that $\sigma|_{T_{1}}$ lies in the fundamental
dominant Weyl chamber (see \cite[p. 46]{CM17}). Denote $\nu_{\sigma}=\sigma|_{T_{1}}\otimes\frac{1}{2}\in\mathfrak{a}_{0}^{*}$,
in which case $w_{\sigma}:=w_{\nu_{\sigma}}$ is equal to half the
weighted Dynkin diagram corresponding to $\sigma$, and define its
rate invariant to be 
\[
r\left(\sigma\right):=r\left(\nu_{\sigma}\right)=2\max_{1\leq i\leq n}\left(1-\frac{\left(\bar{C}w_{\sigma}\right)_{i}}{\left(\bar{C}\mathbf{1}\right)_{i}}\right)^{-1}.
\]
\end{defn}

\begin{lem}
\label{lem:r-formula} For any $\nu\in\mathfrak{a}_{0}^{*}$ such
that $0\preccurlyeq\nu\preccurlyeq\rho$, 
\[
\inf\left\{ r\geq2\,:\,\nu\preccurlyeq\left(1-\frac{2}{r}\right)\rho\right\} =r\left(\nu\right).
\]
\end{lem}

\begin{proof}
First, note that $\omega_{i}^{\vee}=\sum_{j}\bar{C}_{ij}\alpha_{j}^{\vee}$
for any $i$, since 
\[
\langle\alpha_{k},\sum_{j}\bar{C}_{ij}\alpha_{j}^{\vee}\rangle=\sum_{j}\bar{C}_{ij}\langle\alpha_{k},\alpha_{j}^{\vee}\rangle=\sum_{j}\bar{C}_{ij}C_{kj}=\sum_{j}C_{kj}C_{ji}^{-1}=\delta_{ki},
\]
and therefore 
\[
\left(\bar{C}w_{\nu}\right)_{i}=\sum_{j}\bar{C}_{ij}\langle\nu,\alpha_{j}^{\vee}\rangle=\langle\nu,\sum_{j}\bar{C}_{ij}\alpha_{j}^{\vee}\rangle=\langle\nu,\omega_{i}\rangle.
\]
Second, observe that, since $\langle\rho,\alpha_{i}^{\vee}\rangle=1$
for any $i$, then $w_{\rho}=\mathbf{1}$ and by the previous claim
$\langle\rho,\omega_{i}^{\vee}\rangle=\left(\bar{C}\mathbf{1}\right)_{i}$.
Note that by the explicit formula of $C^{-1}$ in \cite{WZ17}, we
get that $\bar{C}_{ij}\geq0$ for any $i,j$, hence $\langle\rho,\omega_{i}^{\vee}\rangle=\left(\bar{C}\mathbf{1}\right)_{i}\geq0$
for any $i$. Finally, since $\Lambda_{+}=\mbox{span}_{\mathbb{N}}\left\{ \omega_{i}^{\vee}\right\} $,
we get
\[
\inf\left\{ r\geq2\,:\,\nu\preccurlyeq\left(1-\frac{2}{r}\right)\rho\right\} =\inf\left\{ r\geq2\,:\,\langle\nu-\left(1-\frac{2}{r}\right)\rho,\omega_{i}^{\vee}\rangle\leq0,\;\forall1\leq i\leq n\right\} 
\]
and since $0\preccurlyeq\nu\preccurlyeq\rho$ implies $0\leq\frac{\langle\nu,\omega_{i}^{\vee}\rangle}{\langle\rho,\omega_{i}^{\vee}\rangle}\leq1$
for any $i$, we get
\[
=2\max_{1\leq i\leq n}\left(1-\frac{\langle\nu,\omega_{i}^{\vee}\rangle}{\langle\rho,\omega_{i}^{\vee}\rangle}\right)^{-1}=2\max_{1\leq i\leq n}\left(1-\frac{\left(\bar{C}w_{\nu}\right)_{i}}{\left(\bar{C}\mathbf{1}\right)_{i}}\right)^{-1}.
\]
\end{proof}
\begin{example}
\label{exa:r-SO5} For $G=SO_{5}$, then $n=2$ is the rank, $\alpha_{1}=\left(\begin{smallmatrix}1\\
-1
\end{smallmatrix}\right)$ and $\alpha_{2}=\left(\begin{smallmatrix}0\\
1
\end{smallmatrix}\right)$ are the simple roots, $\alpha_{1}^{\vee}=\left(\begin{smallmatrix}1\\
-1
\end{smallmatrix}\right)$ and $\alpha_{2}^{\vee}=\left(\begin{smallmatrix}0\\
2
\end{smallmatrix}\right)$ are the simple coroots, $C=\left(\begin{smallmatrix}2 & -2\\
-1 & 2
\end{smallmatrix}\right)$ is the Cartan matrix and $\bar{C}=\left(\begin{smallmatrix}1 & 1/2\\
1 & 1
\end{smallmatrix}\right)$ its inverse transpose. For $\nu=\left(\begin{smallmatrix}\nu_{1}\\
\nu_{2}
\end{smallmatrix}\right)\in\mathfrak{a}_{0}^{*}$, its half weighted Dynkin diagram is $w_{\nu}=\left(\begin{smallmatrix}w_{1}\\
w_{2}
\end{smallmatrix}\right)=\left(\begin{smallmatrix}\nu_{1}-\nu_{2}\\
2\nu_{2}
\end{smallmatrix}\right)$, and a simple calculation shows that its rate invariant is 
\[
r\left(\nu\right)=2\max\left\{ \left(1-\frac{2w_{1}+w_{2}}{3}\right)^{-1},\left(1-\frac{w_{1}+w_{2}}{2}\right)^{-1}\right\} .
\]
Recall from \eqref{eq:SO5-Arthur-SL2}, see also Table~\eqref{tab:sp4-orbits},
the four Arthur $SL_{2}$-types of $SO_{5}$, are the trivial $\sigma_{\mathrm{triv}}=\nu(1)^{4}=\sigma_{(\mathbf{G})}=\sigma_{(\mathbf{Y})}$,
the principal $\sigma_{\mathrm{princ}}=\nu(4)=\sigma_{(\mathbf{F})}$,
the subregular $\sigma_{\mathrm{subreg}}=\nu(2)^{2}=\sigma_{(\mathbf{B})}=\sigma_{(\mathbf{Q})}$,
and the minimal $\sigma_{\mathrm{min}}=\nu(2)\oplus\nu(1)^{2}=\sigma_{(\mathbf{P})}$,
which corresponds to the trivial, principal, short-root and long-root
unipotent conjugacy classes in $\hat{G}=Sp_{4}\left(\mathbb{C}\right)$,
respectively (see for instance \cite[Section 3.8]{Gan08}). Their
half weighted Dynkin diagrams are $w_{\sigma_{\mathrm{triv}}}=\left(\begin{smallmatrix}0\\
0
\end{smallmatrix}\right)$ , $w_{\sigma_{\mathrm{princ}}}=\left(\begin{smallmatrix}1\\
1
\end{smallmatrix}\right)$ , $w_{\sigma_{\mathrm{subreg}}}=\left(\begin{smallmatrix}0\\
1
\end{smallmatrix}\right)$ and $w_{\sigma_{\mathrm{min}}}=\left(\begin{smallmatrix}1/2\\
0
\end{smallmatrix}\right)$ , respectively, (see Table \eqref{tab:sp4-orbits}) and therefore
\[
r\left(\sigma_{\mathrm{triv}}\right)=2,\quad r\left(\sigma_{\mathrm{princ}}\right)=\infty,\quad r\left(\sigma_{\mathrm{subreg}}\right)=4\quad\mbox{and}\quad r\left(\sigma_{\mathrm{min}}\right)=3.
\]
\end{example}

We take this opportunity to raise the following conjecture regarding
the rate of decay of matrix coefficients of members in local $A$-packets,
which depends only on the Arthur $SL_{2}$-type of the $A$-parameter.
This conjecture, which is part of the Arthur--Ramanujan Conjectures
\cite{Art89} (see also \cite{Sha11}), is probably well known to
the experts but we could not find it stated explicitly in the literature. 
\begin{conjecture}
\label{conj:r-A-formula} Let $G$ be connected reductive group defined
over $F_{v}$ a local field. For any $\psi\in\Psi\left(G\left(F_{v}\right)\right)$
with Arthur $SL_{2}$-type $\psi|_{SL_{2}^{A}}=\sigma\in\mathcal{D}\left(G\right)$,
then
\[
\max_{\pi\in\Pi_{\psi_{v}}}r(\pi)\leq r(\sigma),
\]
with equality when $G$ is split over $F_{v}$.
\end{conjecture}

We note that by the work of Shahidi \cite{Sha11}, Conjecture~\ref{conj:r-A-formula}
holds if we replaced the local $A$-packet with its corresponding
$L$-packet $\Pi_{\phi_{\psi_{v}}}\subset\Pi_{\psi_{v}}$. The proof
follows from Proposition~\ref{prop:Shahidi}, the following formula
for the rate of decay of matrix coefficients of Langlands quotients
(as in the notations of Proposition~\ref{prop:Shahidi})
\[
r\left(j\left(P,\sigma,\nu\right)\right)=\inf\left\{ r\geq2\,:\,\nu\preccurlyeq\left(1-\frac{2}{r}\right)\rho\right\} ,
\]
which was proved by Knapp in the archimedean case (see Theorem~\ref{thm:Knapp-8.48}
below, see also Proposition~\ref{prop:r-ur} for the unramified case),
and combined with Lemma~\ref{lem:r-formula}. Therefore Conjecture~\ref{conj:r-A-formula}
will follow from extending Theorem~\ref{thm:Knapp-8.48} to all local
fields $F_{v}$ and by proving the following conjecture raised by
Clozel \cite{Clo11} (see also \cite{Moe09}), that for any local
$A$-parameter $\psi_{v}$, 
\[
\max\left\{ r(\pi)\,:\,\pi\in\Pi_{\psi_{v}}\setminus\Pi_{\phi_{\psi_{v}}}\right\} \leq\max\left\{ r(\pi)\,:\,\pi\in\Pi_{\phi_{\psi_{v}}}\right\} .
\]

In the following two subsections we give the proofs of two special
cases of Conjecture~\ref{conj:r-A-formula}: (i) Unramified representations,
and (ii) Cohomological representations.

\subsection{Bounds for unramified representations}

We begin with the case of $F_{v}$ non-archimedean, with uniformizer
$\varpi_{v}$ and residue degree $p_{v}$. Assume further that $G$
splits over $F_{v}$, and let $T=A_{0}$ be a split maximal torus
and $B=P_{0}$ a Borel subgroup. Let $X_{+}^{\mathrm{ur}}\left(T\left(F_{v}\right)\right)\leq X^{\mathrm{ur}}\left(T\left(F_{v}\right)\right)$
be the subgroup of positive real-valued unramified characters of $T\left(F_{v}\right)$
and denote $|\cdot|\,:\,X^{\mathrm{ur}}\left(T\left(F_{v}\right)\right)\rightarrow X_{+}^{\mathrm{ur}}\left(T\left(F_{v}\right)\right)$,
$|\chi|\left(t\right)=|\chi\left(t\right)|_{\mathbb{C}}$. Note that
$\mathfrak{a}_{0}^{*}:=X^{*}\left(T\right)\otimes_{\mathbb{Z}}\mathbb{R}$,
and by \cite[(2)]{SZ18} we have the following abelian group isomorphism
\begin{equation}
\mathfrak{a}_{0}^{*}\rightarrow X_{+}^{\mathrm{ur}}\left(T\left(F_{v}\right)\right),\qquad\nu\mapsto\chi_{\nu},\qquad\chi_{\sum\theta_{i}\otimes r_{i}}\left(t\right)=\prod|\theta_{i}\left(t\right)|_{F_{v}}^{r_{i}}.\label{eq:unramified=000020isomorphism}
\end{equation}
Denote the inverse map $X_{+}^{\mathrm{ur}}\left(T\left(F_{v}\right)\right)\rightarrow\mathfrak{a}_{0}^{*}$
by $\chi\mapsto\nu_{\chi}$, and note that $\langle\nu_{\chi},\lambda\rangle=\log_{p_{v}}\chi\left(\lambda(\varpi_{v})\right)$,
for any $\lambda\in X_{*}\left(T\right)$. An important example is
$\delta\in X_{+}^{\mathrm{ur}}\left(T\left(F_{v}\right)\right)$,
the modular character of $B$. By \cite[(3.3)]{Gro98}, then $\langle2\rho,\lambda\rangle=\log_{p_{v}}\delta^{-1}\left(\lambda(\varpi_{v})\right)$
for any $\lambda\in X_{*}\left(T\right)$, hence $\nu_{\delta}=-2\rho$. 

For any $\pi\in\Pi^{\mathrm{ur}}\left(G\left(F_{v}\right)\right)$,
by \eqref{eq:BCMS}, there exists a unique $\mathcal{W}$-orbit of
unramified characters $\chi_{\pi}$ of $T\left(F_{v}\right)$, such
that $\pi$ is a subquotient of the normalized parabolic induction
$\mathrm{Ind}_{B}^{G}\chi_{\pi}$. Let $\nu_{\pi}\in\mathfrak{a}_{0}^{*}$
be $\nu_{\pi}=\nu_{|\chi_{\pi}|}$ such that $w.\nu_{\pi}\preccurlyeq\nu_{\pi}$
for any $w\in\mathcal{W}$. 
\begin{prop}
\label{prop:r-ur} For any $\pi\in\Pi^{\mathrm{ur}}\left(G\left(F_{v}\right)\right)$,
then 
\[
r\left(\pi\right)=\inf\left\{ r\geq2\,:\,\nu_{\pi}\preccurlyeq\left(1-\frac{2}{r}\right)\rho\right\} .
\]
\end{prop}

\begin{proof}
Let $c$ be the unique (normalized) bi-$K$-invariant matrix coefficient
of $\pi$. Following the Cartan decomposition
\[
G\left(F_{v}\right)=\bigsqcup_{\lambda\in\Lambda_{+}}K\lambda\left(\varpi_{v}\right)K,
\]
the function $c$ is determined by its value on $\lambda\left(\varpi_{v}\right)$
for $\lambda\in\Lambda_{+}$. By \cite[Section 3.2]{Mac71}, 
\begin{equation}
p_{v}^{\langle2\rho,\lambda\rangle}=\delta\left(\lambda\left(\varpi_{v}\right)\right)^{-1}\asymp|K\lambda\left(\varpi_{v}\right)K/K|,\qquad\forall\lambda\in\Lambda_{+}.\label{eq:=000020Macdonald-1}
\end{equation}
By Macdonald's formula \cite{Mac71} (see also \cite[Theorem 4.2]{Cas80}),
there exists explicitly computable constants $\mu,\gamma(w.\chi_{\pi})\in\mathbb{R}_{>0}$,
such that

\begin{equation}
c\left(\lambda\left(\varpi_{v}\right)\right)=\mu\sum_{w\in\mathcal{W}}\gamma(w.\chi_{\pi})\left(\delta^{1/2}w.\chi_{\pi}\right)\left(\lambda\left(\varpi_{v}\right)\right)\asymp\sum_{w\in\mathcal{W}}p_{v}^{\langle w.\nu_{\pi}-\rho,\lambda\rangle},\qquad\forall\lambda\in\Lambda_{+},\label{eq:=000020Macdonald-2}
\end{equation}
Combining \eqref{eq:=000020Macdonald-1}, \eqref{eq:=000020Macdonald-2}
and the fact that $w.\nu_{\pi}\preccurlyeq\nu_{\pi}$ for any $w\in\mathcal{W}$,
we get for any $r\geq2$,
\[
\int_{G\left(F_{v}\right)}|c(g)|^{r}dg=\sum_{\lambda\in\Lambda^{+}}|K\lambda(\varpi_{v})K/K|\cdot|c\left(\lambda\left(\varpi_{v}\right)\right)|^{r}\asymp\sum_{\lambda\in\Lambda^{+}}p_{v}^{\langle2\rho+r\cdot\left(\nu_{\pi}-\rho\right),\lambda\rangle}.
\]
Since $\Lambda_{+}$ is a finitely generated abelian semi-group, we
get that 
\[
\int_{G\left(F_{v}\right)}|c(g)|^{r}dg<\infty\qquad\Leftrightarrow\qquad\langle2\rho+r\cdot\left(\nu_{\pi}-\rho\right),\lambda\rangle<0,\;\forall\lambda\in\Lambda_{+}.
\]
Therefore $\text{\ensuremath{\inf\left\{  r\geq2\,:\,\nu_{\pi}\preccurlyeq\left(1-\frac{2}{r}\right)\rho\right\} } }$
is equal to the smallest value $r$ such that the $K$-invariant matrix
coefficient $c$ is $r$-integrable. Since $\pi$ is unramified $\dim\pi^{K}=1$,
and since $\pi$ is irreducible any $K$-finite vector of it is a
linear combination of translations of the (unique up to scalars) $K$-invaraint
vector of $\pi$, hence any $K$-finite matrix coefficient of $\pi$
is a linear combination of translations of the (unique up to scalars)
$K$-invariant matrix coefficient $c$, and the claim follows. 
\end{proof}
\begin{lem}
\label{lem:nu-A-ur} For any $\psi\in\Psi^{\mathrm{ur}}\left(G\left(F_{v}\right)\right)$
with $\psi|_{SL_{2}^{A}}=\sigma$, then $\nu_{\pi_{\psi}}=\nu_{\sigma}$.
\end{lem}

\begin{proof}
Recall $\pi_{\psi}=\pi_{\phi_{\psi}}$, where $\phi_{\psi}\in\Phi^{\mathrm{ur}}\left(G\left(F_{v}\right)\right)$
is such that $\phi_{\psi}\left(\mathrm{Frob}_{F_{v}}\right)=\psi\left(\mathrm{Frob}_{F_{v}}\right)\cdot\sigma\left(\mbox{\ensuremath{\mathrm{diag}}\ensuremath{\ensuremath{\left(|\mathrm{Frob}_{F_{v}}|_{F_{v}}^{1/2},|\mathrm{Frob}_{F_{v}}|_{F_{v}}^{-1/2}\right)}}}\right)\in\hat{T}$.
Denote the unramified characters associated to $\phi_{\psi}\left(\mathrm{Frob}_{F_{v}}\right)$,
$\psi\left(\mathrm{Frob}_{F_{v}}\right)$ and $\sigma\left(\mbox{\ensuremath{\mathrm{diag}}\ensuremath{\ensuremath{\left(|\mathrm{Frob}_{F_{v}}|_{F_{v}}^{1/2},|\mathrm{Frob}_{F_{v}}|_{F_{v}}^{-1/2}\right)}}}\right)$
under the isomorphism \eqref{eq:unramified=000020isomorphism} by
$\chi_{\phi_{\psi}}$, $\chi_{\psi}$ and $\chi_{\sigma}$, and note
that $\chi_{\phi_{\psi}}=\chi_{\psi}\cdot\chi_{\sigma}$. Since $\psi$
is an $A$-parameter, $\psi\left(\mathrm{Frob}_{F_{v}}\right)$ sits
in a compact subgroup of $\hat{T}$, hence $|\chi_{\psi}|=1$, and
therefore $|\chi_{\phi_{\psi}}|=|\chi_{\sigma}|$. The claim now follows
from the fact that $\chi_{\sigma}\in X_{+}^{\mathrm{ur}}\left(T\left(F_{v}\right)\right)$,
and that by the isomorphism $X^{\mathrm{ur}}\left(T\left(F_{v}\right)\right)\cong\hat{T}$,
$\chi_{\nu_{\sigma}}$ is mapped to $\sigma\left(\mbox{\ensuremath{\mathrm{diag}}\ensuremath{\ensuremath{\left(|\mathrm{Frob}_{F_{v}}|_{F_{v}}^{1/2},|\mathrm{Frob}_{F_{v}}|_{F_{v}}^{-1/2}\right)}}}\right)$.
\end{proof}
The following Corollary is the statement that Conjecture~\ref{conj:r-A-formula}
holds for unramified local $A$-parameters and their corresponding
unramified representations for split groups over non-archimedean local
fields.
\begin{cor}
\label{cor:r-A-ur} For any $\psi\in\Psi^{\mathrm{ur}}\left(G\left(F_{v}\right)\right)$
with $\psi|_{SL_{2}^{A}}=\sigma$, then $r\left(\pi_{\psi}\right)=r\left(\sigma\right)$. 
\end{cor}

\begin{proof}
This is a combination of Proposition~\ref{prop:r-ur} and Lemmas~\ref{lem:r-formula}
and \ref{lem:nu-A-ur}.
\end{proof}
We are now in a position to give the exact worst-case rate of decay
of matrix coefficients for the $A$-shapes $(\mathbf{B})$ and $(\mathbf{Q})$.
\begin{cor}
\label{cor:r-exact-BQ} Let $G$ be a Gross inner form of $SO_{5}$.
Then $r(\mathbf{B})=r(\mathbf{Q})=4$. 
\end{cor}

\begin{proof}
On the one hand, by Corollary~\ref{cor:r-weak-BQP}, $r(\mathbf{B}),r(\mathbf{Q})\leq4$.
On the other hand, by Corollary~\ref{cor:r-A-ur} and Example~\ref{exa:r-SO5},
we get that for any finite place $v$, any $\psi\in\Psi^{\mathrm{ur}}\left(G\left(F_{v}\right)\right)$
with $\psi|_{SL_{2}^{A}}=\sigma_{\mathrm{subreg}}$, then $r\left(\pi_{\psi}\right)=r\left(\sigma_{\mathrm{subreg}}\right)=4$,
which completes the proof.
\end{proof}

\subsection{Bounds for cohomological representations}

For the archimedean case, i.e. $F_{v}=\mathbb{R}$, we follow \cite{Kna01,BW00}.
Let $G\left(\mathbb{R}\right)$ be a connected reductive Lie group
(real or complex) and fix $K$ a maximal compact subgroup and $A_{0}\leq P_{0}\leq G$
a pair of a maximal split torus and a minimal parabolic subgroup.
Recall $\mathfrak{a}_{0}=X_{*}\left(A_{0}\right)\otimes\mathbb{R}$
and $\mathfrak{a}_{0}^{*}=X^{*}\left(A_{0}\right)\otimes\mathbb{R}$,
and denote $\mathfrak{a}^{*}=\mathfrak{a}_{0}^{*}\otimes_{\mathbb{R}}\mathbb{C}$
and $\mathfrak{a}_{0}^{+}=\left\{ \nu\in\mathfrak{a}_{0}\,:\,\langle\nu,\alpha^{\vee}\rangle>0,\;\forall\alpha\in\Delta^{*}\left(G,P_{0},A_{0}\right)\right\} $.
Note that our notations slightly differ from the notations of \cite{BW00},
where $\mathfrak{a}_{0}^{*}$ is denoted by $\mathfrak{a}^{*}$ and
$\mathfrak{a}^{*}$ is denoted by $(\mathfrak{a}_{0}^{*})_{\mathbb{C}}$.

Let $\pi\in\Pi\left(G\left(\mathbb{R}\right)\right)$, i.e.\ $\pi$
is an irreducible admissible representation of $G\left(F_{v}\right)$.
Then Harish-Chandra (see for example \cite[Chapter IV]{BW00}) showed
that the matrix coefficients $\langle\pi(\exp h)v_{1},v_{2}\rangle$
for $K$-finite vectors $v_{1},v_{2}$ and $h\in\mathfrak{a}_{0}^{+}$
can be expressed as a sum over a countable set of weights $\Lambda\in E(P_{0},\pi)\subset\mathfrak{a}^{*}$
of $e^{\Lambda(h)}$ times polynomials in $h$. The elements of $E(\pi)=E(P_{0},\pi)$
are called the \textit{exponents} of $\pi$ with respect to $P_{0}$.
We will in the following suppress the dependence of $P_{0}$ which
is kept fixed.

We say that two exponents $\lambda$ and $\tilde{\lambda}$ of $\pi$
are \textit{integrally equivalent} if $\tilde{\lambda}-\lambda$ is
an integral combination of simple roots. We write that $\lambda\leqslant\tilde{\lambda}$
if these integer coefficients are nonnegative. There is a finite,
nonempty subset of $E(\pi)$ of so-called \textit{leading exponents}
$\tilde{\lambda}$ such that $\lambda\leqslant\tilde{\lambda}$ for
all exponents $\lambda$ of $\pi$ integrally equivalent to $\tilde{\lambda}$.
It is customary and convenient to parametrize exponents by including
a shift by the Weyl vector $\rho$, as $\lambda-\rho\in E(\pi)$. 
\begin{thm}
\cite[Theorem 8.48]{Kna01} \label{thm:Knapp-8.48} Fix $2\leqslant r\leqslant\infty$
and let $\pi\in\Pi\left(G\left(\mathbb{R}\right)\right)$. Then the
following are equivalent: 

$(a)$ Every (leading) exponent $\lambda-\rho$ of $\pi$ has $\ensuremath{\R}\lambda\prec\left(1-\tfrac{2}{r}\right)\rho$. 

$(b)$ Every $K$-finite matrix coefficient of $\pi$ is in $L^{r}\left(G\left(\mathbb{R}\right)\right)$.
\end{thm}

The term ``leading'' has been put in parenthesis since by \cite[Corollary 8.34]{Kna01}
there exists for every exponent $\lambda-\rho$ of $\pi$ a leading
exponent $\tilde{\lambda}-\rho$ of $\pi$ such that $\lambda\leqslant\tilde{\lambda}$.
In particular this means that $\ensuremath{\R}\lambda\preccurlyeq\R\tilde{\lambda}$
and thus $\R\tilde{\lambda}\prec\left(1-\tfrac{2}{r}\right)\rho$
implies $\R\lambda\prec\left(1-\tfrac{2}{r}\right)\rho$. 

Next we wish to describe the Langlands classification of irreducible
admissible representations in terms of Langlands triples. Let us fix
some notations. For any standard parabolic $P=MAN\supseteq P_{0}$
in which necessarily $A\subset A_{0}$, denote $\mathfrak{a}_{0,P}^{*}=X^{*}\left(A\right)\otimes\mathbb{R}$
and $\mathfrak{a}_{P}^{*}=X^{*}\left(A\right)\otimes\mathbb{C}$,
both of which can be embedded naturally as subspaces of $\mathfrak{a}_{0}^{*}$
and $\mathfrak{a}^{*}$, respectively, and let $\mathfrak{a}_{0,P}^{*,+}=\{\nu\in\mathfrak{a}_{0,P}^{*}\,:\,\langle\nu,\alpha^{\vee}\rangle>0,\;\forall\alpha\in\Delta^{*}\left(A,P\right)\}$
and $\mathfrak{a}_{P}^{*,+}=\{\nu\in\mathfrak{a}_{P}^{*}\,:\,\R\nu\in\mathfrak{a}_{0,P}^{*,+}\}$,
where $\Delta^{*}\left(A,P\right)$ is the set of simple roots of
$A$ w.r.t.\ $P$. The following theorem was first proved by Langlands
in \cite{Lan89}. 
\begin{thm}
\cite[Theorem 8.54]{Kna01} \label{thm:Knapp-8.54} For any $\pi\in\Pi\left(G\left(\mathbb{R}\right)\right)$
there exists a unique triple $\left(P,\sigma,\nu\right)$, where $P=MAN$
is a standard parabolic subgroup, $\sigma\in\Pi^{\mathrm{temp}}\left(M\left(\mathbb{R}\right)\right)$
and $\nu\in\mathfrak{a}_{P}^{*,+}$, such that $\pi=J\left(P,\sigma,\nu\right)$
is the unique Langlands quotient of the (normalized) parabolically
induced representation $I\left(P,\sigma,\nu\right)=\mathrm{Ind}_{P\left(F_{v}\right)}^{G\left(F_{v}\right)}\left(\sigma\otimes e^{\nu}\otimes1\right)$.
\end{thm}

We call $\left(P,\sigma,\nu\right)$ in Theorem~\ref{thm:Knapp-8.54}
the Langlands triple of $\pi$. We note that $\R\nu\in\mathfrak{a}_{0}^{*}$
is the analogue of $\nu_{\pi}$ in the non-archimedean case. We also
note that in \cite[Appendix A.1]{SZ18}, the authors considered a
slight variant of the Langlands triple associated to $\pi$, namely
$\left(P,\sigma',\R\nu\right)$, where $\sigma'=\sigma\otimes e^{\mathrm{Im}\nu}$
is a tempered representation of the Levi subgroup $MA$ and $\R\nu\in\mathfrak{a}_{0,P}^{*,+}$.
In Section~\ref{sec:Schmidt}, when describing the distinguish member
$\pi_{\psi}$ of the local $A$-packet as $j\left(P,\sigma',\nu\right)$,
as well as in Proposition~\ref{prop:Shahidi}, our notation followed
the convention of \cite{SZ18}. Note that $j\left(P,\sigma',\nu\right)=J\left(P,\sigma',\R\nu\right)$.
\begin{thm}
\label{thm:r-arch} Let $\pi\in\Pi\left(G\left(\mathbb{R}\right)\right)$
with Langlands triple $\left(P,\sigma,\nu\right)$. Then
\[
r(\pi)=\inf\left\{ r>2:\R\nu\preccurlyeq\left(1-\tfrac{2}{r}\right)\rho\right\} .
\]
\end{thm}

The proof follows from Theorem~\ref{thm:Knapp-8.48} combined with
\cite[Proposition 8.61]{Kna01}. In order to give the details, we
introduce some concepts and a lemma. 

Let $\Lambda^{+}\otimes\mathbb{R}_{\geq0}=\mbox{span}_{\mathbb{R}_{\geq0}}\left\{ \omega_{i}\right\} _{i=1}^{n}\subset\mathfrak{a}_{0}^{\ast}$
be the positive Weyl chamber, and for $\lambda\in\mathfrak{a}_{0}^{\ast}$
denote by $\lambda_{0}$ the unique point in $\Lambda^{*}\otimes\mathbb{R}_{\geq0}$
closest to $\lambda$. By \cite[Lemma 8.56]{Kna01} there exists for
each $\lambda$ a unique subset $\mathcal{F}\subseteq\left\{ 1,\ldots,n\right\} $
such that 
\begin{equation}
\lambda=\sum_{j\not\in\mathcal{F}}b_{j}\omega_{j}-\sum_{j\in\mathcal{F}}a_{j}\alpha_{j}\label{eq:exponent-projection}
\end{equation}
with $b_{j}>0$ and $a_{j}\geqslant0$. Then $\lambda_{0}=\sum_{j\not\in\mathcal{F}}b_{j}\omega_{j}$.
\begin{lem}
\label{lem:exponents} Let $t>0$, $\lambda\in\mathfrak{a}_{0}^{\ast}$
and $\lambda_{0}$ be defined as above. If $\lambda_{0}=t\rho$ then
$\lambda=\lambda_{0}$. 
\end{lem}

\begin{proof}
Since $\lambda_{0}=t\rho=\sum_{i=1}^{n}t\omega_{i}$ we get that $\mathcal{F}$
in \eqref{eq:exponent-projection} is the empty set for $\lambda$
and thus $\lambda=\lambda_{0}$. 
\end{proof}
~
\begin{proof}[\textbf{Proof of Theorem~\ref{thm:r-arch}}]
 We will show that the condition of Theorem~\ref{thm:Knapp-8.48}(a)
for $\pi$ with Langlands triple $(P,\sigma,\nu)$ is equivalent to
$\R\nu\prec\left(1-\tfrac{2}{r}\right)\rho$. The statement then follows
from Theorem~\ref{thm:Knapp-8.48} and the definition of the decay
parameter $r(\pi)$.

First assume that Theorem~\ref{thm:Knapp-8.48}(a) is satisfied,
that is, that for every exponent $\lambda-\rho$ of $\pi$ we have
that $\R\lambda\prec\left(1-\tfrac{2}{r}\right)\rho$. By \cite[Lemma 8.59]{Kna01}
we get $(\R\lambda)_{0}\preccurlyeq\left(1-\frac{2}{r}\right)\rho_{0}=\left(1-\frac{2}{r}\right)\rho$,
and since $\R\lambda\neq\left(1-\tfrac{2}{r}\right)\rho$ we get from
Lemma~\ref{lem:exponents} that $(\R\lambda)_{0}\prec\left(1-\tfrac{2}{r}\right)\rho$
with strict inequality. By \cite[Proposition 8.61 (a)]{Kna01} and
\cite[(8.117)]{Kna01} there exists an exponent $\tilde{\nu}-\rho$
of $\pi$ such that $(\R\tilde{\nu})_{0}=\R\nu$. Thus, taking $\lambda=\tilde{\nu}$
we get that $\R\nu=(\ensuremath{\R}\tilde{\nu})_{0}\prec\left(1-\tfrac{2}{r}\right)\rho$.

Now assume that Theorem~\ref{thm:Knapp-8.48}(a) is not satisfied,
that is, there exists an exponent $\lambda-\rho$ of $\pi$ such that
$\R\lambda\nprec\left(1-\tfrac{2}{r}\right)\rho$. In other words,
for such an exponent $\lambda-\rho$ there exists a fundamental coweight
$\omega_{j}^{\vee}$ such that $\langle\R\lambda,\omega_{j}^{\vee}\rangle\geqslant\langle\left(1-\tfrac{2}{r}\right)\rho,\omega_{j}^{\vee}\rangle$.
From \cite[Proposition 8.61 (b)]{Kna01} we get $\R\lambda\preccurlyeq\R\nu$,
which means that $\langle\R\lambda,\omega_{j}^{\vee}\rangle\leqslant\langle\R\nu,\omega_{j}^{\vee}\rangle$,
hence $\langle\R\nu,\omega_{j}^{\vee}\rangle\geqslant\langle\left(1-\tfrac{2}{r}\right)\rho,\omega_{j}^{\vee}\rangle$,
and therefore $\R\nu\nprec\left(1-\frac{2}{r}\right)\rho$. 
\end{proof}
In \cite[Section 6]{VZ84}, Vogan and Zuckerman give an explicit algorithm
to determine the Langlands triples of cohomological representations
for real connected semisimple Lie groups $G\left(\mathbb{R}\right)$.
More precisely, for any $A_{\mathfrak{q}}\left(\lambda_{E}\right)\in\Pi^{\mathrm{coh}}\left(G\left(\mathbb{R}\right)\right)$,
where $E\in\Pi^{\mathrm{alg}}\left(G\left(\mathbb{R}\right)\right)$
and $\mathfrak{q}$ is a $\theta$-stable parabolic subalgebra, the
authors described its corresponding Langlands triple $\left(P,\sigma,\nu\right)$
in \cite[(6.6) – (6.15)]{VZ84}. Their construction goes through
a second triple $\left(P^{\mathrm{d}},\sigma^{\mathrm{d}},\nu^{\mathrm{d}}\right)$,
through the recipe of \cite[(6.8)]{VZ84}, where it is shown that
$\nu=\nu^{\mathrm{d}}$ as elements of~$\mathfrak{a}$. The description
of the third component of the second triple $\nu^{\mathrm{d}}$, appears
in \cite[(6.11)]{VZ84}. The following special case of \cite[Theorem 6.1.6]{VZ84},
summarizes what we need.
\begin{prop}
\cite[Theorem 6.1.6]{VZ84} \label{prop:Vogan-Zuckerman-Langlands}
Let $E\in\Pi^{\mathrm{alg}}\left(G\left(\mathbb{R}\right)\right)$,
let $\mathfrak{q}$ be a $\theta$-stable parabolic subalgebra, and
let $L=\mathrm{Stab}_{G\left(\mathbb{R}\right)}\left(\mathfrak{q}\right)$
be the Levi subgroup associated to $\mathfrak{q}$, such that $\hat{L}\in\mathcal{L}_{E}$
in the notation of Subsection \ref{subsec:Cohomological-rep-par}.
Let $A^{\mathrm{d}}$ be the split component of a maximally split
$\theta$-stable Cartan subgroup of $L$, let $\mathfrak{a}^{\mathrm{d}}$
be its Lie algebra and let $\rho_{L}\in(\mathfrak{a}_{0}^{\mathrm{d}})^{*}$
be half the sum of positive roots of $A^{\mathrm{d}}$ in $L$. If
$\left(P,\sigma,\nu\right)$ is the Langlands triple of $A_{\mathfrak{q}}\left(\lambda_{E}\right)\in\Pi^{\mathrm{coh}}\left(G\left(\mathbb{R}\right)\right)$,
with $\nu\in\mathfrak{a}_{P}^{*,+}\subset\mathfrak{a}^{*}$, then
\[
\nu=\rho_{L}+\lambda_{E}\mid_{\mathfrak{a}^{\mathrm{d}}}\in(\mathfrak{a}^{\mathrm{d}})^{*}\subset\mathfrak{a}^{*}.
\]
Moreover, $\R\nu=\rho_{L}\in(\mathfrak{a}_{0}^{\mathrm{d}})^{*}$. 
\end{prop}

\begin{proof}
The only fact that doesn't follow directly from \cite[Section 6]{VZ84}
is the last one. For this, we spell out the construction of the character
$\lambda_{E}:\mathfrak{l}\to\mathbb{C}$ associated to $E$. Let $\mathfrak{l}=\mathfrak{t}^{c}\oplus\bigoplus_{\alpha}\mathfrak{l}_{\alpha}$
be the root spaces decomposition of the Levi subalgebra $\mathfrak{l}$
of $L$ with respect to the Cartan subalgebra $\mathfrak{t}^{c}$
of the maximal compact torus $T^{c}\left(\mathbb{R}\right)$. The
assumption that $L\in\mathcal{L}_{E}$ is by definition that the highest
$\mathfrak{t}^{c}$-weight $\lambda_{E}$ of $E$ satisfies $\langle\lambda_{E},\alpha^{\vee}\rangle=0$
for all $\alpha^{\vee}\in\Delta_{*}(\mathfrak{l},\mathfrak{b},\mathfrak{t}^{c})$.
This implies that all the root spaces $\mathfrak{l}_{\alpha}$ for
$\alpha\in\pm\Delta^{*}(\mathfrak{l},\mathfrak{b},\mathfrak{t}^{c})$
act trivially on the weight space $E_{\lambda_{E}}$ (otherwise the
space $E_{\lambda_{E}+\alpha}$ would be a $\mathcal{W}$-translate
of a higher weight). This implies in turn that the root space $\mathfrak{l}_{\alpha}$
acts trivially for any $\alpha\in\Phi(\mathfrak{l},\mathfrak{t}^{c})$,
so that the one-dimensional space $E_{\lambda_{E}}$ is in fact a
representation of $\mathfrak{l}$, obtained by extending the highest
weight $\lambda_{E}\in(\mathfrak{t}^{c})^{*}$ by zero to the root
spaces. Now $\lambda_{E}$ was the highest weight of $E$ with respect
to the Lie algebra $\mathfrak{t}^{c}$ of a compact torus $T^{c}$.
As such, it is real-valued on $i\mathfrak{t}_{0}^{c}$ \cite[IV. 7]{Kna01},
which is to say that $\lambda_{E}(\mathfrak{t}_{0}^{c})\in i\mathbb{R}$,
and by extension, $\lambda_{E}(\mathfrak{l}_{0})\in i\mathbb{R}$,
and the real part of its restriction to $\mathfrak{a}_{0}^{d}$ is
trivial. 
\end{proof}
\begin{rem*}
The quotient obtained from $A^{\mathrm{d}}$ in \cite[Section 6]{VZ84}
agrees with the traditional Langlands quotient if an only if $\langle\rho_{L},\alpha^{\vee}\rangle>0$
for all coroots of $\mathfrak{a}^{\mathrm{d}}$. This condition is
automatically satisfied when $L$ is (up to isogeny) the product $L'\times K_{L}$
of a simple Lie group and a compact group. In this case $A^{\mathrm{d}}\subset L'$
and $\rho_{L}=\rho_{L'}$, so that the positivity condition follows
from elementary properties of $\rho_{L}$. The reader can observe
in the proof of the following corollary that Levi subgroups of inner
forms of $SO_{5}$ are always of the form $L'\times K_{L}$. 
\end{rem*}
We may now combine all of the above results to give a closed formula
for upper bounding the rate of decay of matrix coefficients in a cohomological
$A$-packet, using only the Arthur $SL_{2}$-type of the $A$-parameter. 
\begin{cor}
\label{cor:r-coh} Let $\sigma\in\mathcal{D}\left(G\right)$ and $E\in\Pi^{\mathrm{alg}}\left(G\left(\mathbb{R}\right)\right)$
be such that $\sigma$ is the principal $SL_{2}$ of a Levi subgroup
$\hat{L}\in\mathcal{L}_{E}$. Then for any $\psi\in\Psi^{\mathrm{AJ}}\left(G\left(\mathbb{R}\right);E\right)$
with $\psi|_{SL_{2}^{A}}=\sigma$, 
\[
\max_{\pi\in\Pi_{\psi}}r\left(\pi\right)=\max_{L}r\left(\rho_{L}\right),
\]
where $L$ runs over the set of standard Levi subgroups of $G\left(\mathbb{R}\right)$
whose dual is $\hat{L}$.
\end{cor}

\begin{proof}
By Proposition~\ref{prop:AJ}, any $\pi\in\Pi_{\psi}$ is of the
form $A_{\mathfrak{q}}\left(\lambda_{E}\right)$, such that $L=\mathrm{Stab}_{G\left(\mathbb{R}\right)}\left(\mathfrak{q}\right)$
is a Levi subgroup and $\psi|_{SL_{2}^{A}}$ is the principal $SL_{2}$
of $\hat{L}$. By Proposition~\ref{prop:Vogan-Zuckerman-Langlands},
Theorem~\ref{thm:r-arch} and Lemma~\ref{lem:r-formula}, we get
\[
r\left(A_{\mathfrak{q}}\left(\lambda_{E}\right)\right)=\inf\left\{ r>2:\rho_{L}\preccurlyeq\left(1-\tfrac{2}{r}\right)\rho\right\} =r\left(\rho_{L}\right).
\]
\end{proof}
In particular, from the above corollary we obtain the following consequence
(which is certainly well known to the experts).
\begin{cor}
\label{cor:r-coh-A-triv} If $\psi$ is a cohomological $A$-parameter
of $G\left(\mathbb{R}\right)$ with trivial Arthur $SL_{2}$-type,
then $r\left(\pi\right)=2$ for any $\pi\in\Pi_{\psi}$.
\end{cor}

\begin{proof}
The trivial Arthur $SL_{2}$-type is principal for the dual maximal
torus $\hat{L}=\hat{T}$, and therefore $L$ is equal to the compact
Cartan subgroup $L=T^{c}\left(\mathbb{R}\right)$. Since $L$ is compact
we get that $A^{\mathrm{d}}$ is trivial, hence $\rho_{L}=\left(0,0\right)$
and therefore $r\left(\rho_{L}\right)=2$. The claim now follows from
Corollary~\ref{cor:r-coh}.
\end{proof}
We now specialize to our case of interest where $G$ is a non-compact
inner form of $G^{*}=SO_{5}$ defined over $F_{v}=\mathbb{R}$, namely
$G\left(\mathbb{R}\right)=SO\left(3,2\right)$ or $SO\left(1,4\right)$.
Let us describe for each of the four possible Arthur $SL_{2}$-types
of $G$, their possible Levi subgroups $L$, their split component
of a maximally split Cartan subgroup $A^{\mathrm{d}},$ and their
half sum of positive roots $\rho_{L}$:
\begin{lyxlist}{00.00.0000}
\item [{$\sigma_{\mathrm{triv}}$}] The trivial Arthur $SL_{2}$-type is
principal for the dual maximal torus $\hat{L}=\hat{T}=GL_{1}\left(\mathbb{C}\right)^{2}$,
and the only possible Levi subgroup is the compact Cartan subgroup
$L=T^{c}\left(\mathbb{R}\right)=SO\left(2\right)^{2}$, for either
$G\left(\mathbb{R}\right)=SO\left(3,2\right)$ or $SO\left(1,4\right)$.
Since $L$ is compact we get that $A^{\mathrm{d}}$ is trivial, hence
$\rho_{L}=\left(0,0\right)$ and therefore 
\[
r\left(\rho_{L}\right)=2.
\]
\item [{$\sigma_{\mathrm{princ}}$}] The principal Arthur $SL_{2}$-type
is principal for the dual group $\hat{L}=\hat{G}=Sp_{4}\left(\mathbb{C}\right)$,
and the only possible Levi subgroup is the full group $L=G\left(\mathbb{R}\right)$,
for either $G\left(\mathbb{R}\right)=SO\left(3,2\right)$ or $SO\left(1,4\right)$.
Since $L=G\left(\mathbb{R}\right)$ we get that $A^{\mathrm{d}}=A_{0}$
hence $\rho_{L}=\rho=\left(3/2,1/2\right)$ and therefore
\[
r\left(\rho_{L}\right)=\infty.
\]
\item [{$\sigma_{\mathrm{subreg}}$}] The subregular Arthur $SL_{2}$-type
is principal for the dual Levi subgroup $\hat{L}=GL_{2}\left(\mathbb{C}\right)$,
and the possible Levi subgroups are, $L_{1}=U\left(1,1\right)$ when
$G\left(\mathbb{R}\right)=SO\left(3,2\right)$, and $L_{2}=U\left(2\right)$
when $G\left(\mathbb{R}\right)=SO\left(1,4\right)$. The case $L_{2}=U\left(2\right)$
is compact, we get that $A^{\mathrm{d}}$ is trivial, hence $\rho_{L_{2}}=\left(0,0\right)$.
The case $L_{1}=U\left(1,1\right)$ gives $A^{\mathrm{d}}=SO\left(1,1\right)\leq U\left(1,1\right)$,
hence $\rho_{L_{1}}=\left(1/2,1/2\right)$. Therefore 
\[
r\left(\rho_{L_{1}}\right)=4\qquad\mbox{and}\qquad r\left(\rho_{L_{2}}\right)=2.
\]
\item [{$\sigma_{\mathrm{min}}$}] The minimal Arthur $SL_{2}$-type is
principal for the dual Levi subgroup $\hat{L}=SL_{2}\left(\mathbb{C}\right)\times GL_{1}\left(\mathbb{C}\right)$,
and the possible Levi subgroups are, $L_{1}=SO\left(1,2\right)\times SO\left(2\right)$
or $L_{2}=SO\left(3\right)\times SO\left(2\right)$ when $G\left(\mathbb{R}\right)=SO\left(3,2\right)$,
and $L_{1}=SO\left(1,2\right)\times SO\left(2\right)$ when $G\left(\mathbb{R}\right)=SO\left(1,4\right)$.
The case $L_{2}=SO\left(3\right)\times SO\left(2\right)$ is compact,
we get that $A^{\mathrm{d}}$ is trivial, hence $\rho_{L_{2}}=\left(0,0\right)$.
The case $L_{1}=SO\left(1,2\right)\times SO\left(2\right)$ gives
$A^{\mathrm{d}}=SO\left(1,1\right)\leq SO\left(1,2\right)$, hence
$\rho_{L_{1}}=\left(1/2,0\right)$. Therefore 
\[
r\left(\rho_{L_{1}}\right)=3\qquad\mbox{and}\qquad r\left(\rho_{L_{2}}\right)=2.
\]
\end{lyxlist}
\begin{cor}
\label{cor:r-arch-BQP} Let $G\left(\mathbb{R}\right)=SO\left(3,2\right)$
or $SO\left(1,4\right)$, let $\sigma\in\mathcal{D}\left(G\right)$
and let $\psi\in\Psi^{\mathrm{AJ}}\left(G\left(\mathbb{R}\right);E\right)$
with $\psi|_{SL_{2}^{A}}=\sigma$. If $\sigma\ne\sigma_{\mathrm{subreg}}$,
then
\[
\max_{\pi\in\Pi_{\psi}}r(\pi)=\begin{cases}
\begin{array}{c}
\infty\\
2\\
3
\end{array} & \begin{array}{c}
\sigma=\sigma_{\mathrm{princ}}\\
\sigma=\sigma_{\mathrm{triv}}\\
\sigma=\sigma_{\mathrm{min}}
\end{array}\end{cases},
\]
and if $\sigma=\sigma_{\mathrm{subreg}}$, then 
\[
\max_{\pi\in\Pi_{\psi}}r(\pi)=\begin{cases}
\begin{array}{c}
2\\
4
\end{array} & \begin{array}{c}
G\left(\mathbb{R}\right)=SO\left(1,4\right)\\
G\left(\mathbb{R}\right)=SO\left(3,2\right)
\end{array}\end{cases}.
\]
\end{cor}

\begin{proof}
This follows from Corollary~\ref{cor:r-coh} and the above classification
of the Levi subgroups corresponding to the Arthur $SL_{2}$-types
of $G\left(\mathbb{R}\right)$.
\end{proof}
We end this section with the following statement, which gives the
exact worst-case rate of decay of matrix coefficients of the local
factors of automorphic representations whose $A$-parameters have
$A$-shape $(\mathbf{P})$.
\begin{cor}
\label{cor:r-exact-P} Let $G$ be a Gross inner form of $SO_{5}$.
Then $r(\mathbf{P})=3$. 
\end{cor}

\begin{proof}
By Corollary~\ref{cor:r-A-ur} and Example~\ref{exa:r-SO5}, we
get that for any finite place $v$, any $\psi\in\Psi^{\mathrm{ur}}\left(G\left(F_{v}\right)\right)$
with $\psi|_{SL_{2}^{A}}=\sigma_{\mathrm{min}}$, then $r\left(\pi_{\psi}\right)=r\left(\sigma_{\mathrm{min}}\right)=3$,
and by Corollary~\ref{cor:r-arch-BQP}, we get that for any $\psi\in\Psi^{\mathrm{AJ}}\left(G\left(\mathbb{R}\right)\right)$
with $\psi|_{SL_{2}^{A}}=\sigma_{\mathrm{min}}$, then $\max_{\pi\in\Pi_{\psi}}r\left(\pi\right)=3$.
\end{proof}

\section{Bounds on cohomological dimensions $h\left(G,\varsigma;q,E\right)$
\protect\label{sec:Bounds-on-cohomology}}

In this section we give upper bounds on the cohomological dimensions
$h\left(G,\varsigma;q,E\right)$, for Gross inner forms $G$ of a
split classical groups $G^{*}$ defined over totally real fields,
for the different $A$-shapes $\varsigma\in\mathcal{M}\left(G\right)$,
where $E\in\Pi^{\mathrm{alg}}\left(G\right)$ is fixed and $|q|\rightarrow\infty$.
In Subsections~\ref{subsec:h-GYF}, \ref{subsec:h-BQ} and \ref{subsec:h-P},
we shall restrict to our special case of interest where $G^{*}=SO_{5}$. 

\subsection{Bounds on $h\left(G,\varsigma;q,E\right)$ for general $G$\protect\label{subsec:h-general}}

Let $G$ be a Gross inner form of a split classical group. 
\begin{lem}
\label{lem:h-triv-bound} Fix $E\in\Pi^{\mathrm{alg}}\left(G\right)$
and let $|q|\rightarrow\infty$. Then
\[
h\left(G;q,E\right):=\sum_{\varsigma\in\mathcal{M}\left(G\right)}h\left(G,\varsigma;q,E\right)\asymp\vol\left(X\left(q\right)\right).
\]
In particular for any $\varsigma\in\mathcal{M}\left(G\right)$, 
\[
h\left(G,\varsigma;q,E\right)\ll\vol\left(X\left(q\right)\right).
\]
\end{lem}

\begin{proof}
By the Arthur--Matsushima decomposition \eqref{eq:Arthur-Matsushima},
combined with \eqref{eq:coh-vol-dim}, we get
\[
\sum_{\varsigma\in\mathcal{M}\left(G\right)}h\left(G,\varsigma;q,E\right)=\sum_{\psi\in\Psi_{2}^{\mathrm{AJ}}\left(G;q,E\right)}\sum_{\pi\in\Pi_{\psi}(\epsilon_{\psi})}h\left(\pi;q,E\right)=\dim H_{(2)}^{*}\left(X\left(q\right);E\right)\asymp\vol\left(X\left(q\right)\right).
\]
The second claim follows from the first.
\end{proof}
Recall $h\left(\psi;q,E\right)=\sum_{\pi\in\Pi_{\psi}(\epsilon_{\psi})}h\left(\pi;q,E\right)$,
where $\Pi_{\psi}(\epsilon_{\psi})\subset\Pi_{\psi}$ is defined via
the global sign character $\epsilon_{\psi}\in\widehat{\mathcal{S}_{\psi}}$.
Define the analogous sum over the full global $A$-packet $\Pi_{\psi}$,
to be
\[
h'\left(\psi;q,E\right):=\sum_{\pi\in\Pi_{\psi}}h\left(\pi;q,E\right).
\]
Note that $h\left(\psi;q,E\right)\leq h'\left(\psi;q,E\right)$ for
any $\psi\in\Psi_{2}\left(G\right)$, and that $h\left(\psi;q,E\right)=h'\left(\psi;q,E\right)$
if $\mathcal{S}_{\psi}=\left\{ 1\right\} $. The following Proposition
shows that $h'$ is comparable to $h$ for $A$-parameters with generic
$A$-shapes.
\begin{prop}
\label{prop:h'-generic} Let $\varsigma\in\mathcal{M}^{g}\left(G\right)$.
Then for any $\epsilon>0$, there exists $c>0$, such that: For any
$E\in\Pi^{\mathrm{alg}}\left(G\right)$, $q\lhd\mathcal{O}$ and any
$\psi\in\Psi_{2}^{\mathrm{AJ}}\left(G,\varsigma\right)$,
\begin{equation}
h'\left(\psi;q,E\right)\leq c|q|^{\epsilon}\cdot h\left(\psi;c|q|^{\epsilon}\cdot q,E\right).\label{eq:h'-generic}
\end{equation}
\end{prop}

\begin{proof}
Denote $\Pi_{\psi}\left(q,E\right)=\left\{ \pi\in\Pi_{\psi}\,:\,h\left(\pi;q,E\right)\ne0\right\} $
for any $\psi\in\Psi_{2}\left(G,\varsigma\right)$, and note that
\begin{equation}
h'\left(\psi;q,E\right)\leq|\Pi_{\psi}\left(q,E\right)|\cdot\max_{\pi\in\Pi_{\psi}}h\left(\pi;q,E\right).\label{eq:h'-generic-trivial}
\end{equation}
By Proposition~\ref{prop:Moeglin-A-ur}, which says that every unramified
local $A$-packet is a singleton, Proposition~\ref{thm:Moeglin-A-size},
which says that all local $A$-packet are bounded by some uniform
constant $C>0$, and Lemma~\ref{lem:prime-omega}, we get that for
any $\epsilon>0$, there exists $c$, such that for any $\psi$,
\begin{equation}
\left|\Pi_{\psi}\left(q,E\right)\right|\leq\prod_{v\mid q\cdot\infty}|\Pi_{\psi_{v}}|\leq C^{\omega\left(q\right)+|\{v\,:\,v\mid\infty\}|}\leq c|q|^{\epsilon}.\label{eq:h'-generic-packet}
\end{equation}
From \eqref{eq:h'-generic-trivial} and \eqref{eq:h'-generic-packet}
we get that \eqref{eq:h'-generic} will follow from the following
claim
\begin{equation}
\max_{\pi\in\Pi_{\psi}}h\left(\pi;q,E\right)\leq\max_{\pi'\in\Pi_{\psi}(\epsilon_{\psi})}h\left(\pi';c|q|^{\epsilon}\cdot q,E\right).\label{eq:h'-generic-bound}
\end{equation}
Let $\pi\in\Pi_{\psi}$ which maximize $h\left(\pi;q,E\right)$ and
let $u$ be the smallest finite place coprime to $q$, denote by $p_{u}$
the residue degree of $F_{u}$. By Lemma~\ref{lem:prime-smallest},
for any $\epsilon>0$ there exists $c>0$, such that $p_{u}^{2}\leq c|q|^{\epsilon}$.
By Proposition~\ref{prop:GRPC-classical}, $\psi_{u}$ is tempered
for any $\psi\in\Psi_{2}^{\mathrm{AJ}}\left(G,\varsigma\right)$,
and by Theorem~\ref{thm:Arthur=0000201.5.1}, the map $\langle,\rangle_{\psi_{u}}\,:\,\Pi_{\psi_{u}}\rightarrow\widehat{\mathcal{S}_{\psi_{u}}}$
is a bijection. Define $\pi'\in\Pi_{\psi}(\epsilon_{\psi})$ to be
$\pi'_{v}=\pi_{v}$ for any $v\ne u$, and $\pi'_{u}\in\Pi_{\psi_{u}}$
to be the unique element in the local $A$-packet such that $\langle\cdot,\pi'_{u}\rangle_{\psi_{u}}=\epsilon_{\psi}\prod_{v\ne u}\langle\cdot,\pi_{v}\rangle_{\psi_{v}}^{-1}\in\widehat{\mathcal{S}_{\psi_{u}}}$.
Since $\psi_{u}$ is unramified, hence of depth $0$, by Lemma~\ref{lem:depth-level}
and Proposition~\ref{prop:GV-A-SO5}, we get that $\pi'_{u}$ is
of level at most $2$. Hence 
\[
h\left(\pi;q,E\right)=\dim H^{*}\left(\pi_{\infty};E\right)\prod_{v\mid q}\dim\pi_{v}^{K_{v}(p_{v}^{\mathrm{ord}_{v}(q)})}\prod_{v\nmid q\infty}\dim\pi_{v}^{K_{v}}
\]
\[
\leq\dim H^{*}\left(\pi_{\infty};E\right)\dim\left(\pi'_{u}\right)^{K_{u}\left(p_{u}^{2}\right)}\prod_{v\mid q}\dim\pi_{v}^{K_{v}(p_{v}^{\mathrm{ord}_{v}(q)})}\prod_{v\nmid uq\infty}\dim\pi_{v}^{K_{v}}=h\left(\pi';p_{u}^{2}\cdot q,E\right).
\]
This proves \eqref{eq:h'-generic-bound}, and therefore \eqref{eq:h'-generic}. 
\end{proof}
We now introduce the notations needed for Subsection~\ref{subsec:h-BQ}.
\begin{defn}
\label{def:h1-h2} For any $A$-shape $\varsigma\in\mathcal{M}\left(G\right)$,
denote 
\[
h_{1}\left(G,\varsigma;q,E\right):=|\Psi_{2}^{\mathrm{AJ}}\left(G,\varsigma;q,E\right)|,
\]
and
\[
h_{2}\left(G,\varsigma;q,E\right):=\max_{\psi\in\Psi_{2}^{\mathrm{AJ}}\left(G,\varsigma;q,E\right)}\prod_{v\mid q}\max_{\pi_{v}\in\Pi_{\psi_{v}}}\dim\pi_{v}^{K_{v}\left(q\right)}.
\]
\end{defn}

\begin{lem}
\label{lem:h-h1h2} Fix $E\in\Pi^{\mathrm{alg}}\left(G\right)$ and
let $|q|\rightarrow\infty$. Then for any $\varsigma\in\mathcal{M}\left(G\right)$,
\begin{equation}
h\left(G,\varsigma;q,E\right)\ll h_{1}\left(G,\varsigma;q,E\right)\cdot h_{2}\left(G,\varsigma;q,E\right).\label{eq:h-h1h2}
\end{equation}
\end{lem}

\begin{proof}
Since 
\[
h\left(G,\varsigma;q,E\right)=\sum_{\psi\in\Psi_{2}^{\mathrm{AJ}}\left(G,\varsigma;q,E\right)}h\left(\psi;q,E\right)\leq h_{1}\left(G,\varsigma;q,E\right)\cdot\left(\max_{\psi\in\Psi_{2}^{\mathrm{AJ}}\left(G,\varsigma;q,E\right)}h\left(\psi;q,E\right)\right),
\]
it suffices to prove that there exist a constant $C>0$ which depends
only on $\mathrm{rank}\left(G\right)$, such that for any $\psi\in\Psi_{2}^{\mathrm{AJ}}\left(G,\varsigma;q,E\right)$,
\begin{equation}
h\left(\psi;q,E\right)\leq C\cdot\prod_{v\mid q}\max_{\pi_{v}\in\Pi_{\psi_{v}}}\dim\pi_{v}^{K_{v}\left(q\right)}.\label{eq:h-=00005Cpsi-max}
\end{equation}
By Subsection \ref{subsec:Unramified-rep-par}, recall that $\dim\pi_{v}^{K_{v}}\leq1$
for any $\pi_{v}\in\Pi\left(G_{v}\right)$ and $v\nmid\infty$, and
that $K_{v}\left(q\right)=K_{v}$ for $v\nmid q$. By Proposition
\ref{prop:NP-uniform}, there exists a constant $C_{1}$, which depends
only on $\mathrm{rank}\left(G\right)$, such that $\dim H^{*}\left(\pi_{\infty};E\right)\leq C_{1}$
for any $\pi_{\infty}\in\Pi^{\mathrm{coh}}\left(G_{\infty};E\right)$.
By Proposition~\ref{thm:Moeglin-A-size}, there exists a constant
$C_{2}$, which depends only on $\mathrm{rank}\left(G\right)$, such
that $|\Pi_{\psi_{v}}|\leq C_{2}$, for any $v$ and any $\psi_{v}\in\Psi\left(G_{v}\right)$.
Hence 
\[
h\left(\psi;q,E\right)=\sum_{\pi\in\Pi_{\psi}(\epsilon_{\psi})}\dim H^{*}\left(\pi_{\infty};E\right)\dim\pi_{f}^{K_{f}\left(q\right)}\leq C_{1}\sum_{\pi\in\Pi_{\psi}}\prod_{v}\dim\pi_{v}^{K_{v}\left(q\right)},
\]
\[
\leq C_{1}\sum_{\pi\in\Pi_{\psi}}\prod_{v\mid q}\dim\pi_{v}^{K_{v}\left(q\right)}\leq C_{1}\prod_{v\mid q}\sum_{\pi_{v}\in\Pi_{\psi_{v}}}\dim\pi_{v}^{K_{v}\left(q\right)}\leq C_{1}C_{2}\prod_{v\mid q}\max_{\pi_{v}\in\Pi_{\psi_{v}}}\dim\pi_{v}^{K_{v}\left(q\right)},
\]
and we get \eqref{eq:h-=00005Cpsi-max}.
\end{proof}
In the rest of this subsection our goal will be to bound $h_{1}\left(G,\varsigma;q,E\right)$.
We note that the invariant $h_{2}\left(G,\varsigma;q,E\right)$ resembles
the notion of the Gelfand-Kirillov dimension, which is not only uniform
in the level aspect but also in the representation aspect and the
local field aspect, for representations which sit in an $A$-packet
of an $A$-parameter of a certain $A$-shape.

Recall the group $G^{\{\varsigma\}}$ from Definition \ref{def:shape-group},
which is a product of almost-simple quasi-split classical groups $G^{\{\varsigma\}}=\prod_{i}G_{i}^{\{\varsigma\}}$,
and define $\mathcal{M}(G^{\{\varsigma\}})$ to be the concatenation
of $A$-shapes of $\mathcal{M}(G_{i}^{\{\varsigma\}})$ for all $i$,
and similarly for $\mathcal{M}^{g}(G^{\{\varsigma\}})$.
\begin{lem}
\label{lem:h1-varsigma} Fix $E\in\Pi^{\mathrm{alg}}\left(G\right)$
and let $q\rightarrow\infty$. Then for any $\varsigma\in\mathcal{M}\left(G\right)$,
\[
h_{1}\left(G,\varsigma;q,E\right)\asymp\sum_{\varphi\in\mathcal{M}^{g}(G^{\{\varsigma\}})}h_{1}(G^{\{\varsigma\}},\varphi;q,E^{\{\varsigma\}}),
\]
where $E^{\{\varsigma\}}\in\Pi^{\mathrm{alg}}\left(G^{\{\varsigma\}}\right)$
is such that $\lambda_{E^{\{\varsigma\}}}=\rho_{G}-\rho_{G^{\{\varsigma\}}}+\lambda_{E}$.
\end{lem}

\begin{proof}
By Propositions~\ref{prop:gen-func-global} and \ref{prop:gen-func-local},
there is a surjective map with finite fibers of size bound by a constant
which depends only on $G$, from $\Psi_{2}^{\mathrm{AJ}}\left(G,\varsigma\right)$
to $\Psi_{2}^{\mathrm{AJ}}\left(G^{\{\varsigma\}}\right)$, $\psi\mapsto\psi^{\{\varsigma\}}$,
such that $\psi_{v}|_{L_{F_{v}}}\equiv\iota_{\varsigma}^{G}\circ\psi_{v}^{\{\varsigma\}}|_{L_{F_{v}}}$
for any $v$. From the last condition we get that $d\left(\psi_{v}\right)=d(\psi_{v}^{\{\varsigma\}})$
for any finite $v$, that $\psi_{v}\in\Psi^{\mathrm{ur}}\left(G_{v}\right)$
if and only if $\psi_{v}^{\{\varsigma\}}\in\Psi^{\mathrm{ur}}(G_{v}^{\{\varsigma\}})$,
for any finite $v\nmid q$, and that $\psi_{\infty}\in\Psi^{\mathrm{AJ}}\left(G_{\infty};E\right)$
if and only if $\psi_{\infty}^{\{\varsigma\}}\in\Psi^{\mathrm{AJ}}(G_{\infty}^{\{\varsigma\}};E^{\{\varsigma\}})$.
\end{proof}
\begin{defn}
For $q=\prod_{v\mid q}\varpi_{v}^{\mathrm{ord}_{v}(q)}$, define its
radical to be $r(q)=\prod_{v\mid q}\varpi_{v}$. Say that $|q|\rightarrow\infty$
smoothly if $|r(q)|\asymp1$. 

Examples of sequences of ideals going smoothly to $\infty$ are either
$q_{n}=p^{n}$ for a single prime.
\end{defn}

\begin{prop}
\label{prop:h1-bound} Fix $E\in\Pi^{\mathrm{alg}}\left(G\right)$
and let $|q|\rightarrow\infty$ smoothly. Then for any $\varsigma\in\mathcal{M}\left(G\right)$,
\[
h_{1}\left(G,\varsigma;q,E\right)\ll|q|^{\dim G^{\{\varsigma\}}}.
\]
\end{prop}

\begin{proof}
First consider $\varsigma\in\mathcal{M}^{g}\left(G\right)$. Let any
$\epsilon>0$ let $c>0$ be such that $|r(q)|\leq c|q|^{\epsilon}$
and such that Proposition~\ref{prop:h'-generic} holds for $\epsilon$
and $c$. Then for any $\psi\in\Psi_{2}^{\mathrm{AJ}}\left(G,\varsigma_{1};q,E\right)$
and any $\pi\in\Pi_{\psi}$, such that $\pi_{v}^{K_{v}}\ne0$ for
any $v\nmid q\infty$, by Propositions~\ref{prop:h'-generic} and
\ref{prop:coh-dep-A-par-pac}, we get 
\[
h\left(\psi;c^{3}|q|^{3\epsilon}\cdot q,E\right)\geq\dim H^{*}\left(\pi;s(q)^{2}q,E\right)\geq1.
\]
Hence 
\[
h\left(G,\varsigma_{1};c^{3}|q|^{3\epsilon}\cdot q,E\right)=\sum_{\psi\in\Psi_{2}^{\mathrm{AJ}}\left(G,\varsigma_{1};r(q)^{2}q,E\right)}h\left(\psi;c^{3}|q|^{3\epsilon}\cdot q,E\right)
\]
\[
\geq\sum_{\psi\in\Psi_{2}^{\mathrm{AJ}}\left(G,\varsigma_{1};q,E\right)}h\left(\psi;c^{3}|q|^{3\epsilon}\cdot q,E\right)\geq\left|\Psi_{2}^{\mathrm{AJ}}\left(G,\varsigma_{1};q,E\right)\right|.
\]
Combined with Lemma~\ref{lem:h-triv-bound} and \eqref{eq:coh-vol-dim},
and since the above holds for any $\epsilon>0$, we get
\[
h_{1}\left(G,\varsigma;q,E\right)\leq h\left(G,\varsigma_{1};c^{3}|q|^{3\epsilon}\cdot q,E\right)\ll|q|^{\dim G}.
\]
The general case reduces to the case of generic $A$-shapes by Lemma~\ref{lem:h1-varsigma}.
\end{proof}
\begin{cor}
\label{cor:h-F} Let $(\mathbf{F})=\left(\left(1,N\right)\right)\in\mathcal{M}\left(G\right)$.
Then $h\left(G,(\mathbf{F});q,E\right)\ll1$.
\end{cor}

\begin{proof}
By Proposition~\ref{prop:aut-dim-1}, any automorphic representation
of $A$-shape $(\mathbf{F})$ is $1$-dimensional, i.e. $h_{2}\left(G,(\mathbf{F});q,E\right)\leq1$
and therefore $h\left(G,(\mathbf{F});q,E\right)\ll h_{1}\left(G,(\mathbf{F});q,E\right)$
by Lemma~\ref{lem:h-h1h2}. We are left to bound $h_{1}\left(G,(\mathbf{F});q,E\right)=|\Psi_{2}^{\mathrm{AJ}}\left(G,(\mathbf{F});q,E\right)|$.
By Lemma~\ref{lem:h1-varsigma}, and using the fact that $\widehat{G^{\{(\mathbf{F})\}}}=\hat{Z}$
is the center of $\hat{G}$, which is either trivial or $\hat{Z}=O\left(1\right)$,
hence of size at most $2$, we get that $h_{1}\left(G,(\mathbf{F});q,E\right)\asymp h_{1}(Z;q,E^{\{(\mathbf{F})\}})$.
Note that $h_{1}(Z;q,E^{\{(\mathbf{F})\}})=1$ when $\hat{Z}$ is
trivial. When $\hat{Z}=O\left(1\right)$, $h_{1}(Z;q,E^{\{(\mathbf{F})\}})$
counts the number of quadratic Hecke characters with a fixed weight
at infinity and of level $q$, hence ramified only at the primes dividing
$q$, and therefore $h_{1}(Z;q,E^{\{(\mathbf{F})\}})\leq2^{\omega(q)}$
when where $\omega(q)$ is the number of prime dividing $q$. We get
that $h_{1}(Z;q,E^{\{(\mathbf{F})\}})\ll1$, which completes the proof.
\end{proof}

\subsection{Bounds for $(\mathbf{G})$, $(\mathbf{Y})$ and $(\mathbf{F})$\protect\label{subsec:h-GYF}}

Let $G$ be a Gross inner form of $SO_{5}$. The purpose of this short
subsection is to prove the requisite bounds on $h\left(G,\varsigma;q,E\right)$,
where $E\in\Pi^{\mathrm{alg}}\left(G\right)$ and $|q|\rightarrow\infty$,
for the $A$-shapes $(\mathbf{G})$, $(\mathbf{Y})$ and $(\mathbf{F})$,
and to deduce Conjecture~\ref{conj:CSXDH-shape} (CSXDH-shape) for
these shapes. These are the extreme (and easiest) cases of the the
Sarnak--Xue Density Hypothesis (at least in the cohomological case),
namely the generic $A$-shapes $(\mathbf{G})$ and $(\mathbf{Y})$
and the trivial $A$-shape $(\mathbf{F})$.
\begin{cor}
\label{cor:h-GYF} (i) For $\varsigma\in\left\{ (\mathbf{G}),(\mathbf{Y})\right\} $,
we have $h\left(G,\varsigma;q,E\right)\ll\vol\left(X\left(q\right)\right)$.

(ii) For $\varsigma=(\mathbf{F})$, we have $h\left(G,\varsigma;q,E\right)\ll1$.
\end{cor}

\begin{proof}
The first statement is a special case of Lemma~\ref{lem:h-triv-bound},
namely 
\[
h\left(G,\varsigma;q,E\right)\leq h\left(G;q,E\right)\asymp\vol\left(X\left(q\right)\right).
\]
The second statement is a special case of Corollary~\ref{cor:h-F}.
\end{proof}
From Corollaries~\ref{cor:r-GY}, \ref{cor:r-F} and \ref{cor:h-GYF},
we then deduce:
\begin{cor}
\label{cor:CSXDH-GYF} Conjecture~\ref{conj:CSXDH-shape} (CSXDH-shape)
holds for $G$ a Gross inner form of $SO_{5}$ and $\varsigma\in\left\{ (\mathbf{G}),(\mathbf{Y}),(\mathbf{F})\right\} $.
\end{cor}

We note that the above holds more generally, namely, if $G$ is a
Gross form of a split classical group, then the above arguments imply
Conjecture~\ref{conj:CSXDH-shape} for any generic shape $\varsigma\in\mathcal{M}^{g}\left(G\right)$
(by combining Proposition~\ref{prop:GRPC-classical} and Lemma~\ref{lem:h-triv-bound})
as well as for the trivial shape $\varsigma=\left(\left(1,N\right)\right)\in\mathcal{M}\left(G\right)$
(by combining Corollaries~\ref{cor:r-F} and \ref{cor:h-F}).

\subsection{Bounds for $(\mathbf{B})$ and $(\mathbf{Q})$\protect\label{subsec:h-BQ}}

Let $G$ be a Gross inner form of $SO_{5}$. We now provide a bound
on $h\left(G,\varsigma;q,E\right)$ for $\varsigma\in\left\{ (\mathbf{B}),(\mathbf{Q})\right\} $,
for a fixed $E\in\Pi^{\mathrm{alg}}\left(G\right)$ and $|q|\rightarrow\infty$.
Recall from Definition~\ref{def:h1-h2} the two quantities $h_{i}(G,\varsigma;q,E)$
for $i=1,2.$
\begin{prop}
\label{prop:h1-BQ} $h_{1}\left(G,(\mathbf{B});q,E\right)\ll1$ and
$h_{1}\left(G,(\mathbf{Q});q,E\right)\ll|q|$.
\end{prop}

\begin{proof}
If $|q|\rightarrow\infty$ smoothly, then this follows from Proposition~\ref{prop:h1-bound},
since $G^{\left\{ \mathbf{B}\right\} }\cong O_{1}\times O_{1}$ hence
$\dim G^{\left\{ \mathbf{B}\right\} }=0$, and $G^{\left\{ \mathbf{Q}\right\} }\cong O_{2}$
hence $\dim G^{\{\mathbf{Q}\}}=1$. Since we do not wish to restrict
ourselves only to $|q|\rightarrow\infty$ smoothly, we shall give
another proof without this assumption. By Lemma~\ref{lem:h1-varsigma},
it suffices to prove that 
\[
h_{1}\left(O_{1}\times O_{1};q,E'\right)\ll1\qquad\mbox{and}\qquad h_{1}\left(O_{2};q,E'\right)\ll|q|.
\]

We begin with proving the estimate on $h_{1}\left(O_{1}\times O_{1};q,E'\right)$.
By \cite{Art04}, $\Psi_{2}\left(O_{1}\times O_{1}\right)$ is the
set of pairs of quadratic Hecke characters of $GL_{1}/F$, and $h_{1}\left(O_{1}\times O_{1};q,E'\right)$
counts the number pairs of such quadratic Hecke characters of level
$q$ with a certain condition at the infinite places (coming from
$E'$), and therefore it is bounded by the square of $h_{1}^{1}\left(q\right)$,
the number of quadratic Hecke characters of level $q$, hence it suffice
to prove that $h_{1}^{1}\left(q\right)\ll1$. Since $O_{1}\left(F_{v}\right)\cong\mathbb{Z}/2\mathbb{Z}$
for any $v$, the local group admits only $2$ characters. For any
Hecke character $\chi$ of level $q$, if $v\nmid q\infty$ then $\chi_{v}=\mathbf{1}$,
so that $h_{1}^{1}\left(q\right)$ is bounded by $2^{\omega(q)}$.
By Lemma~\ref{lem:prime-omega}, $2^{\omega(q)}\ll1$, which provides
the desired bound.

Next we shall prove the estimate on $h_{1}\left(O_{2};q,E'\right)$.
By \cite[Section 5]{Art04}, $\Psi_{2}\left(O_{2}\right)$ is the
set of self-dual, unitary, cuspidal automorphic representations of
$GL_{2}/F$ of the form $\mu=\mu(\theta)$, an automorphic induction
of a Hecke character $\theta$ of $GL_{1}/K$ for a quadratic extension
$K/F$, and $h_{1}\left(O_{2};q,E'\right)$ is the number of such
automorphic representations of level $q$ with a certain algebraic
condition at the infinite places (coming from $E'$), and therefore
$h_{1}\left(O_{2};q,E'\right)$ is bounded by $h_{1}^{2}\left(q\right)$,
the number of pairs $\left(K,\theta\right)$ of the above form. The
level $q$ condition requires that $\left(K,\theta\right)$ are ramified
only at the primes dividing $q$, and by class field theory we get
that there are only $\log|q|\ll1$ such quadratic fields $K/F$, and
for each such $K$ there are at most $h_{1}\left(GL_{1}/K;q,E''\right)$
such Hecke characters (where $E''$ is determined uniquely by $E'$).
Finally we get that $h_{1}\left(GL_{1}/K;q,E''\right)\asymp|q|$,
since the $A$-parameter of $GL_{1}/K$ are in bijection with the
automorphic representations of $GL_{1}/K$ (again by class field theory).
\end{proof}
\begin{prop}
\label{prop:h2-BQ} For any $\epsilon>0$, there exists $c_{\epsilon}>0$,
such that for any $\varsigma\in\left\{ (\mathbf{B}),(\mathbf{Q})\right\} $,
$\psi\in\Psi_{2}\left(G,\varsigma\right)$, finite place $v$, $\pi\in\Pi_{\psi_{v}}$
and $n\in\mathbb{N}$, 
\[
\dim\pi^{K_{v}\left(\varpi_{v}^{n}\right)}\leq c_{\epsilon}p_{v}^{(4+\epsilon)n}.
\]
\end{prop}

\begin{proof}
By Proposition~\ref{prop:Schmidt-BQ}, $\pi$ is a subquotient of
a parabolically induced representation $\mbox{ind}\,\sigma$, where
$\sigma\in\Pi\left(M\left(F_{v}\right)\right)$ and $M$ is a Levi
subgroup of either the Borel subgroup $B$ or the Klingen subgroup
$Q$. By Frobenius reciprocity and Mackey theory we get 
\[
\dim\left(\pi\right)^{K_{v}\left(\varpi_{v}^{n}\right)}\leq\dim\left(\mbox{ind\,}\sigma\right)^{K_{v}\left(\varpi_{v}^{n}\right)}\leq p_{v}^{\left(\frac{\dim G-\dim M}{2}\right)\cdot n}\dim\left(\sigma\right)^{M\left(F_{v}\right)\cap K_{v}\left(\varpi_{v}^{n}\right)}.
\]
In the Borel case $M\cong GL_{1}\times GL_{1}$, hence $\dim G-\dim M=8$,
and since $M$ is abelian then $\dim\left(\sigma\right)=1$ for any
$\sigma$, and we get the claim. In the Klingen case $M\cong GL_{2}$,
hence $\dim G-\dim M=6$, and by \cite[Corollary A.4]{MS19} (see
also \cite{Lap19}), for any $\epsilon>0$, there exists $c_{\epsilon}>0$
(in the notation of \cite[Corollary A.4]{MS19}, $c_{\epsilon}:=\max\left\{ C(\epsilon,p_{v})\,:\,p_{v}\leq q(\epsilon)\right\} $),
such that for any finite place $v$ and any $n$,
\[
\dim\left(\sigma\right)^{GL_{2}\left(F_{v}\right)\cap K_{v}\left(\varpi_{v}^{n}\right)}\leq c_{\epsilon}p_{v}^{n+\epsilon},\qquad\forall\sigma\in\Pi\left(GL_{2}\left(F_{v}\right)\right),
\]
which proves the claim in this case as well. 
\end{proof}
\begin{cor}
\label{cor:h2-BQ} For $\varsigma\in\left\{ (\mathbf{B}),(\mathbf{Q})\right\} $,
we have 
\[
h_{2}\left(G,\varsigma;q,E\right)\ll|q|^{4}.
\]
\end{cor}

\begin{proof}
Let $\psi\in\Psi_{2}\left(G,\varsigma;q,E\right)$ and let $q=\prod_{i=1}^{r}\varpi_{v_{i}}^{n_{i}}$
be its decomposition into prime factors. Then by Proposition~\ref{prop:h2-BQ}
we get 
\[
\prod_{v\mid q}\max_{\pi_{v}\in\Pi_{\psi_{v}}}\dim\pi_{v}^{K_{v}\left(q\right)}=\prod_{i=1}^{r}\max_{\pi_{v_{i}}\in\Pi_{\psi_{v_{i}}}}\dim\pi_{v_{i}}^{K_{v_{i}}(\varpi_{v_{i}}^{n_{i}})}\leq\prod_{i=1}^{r}c_{\epsilon}p_{v_{i}}^{(4+\epsilon)n_{i}}\ll|q|^{4},
\]
where the last estimate follows from Lemma~\ref{lem:asymptotic},
which proves the claim.
\end{proof}
We are now able to prove the CSXDH for gross inner forms of $SO_{5}$,
for the $A$-shapes of Howe--Piatetski-Shapiro $(\mathbf{B})$ and
Soudry $(\mathbf{Q})$ .
\begin{cor}
\label{cor:h-BQ} For $\varsigma=(\mathbf{Q})$, we have $h\left(G,\varsigma;q,E\right)\ll\vol\left(X\left(q\right)\right)^{\frac{1}{2}}$. 

For $\varsigma=(\mathbf{B})$, we have $h\left(G,\varsigma;q,E\right)\ll\vol\left(X\left(q\right)\right)^{\frac{2}{5}}$.
\end{cor}

\begin{proof}
By combining Lemma~\ref{lem:h-h1h2}, Proposition~\ref{prop:h1-BQ}
and Corollary~\ref{cor:h2-BQ}, we get 
\[
h\left(G,\varsigma;q,E\right)\ll h_{1}\left(G,\varsigma;q,E\right)\cdot h_{2}\left(G,\varsigma;q,E\right)\ll|q|^{\delta+4},
\]
where $\delta=1$ for $\varsigma=(\mathbf{Q})$ and $\delta=0$ for
$\varsigma=(\mathbf{B})$. The claim now follows from Lemma~\ref{lem:coh-vol-dim}. 
\end{proof}
\begin{cor}
\label{cor:CSXDH-BQ} Conjecture~\ref{conj:CSXDH-shape} (CSXDH-shape)
holds for $G$ a Gross inner form of $SO_{5}$ and $\varsigma\in\left\{ (\mathbf{B}),(\mathbf{Q})\right\} $.
\end{cor}

\begin{proof}
Follows from Corollaries~\ref{cor:r-weak-BQP} and \ref{cor:h-BQ},
\[
h\left(G,\varsigma;q,E\right)\ll\vol\left(X\left(q\right)\right)^{\frac{1}{2}}=\vol\left(X\left(q\right)\right)^{\frac{2}{r\left(\varsigma\right)}}.
\]
\end{proof}

\subsection{Bounds via endoscopy\protect\label{subsec:h-endoscopy}}

Let $G$ be a Gross inner form of a split classical group, defined
over a totally real field $F$. The purpose of this section will be
to give general bounds on the terms $h\left(\psi;q,E\right)$ and
$h\left(\pi;q,E\right)$ from Definition~\ref{def:h(pi,q,E)}, using
the endoscopic character relations (see \cite[Chapter 2]{Art13}).

For each place $v$, recall that $\mathcal{\mathcal{H}}\left(G_{v}\right)$
is the local Hecke algebra of smooth, compactly supported, bi-$K_{v}$-finite
functions, and for each $\pi_{v}\in\Pi\left(G_{v}\right)$, consider
the distribution $\mathrm{tr}\,\pi_{v}:\mathcal{H}\left(G_{v}\right)\to\mathbb{C}$,
defined as the trace of the operator 
\begin{equation}
\pi_{v}(f_{v})(x):=\int_{G_{v}}f_{v}(g)\pi_{v}(g)(x)d\mu_{G_{v}}.\label{eq:local-integral}
\end{equation}
Let $\mathcal{\mathcal{H}}\left(G\right)=\otimes'_{v}\mathcal{H}\left(G_{v}\right)$
be the corresponding global Hecke algebra, and for each $\pi\in\Pi\left(G\left(\mathbb{A}\right)\right)$,
consider the distribution $\mathrm{tr}\,\pi:\mathcal{H}\left(G\right)\to\mathbb{C}$,
i.e. the trace of the operator 
\[
\pi(f)(x):=\int_{G(\mathbb{A})}f(g)\pi(g)(x)d\mu
\]
Note that if $f=\prod_{v}f_{v}$$,$ then $\mathrm{tr}\,\pi(f)=\prod_{v}\mathrm{tr}\,\pi_{v}(f_{v}).$
\begin{defn}
\label{def:f(q,E)} Let $E=\otimes_{v\mid\infty}E_{v}$, $E_{v}\in\Pi^{\mathrm{alg}}\left(G\right)$,
and let $q\in\mathcal{O}$. Define the global function $f\left(q,E\right):=\prod_{v\mid\infty}f_{v}\left(E_{v}\right)\prod_{v\nmid\infty}f_{v}\left(q\right)\in\mathcal{H}\left(G\right)$,
as follows:
\begin{itemize}
\item For $v\nmid\infty$, let $f_{v}\left(q\right)$ be the normalized
indicator function $\tilde{\mathbf{1}}_{K_{v}\left(q\right)}:=\mu_{G_{v}}^{-1}\left(K_{v}\left(q\right)\right)\cdot\mathbf{1}_{K_{v}\left(q\right)}$.
Note that $f_{v}\left(q\right)=\mathbf{1}_{K_{v}}$ for $v\nmid q$. 
\item For $v\mid\infty$, let $f_{v}\left(E_{v}\right)$ be the Euler-Poincaré
function (see \cite{CD90}) such that $\mathrm{tr}\,\pi_{v}\left(f_{v}\left(E_{v}\right)\right)=\dim H^{*}\left(\pi_{v};E_{v}\right)$
for all $\pi_{v}\in\Pi\left(G_{v}\right)$. 
\end{itemize}
\end{defn}

\begin{lem}
\label{lem:f(q,E)} Let $f\left(q,E\right)$ be as in Definition~\ref{def:f(q,E)}.
Then
\[
\mathrm{\mathrm{tr}}\,\pi\left(f\left(q,E\right)\right)=h\left(\pi;q,E\right),\qquad\forall\pi\in\Pi\left(G\left(\mathbb{A}\right)\right).
\]
\end{lem}

\begin{proof}
By the definition of the Euler-Poincaré functions we get that $\prod_{v\mid\infty}\mathrm{\mathrm{tr}}\,\pi_{v}\left(f_{v}\left(E_{v}\right)\right)=\dim H^{*}\left(\pi_{\infty};E\right)$,
for any $\pi_{\infty}=\otimes_{v\mid\infty}\pi_{v}$. By~\eqref{eq:local-integral},
we get that $\pi_{v}\left(\tilde{\mathbf{1}}_{K_{v}\left(q\right)}\right)$
is the projector onto the subspace of $K_{v}\left(q\right)$-fixed
vectors in $\pi_{v}$, hence $\mathrm{\mathrm{tr}}\,\pi_{v}\left(f_{v}\left(q,E\right)\right)=\dim\pi_{v}^{K_{v}\left(q\right)}$,
for any finite place $v$. Combining it all we get 
\[
\mathrm{\mathrm{tr}}\,\pi\left(f\left(q,E\right)\right)=\prod_{v}\mathrm{\mathrm{tr}}\,\pi_{v}\left(f_{v}\left(q,E\right)\right)=\dim H^{*}\left(\pi_{\infty};E\right)\prod_{v\nmid\infty}\dim\pi_{v}^{K_{v}\left(q\right)}=h\left(\pi;q,E\right).
\]
\end{proof}
Let $\mathcal{E}_{2}\left(G\right)$ be the set of isomorphism classes
of elliptic endoscopic datum of $G$ as defined in \cite[Section 2.3]{Tai18}
(for groups over both local and global fields). For $\mathfrak{e}\in\mathcal{E}_{2}\left(G\right)$,
denote by $G^{\mathfrak{e}}$ the elliptic endoscopic group of $G$,
whose dual is $\widehat{G^{\mathfrak{e}}}=C_{\hat{G}}\left(s^{\mathfrak{e}}\right)^{0}$,
the connected centralizer of the semisimple element $s^{\mathfrak{e}}\in\hat{T}$,
and $\xi^{\mathfrak{e}}\,:\,\widehat{G^{\mathfrak{e}}}\rightarrow\widehat{G}$
the associated embedding. Such a group is of the form $G^{\mathfrak{e}}=G_{1}^{\mathfrak{e}}\times G_{2}^{\mathfrak{e}}$,
where $G_{i}^{\mathfrak{e}}$ is a quasi-split classical group. Call
$\mathfrak{e}\in\mathcal{E}_{2}\left(G\right)$ non-simple if $G_{1}^{\mathfrak{e}}$
and $G_{2}^{\mathfrak{e}}$ are non-trivial, and call $\mathfrak{e}\in\mathcal{E}_{2}\left(G\right)$
split if $G_{1}^{\mathfrak{e}}$ and $G_{2}^{\mathfrak{e}}$ are split
classical groups. For example:
\begin{itemize}
\item If $G=SO_{2n+1}$, then $\mathcal{E}_{2}\left(G\right)$ is parametrized
by $a\in\left\{ 0,1,\ldots,\left\lfloor \frac{n}{2}\right\rfloor \right\} $,
such that $G^{a}=SO_{2a+1}\times SO_{2(n-a)+1}$. The endoscopic datum
of $a$ is always split and it is non-simple when $a\ne0$. 
\item If $G=Sp_{2n}$, then $\mathcal{E}_{2}\left(G\right)$ is parametrized
by $\left(a,\alpha\right)\in\left\{ 0,1,\ldots,n\right\} \times\left\{ \pm\right\} $,
$\left(a,\alpha\right)\ne\left(n-1,+\right),\left(n,-\right)$, such
that $G^{a,\alpha}=Sp_{2a}\times SO_{2(n-a)}^{\alpha}$, where $SO_{2(n-a)}^{+}=SO_{2(n-a)}$
is split and $SO_{2(n-a)}^{-}$ is quasi-split but non-split. The
endoscopic datum of $\left(a,\alpha\right)$ is split when $\alpha=+$
and it is non-simple when $\left(a,\alpha\right)\ne\left(0,\pm\right),\left(n,+\right)$. 
\item If $G=SO_{2n}$, then $\mathcal{E}_{2}\left(G\right)$ is parametrized
by $\left(a,\alpha\right)\in\left\{ 0,1,\ldots,\left\lfloor \frac{n}{2}\right\rfloor \right\} \times\left\{ \pm\right\} $,
$\left(a,\alpha\right)\ne\left(0,-\right),\left(1,+\right)$, such
that $G^{a,\alpha}=SO_{2a}^{\alpha}\times SO_{2(n-a)}^{\alpha}$.
The endoscopic datum of $\left(a,\alpha\right)$ is split when $\alpha=+$
and it is non-simple when $\left(a,\alpha\right)\ne\left(0,+\right)$. 
\end{itemize}
Recall that each $A$-parameter $\psi$ of $G$ determines a group
$\mathcal{S}_{\psi}=S_{\psi}/S_{\psi}^{0}\hat{Z}$, where $S_{\psi}$
is the centralizer of the image of $\psi$ in $\hat{G}$ with identity
component $S_{\psi}^{0}$, and $\hat{Z}$ is the center of $\hat{G}$.
\begin{prop}
\cite[Section 1.4]{Art13}\cite[Proposition 2.4.1]{Tai18} \label{prop:endoscopy-bijection}
Let $\psi$ be a local or global $A$-parameter of $G$. Then for
any $\bar{s}\in\mathcal{S_{\psi}}$, there exists a unique $\mathfrak{e}^{\bar{s}}=\left(G^{\bar{s}},\bar{s},\xi^{\bar{s}}\right)\in\mathcal{E}_{2}\left(G\right)$,
and a unique $A$-parameter $\psi^{\bar{s}}$ of $G^{\bar{s}}$, such
that $\xi^{\bar{s}}\circ\psi^{\bar{s}}=\psi$, and the following map
is a bijection
\[
\mathcal{S}_{\psi}\leftrightarrow\left\{ \left(\mathfrak{e},\psi'\right)\,:\,\mathfrak{e}\in\mathcal{E}_{2}\left(G\right),\;\psi'\in\Psi\left(G^{\mathfrak{e}}\right),\;\xi^{\mathfrak{e}}\circ\psi'=\psi\right\} ,\qquad\bar{s}\mapsto\left(\mathfrak{e}^{\bar{s}},\psi^{\bar{s}}\right).
\]
\end{prop}

Among all elliptic endoscopic data of $G$, we shall be particularly
interested in those associated to Arthur $SL_{2}$-types as below.
Recall that an Arthur $SL_{2}$-type is an element of $\mathcal{D}\left(G\right):=\mbox{Hom}(SL_{2}^{A},\hat{G})/\hat{G}$.
\begin{defn}
\label{def:assoc-endiscopic-datum} Let $\sigma\in\mathcal{D}\left(G\right)$
be an Arthur $SL_{2}$-type. Its corresponding split elliptic endoscopic
datum is $\mathfrak{e}^{\sigma}=\left(G^{\sigma},s^{\sigma},\xi^{\sigma}\right)\in\mathcal{E}_{2}\left(G\right)$,
where $s^{\sigma}=\sigma\left(-I\right)$, defined up to conjugacy,
the group $G^{\sigma}$ has dual $\widehat{G^{\sigma}}=C_{\hat{G}}\left(s^{\sigma}\right)^{0}$,
and $\xi^{\sigma}$ is the natural embedding of $\widehat{G^{\sigma}}$
into $\hat{G}$. The deficiency of $\sigma$ is 
\[
\delta\left(\sigma\right):=\frac{1}{2}\left(\dim G-\dim G^{\sigma}\right)=\frac{1}{2}\left(\dim\widehat{G}-\dim\widehat{G^{\sigma}}\right).
\]
\end{defn}

The definition of $G^{\sigma}$ is motivated by the following result,
which describes a particularly simple case of the endoscopic character
relations of the local Arthur classification, proved in the split
case by Arthur in \cite{Art13} and extended to the non-split real
case by Taïbi \cite{Tai18}. 
\begin{prop}
\label{prop:endoscopy-local} Let $v$ be a place and $\psi_{v}\in\Psi\left(G_{v}\right)$
with Arthur $SL_{2}$-type $\sigma=\psi_{v}|_{SL_{2}^{A}}\in\mathcal{D}\left(G\right)$.
Then $s_{\psi_{v}}=s^{\sigma}$ and $G_{v}^{\bar{s}_{\psi_{v}}}=G_{v}^{\sigma}$.
Denote $\psi_{v}^{\sigma}:=\psi_{v}^{\bar{s}_{\psi_{v}}}\in\Psi\left(G_{v}^{\sigma}\right)$,
and for any $f_{v}\in\mathcal{H}\left(G_{v}\right)$, let $f_{v}^{\sigma}:=f_{v}^{G_{v}^{\sigma}}\in\mathcal{H}\left(G_{v}^{\sigma}\right)$
be a transfer of $f_{v}$ as in \cite[Section 2.1]{Art13}. Then 
\[
e\left(G_{v}\right)\sum_{\pi_{v}\in\Pi_{\psi_{v}}}\mathrm{tr}\,\pi_{v}\left(f_{v}\right)=\sum_{\pi'_{v}\in\Pi_{\psi_{v}^{\sigma}}}\mathrm{tr}\,\pi'_{v}\left(f_{v}^{\sigma}\right),
\]
where $e\left(G_{v}\right)$ is the Kottwitz sign defined in \eqref{eq:Kottwitz-sign},
in particular $e\left(G_{v}\right)=1$ for $G_{v}$ split. 
\end{prop}

\begin{proof}
Clearly $s_{\psi_{v}}=\psi_{v}\left(1,-I\right)=\sigma\left(-I\right)=s^{\sigma}$,
and by \cite[Theorem 2.2.1]{Art13} in the split case, and \cite[Proposition 3.2.5]{Tai18}
in the non-split case, we get 
\[
e\left(G_{v}\right)\cdot\sum_{\pi_{v}\in\Pi_{\psi_{v}}}\langle\bar{s}_{\psi_{v}}\cdot\bar{s}_{\psi_{v}},\pi_{v}\rangle_{\psi_{v}}\mathrm{tr}\,\pi_{v}\left(f_{v}\right)=\sum_{\pi_{v}\in\Pi_{\psi_{v}^{\sigma}}}\langle\bar{s}_{\psi_{v}},\pi'_{v}\rangle_{\psi_{v}^{\sigma}}\mathrm{tr}\,\pi'_{v}\left(f_{v}^{\sigma}\right).
\]
The claim follows from the fact that $s_{\psi_{v}}^{2}=1$ and $s_{\psi_{v}}\in Z\left(G_{v}^{\sigma}\right)$,
hence $\langle\bar{s}_{\psi_{v}}\cdot\bar{s}_{\psi_{v}},\cdot\rangle_{\psi_{v}}\equiv1$
and $\langle\bar{s}_{\psi_{v}},\cdot\rangle_{\psi_{v}^{\sigma}}\equiv1$.
\end{proof}
The local results give rise to the following global identity. 
\begin{cor}
\label{cor:endoscopy-global} Let $\psi\in\Psi_{2}\left(G\right)$
with Arthur $SL_{2}$-type $\sigma=\tilde{\psi}|_{SL_{2}^{A}}\in\mathcal{D}\left(G\right)$.
For $f=\otimes_{v}f_{v}\in\mathcal{H}\left(G\right)$, if $f^{\sigma}=\otimes_{v}f_{v}^{\sigma}\in\mathcal{H}\left(G^{\sigma}\right)$,
is such that $f_{v}^{\sigma}$ is a transfer of $f_{v}$ for any $v$,
then 
\[
\sum_{\pi\in\Pi_{\psi}}\mathrm{tr}\,\pi\left(f\right)=\sum_{\pi'\in\Pi_{\psi^{\sigma}}}\mathrm{tr}\,\pi'\left(f^{\sigma}\right).
\]
\end{cor}

\begin{proof}
By Proposition~\ref{prop:endoscopy-local}, 
\[
\sum_{\pi\in\Pi_{\psi}}\mathrm{tr}\,\pi\left(f\right)=\prod_{v}\sum_{\pi_{v}\in\Pi_{\psi_{v}}}\mathrm{tr}\,\pi_{v}\left(f_{v}\right)=\prod_{v}e\left(G_{v}\right)\sum_{\pi'_{v}\in\Pi_{\psi_{v}^{\sigma}}}\mathrm{tr}\,\pi'_{v}\left(f_{v}^{\sigma}\right)=\prod_{v}e\left(G_{v}\right)\cdot\sum_{\pi'\in\Pi_{\psi^{\sigma}}}\mathrm{tr}\,\pi'\left(f^{\sigma}\right),
\]
and by Proposition~\ref{prop:Gross-forms}, $\prod_{v}e\left(G_{v}\right)=1$,
which completes the proof.
\end{proof}
The following result summarizes the Fundamental Lemma of Ngo and Waldspurger
\cite{Ngo10,Wal97}, together with its generalization to congruence
subgroups by Ferrari \cite{Fer07}, as it pertains to our study of
endoscopic subgroups related to Arthur $SL_{2}$-type.
\begin{prop}
\label{prop:Ferrari-local} Let $\sigma\in\mathcal{D}\left(G\right)$,
$v$ a finite place, $K_{v}^{\sigma}\leq G_{v}^{\sigma}$ a hyperspecial
maximal compact subgroup and $K_{v}^{\sigma}\left(\varpi_{v}^{n}\right)$
is its level $n$ congruence subgroup for $n\in\mathbb{N}_{0}$. Then:
\begin{enumerate}
\item $\mathbf{1}_{K_{v}^{\sigma}}\in\mathcal{H}\left(G_{v}^{\sigma}\right)$
is a transfer of $\mathbf{1}_{K_{v}}\in\mathcal{H}\left(G_{v}\right)$.
\item $p_{v}^{-\delta\left(\sigma\right)n}\cdot\frac{\mu_{G_{v}^{\sigma}}\left(K_{v}^{\sigma}\left(\varpi_{v}^{n}\right)\right)}{\mu_{G_{v}}\left(K_{v}\left(\varpi_{v}^{n}\right)\right)}\cdot\tilde{\mathbf{1}}_{K_{v}^{\sigma}\left(\varpi_{v}^{n}\right)}\in\mathcal{H}\left(G_{v}^{\sigma}\right)$
is a transfer of $\mathbf{1}_{K_{v}\left(\varpi^{n}\right)}\in\mathcal{H}\left(G_{v}\right)$,
for any $n\in\mathbb{N}$, assuming $p_{v}>10\cdot[F:\mathbb{Q}]+1$.
\end{enumerate}
\end{prop}

\begin{proof}
(1) Follows from the Fundamental Lemma proved by Ngo \cite{Ngo10}.
(2) Ferrari \cite[Théorème 3.2.3]{Fer07} proved that the transfer
from $\mathcal{H}\left(G_{v}\right)$ to $\mathcal{H}\left(G_{v}^{\sigma}\right)$
of $\mathbf{1}_{K_{v}\left(\varpi_{v}^{n}\right)}$ is
\[
\chi_{G_{v}/G_{v}^{\sigma}}\left(\varpi_{v}^{n}\right)\cdot p_{v}^{-\delta\left(\sigma\right)n}\cdot\mu_{G_{v}}^{-1}\left(K_{v}\left(\varpi_{v}^{n}\right)\right)\cdot\mathbf{1}_{K_{v}^{G^{\sigma}}\left(\varpi_{v}^{n}\right)}.
\]
The character $\chi_{G_{v}/G_{v}^{\sigma}}$ measures the difference
in Galois actions on the maximal tori of $G_{v}$ and $G_{v}^{\sigma}$,
respectively, and since both groups are split, we have $\chi_{G_{v}/G_{v}^{\sigma}}\equiv1$.
\end{proof}
Denote by $\mathcal{H}\left(G_{f}\right)=\bigotimes'_{v\nmid\infty}\mathcal{H}\left(G_{v}\right)$
the Hecke algebra of the finite adelic group $G\left(\mathbb{A}_{f}\right)=\prod'_{v\nmid\infty}G_{v}$.
Assume from now on that $q\lhd\mathcal{O}$ is such that for any finite
place dividing it $v\mid q$, its residue degree is $p_{v}>10\cdot[F:\mathbb{Q}]+1$.
\begin{cor}
\label{cor:Ferrari-global} Let $\sigma\in\mathcal{D}\left(G\right)$.
Denote by $f_{f}\left(q\right)^{\sigma}=\prod_{v\nmid\infty}f_{v}\left(q\right)^{\sigma}$
the transfer from $\mathcal{H}\left(G_{f}\right)$ to $\mathcal{H}(G_{f}^{\sigma})$
of $f_{f}\left(q\right)=\prod_{v\nmid\infty}f_{v}\left(q\right)$,
with $f_{v}\left(q\right)$ as in Definition~\ref{def:f(q,E)}. Then
for any $\pi'_{f}\in\Pi\left(G^{\sigma}\left(\mathbb{A}_{f}\right)\right)$,
\[
\mathrm{\mathrm{tr}}\,\pi'_{f}\left(f_{f}\left(q\right)^{\sigma}\right)\asymp|q|^{\delta\left(\sigma\right)}\cdot\dim\left(\pi'_{f}\right)^{K^{\sigma}\left(q\right)}.
\]
\end{cor}

\begin{proof}
By Proposition~\ref{prop:Ferrari-local}, 
\[
f_{f}\left(q\right)^{\sigma}=\prod_{v\nmid\infty}f_{v}\left(q\right)^{\sigma}=\prod_{v\nmid\infty}p_{v}^{-\delta\left(\sigma\right)n}\cdot\frac{\mu_{G_{v}^{\sigma}}\left(K_{v}^{\sigma}\left(\varpi_{v}^{n}\right)\right)}{\mu_{G_{v}}\left(K_{v}\left(\varpi_{v}^{n}\right)\right)}\cdot\tilde{\mathbf{1}}_{K_{v}^{\sigma}\left(\varpi_{v}^{n}\right)}.
\]
Since $\mathrm{tr}\,\pi'_{v}\left(\tilde{\mathbf{1}}_{K_{v}^{\sigma}\left(\varpi_{v}^{n}\right)}\right)=\dim\left(\pi'_{v}\right)^{K_{v}^{\sigma}\left(\varpi_{v}^{n}\right)}$,
we get that
\[
\mathrm{\mathrm{tr}\,}\pi'\left(f_{f}\left(q\right)^{\sigma}\right)=\left(\prod_{v\mid q}p_{v}^{-\delta\left(\sigma\right)n}\cdot\frac{\mu_{G_{v}^{\sigma}}\left(K_{v}^{\sigma}\left(\varpi_{v}^{n}\right)\right)}{\mu_{G_{v}}\left(K_{v}\left(\varpi_{v}^{n}\right)\right)}\right)\cdot\dim\left(\pi'_{f}\right)^{K^{\sigma}\left(q\right)}.
\]
For a fixed finite $v$, by Lemma~\ref{lem:coh-vol-dim} we have
$\mu_{H_{v}}^{-1}\left(K_{v}^{H}\left(\varpi_{v}^{n}\right)\right)\asymp|H\left(\mathcal{O}_{v}/\varpi_{v}^{n}\mathcal{O}_{v}\right)|\asymp p_{v}^{n\cdot\dim H}$,
for any split group $H$, hence combined with Lemma~\ref{lem:asymptotic}
we get 
\[
\prod_{v\mid q}p_{v}^{-\delta\left(\sigma\right)n}\cdot\frac{\mu_{G_{v}^{\sigma}}\left(K_{v}^{\sigma}\left(\varpi_{v}^{n}\right)\right)}{\mu_{G_{v}}\left(K_{v}\left(\varpi_{v}^{n}\right)\right)}\asymp\prod_{v\mid q}p_{v}^{\delta\left(\sigma\right)n}\asymp|q|^{\delta\left(\sigma\right)}.
\]
\end{proof}
Next we wish to describe the endoscopic transfer in the cohomological
case. Namely, for a given cohomological $A$-parameter $\psi$ with
Arthur $SL_{2}$-type $\sigma$ and an infinitisimal character corresponding
to the algebraic representation $E$, show that its endoscopic transfer
$\psi^{\sigma}$ is cohomological and determine its possible infinitisimal
characters. This will be the content of Proposition~\ref{prop:endoscopy-cohomology}.
We first introduce some concepts and a lemma. 

Let $\mathfrak{e}\in\mathcal{E}_{2}\left(G\right)$ and let $G^{\mathfrak{e}}=G_{1}^{\mathfrak{e}}\times G_{2}^{\mathfrak{e}}$
its corresponding endoscopic group. Note that $\hat{T}$ is a maximal
torus of $\widehat{G^{\mathfrak{e}}}$, that $\hat{T}_{i}:=\hat{T}\cap\widehat{G_{i}^{\mathfrak{e}}}$
is a maximal torus of $\widehat{G_{i}^{\mathfrak{e}}}$, for $i=1,2$,
and that $\hat{T}=\hat{T}_{1}\times\hat{T}_{2}$, hence $X_{*}(\hat{T})=X_{*}(\hat{T}_{1})\oplus X_{*}(\hat{T}_{2})$.
For any $\chi\in X_{*}(\hat{T})\otimes\mathbb{R}$ denote by $\chi_{i}\in X_{*}(\hat{T}_{i})\otimes\mathbb{R}$
its projection, for $i=1,2$. Denote by $\rho,\rho_{\mathfrak{e}}\in X_{*}(\hat{T})\otimes\mathbb{R}$
and $\rho_{\mathfrak{e}}^{i}\in X_{*}(\hat{T}_{i})\otimes\mathbb{R}$,
for $i=1,2$, the half-sums of positive coroots of $\hat{G},\widehat{G^{\mathfrak{e}}}$
and $\widehat{G_{i}^{\mathfrak{e}}}$, respectively. 
\begin{lem}
\label{lem:endoscopy-algebraic} Let $\mathfrak{e}\in\mathcal{E}_{2}\left(G\right)$.
Then $\rho_{i}-\rho_{\mathfrak{e}}^{i}\in X_{*}(\hat{T}_{i})$ , for
$i=1,2$, and $\rho-\rho_{\mathfrak{e}}\in X_{*}(\hat{T})$. 
\end{lem}

\begin{proof}
If $G=SO_{2n+1}$ and $G^{\mathfrak{e}}=SO_{2a+1}\times SO_{2(n-a)+1}$
for $0<a\leq\left\lfloor \frac{n}{2}\right\rfloor $, then $\rho=(\frac{2n-1}{2},\ldots,\frac{1}{2})$,
$\rho_{1}=(\frac{2n-1}{2},\ldots,\frac{2\left(n-a\right)+1}{2})$,
$\rho_{2}=(\frac{2\left(n-a\right)-1}{2},\ldots,\frac{1}{2})$, $\rho_{\mathfrak{e}}^{1}=(\frac{2a-1}{2},\ldots,\frac{1}{2})$,
$\rho_{\mathfrak{e}}^{1}=(\frac{2\left(n-a\right)-1}{2},\ldots,\frac{1}{2})$.
If $G=Sp_{2n}$ and $G^{\mathfrak{e}}=Sp_{2a}\times SO_{2(n-a)}^{\alpha}$
for $0<a<n-1$ and $\alpha=\pm$, then $\rho=\left(n,\ldots,1\right)$,
$\rho_{1}=\left(n,\ldots,n-a+1\right)$, $\rho_{2}=\left(n-a,\ldots,1\right)$,
$\rho_{\mathfrak{e}}^{1}=\left(a,\ldots,1\right)$, $\rho_{\mathfrak{e}}^{2}=\left(n-a,\ldots,1\right)$.
Finally, if $G=SO_{2n}$ and $G^{\mathfrak{e}}=SO_{2a}^{\alpha}\times SO_{2(n-a)}^{\alpha}$
for $0<a\leq\left\lfloor \frac{n}{2}\right\rfloor $ and $\alpha=\pm$,
then $\rho=\left(n,\ldots,1\right)$, $\rho_{1}=\left(n,\ldots,n-a+1\right)$,
$\rho_{2}=\left(n-a,\ldots,1\right)$, $\rho_{\mathfrak{e}}^{1}=\left(a,\ldots,1\right)$,
$\rho_{\mathfrak{e}}^{2}=\left(n-a,\ldots,1\right)$. Therefore in
all cases, $\rho_{1}-\rho_{\mathfrak{e}}^{1}=\left(n-a\right)\mathbf{1}$
and $\rho_{2}-\rho_{\mathfrak{e}}^{2}=0$ , hence $\rho_{i}-\rho_{\mathfrak{e}}^{i}\in X_{*}(\hat{T}_{i})$,
for $i=1,2$. Finally, since $\widehat{G^{\mathfrak{e}}}=\widehat{G_{1}^{\mathfrak{e}}}\times\widehat{G_{2}^{\mathfrak{e}}}$,
$\hat{T}=\hat{T}_{1}\times\hat{T}_{2}$ and $\rho_{\mathfrak{e}}=\rho_{\mathfrak{e}}^{1}+\rho_{\mathfrak{e}}^{2}$,
then $(\rho_{\mathfrak{e}})_{i}=\rho_{\mathfrak{e}}^{i}$ for $i=1,2$,
hence $\rho-\rho_{\mathfrak{e}}\in X_{*}(\hat{T})$ .
\end{proof}
For any $\psi\in\Psi_{2}\left(G\right)$ with Arthur $SL_{2}$-type
$\psi|_{SL_{2}^{A}}=\sigma$, let $\psi^{\sigma}=\psi_{1}^{\sigma}\times\psi_{2}^{\sigma}\in\Psi_{2}\left(G^{\sigma}\right):=\Psi_{2}\left(G_{1}^{\sigma}\right)\times\Psi_{2}\left(G_{2}^{\sigma}\right)$
be the unique $A$-parameter defined in Proposition~\ref{prop:endoscopy-bijection}.
One of the factors in $G^{\sigma}=G_{1}^{\sigma}\times G_{2}^{\sigma}$
may be trivial so that $G^{\sigma}=G$; the following lemma considers
the infinitesimal characters of the representation of $G^{\sigma}$
associated to $\psi^{\sigma}$ when this is not the case. Recall that
the infinitesimal character of an $A$-parameter $\psi$ is $\chi_{\psi}\in X_{*}(\hat{T})$,
such that $\phi_{\psi_{\infty}}|_{W_{\mathbb{C}}}$ sends $z$ to
$z^{\chi_{\psi}}\bar{z}^{\nu_{\psi}}$. Additionally, for $E\in\Pi^{\mathrm{alg}}\left(G\right)$
a finite-dimensional representation and $\lambda_{E}\in X_{*}(\hat{T})$
its highest weight, we have $\psi\in\Psi_{2}^{\mathrm{AJ}}\left(G;E\right)$
when $\chi_{\psi}$ is in the $\mathcal{W}$-orbit of $\rho+\lambda_{E}$
for $\mathcal{W}$ the Weyl group of $G$.
\begin{prop}
\label{prop:endoscopy-cohomology} Let $\sigma\in\mathcal{D}\left(G\right)$
with $G^{\sigma}=G_{1}^{\sigma}\times G_{2}^{\sigma}$ non-simple.
Then for any $E\in\Pi^{\mathrm{alg}}\left(G\right)$, there exists
finite subsets $\Pi^{\mathrm{alg}}\left(G^{\sigma};E\right)\subset\Pi^{\mathrm{alg}}\left(G^{\sigma}\right)$
and $\Pi^{\mathrm{alg}}\left(G_{i}^{\sigma};E\right)\subset\Pi^{\mathrm{alg}}\left(G_{i}^{\sigma}\right)$,
for $i=1,2$, whose sizes are bounded by $|\mathcal{W}|$, and such
that for any $\psi\in\Psi_{2}^{\mathrm{AJ}}\left(G;E\right)$, 
\[
\psi^{\sigma}\in\bigsqcup_{E'\in\Pi^{\mathrm{alg}}\left(G^{\sigma};E\right)}\Psi_{2}^{\mathrm{AJ}}\left(G^{\sigma};E'\right),\qquad\psi_{i}^{\sigma}\in\bigsqcup_{E'\in\Pi^{\mathrm{alg}}\left(G_{i}^{\sigma};E\right)}\Psi_{2}^{\mathrm{AJ}}\left(G_{i}^{\sigma};E'\right),\quad i=1,2.
\]
\end{prop}

\begin{proof}
Let $\rho_{\sigma}$ be half the sum of the positive roots of the
endoscopic group $G^{\sigma}$. The weight $\rho-\rho_{\sigma}+\lambda_{E}$
depends on the relative positions of $G^{\sigma}$ and $\lambda_{E}$,
i.e.~on the way in which the coordinates of $\lambda_{E}$ are distributed
among the factors $\psi_{i}^{\sigma}$ of a given parameter $\psi^{\sigma}$.
Yet, by Lemma~\ref{lem:endoscopy-algebraic}, these weights are always
integral, i.e.~we have $\rho-\rho_{\sigma}+\lambda_{E}\in X_{*}(\hat{T})$
and $\rho_{i}-\rho_{\sigma}^{i}+(\lambda_{E})_{i}\in X_{*}(\hat{T}_{i})$,
$i=1,2$. Define the sets $\Pi^{\mathrm{alg}}\left(G^{\sigma};E\right)$
and $\Pi^{\mathrm{alg}}\left(G_{i}^{\sigma};E\right)$ to consist
of the finite-dimensional representations of $G^{\sigma}$ and $G_{i}^{\sigma}$
whose highest weights can be realized inside of to the $\mathcal{W}$-orbits
of $\rho-\rho_{\sigma}+\lambda_{E}$ and $\rho_{i}-\rho_{\sigma}^{i}+(\lambda_{E})_{i}$.
Since the various arrangements of the coordinates of $\lambda_{E}$
inside $\psi_{i}^{\sigma}$ belong to a single $\mathcal{W}$-orbit,
the cardinalities of both sets are bounded above by $|\mathcal{W}|.$
The main claim now follows from the fact that $\chi_{\psi^{\sigma}}=\chi_{\psi}$
as elements of $X_{*}(\hat{T})$, and that $\chi_{\psi_{i}^{\sigma}}=(\chi_{\psi})_{i}$
as elements of $X_{*}(\hat{T}_{i})$. 
\end{proof}
Recall $h\left(\psi;q,E\right):=\sum_{\pi\in\Pi_{\psi}(\epsilon_{\psi})}h\left(\pi;q,E\right)$
and $h'\left(\psi;q,E\right):=\sum_{\pi\in\Pi_{\psi}}h\left(\pi;q,E\right)$,
where $h\left(\pi;q,E\right)=\dim H^{*}\left(\pi_{\infty};E\right)\cdot\dim\pi_{f}^{K_{f}\left(q\right)}$.
Note that $h'\left(\psi;q,E\right)$ (resp. $h\left(\psi;q,E\right)$)
is the cohomological dimension running over all level $q$ weight
$E$ adelic (resp. discrete automorphic) representations in the global
$A$-packet of $\psi$. Also note that $h\left(\psi;q,E\right)\leq h'\left(\psi;q,E\right)$,
and that $h\left(\psi;q,E\right)=h'\left(\psi;q,E\right)$ if $\mathcal{S}_{\psi}$
is trivial. 
\begin{prop}
\label{prop:endoscopy-general} Let $\sigma\in\mathcal{D}\left(G\right)$,
$E\in\Pi^{\mathrm{alg}}\left(G\right)$ and let $\Pi^{\mathrm{alg}}\left(G^{\sigma};E\right)\subset\Pi^{\mathrm{alg}}\left(G^{\sigma}\right)$
be as in Proposition~\ref{prop:endoscopy-cohomology}. For any $\psi\in\Psi_{2}^{\mathrm{AJ}}\left(G;E\right)$
with Arthur $SL_{2}$-type $\psi|_{SL_{2}^{A}}=\sigma$, then
\[
h'\left(\psi;q,E\right)\asymp|q|^{\delta\left(\sigma\right)}\sum_{E'\in\Pi^{\mathrm{alg}}\left(G^{\sigma};E\right)}h'\left(\psi^{\sigma};q,E'\right).
\]
In particular, if $\mathcal{S}_{\psi^{\sigma}}$ is trivial then 
\[
h\left(\psi;q,E\right)\ll|q|^{\delta\left(\sigma\right)}\sum_{E'\in\Pi^{\mathrm{alg}}\left(G^{\sigma};E\right)}h\left(\psi^{\sigma};q,E'\right).
\]
\end{prop}

\begin{proof}
Denote by $f\left(q,E\right)^{\sigma}=f_{\infty}\left(E\right)^{\sigma}\cdot f_{f}\left(q\right)^{\sigma}=\left(\prod_{v\mid\infty}f_{v}\left(E_{v}\right)^{\sigma}\right)\cdot\left(\prod_{v\nmid\infty}f_{v}\left(q\right)^{\sigma}\right)$,
the endoscopic transfer from $\mathcal{H}\left(G\right)$ to $\mathcal{H}\left(G^{\sigma}\right)$
of $f\left(q,E\right)=\prod_{v\mid\infty}f_{v}\left(E_{v}\right)\prod_{v\nmid\infty}f_{v}\left(q\right)$.
By Lemma~\ref{lem:f(q,E)} and Corollary~\ref{cor:endoscopy-global},
we get 
\begin{multline*}
h'\left(\psi;q,E\right)=\sum_{\pi\in\Pi_{\psi}}\mathrm{\mathrm{tr}}\,\pi\left(f\left(q,E\right)\right)=\sum_{\pi'\in\Pi_{\psi^{\sigma}}}\mathrm{tr}\,\pi'\left(f\left(q,E\right)^{\sigma}\right)\\
=\left(\sum_{\pi'_{\infty}\in\Pi_{\psi_{\infty}^{\sigma}}}\mathrm{tr}\,\pi'_{\infty}\left(f_{\infty}\left(E\right)^{\sigma}\right)\right)\cdot\left(\sum_{\pi'_{f}\in\Pi_{\psi_{f}^{\sigma}}}\mathrm{tr}\,\pi'_{f}\left(f_{f}\left(q\right)^{\sigma}\right)\right).
\end{multline*}
By Proposition~\ref{prop:endoscopy-cohomology}, we get that $\psi^{\sigma}$
is cohomological and that if $\pi'\in\Pi_{\psi^{\sigma}}$ is such
that $\dim H^{*}\left(\pi'_{\infty};E'\right)\ne0$, then $E'\in\Pi^{\mathrm{alg}}\left(G^{\sigma};E\right)$.
By Proposition~\ref{prop:NP-uniform} we get that $\sum_{E'\in\Pi^{\mathrm{alg}}\left(G^{\sigma};E\right)}\sum_{\pi'_{\infty}\in\Pi_{\psi_{\infty}^{\sigma}}}\dim H^{*}\left(\pi'_{\infty};E'\right)$
and $\sum_{\pi'_{\infty}\in\Pi_{\psi_{\infty}^{\sigma}}}\mathrm{tr}\,\pi'_{\infty}\left(f_{\infty}\left(E\right)^{\sigma}\right)$
are non-zero constants (i.e. independent of $q$), which gives 
\[
\sum_{\pi'_{\infty}\in\Pi_{\psi_{\infty}^{\sigma}}}\mathrm{tr}\,\pi'_{\infty}\left(f_{\infty}\left(E\right)^{\sigma}\right)\asymp1\asymp\sum_{E'\in\Pi^{\mathrm{alg}}\left(G^{\sigma};E\right)}\sum_{\pi'_{\infty}\in\Pi_{\psi_{\infty}^{\sigma}}}\dim H^{*}\left(\pi'_{\infty};E'\right).
\]
By Corollary~\ref{cor:Ferrari-global}, 
\[
\sum_{\pi'_{f}\in\Pi_{\psi_{f}^{\sigma}}}\mathrm{tr}\,\pi'_{f}\left(f_{f}\left(q\right)^{\sigma}\right)\asymp|q|^{\delta\left(\sigma\right)}\sum_{\pi'_{f}\in\Pi_{\psi_{f}^{\sigma}}}\dim\left(\pi'_{f}\right)^{K^{\sigma}\left(q\right)}.
\]
Combining all of the above we get 
\[
h'\left(\psi;q,E\right)\asymp|q|^{\delta\left(\sigma\right)}\sum_{E'\in\Pi^{\mathrm{alg}}\left(G^{\sigma};E\right)}\sum_{\pi'\in\Pi_{\psi^{\sigma}}}h\left(\pi';q,E'\right)=|q|^{\delta\left(\sigma\right)}\sum_{E'\in\Pi^{\mathrm{alg}}\left(G^{\sigma};E\right)}h'\left(\psi^{\sigma};q,E'\right),
\]
which proves the first claim. The second claim follows from the first.
\end{proof}
\begin{prop}
\label{prop:endoscopy-prod} Let $\sigma\in\mathcal{D}\left(G\right)$,
$E\in\Pi^{\mathrm{alg}}\left(G\right)$ and let $\Pi^{\mathrm{alg}}\left(G^{\sigma};E\right)\subset\Pi^{\mathrm{alg}}\left(G^{\sigma}\right)$
and $\Pi^{\mathrm{alg}}\left(G_{i}^{\sigma};E\right)\subset\Pi^{\mathrm{alg}}\left(G_{i}^{\sigma}\right)$,
for $i=1,2$, be as in Proposition~\ref{prop:endoscopy-cohomology}.
For any $\psi\in\Psi_{2}^{\mathrm{AJ}}\left(G;E\right)$ with Arthur
$SL_{2}$-type $\psi|_{SL_{2}^{A}}=\sigma$, there exists $E'\in\Pi^{\mathrm{alg}}\left(G^{\sigma};E\right)$
and $E'_{i}\in\Pi^{\mathrm{alg}}\left(G_{i}^{\sigma};E\right)$, for
$i=1,2$, such that $\psi^{\sigma}\in\Psi_{2}^{\mathrm{AJ}}\left(G_{i}^{\sigma};q,E'\right)$
and $\psi_{i}^{\sigma}\in\Psi_{2}^{\mathrm{AJ}}\left(G_{i}^{\sigma};q,E'_{i}\right)$,
for $i=1,2$, and
\[
h'\left(\psi^{\sigma};q,E'\right)=h'\left(\psi_{1}^{\sigma};q,E'_{1}\right)\cdot h'\left(\psi_{2}^{\sigma};q,E'_{2}\right).
\]
\end{prop}

\begin{proof}
For a finite place $v$, if $\psi_{v}^{\sigma}|_{I_{v}\times SL_{2}^{D}}\equiv1$
then $\psi_{i,v}^{\sigma}|_{I_{v}\times SL_{2}^{D}}\equiv1$ for $i=1,2$;
in general, the depths satisfy $d\left(\psi_{v}^{\sigma}\right)=\min\left\{ d\left(\psi_{1,v}^{\sigma}\right),d\left(\psi_{2,v}^{\sigma}\right)\right\} $.
By Proposition~\ref{prop:endoscopy-cohomology}, we get the existence
of $E'$ and $E'_{i}$, such that $\psi^{\sigma}\in\Psi_{2}^{\mathrm{AJ}}\left(G_{i}^{\sigma};q,E'\right)$
and $\psi_{i}^{\sigma}\in\Psi_{2}^{\mathrm{AJ}}\left(G_{i}^{\sigma};q,E'_{i}\right)$,
for $i=1,2$. We are left to prove the identity. Note that any $\pi'\in\Pi\left(G^{\sigma}\left(\mathbb{A}\right)\right)$
is of the form $\pi'=\pi'_{1}\otimes\pi'_{2}$, where $\pi'_{i}\in\Pi\left(G_{i}^{\sigma}\left(\mathbb{A}\right)\right)$
for $i=1,2$, and we get 
\begin{multline*}
h\left(\pi';q,E^{\sigma}\right)=\dim H^{*}\left(\pi'_{1}\otimes\pi'_{2};E'\right)\dim\left(\pi'_{1}\otimes\pi'_{2}\right)_{f}^{K_{f}^{G^{\sigma}}\left(q\right)}=\\
=\dim H^{*}\left(\pi'_{1};E'_{1}\right)\dim(\pi'_{1})_{f}^{K_{f}^{G_{1}^{\sigma}}\left(q\right)}\dim H^{*}\left(\pi'_{2};E'_{2}\right)\dim(\pi'_{2})_{f}^{K_{f}^{G_{2}^{\sigma}}\left(q\right)}=h\left(\pi'_{1};q,E'_{1}\right)\cdot h\left(\pi'_{2};q,E'_{2}\right).
\end{multline*}
Therefore, since $\Pi_{\psi^{\sigma}}=\Pi_{\psi_{1}^{\sigma}}\times\Pi_{\psi_{2}^{\sigma}}$,
we get
\begin{multline*}
h'\left(\psi^{\sigma};q,E'\right)=\sum_{\pi'\in\Pi_{\psi^{\sigma}}}h\left(\pi';q,E'\right)=\sum_{\pi'_{1}\in\Pi_{\psi_{1}^{\sigma}}}\sum_{\pi'_{2}\in\Pi_{\psi_{2}^{\sigma}}}h\left(\pi'_{1}\otimes\pi'_{2};q,E'\right)\\
=\left(\sum_{\pi'_{1}\in\Pi_{\psi_{1}^{\sigma}}}h\left(\pi'_{1};q,E'_{1}\right)\right)\left(\sum_{\pi'_{2}\in\Pi_{\psi_{2}^{\sigma}}}h\left(\pi'_{2};q,E'_{2}\right)\right)=h'\left(\psi_{1}^{\sigma};q,E'_{1}\right)\cdot h'\left(\psi_{2}^{\sigma};q,E'_{2}\right).
\end{multline*}
\end{proof}

\subsection{Generalized Saito--Kurokawa parameters\protect\label{subsec:Generalized-Saito-Kurokawa-param}}

So far we made no assumption on the Arthur $SL_{2}$-type. The following
special class of Arthur $SL_{2}$-type was first studied in \cite{GG21}
and \cite{DGG22}.
\begin{defn}
Say that $\sigma\in\mathcal{D}\left(G\right)$ is (even) Generalized
Saito--Kurokawa, or GSK for short, if it is of the form $\sigma=\nu(2k)\oplus\nu(1)^{N-2k}$,
for $1\leq k<\frac{N}{2}$. Note that $s^{\sigma}=\mbox{diag}\left(-1^{(2k)},1^{(N-2k)}\right)$
and that $G^{\sigma}=G_{1}^{\sigma}\times G_{2}^{\sigma}$, where
$G_{1}^{\sigma}$ and $G_{2}^{\sigma}$ are classical groups whose
dual group have standard representations of dimension $2k$ and $N-2k$,
respectively. Finally, define the GSK-deficiency of $\sigma$ by
\[
\delta_{2}\left(\sigma\right):=\delta\left(\sigma\right)+\dim G_{2}^{\sigma}.
\]
Say that an $A$-shape $\varsigma\in\mathcal{M}\left(G\right)$ is
GSK, if its Arthur $SL_{2}$-type $\sigma_{\varsigma}$ is.
\end{defn}

For example, if $G=SO_{2n+1}$ then the GSK Arthur $SL_{2}$-types
are $\nu(2k)\oplus\nu(1)^{2(n-k)}\in\mathcal{D}\left(G\right)$, where
$1\leq k<n$, and note that $G^{\sigma}=SO_{2k+1}\times SO_{2(n-k)+1}$
and 
\[
\delta_{2}\left(\sigma\right)=(2n+1)(n-k).
\]

\begin{prop}
\label{prop:GSK-prod} Let $\sigma\in\mathcal{D}\left(G\right)$ be
GSK. Then for any $\psi\in\Psi_{2}^{\mathrm{AJ}}\left(G;q,E\right)$
of Arthur $SL_{2}$-type $\sigma$, we have
\[
\psi_{1}^{\sigma}\in\Psi_{2}^{\mathrm{AJ}}\left(G_{1}^{\sigma},(\mathbf{F});q,E_{1}^{\sigma}\right)\qquad\mbox{and}\qquad\psi_{2}^{\sigma}\in\bigsqcup_{\varsigma_{2}\in\mathcal{M}^{g}\left(G_{2}^{\sigma}\right)}\Psi_{2}^{\mathrm{AJ}}\left(G_{2}^{\sigma},\varsigma_{2};q,E_{2}^{\sigma}\right).
\]
\end{prop}

\begin{proof}
This follows from Proposition~\ref{prop:endoscopy-prod} and the
fact that the only $A$-shape whose Arthur $SL_{2}$-type is principal
is $(\mathbf{F})$, hence $\psi_{1}^{\sigma}|_{SL_{2}^{A}}\equiv\nu(2k)$
implies $\varsigma\left(\psi_{1}^{\sigma}\right)=(\mathbf{F})\in\mathcal{M}^{g}\left(G_{1}^{\sigma}\right)$,
and the fact that the only $A$-shapes whose Arthur $SL_{2}$-type
is trivial are the generic $A$-shapes, hence $\psi_{2}^{\sigma}|_{SL_{2}^{A}}\equiv\nu(1)^{N-2k}$
implies $\varsigma_{2}=\varsigma\left(\psi_{2}^{\sigma}\right)\in\mathcal{M}^{g}\left(G_{2}^{\sigma}\right)$. 
\end{proof}
For any $\varsigma\in\mathcal{M}\left(G\right)$, define the following
quantity
\[
h'\left(G,\varsigma;q,E\right):=\sum_{\psi\in\Psi_{2}^{\mathrm{AJ}}\left(G,\varsigma;q,E\right)}h'\left(\psi;q,E\right)=\sum_{\psi\in\Psi_{2}^{\mathrm{AJ}}\left(G,\varsigma;q,E\right)}\sum_{\pi\in\Pi_{\psi}}h\left(\pi;q,E\right).
\]
Note that $h\left(G,\varsigma;q,E\right)\leq h'\left(G,\varsigma;q,E\right)$
and that $h\left(G,\varsigma;q,E\right)=h'\left(G,\varsigma;q,E\right)$
if $\mathcal{S}_{\psi}$ is trivial for any $\psi\in\Psi_{2}^{\mathrm{AJ}}\left(G,\varsigma\right)$. 
\begin{prop}
\label{lem:h'-F} We have 
\[
h'\left(G,(\mathbf{F});q,E\right)=h\left(G,(\mathbf{F});q,E\right).
\]
In particular
\[
h'\left(G,(\mathbf{F});q,E\right)\ll1=|q|^{0}.
\]
\end{prop}

\begin{proof}
Note that its Arthur $SL_{2}$-type of $\varsigma=(\mathbf{F})$ is
$\sigma_{\mathrm{princ}}$, the principal $SL_{2}$ of $\hat{G}$.
Hence $C_{\hat{G}}\left(\psi\right)\subset C_{\hat{G}}\left(\sigma_{\mathrm{princ}}\right)=\hat{Z}$,
the center of $\hat{G}$, and therefore $\mathcal{S}_{\psi}=\left\{ 1\right\} $
for any $\psi\in\Psi_{2}^{\mathrm{AJ}}\left(G,\varsigma\right)$,
which implies $h\left(G,\varsigma;q,E\right)=h'\left(G,\varsigma;q,E\right)$.
The second claim now follows from Corollary~\ref{cor:h-F}.
\end{proof}
\begin{prop}
\label{prop:h'-g} We have
\[
\sum_{\varsigma\in\mathcal{M}^{g}\left(G\right)}h'\left(G,\varsigma;q,E\right)\ll|q|^{\dim G}.
\]
\end{prop}

\begin{proof}
By Proposition~\ref{prop:h'-generic}, combined with Lemmas~\ref{lem:h-triv-bound}
and~\ref{lem:coh-vol-dim}, for any $\epsilon>0$ there exists $c>0$,
such that 
\[
\sum_{\varsigma\in\mathcal{M}^{g}\left(G\right)}h'\left(G,\varsigma;q,E\right)\leq\sum_{\varsigma\in\mathcal{M}^{g}\left(G\right)}c|q|^{\epsilon}\cdot h\left(G,\varsigma;c|q|^{\epsilon}\cdot q,E\right)\ll|q|^{\dim G}.
\]
\end{proof}
\begin{thm}
\label{thm:h-GSK} Let $\varsigma\in\mathcal{M}\left(G\right)$ be
a GSK $A$-shape. Then
\[
h\left(G,\varsigma;q,E\right)\ll|q|^{\delta_{2}\left(\sigma\right)}.
\]
\end{thm}

\begin{proof}
Let $\sigma=\sigma_{\varsigma}\in\mathcal{D}\left(G\right)$ be the
Arthur $SL_{2}$-type of $\varsigma$. By Proposition~\ref{prop:endoscopy-general},
\[
h\left(G,\varsigma;q,E\right)=\sum_{\psi\in\Psi_{2}^{\mathrm{AJ}}\left(G,\varsigma;q,E\right)}h\left(\psi;q,E\right)\ll|q|^{\delta\left(G\right)}\sum_{\psi\in\Psi_{2}^{\mathrm{AJ}}\left(G,\varsigma;q,E\right)}\sum_{E^{\sigma}\in\Pi^{\mathrm{alg}}\left(G^{\sigma};E\right)}h'\left(\psi^{\sigma};q,E^{\sigma}\right),
\]
 combined with Propositions~\ref{prop:GSK-prod} and \ref{prop:endoscopy-prod}
we get 
\[
\leq|q|^{\delta\left(G\right)}\left(\sum_{E_{1}^{\sigma}\in\Pi^{\mathrm{alg}}\left(G_{1}^{\sigma};E\right)}h'\left(G_{1}^{\sigma},(\mathbf{F});q,E_{1}^{\sigma}\right)\right)\left(\sum_{\varsigma_{2}\in\mathcal{M}^{g}\left(G_{2}^{\sigma}\right)}\sum_{E_{2}^{\sigma}\in\Pi^{\mathrm{alg}}\left(G_{2}^{\sigma};E\right)}h'\left(G_{2}^{\sigma},\varsigma_{2};q,E_{2}^{\sigma}\right)\right),
\]
and by Lemma~\ref{lem:h'-F} and Proposition~\ref{prop:h'-g} we
get 
\[
\ll|q|^{\delta\left(G\right)}\cdot|q|^{0}\cdot|q|^{\dim G_{2}^{\sigma}}=|q|^{\delta_{2}\left(\sigma\right)}.
\]
\end{proof}

\subsection{Bounds for $(\mathbf{P})$ and proof of the main theorem\protect\label{subsec:h-P} }

Let $G$ be a Gross inner form of $SO_{5}$ defined over a totally
real number field $F$. Our goal is to give bounds on the contribution
of $h\left(\psi;q,E\right)$, for each $\psi\in\Psi_{2}^{\mathrm{AJ}}\left(G,(\mathbf{P})\right)$,
to $h\left(G,(\mathbf{P});q,E\right)$ (see Corollary~\ref{cor:h-P}),
using the endoscopic results described in the previous two subsections. 

Let $\sigma^{(\mathbf{P})}=\nu(2)\oplus\nu(1)^{2}\in\mathcal{D}\left(G\right)$
which is the Arthur $SL_{2}$-type of $(\mathbf{P})$. This is the
only GSK Arthur $SL_{2}$-type of $SO_{5}$. The endoscopic group
appearing in the associated datum $\left(G^{(\mathbf{P})},s^{(\mathbf{P})},\xi^{(\mathbf{P})}\right)$
is $G^{(\mathbf{P})}\simeq SO_{3}\times SO_{3}$. We compute 
\begin{equation}
\delta(\sigma^{(\mathbf{P})})=\frac{1}{2}\left(\dim G-\dim G^{(\mathbf{P})}\right)=2,\qquad\delta_{2}(\sigma^{(\mathbf{P})})=\delta(\sigma^{(\mathbf{P})})+\dim PGL_{2}=5.\label{eq:bound-shape-P}
\end{equation}

This allows us to deduce our sought-after bound for the $A$-shape
$(\mathbf{P})$ as a corollary of Theorem~\ref{thm:h-GSK} for the
special case of a Gross inner form of $SO_{5}$.
\begin{cor}
\label{cor:h-P} For $\varsigma=(\mathbf{P})$, we have $h\left(G,\varsigma;q,E\right)\ll\vol\left(X\left(q\right)\right)^{\frac{1}{2}}$.
\end{cor}

\begin{proof}
This follows from Theorem~\ref{thm:h-GSK} and \eqref{eq:bound-shape-P}. 
\end{proof}
\begin{cor}
\label{cor:CSXDH-P} Conjecture~\ref{conj:CSXDH-shape} (CSXDH-shape)
holds for Gross inner forms of $SO_{5}$ and $\varsigma=(\mathbf{P})$.
\end{cor}

\begin{proof}
For $\varsigma=(\mathbf{P})$ by Corollary~\ref{cor:r-weak-BQP}
and Corollary~\ref{cor:h-P},
\[
h\left(G,\varsigma;q,E\right)\ll\vol\left(X\left(q\right)\right)^{\frac{1}{2}}\leq\vol\left(X\left(q\right)\right)^{\frac{2}{r\left(\varsigma\right)}}.
\]
Note that by Corollary~\ref{cor:r-exact-P}, we actually have $r\left(\mathbf{P}\right)=3$.
\end{proof}
We are now in a position to prove our main result: Theorem~\ref{thm:CSXDH-SO5-shape}. 
\begin{proof}[\textbf{Proof of Theorem~\ref{thm:CSXDH-SO5-shape}}]
 By Corollaries~\ref{cor:CSXDH-GYF}, \ref{cor:CSXDH-BQ} and \ref{cor:CSXDH-P},
we get that Conjecture~\ref{conj:CSXDH-shape} holds for Gross inner
forms of $SO_{5}$ for any of the six possible $A$-shapes $\varsigma\in\left\{ (\mathbf{G}),(\mathbf{Y}),(\mathbf{F}),(\mathbf{B}),(\mathbf{Q}),(\mathbf{P})\right\} $.
\end{proof}
Finally we prove Theorems~\ref{thm:CSXDH-SO5-Betti} and \ref{thm:CSXDH-SO5-Betti-Gross},
which we combine into the following single uniform statement which
generalizes both theorems. Recall that $h^{d}\left(G;q,E\right)$
is the $d$-th Betti number ($L^{2}$-Betti when $G=SO_{5}$ and usual
for uniform Gross inner forms) of the congruence manifolds $X\left(q\right)=\Gamma\left(q\right)\backslash G_{\infty}/K_{\infty}$. 
\begin{cor}
Let $G$ be either the split group $SO_{5}$ defined over $\mathbb{Q}$
or a uniform Gross inner form of $SO_{5}$ defined over a totally
real number field. Then 
\[
h^{d}\left(G;q,E\right)\ll\vol\left(X\left(q\right)\right)^{\frac{1}{2}},
\]
where $d=\frac{ab}{2}-1$ for $G_{\infty}=SO\left(a,b\right)\times\mbox{compact}$
(namely, $d=1$ for $SO\left(1,4\right)$ and $d=2$ for $SO\left(3,2\right)$). 
\end{cor}

\begin{proof}
By Theorem~\ref{thm:Arthur-Matsushima},
\[
h^{d}\left(G;q,E\right)=\dim H^{d}\left(X\left(q\right);E\right)\cong\sum_{\varsigma\in\mathcal{M}\left(G\right)}\sum_{\psi\in\Psi_{2}^{\mathrm{AJ}}\left(G,\varsigma;q,E\right)}\sum_{\pi\in\Pi_{\psi}(\epsilon_{\psi})}\dim H^{d}\left(\pi;q,E\right).
\]
By Proposition~\ref{prop:coh-deg-par}, $\dim H^{d}\left(\pi;q,E\right)=0$
if $\varsigma\in\left\{ (\mathbf{G}),(\mathbf{Y})\right\} $, hence
\[
h^{2}\left(G;q,E\right)\leq\sum_{\varsigma\not\in\left\{ (\mathbf{G}),(\mathbf{Y})\right\} }h\left(G,\varsigma;q,E\right).
\]
The claim now follows from Corollaries~\ref{cor:h-GYF}, \ref{cor:CSXDH-BQ}
and \ref{cor:CSXDH-P}.
\end{proof}

\section{Definite Gross inner forms \protect\label{sec:Gross-forms}}

In this section we consider the case of definite Gross inner forms
and describe several applications of the CSXDH, including density-Ramanujan
complexes, the cutoff phenomenon and the optimal strong approximation
property. Let $G$ be a definite Gross inner form of a split classical
group $G^{*}$ over a totally real number field $F$. Namely $G_{v}:=G\left(F_{v}\right)\cong G^{*}\left(F_{v}\right)$
for any finite place $v\nmid\infty$ of $F$, and $G_{\infty}:=\prod_{v\mid\infty}G\left(F_{v}\right)$
is a compact Lie group. 

\subsection{Ramanujan and density-Ramanujan complexes}

Fix a finite place $\mathfrak{p}$, let $\tilde{X}_{\mathfrak{p}}$
be the Bruhat-Tits building associated to $G_{\mathfrak{p}}$ (see
\cite{Tit79}). Denote by $I_{\mathfrak{p}}=\left\{ g\in G\left(\mathcal{O}_{\mathfrak{p}}\right)\,:\,g\mod{\mathfrak{p}}\in B\left(\mathbb{F}_{\mathfrak{p}}\right)\right\} $,
where $\mathcal{O_{\mathfrak{p}}}$ is the ring of integers of $F_{\mathfrak{p}}$,
$\mathbb{F}_{\mathfrak{p}}=\mathcal{O}_{\mathfrak{p}}/\left(\varpi_{\mathfrak{p}}\right)$
the residue field, $\bmod\,\mathfrak{p}\,:\,G\left(\mathcal{O}_{\mathfrak{p}}\right)\rightarrow G\left(\mathbb{F}_{\mathfrak{p}}\right)$
the reduction modulo $\varpi_{\mathfrak{p}}$ homomorphism, and $B\left(\mathbb{F}_{\mathfrak{p}}\right)\leq G\left(\mathbb{F}_{\mathfrak{p}}\right)$
the standard Borel subgroup, the standard Iwahori subgroup of $G_{\mathfrak{p}}$.
Let $\tilde{Y}_{\mathfrak{p}}$ be the set of maximal faces (a.k.a.
chambers) of $\tilde{X}_{\mathfrak{p}}$ in the case where $G_{\mathfrak{p}}$
acts in a type-preserving manner on $\tilde{X}_{\mathfrak{p}}$ (e.g.
$G_{\mathfrak{p}}\cong Sp_{2n}\left(F_{\mathfrak{p}}\right)$), and
otherwise let $\tilde{Y}_{\mathfrak{p}}$ be the set of ``directed''
maximal faces, such that in both cases $\tilde{Y}_{\mathfrak{p}}$
is isomorphic as a a $G_{\mathfrak{p}}$-set with $G_{\mathfrak{p}}/I_{\mathfrak{p}}$.

For any $q\lhd\mathcal{O}$ such that $\mathfrak{p}\nmid q$, denote
$K\left(q\right)=I_{\mathfrak{p}}K^{\mathfrak{p}}\left(q\right)\leq G\left(\mathbb{A}\right)$,
where $K^{\mathfrak{p}}\left(q\right)=G_{\infty}\prod_{\ell\nmid\mathfrak{p}\infty}K_{\ell}\left(q\right)$,
and denote the level $q$ principal congruence $\mathfrak{p}$-arithmetic
subgroup $\Gamma_{\mathfrak{p}}\left(q\right)=G\left(F\right)\cap K^{\mathfrak{p}}\left(q\right)\leq G_{\mathfrak{p}}$.
If $G$ satisfies the strong approximation property (i.e. if $G$
is a symplectic group), $G\left(\mathbb{A}\right)=G\left(F\right)G_{\mathfrak{p}}K^{\mathfrak{p}}\left(q\right)$,
then we define the (connected) level $q$ congruence complex and its
set of (directed) maximal faces by
\[
X_{\mathfrak{p}}\left(q\right):=\Gamma_{\mathfrak{p}}\left(q\right)\backslash\tilde{X}_{\mathfrak{p}},\qquad Y_{\mathfrak{p}}\left(q\right)\cong\Gamma_{\mathfrak{p}}\left(q\right)\backslash G_{\mathfrak{p}}/I_{\mathfrak{p}}\cong G\left(F\right)\backslash G\left(\mathbb{A}\right)/K\left(q\right).
\]
Otherwise (i.e. if $G$ is a special orthogonal group) there exists
a finite set $\left\{ g_{1}=1,\ldots,g_{h}\right\} \subset G\left(\mathbb{A}\right)$,
of representatives for the double coset space $G\left(F\right)\backslash G\left(\mathbb{A}\right)/G_{\mathfrak{p}}K^{\mathfrak{p}}\left(q\right)$.
Denote $\Gamma_{\mathfrak{p}}^{i}\left(q\right)=G\left(F\right)\cap g_{i}K^{\mathfrak{p}}\left(q\right)g_{i}^{-1}\leq G_{\mathfrak{p}}$
for any $i=1,\ldots,h$, and denote the (possibly disconnected) level
$q$ congruence complex and its set of (directed) maximal faces by
\[
X_{\mathfrak{p}}\left(q\right):=\bigsqcup_{i=1}^{h}\Gamma_{\mathfrak{p}}^{i}\left(q\right)\backslash\tilde{X}_{\mathfrak{p}},\qquad Y_{\mathfrak{p}}\left(q\right)\cong\bigsqcup_{i=1}^{h}\Gamma_{\mathfrak{p}}^{i}\left(q\right)\backslash G_{\mathfrak{p}}/I_{\mathfrak{p}}\cong G\left(F\right)\backslash G\left(\mathbb{A}\right)/K\left(q\right).
\]
Note that each $X_{\mathfrak{p}}^{i}\left(q\right)=\Gamma_{\mathfrak{p}}^{i}\left(q\right)\backslash\tilde{X}_{\mathfrak{p}}$
is a connected component of $X_{\mathfrak{p}}\left(q\right)$, and
the number of connected component of $X_{\mathfrak{p}}\left(q\right)$
is $h$. For example, since $g_{1}=1$, $\Gamma_{\mathfrak{p}}^{1}\left(q\right)=\Gamma_{\mathfrak{p}}\left(q\right)$,
and $X_{\mathfrak{p}}^{1}\left(q\right)=\Gamma_{\mathfrak{p}}\left(q\right)\backslash\tilde{X}_{\mathfrak{p}}$
is a connected component of $X_{\mathfrak{p}}\left(q\right)$. Recall
that by Lemma~\ref{lem:str-app-SO}, we have $h\asymp1$ and $\left|X_{\mathfrak{p}}^{i}\left(q\right)\right|\asymp\left|X_{\mathfrak{p}}\left(q\right)\right|$
for any $i$.

Let $\bigoplus_{i=1}^{h}L^{2}\left(\Gamma_{\mathfrak{p}}^{i}\left(q\right)\backslash G_{\mathfrak{p}}\right)$
be the right regular $G_{\mathfrak{p}}$-representation associated
to $\left\{ \Gamma_{\mathfrak{p}}^{i}\left(q\right)\right\} _{i=1}^{h}$.
Since for each $i=1,\ldots,h$, $\Gamma_{\mathfrak{p}}^{i}\left(q\right)$
is a cocompact lattice of $G_{\mathfrak{p}}$ (see \cite[Theorem 5.5]{PR93}),
this representation decomposes into a Hilbert direct sum of irreducible
representations
\[
\bigoplus_{i=1}^{h}L^{2}\left(\Gamma_{\mathfrak{p}}^{i}\left(q\right)\backslash G_{\mathfrak{p}}\right)=\bigoplus_{\pi_{\mathfrak{p}}\in\Pi\left(G_{\mathfrak{p}}\right)}m\left(\pi_{\mathfrak{p}};q\right)\cdot\pi_{\mathfrak{p}}.
\]
By considering only the $I_{\mathfrak{p}}$-fixed part of the above
decomposition we get
\begin{equation}
V\left(G;q\right):=L^{2}\left(Y_{\mathfrak{p}}\left(q\right)\right)\cong\left(\bigoplus_{i=1}^{h}L^{2}\left(\Gamma_{\mathfrak{p}}^{i}\left(q\right)\backslash G_{\mathfrak{p}}\right)\right)^{I_{\mathfrak{p}}}\cong\bigoplus_{\pi_{\mathfrak{p}}\in\Pi\left(G_{\mathfrak{p}}\right)}m\left(\pi_{\mathfrak{p}};q\right)\cdot\pi_{\mathfrak{p}}^{I_{\mathfrak{p}}}.\label{eq:=000020V(G;q)}
\end{equation}
Let us decompose this space according to the rate of decay of the
matrix coefficients
\[
V\left(G;q\right)=\bigoplus_{r\in[2,\infty]}V\left(G,r;q\right),\qquad V\left(G,r;q\right):=\bigoplus_{r(\pi_{\mathfrak{p}})=r}m\left(\pi_{\mathfrak{p}};q\right)\pi_{\mathfrak{p}}^{I_{\mathfrak{p}}}.
\]
Note that $V\left(G,2;q\right)$ is comprised of the tempered representations
of $G_{\mathfrak{p}}$, while $V\left(G,\infty;q\right)$ is comprised
of the one-dimensional representations of $G_{\mathfrak{p}}$ by Proposition~\ref{prop:aut-dim-1},
hence
\[
\dim V\left(G,2;q\right)\leq\dim V\left(G;q\right)\asymp|X_{\mathfrak{p}}\left(q\right)|,\qquad\dim V\left(G,\infty;q\right)\asymp1.
\]
Using these notations, the definitions of Ramanujan and density-Ramanujan
complexes (Definitions~\ref{def:Ram-comp} and \ref{def:den-Ram-comp})
from the introduction are equivalent to the following ones:
\begin{defn}
1. Call $X_{\mathfrak{p}}\left(q\right)$ a Ramanujan complex if for
any $2<r<\infty$,
\[
\dim V\left(G,r;q\right)=0.
\]
2. Call $\left\{ X_{\mathfrak{p}}\left(q\right)\right\} $ a family
of density-Ramanujan complexes if for any $2<r<\infty$,
\[
\sum_{t\geq r}\dim V\left(G,t;q\right)=\sum_{r(\pi_{\mathfrak{p}})\geq r}m\left(\pi_{\mathfrak{p}};q\right)\dim\pi_{\mathfrak{p}}^{I_{\mathfrak{p}}}\ll|X_{\mathfrak{p}}\left(q\right)|^{\frac{2}{r}}.
\]
\end{defn}

\begin{rem}
\label{rem:disc-comp} A family of Ramanujan complexes is clearly
also a family of density-Ramanujan complexes. In the above definition
we allow for disconnected Ramanujan complexes. We note that a disconnected
complex is Ramanujan if and only if each connected component of it
is Ramanujan. 
\end{rem}

\begin{lem}
\label{lem:conn-den-ran-comp} Assume $\left\{ X_{\mathfrak{p}}\left(q\right)\right\} $
is a family of density-Ramanujan complexes and let $1\leq i\leq h$.
Then $\left\{ X_{\mathfrak{p}}^{i}\left(q\right)\right\} $ is a family
of connected density-Ramanujan complexes. 
\end{lem}

\begin{proof}
Denote $V^{i}\left(G;q\right)=L^{2}\left(Y_{\mathfrak{p}}^{i}\left(q\right)\right)\cong L^{2}\left(\Gamma_{\mathfrak{p}}^{i}\left(q\right)\backslash G_{\mathfrak{p}}\right)^{I_{\mathfrak{p}}}\cong\bigoplus_{\pi_{\mathfrak{p}}\in\Pi\left(G_{\mathfrak{p}}\right)}m^{i}\left(\pi_{\mathfrak{p}};q\right)\cdot\pi_{\mathfrak{p}}^{I_{\mathfrak{p}}}$
, hence 
\[
V^{i}\left(G,r;q\right)=V\left(G,r;q\right)\cap V^{i}\left(G;q\right)=\bigoplus_{r(\pi_{\mathfrak{p}})=r}m^{i}\left(\pi_{\mathfrak{p}};q\right)\pi_{\mathfrak{p}}^{I_{\mathfrak{p}}}.
\]
Then $V\left(G;q\right)=\oplus_{i=1}^{h}V^{i}\left(G;q\right)$ and
$V\left(G,r;q\right)=\oplus_{i=1}^{h}V^{i}\left(G,r;q\right)$ for
any $r$. Hence $\sum_{t\geq r}\dim V^{i}\left(G,t;q\right)\leq\sum_{t\geq r}\dim V\left(G,t;q\right)\ll|X_{\mathfrak{p}}\left(q\right)|^{\frac{2}{r}}$,
and since $|X_{\mathfrak{p}}^{i}\left(q\right)|\asymp|X_{\mathfrak{p}}\left(q\right)|$
by Lemma~\ref{lem:str-app-SO}, the claim follows.
\end{proof}
The following weaker property compared to Ramanujan, called spectral
expansion (see \cite{Lub18}), holds for all congruence complexes
(and in fact for all finite quotients of Bruhat-Tits buildings of
dim $\geq2$).
\begin{prop}
\label{prop:spec-exp} Let $G$ be a definite Gross inner form of
a classical group of rank $n\geq2$. Then for any prime $\mathfrak{p}$,
any level $q$, and any $2n<r<\infty$,
\[
\dim V\left(G,r;q\right)=0.
\]
\end{prop}

\begin{proof}
It follows from Oh's work on the uniform Kazhdan property (T) \cite[Section 7]{Oh02},
that if $\sigma$ is a unitary irreducible representation of the split
classical $p$-adic group of rank $n$ then $r\left(\sigma\right)\leq2n$.
Since for any automorphic representation $\pi$ of $G$, its local
factor $\pi_{\mathfrak{p}}$ is unitary, we get $r(\pi_{\mathfrak{p}})\leq2n$,
which proves the claim.
\end{proof}
As mentioned in the introduction, currently the only known constructions
of Ramanujan complexes comes from inner and outer forms of $G^{*}=PGL_{d}$
over global fields, and no construction of Ramanujan complexes is
known for forms of other classical groups; for some of their expected
properties see the work of \cite{FLW13,KLW18} on inner forms of $PSp_{4}\cong SO_{5}$.
The following proposition shows that for definite Gross inner forms
of $SO_{5}$, most congruence complexes are not Ramanujan.
\begin{prop}
\label{prop:Non-Ram} Let $G$ be definite Gross inner form of $SO_{5}$
defined over a totally real field $F$ of even degree. Then for any
large enough (depending on $F$) square-free level $q$, and any finite
place $\mathfrak{p}$ coprime to $q$, the congruence complex $X_{\mathfrak{p}}\left(q\right)$
is not Ramanujan.
\end{prop}

\begin{proof}
By Theorem~\ref{thm:Arthur-Taibi}, we have the following decomposition
\[
V\left(G;q\right)\cong\left(\bigoplus_{i=1}^{h}L^{2}\left(\Gamma_{\mathfrak{p}}^{i}\left(q\right)\backslash G_{\mathfrak{p}}\right)\right)^{I_{\mathfrak{p}}}\cong L^{2}\left(G\left(F\right)\backslash G\left(\mathbb{A}\right)\right)^{K\left(q\right)}\cong\bigoplus_{\psi\in\Psi_{2}\left(G\right)}\bigoplus_{\pi\in\Pi_{\psi}(\epsilon_{\psi})}\pi^{K\left(q\right)}.
\]
We will exhibit a representation $\pi=\otimes'_{v}\pi_{v}$ of level
$K(q)$ such that $r(\pi_{v})=3$ for all finite $v$. Let $\psi=\mu\boxplus\nu(2)\in\Psi_{2}\left(G\right)$,
a global $A$-parameter of Saito--Kurokawa type, i.e. of $A$-shape
$(\mathbf{P})$, where $\mu$ is an automorphic representation of
$PGL_{2}/F$, associated to a Hilbert modular cusp form of parallel
even weight $4$ and level $\Gamma_{0}\left(q\right)=\left\{ g\in SL_{2}\left(\mathcal{O}\right)\,:\,g_{21}\equiv0\mod q\right\} $.
Such a Hilbert modular cusp form exists for any $q$ large enough
compared to some constant which depends on $F$ by Shimizu's formula
for the dimension of the space of Hilbert modular cusp forms of a
given parallel even weight and congruence level \cite[Theorem 11]{Shi63}.
By \cite{Sch20}, for any $v\mid\infty$, the infinitisimal character
of $\psi_{v}$ is equal to that of the trivial representation, hence
by \cite{Tai18} we get that $\Pi_{\psi_{v}}$ contains only the trivial
representation of $G_{v}$. For any $v\nmid q\cdot\infty$, by construction
$\mu_{v}$ is unramified, hence also the local $A$-parameter $\psi_{v}$
is unramified, and by Proposition~\ref{prop:Moeglin-A-ur} we get
that $\Pi_{\psi_{v}}$ contains a single unramified representation.
For any $v\mid q$, by construction $\mu_{v}$ is fixed by an Iwahori
subgroup, hence the local $A$-parameter $\psi_{v}$ is trivial on
the inertia subgroup, so that $d(\psi_{v})=0$. By Proposition~ \ref{prop:GV-A-SO5},
all members $\pi_{v}$ of $\Pi_{\psi_{v}}$ satisfy $d(\pi_{v})=0$,
and hence $\ell(\pi_{v})\leq1$ by Lemma \ref{lem:depth-level}, i.e.
any $\pi_{v}\in\Pi_{\psi_{v}}$ has a $K_{v}(p_{v})$-fixed vector.
Combining all of the above we produce a $\pi=\otimes'_{v}\pi_{v}\in\Pi_{\psi}$
such that $\pi^{K\left(q\right)}\ne0$ for any $\pi\in\Pi_{\psi}$.
By switching the $v$-component for some $v\mid q$, we may assume
$\pi\in\Pi_{\psi}(\epsilon_{\psi})$. Finally by Corollary~\ref{cor:r-exact-P},
we get that $r(\pi_{v})=3$ for any $v\nmid q$, hence $V\left(G,3;q\right)\ne0$
which completes the proof.
\end{proof}
On the other hand, congruence complexes are conjectured to be density-Ramanujan
complexes. The following conjecture is stated only for the special
case of definite Gross inner forms for simplicity, however it can
be stated more generally for any connected reductive group over a
number field $G/F$, which is definite, i.e. $G_{\infty}:=\prod_{v\mid\infty}G\left(F_{v}\right)$
is compact, and any finite place $\mathfrak{p}$ such that $G_{\mathfrak{p}}$
is non-compact.
\begin{conjecture}
\label{conj:den-Ram-comp} For any definite Gross inner form $G$
and any finite place $\mathfrak{p}$, its congruence complexes $\left\{ X_{\mathfrak{p}}\left(q\right)\right\} $
forms a family of density-Ramanujan complexes.
\end{conjecture}

Note that Theorem~\ref{thm:den-Ram-SO5} is a special case of Conjecture~\ref{conj:den-Ram-comp}
when $G$ is a definite Gross inner form of $G^{*}=SO_{5}$.

\subsection{CSXDH for definite Gross inner forms}

In this subsection we prove Theorems~\ref{thm:CSXDH-SO5-p} and \ref{thm:den-Ram-SO5}
from the introduction, namely we show that if $G$ is a definite Gross
form of $SO_{5}$, then the $p$-adic CSXDH holds for $G$ and that
the congruence complexes$\left\{ X_{\mathfrak{p}}\left(q\right)\right\} _{q}$
form a family of density-Ramanujan complexes.

We first recall the work of \cite{GK23}, which reduces the $p$-adic
CSXDH and the density-Ramanujan property (Conjectures~\ref{conj:CSXDH-p}
and \ref{conj:den-Ram-comp}) to the following special case, which
we call the spherical SXDH (in \cite{GK23} it is called the spherical
density hypothesis).
\begin{defn}
\label{def:sph-SXDH} Let $G$ be a definite Gross inner form and
$\mathfrak{p}$ a finite place. Say that the spherical SXDH holds
(at $\mathfrak{p}$) if for any $2<r<\infty$,
\[
\sum_{r(\pi_{\mathfrak{p}})\geq r}m\left(\pi_{\mathfrak{p}};q\right)\cdot\dim\pi_{\mathfrak{p}}^{K_{\mathfrak{p}}}\ll|X_{\mathfrak{p}}\left(q\right)|^{\frac{2}{r}}.
\]
\end{defn}

The following non-trivial Corollary is a consequence of the work of
\cite{GK23}.
\begin{cor}
\cite[Theorems 1.6 and 1.8]{GK23} \label{cor:sph-SXDH} Let $G$
be a definite Gross inner form and $\mathfrak{p}$ a finite place.
If the spherical SXDH holds (Definition~\ref{def:sph-SXDH}) then
Conjectures~\ref{conj:CSXDH-p} and \ref{conj:den-Ram-comp} hold.
\end{cor}

\begin{proof}
Since the spherical SXDH holds, by combining \cite[Theorems 1.6 and 1.8]{GK23}
we get that given a pre-compact set, $A\subset\Pi\left(G_{\mathfrak{p}}\right)$,
then for any $\epsilon>0$ there exists $C_{A,\epsilon}>0$, such
that 
\[
M\left(A,q,r\right):=\sum_{\pi_{\mathfrak{p}}\in A,\,r(\pi_{\mathfrak{p}})\geq r}m\left(\pi_{\mathfrak{p}};q\right)\leq C_{A,\epsilon}|X_{\mathfrak{p}}\left(q\right)|^{\frac{2}{r}}.
\]
Conjectures~\ref{conj:CSXDH-p} and \ref{conj:den-Ram-comp} follow
from this inequality by taking $A=\left\{ \pi_{\mathfrak{p}}\right\} $
and $A=\{\pi_{\mathfrak{p}}:\pi_{\mathfrak{p}}^{I_{\mathfrak{p}}}\ne0\}$,
respectively, and noting that in the latter case $\dim\pi_{\mathfrak{p}}^{I_{\mathfrak{p}}}\leq|\mathcal{W}|$,
where $\mathcal{W}$ is the finite Weyl group, and in the former case
$C_{A,\epsilon}$ depends on $[K_{\mathfrak{p}}:U]$, where $\pi_{\mathfrak{p}}^{U}\ne0$.
\end{proof}
Next we describe the analogue of Conjecture~\ref{conj:CSXDH-shape}
applied to the case of a definite Gross inner form $G$ and the trivial
coefficient system $E=\mathbb{C}$. Let $\Psi_{2}\left(G\right)$
be the set of discrete global $A$-parameters of $G$, let $\Psi_{2}^{\mathrm{AJ}}\left(G,\varsigma;q\right)$
be the subset of $\Psi_{2}\left(G\right)$ as defined in Subsection~\ref{subsec:=000020Depth-cohom-endos},
for any $\psi\in\Psi_{2}\left(G\right)$ let $\Pi_{\psi}(\epsilon_{\psi})$
be as defined in Theorem~\ref{thm:Arthur-Taibi}, and let $K'\left(q\right)=K_{\mathfrak{p}}K^{\mathfrak{p}}\left(q\right)=G_{\infty}\prod_{\ell\nmid\infty}K_{\ell}\left(q\right)\leq G\left(\mathbb{A}\right)$
and finally denote
\[
h\left(G,\varsigma;q\right):=\sum_{\psi\in\Psi_{2}^{\mathrm{AJ}}\left(G,\varsigma;q\right)}\sum_{\pi\in\Pi_{\psi}(\epsilon_{\psi})}\dim\pi^{K'(q)}.
\]

\begin{conjecture}
\label{conj:DSXDH-shape} (DSXDH-shape) Fix a finite prime $\mathfrak{p}$.
Then for any $\varsigma\in\mathcal{M}\left(G\right)$, 
\begin{equation}
h\left(G,\varsigma;q\right)\ll|X_{\mathfrak{p}}\left(q\right)|^{\frac{2}{r\left(\varsigma\right)}}.\label{eq:DSXDH-shape}
\end{equation}
\end{conjecture}

Let us now show how Conjecture~\ref{conj:DSXDH-shape} implies the
spherical SXDH. The proof is \emph{mutatis mutandis} the proof of
Proposition~\ref{prop:CSXDH-shape->=00005Cinfty}.
\begin{prop}
\label{prop:DSXDH-shape->SXDH-p+den-Ram} Conjecture~\ref{conj:DSXDH-shape}
implies the spherical SXDH (Definition~\ref{def:sph-SXDH}).
\end{prop}

\begin{proof}
Since $\Gamma_{\mathfrak{p}}\left(q\right)=G\left(F\right)\bigcap K^{\mathfrak{p}}\left(q\right)$,
we get that $\Gamma_{\mathfrak{p}}\left(q\right)\backslash G_{\mathfrak{p}}$
embeds as a $G_{\mathfrak{p}}$-set in $G\left(F\right)\backslash G\left(\mathbb{A}\right)/K^{\mathfrak{p}}\left(q\right)$.
Hence for any $\pi_{\mathfrak{p}}\in\Pi\left(G_{\mathfrak{p}}\right)$,

\[
m\left(\pi_{\mathfrak{p}};q\right):=\dim\mbox{Hom}_{G_{\mathfrak{p}}}\left(\pi_{\mathfrak{p}},\bigoplus_{i=1}^{h}L^{2}\left(\Gamma_{\mathfrak{p}}^{i}\left(q\right)\backslash G_{\mathfrak{p}}\right)\right)\leq\dim\mbox{Hom}_{G_{\mathfrak{p}}}\left(\pi_{\mathfrak{p}},L^{2}\left(G\left(F\right)\backslash G\left(\mathbb{A}\right)\right)^{K^{\mathfrak{p}}\left(q\right)}\right)
\]
\[
=\sum_{\pi^{\mathfrak{p}}\in\Pi\left(G\left(\mathbb{A}^{\mathfrak{p}}\right)\right)}m\left(\pi_{\mathfrak{p}}\otimes\pi^{\mathfrak{p}}\right)\dim\left(\pi^{\mathfrak{p}}\right)^{K^{\mathfrak{p}}\left(q\right)}=\sum_{\pi=\pi_{\mathfrak{p}}\otimes\pi^{\mathfrak{p}}}m\left(\pi\right)\dim\left(\pi\right)^{K'\left(q\right)}.
\]
Therefore 
\[
\sum_{\begin{array}{c}
\pi_{\mathfrak{p}}\in\Pi\left(G_{\mathfrak{p}}\right)\\
r\left(\pi_{\mathfrak{p}}\right)\geq r
\end{array}}m\left(\pi_{\mathfrak{p}};q\right)\dim\pi_{\mathfrak{p}}^{K_{\mathfrak{p}}}\leq\sum_{\begin{array}{c}
\pi\in\Pi\left(G\left(\mathbb{A}\right)\right)\\
r\left(\pi_{\mathfrak{p}}\right)\geq r
\end{array}}m\left(\pi\right)\dim\pi^{K'\left(q\right)}=:\hat{m}\left(r;q\right).
\]
By Theorem~\ref{thm:Arthur-Taibi}, if $\pi\in\Pi\left(G\left(\mathbb{A}\right)\right)$
is such that $m\left(\pi\right)\ne0$, then $m\left(\pi\right)=1$
and there exists a unique $\psi\in\Psi_{2}\left(G\right)$ such that
$\pi\in\Pi_{\psi}(\epsilon_{\psi})$. By Proposition~\ref{prop:coh-dep-A-par-pac},
if $\pi\in\Pi_{\psi}$ is such that $\dim\pi^{K'\left(q\right)}\ne0$
then $\psi\in\Psi_{2}\left(G;q\right)$ and $\psi_{\mathfrak{p}}\in\Psi^{\mathrm{ur}}\left(G_{\mathfrak{p}}\right)$.
Denote by $\mathcal{M}\left(G;r\right)$ the set of $\varsigma\left(\psi\right)\in\mathcal{M}\left(G\right)$
for $\,\psi\in\Psi_{2}\left(G\right)$ with $\psi_{\mathfrak{p}}\in\Psi^{\mathrm{ur}}\left(G_{\mathfrak{p}}\right)$
such that $r\left(\pi_{\psi_{\mathfrak{p}}}\right)\geq r$. Then
\[
\hat{m}\left(r;q\right)\leq\sum_{\varsigma\in\mathcal{M}\left(G;r\right)}\sum_{\psi\in\Psi_{2}^{\mathrm{AJ}}\left(G,\varsigma;q\right)}\sum_{\pi\in\Pi_{\psi}(\epsilon_{\psi})}\dim\pi^{K'\left(q\right)}=\sum_{\varsigma\in\mathcal{M}\left(G,r\right)}h\left(G,\varsigma;q\right).
\]
By definition $\varsigma\in\mathcal{M}\left(G;r\right)$ implies $r\left(\varsigma\right)\geq r$.
By Conjecture~\ref{conj:DSXDH-shape}, we get 
\[
\sum_{r\left(\pi_{\mathfrak{p}}\right)\geq r}m\left(\pi_{\mathfrak{p}};q\right)\dim\pi_{\mathfrak{p}}^{K_{\mathfrak{p}}}\leq\sum_{\varsigma\in\mathcal{M}\left(G,r\right)}h\left(G,\varsigma;q\right)\ll|X_{\mathfrak{p}}\left(q\right)|^{\frac{2}{r}},
\]
which proves the spherical SXDH.
\end{proof}
As a consequence of Theorem~\ref{thm:CSXDH-SO5-shape}, we get the
following:
\begin{cor}
Let $G$ be a definite Gross inner form of $SO_{5}$. Then Conjecture~\ref{conj:DSXDH-shape}
holds for $G$, which implies Theorems~\ref{thm:CSXDH-SO5-p} and
\ref{thm:den-Ram-SO5}.
\end{cor}

\begin{proof}
Conjecture~\ref{conj:DSXDH-shape} is a special case of Theorem~\ref{thm:CSXDH-SO5-shape}
for $G$ is a definite Gross inner form of $SO_{5}$ and $E=\mathbb{C}$.
By Proposition~\ref{prop:DSXDH-shape->SXDH-p+den-Ram} and Corollary~\ref{cor:sph-SXDH},
we get that Conjectures~\ref{conj:CSXDH-p} and~\ref{conj:den-Ram-comp}
hold for such $G$, which are exactly Theorems~\ref{thm:CSXDH-SO5-p}
and \ref{thm:den-Ram-SO5}, respectively.
\end{proof}

\subsection{Cutoff phenomena}

In \cite{LLP20}, the authors prove that a family of connected Ramanujan
complexes exhibits the cutoff phenomenon (in total variation) for
the mixing time of the Non Backtracking Random Walk (NBRW) (see below
for the precise definitions). In \cite{EFMP23}, the authors studied
the case of $3\times3$ unitary matrix groups and constructed families
of $(p^{3}+1,p+1)$-biregular bipartite congruence graphs, for any
prime $p\equiv2\bmod3$, which are density-Ramanujan but non-Ramanujan,
and showed that they exhibit the cutoff phenomenon for their NBRW.
Our main result in this section generalizes the results of \cite{LLP20}
and \cite{EFMP23} by proving that a family of density-Ramanujan congruence
complexes exhibits the cutoff phenomenon for their NBRW on their maximal
faces (a.k.a. chambers). If the original congruence complexes $\left\{ X_{\mathfrak{p}}\left(q\right)\right\} _{q}$
are not connected then we replace them with $\left\{ X_{\mathfrak{p}}^{1}\left(q\right)\right\} _{q}$,
in the above notation and invoke Lemma~\ref{lem:conn-den-ran-comp}. 
\begin{thm}
\label{thm:den-Ram->cutoff} Let $\left\{ X_{\mathfrak{p}}\left(q\right)\right\} _{q}$
be a family of connected density-Ramanujan congruence complexes. Then
$\left\{ X_{\mathfrak{p}}\left(q\right)\right\} _{q}$ exhibits the
cutoff phenomenon for NBRW on maximal faces.
\end{thm}

Observe that Theorems~\ref{thm:den-Ram-SO5} and \ref{thm:den-Ram->cutoff}
imply Theorem~\ref{thm:cutoff-SO5} from the introduction. The purpose
of this section will be to define the NBRW on the maximal faces of
the connected congruence complexes, define the cutoff phenomenon and
prove Theorem~\ref{thm:den-Ram->cutoff}. 
\begin{rem}
We note that \cite{LLP20} considered NBRW on all faces except for
vertices. The vertices case is more complicated and was resolved for
forms of $PGL_{n}$ in \cite{CP22} (the notion of classical groups
in the title of said paper refers only to the groups $PGL_{n}$).
Here we restrict to NBRW only on maximal faces for concreteness and
brevity. It is expected that the implication of cutoff phenomena from
the density-Ramanujan property to hold for NBRW on faces in all dimensions,
which we leave as an open problem.
\end{rem}

For simplicity of notation in this section we will drop the subscript
$\mathfrak{p}$, namely $p=p_{\mathfrak{p}}$ the residue degree,
$G=G_{\mathfrak{p}}$ the $\mathfrak{p}$-adic group, $I=I_{\mathfrak{p}}$
the Iwahori subgroup, $\tilde{X}=\tilde{X}_{\mathfrak{p}}$ the Bruhat-Tits
building, $\Gamma\left(q\right)=\Gamma_{\mathfrak{p}}\left(q\right)$
the level $q$ congruence subgroup, $X\left(q\right)=X_{\mathfrak{p}}\left(q\right)$
the level $q$ congruence complex, and denote by $\tilde{Y}=\tilde{Y}_{\mathfrak{p}}$
and $Y\left(q\right)=Y_{\mathfrak{p}}\left(q\right)$ the sets of
maximal faces in $\tilde{X}$ and $X\left(q\right)$, respectively,
and note that
\[
\tilde{Y}\cong G/I,\qquad Y\left(q\right)\cong\Gamma\left(q\right)\backslash G/I,\qquad\forall q.
\]

Let $\widetilde{\mathcal{W}}=\mathcal{W}\ltimes\Lambda$ be the affine
Weyl group of $G$, $\mathcal{W}$ the finite Weyl group, $\Lambda\cong\mathbb{Z}^{r}$
the lattice of coweights and $\Lambda^{+}\cong\mathbb{N}_{0}^{r}\subset\Lambda$
the cone of positive coweights. Then the (affine) Bruhat decomposition
is $G=\bigsqcup_{w\in\widetilde{\mathcal{W}}}IwI$. Let $\ell\,:\,\widetilde{\mathcal{W}}\rightarrow\mathbb{N}_{0}$
be the word length map w.r.t. the Coxeter generators of $\widetilde{\mathcal{W}}$
and note that $\ell\left(w\right)=\log_{p}|IwI/I|$ for any $w\in\widetilde{\mathcal{W}}$.
Fix some $\omega\in\Lambda^{+}$ (for example $\omega=\omega_{1}$,
the standard fundamental coweight), denote $d=|I\omega I/I|=p^{\ell\left(\omega\right)}$
and define the following operator on the complex vector space $L^{2}(\tilde{Y})\cong L^{2}\left(G/I\right)$,
by 
\[
T=T_{\omega}\,:\,L^{2}(\tilde{Y})\rightarrow L^{2}(\tilde{Y}),\qquad\left(Tf\right)\left(gI\right)=\frac{1}{d}\sum_{\iota\omega I\in I\omega I}f\left(g\iota\omega I\right).
\]
Let us show that the operator $T$ is collision-free and geometric
in the terminology of \cite{LLP20}. 
\begin{lem}
The operator $T$ satisfies the following two properties:
\begin{lyxlist}{00.00.0000}
\item [{(Collision-free)}] For any $x\in Y\left(q\right)$ and $n\ne m\in\mathbb{N}$,
then $\mathrm{supp}\left(T^{n}\mathbf{1}_{x}\right)\bigcap\mathrm{supp}\left(T^{m}\mathbf{1}_{x}\right)=\emptyset$.
\item [{(Geometric)}] For any $f\in L^{2}(\tilde{Y})$ and $g\in G$, then
$g.T\left(f\right)=T\left(g.f\right)$.
\end{lyxlist}
\end{lem}

\begin{proof}
Observe that $\mathrm{supp}\left(T^{m}\mathbf{1}_{gI}\right)=g\left(I\omega I\right)^{m}$
for any $m\in\mathbb{N}$, since $\omega\in\Lambda^{+}$, then $\ell\left(\omega^{m}\right)=m\cdot\ell\left(\omega\right)$,
hence by the Bruhat relations we get $\left(I\omega I\right)^{m}=I\omega^{m}I$,
and therefore $T$ is collision-free. Since $G$ acts from the left
on $\tilde{Y}\cong G/I$, hence on $L^{2}(\tilde{Y})$, and $T$ acts
from the right, we get that the operator $T$ is geometric. 
\end{proof}
By the geometric property, $T$ descends to an operator on all finite
quotients of $\tilde{X}$, in particular on the finite congruence
complexes $\left\{ X\left(q\right)\right\} _{q}$. On the finite-dimensional
vector space $\mathbb{C}^{Y\left(q\right)}\cong L^{2}\left(Y\left(q\right)\right)\cong L^{2}\left(\Gamma\left(q\right)\backslash G/I\right)$,
this operator is defined as
\[
T\,:\,L^{2}\left(Y\left(q\right)\right)\rightarrow L^{2}\left(Y\left(q\right)\right),\qquad\left(Tf\right)\left(\Gamma\left(q\right)gI\right)=\frac{1}{d}\sum_{\iota\omega I\in I\omega I}f\left(\Gamma\left(q\right)g\iota\omega I\right).
\]

One may use the operator $T$ to endow $Y\left(q\right)$ with the
structure of a directed graph (digraph), where the set of out-neighbors
of $y=\Gamma\left(q\right)gI\in Y\left(q\right)$ is $N\left(y\right)=\left\{ \Gamma\left(q\right)g\iota\omega I\,:\,\iota\omega I\in I\omega I\right\} \subset Y\left(q\right)$.
Assume that $Y\left(q\right)$ is a simple digraph, namely that $q$
is large enough such that $\Gamma\left(q\right)\bigcap I=\Gamma\left(q\right)\bigcap I\omega I=\emptyset$.
In which case note that $d=|N\left(y\right)|$ for any $y\in Y\left(q\right)$,
and that $T$ is the normalized adjacency operator of the digraph.
In particular, one may consider $\left\{ T\curvearrowright L^{2}\left(Y\left(q\right)\right)\right\} _{q}$
as a sequence of Markov chains. We will consider $T$ as a Non Backtracking
Random Walk (NBRW) on the chambers of the congruence complexes $\left\{ X\left(q\right)\right\} _{q}$
(or, equivalently, on the vertices of the digraphs $\left\{ Y\left(q\right)\right\} _{q}$).
Namely, for any $y\in Y\left(q\right)$, the NBRW moves to one of
its $d$ neighbors in $N\left(y\right)$ with probability $\frac{1}{d}$.

Recall that a distribution on $Y\left(q\right)$ is a real vector
$v\in\mathbb{R}^{Y\left(q\right)}$ with non-negative coefficients
which sum up to $1$. Let $u=\frac{1}{|Y\left(q\right)|}\mathbf{1}_{Y\left(q\right)}$
be the uniform distribution, and note that $Tu=u$. For $x\in Y\left(q\right)$,
consider the distribution $\mathbf{1}_{x}$, the indicator function
of $x$, and define $P_{x}^{\ell}:=T^{\ell}\mathbf{1}_{x}$ to be
the distribution of the NBRW of length $\ell$ starting from $x$,
for any $\ell\in\mathbb{N}$. Note that for any $y\in Y\left(q\right)$,
then $P_{x}^{\ell}\left(y\right)$ is equal to $d^{-\ell}$ times
the number of paths from $x$ to $y$ of length $\ell$ in the digraph.
For any two distribution $\mu$ and $\nu$, their $L^{1}$-distance,
$\|\mu-\nu\|_{1}:=\sum_{x\in Y\left(q\right)}\left|\mu\left(x\right)-\nu\left(x\right)\right|$,
is known to equal twice their total variation distance, $\|\mu-\nu\|_{\mathrm{TV}}:=\max_{A\subset Y\left(q\right)}\left|\mu\left(A\right)-\nu\left(A\right)\right|$.
Denote the total variation distance of the NBRW of time $\ell$ from
the uniform distribution with starting point $x\in Y\left(q\right)$
by 
\[
d_{x}\left(\ell;q\right):=\|P_{x}^{\ell}-u\|_{\mathrm{TV}}=\frac{1}{2}\|P_{x}^{\ell}-u\|_{1}=\frac{1}{2}\sum_{y\in Y\left(q\right)}\left|P_{x}^{\ell}\left(y\right)-\frac{1}{|Y\left(q\right)|}\right|\in[0,1].
\]
For $0<\epsilon<1$, say that the NBRW at time $\ell$ with worst-case
starting point is $\epsilon$-close to the uniform distribution $u$
if $\max_{x}d_{x}\left(\ell;q\right)\leq\epsilon$. Define the associated
$L^{1}$-mixing time by 
\[
t_{\mathrm{mix}}\left(\epsilon;q\right):=\min\left\{ \ell\,:\,\max_{x\in Y\left(q\right)}d_{x}\left(\ell;q\right)\leq\epsilon\right\} .
\]

Observe the following trivial lower bound on the $L^{1}$-mixing time
of the NBRW.
\begin{lem}
\label{lem:cutoff-lower} For any $0<\epsilon<1$, then $t_{\mathrm{mix}}\left(\epsilon;q\right)\geq\log_{d}|Y\left(q\right)|-\log_{d}\left(\frac{1}{1-\epsilon}\right)$.
\end{lem}

\begin{proof}
Note that $\left|\mbox{supp}\left(P_{x}^{\ell}\right)\right|\leq d^{\ell}$,
the maximal possible number of vertices you can visit in a $d$-out-regular
graph by taking paths of length $\ell$, for any $x$ and $\ell$.
Hence 
\[
d_{x}\left(\ell;q\right)\geq\sum_{y\not\in\mbox{supp}\left(P_{x}^{\ell}\right)}\left|P_{x}^{\ell}\left(y\right)-\frac{1}{|Y\left(q\right)|}\right|\geq\frac{|Y\left(q\right)\setminus\mbox{supp}\left(P_{x}^{\ell}\right)|}{|Y\left(q\right)|}\geq1-\frac{d^{\ell}}{|Y\left(q\right)|}.
\]
Taking $\ell=t_{\mathrm{mix}}\left(\epsilon;q\right)$, we get 
\[
\epsilon\geq1-\frac{d^{\ell}}{|Y\left(q\right)|}\quad\Rightarrow\quad d^{\ell}\geq\left(1-\epsilon\right)|Y\left(q\right)|\quad\Rightarrow\quad\ell\geq\log_{d}|Y\left(q\right)|-\log_{d}\left(\frac{1}{1-\epsilon}\right).
\]
\end{proof}
The cutoff phenomena for families of Markov chains in general, and
for random walks on families of graphs in particular, received much
attention in recent years (see \cite{LP16} and the references therein).
Before giving the precise definition, let us explain in words the
cutoff phenomenon: By Lemma~\ref{lem:cutoff-lower}, we get that
the NBRW at time $\ell_{-}=\left(1-o\left(1\right)\right)\log_{d}|Y\left(q\right)|$
is ``very far'' from the uniform distribution. We say that the NBRW
exhibits the cutoff phenomenon if at time $\ell_{+}=\left(1+o\left(1\right)\right)\log_{d}|Y\left(q\right)|$
it is ``very close'' to the uniform distribution, namely, there
is a sharp ``cutoff'' in the distance of the NBRW from the uniform
distribution, around time $\log_{d}|Y\left(q\right)|$. 
\begin{defn}
\label{def:cutoff} Say that $\left\{ X\left(q\right)\right\} _{q}$
exhibits the cutoff phenomenon (in total variation) for the mixing
time of its NBRW (on maximal faces), if for any fixed $0<\epsilon<1$
and for $q$ large enough,
\[
\left(1-\epsilon\right)\log_{d}|Y\left(q\right)|\leq t_{\mathrm{mix}}\left(\epsilon;q\right)\leq\left(1+\epsilon\right)\log_{d}|Y\left(q\right)|.
\]
\end{defn}

\begin{note}
We remark that there are stronger versions of the cutoff phenomenon,
where one is interested in giving tighter bounds for the cutoff window
(which in our case is $\epsilon\log_{d}|Y\left(q\right)|$). In \cite{LP16,LLP20}
the cutoff window grows like $O\left(\log\log_{d}|Y\left(q\right)|\right)$,
while in \cite{NS23} the authors showed it could even be bounded
as a function of $q$. We leave it as an open problem whether the
Sarnak--Xue Density Hypothesis implies the bounded window cutoff
phenomenon. 
\end{note}

Now let us turn to showing that density-Ramanujan complexes display
the cutoff phenomenon. Let $\left\{ X\left(q\right)\right\} _{q}$
be a family of congruence complexes associated to a definite inner
form $\mathcal{G}$ of a split classical group defined over a totally
real number field (note that $G=\mathcal{G}\left(F_{\mathfrak{p}}\right)$).
By \eqref{eq:=000020V(G;q)}, $V\left(q\right):=V\left(G,q\right)=L^{2}\left(\Gamma\left(q\right)\backslash G\right)^{I}\cong\bigoplus_{\pi}m\left(\pi;q\right)\pi^{I}$,
and consider the following orthogonal decomposition according to the
rate of decay of matrix coefficients 
\[
V\left(q\right)=\bigoplus_{r\in[2,\infty]}V\left(r;q\right),\qquad V\left(r;q\right):=\bigoplus_{r\left(\pi\right)=r}m\left(\pi;q\right)\pi^{I}.
\]
These decompositions are invariant under the action of $\mathcal{H}=\mathcal{H}\left(G,I\right)$,
the Iwahori-Hecke algebra of bi-$I$-invariant compactly supported
functions under convolution, which acts on $\pi^{I}$ for any $I$-spherical
representation $\pi$ of $G$. For example, for any $a\in\widetilde{\mathcal{W}}$,
then $T_{a}:=\frac{1}{|IaI/I|}1_{IaI}\in\mathcal{H}$, and in particular
$T=T_{\omega}\in\mathcal{H}$. For any $\mathcal{H}$-invariant subspace
$U\leq V\left(q\right)$ and any $S\in\mathcal{H}$, let $S|_{U}$
be the restriction of $S$ to $U$, and denote by $\|S|_{U}\|$ and
$\lambda\left(S|_{U}\right)$ the operator norm and the largest eigenvalue
in absolute value, respectively. Note that $\|S|_{U}\|=\lambda\left(S|_{U}\right)$
if $S|_{U}$ is normal, and that $\lambda\left(S|_{U}\right)\leq\|S|_{U}\|$
for any $S|_{U}$. 
\begin{lem}
\label{lem:eigenvalue-rate} For any $\pi\in\Pi\left(G\right)$ with
$\pi^{I}\ne0$, then
\[
\lambda\left(T|_{\pi^{I}}\right)\leq d^{-1/r(\pi)}.
\]
\end{lem}

\begin{proof}
Let $\lambda$ be the largest eigenvalue in absolute value for the
action of $T$ on $\pi^{I}$, and let $f\in\pi^{I}$ be a normalized
eigenvector with eigenvalue $\lambda$, i.e. $Tf=\lambda\cdot f$
and $\lambda\left(T|_{\pi^{I}}\right)=|\lambda|$. Consider the corresponding
matrix coefficient $c\,:\,G\rightarrow\mathbb{C}$, i.e. $c(g)=\langle\pi(g)f,f\rangle$.
Note that by the definition of $T$ and since $f$ is $I$-invariant,
we get that for any $\ell$,
\[
c(\omega^{\ell})=\langle\pi(\omega^{\ell})f,f\rangle=\langle T^{\ell}f,f\rangle=\lambda^{\ell}.
\]
Normalizing the Haar measure on $G$ to give $I$ measure $1$, by
the definition of $r=r(\pi)$, for any $\epsilon>0$, we get
\[
\sum_{\ell}\left(d|\lambda|^{r+\epsilon}\right)^{\ell}=\sum_{\ell}|I\omega^{\ell}I/I||c(\omega^{\ell})|^{r+\epsilon}=\int_{I\langle\omega\rangle I}|c(x)|^{r+\epsilon}dx\leq\int_{G}|c(x)|^{r+\epsilon}dx<\infty,
\]
where $I\langle\omega\rangle I=\sqcup_{\ell}I\omega^{\ell}I$, which
implies $d|\lambda|^{r+\epsilon}<1$ for any $\epsilon>0$, hence
$d|\lambda|^{r}\leq1$, and the claim follows.
\end{proof}
The following Proposition follows from \cite[Proposition 4.1]{Par20},
which gives an upper bound on the norm of a large power of a fixed
Iwahori-Hecke operator in terms of its second largest eigenvalue.
\begin{prop}
\label{prop:norm-rate} Fix $r\in[2,\infty]$. Then there exists $c>0$,
such that for any $q$ and any $\ell$, 
\[
\|T^{\ell}|_{V\left(r;q\right)}\|\leq c\cdot\ell^{|\mathcal{W}|}\cdot d^{-\ell/r}.
\]
\end{prop}

\begin{proof}
Recall that $V\left(r;q\right):=\bigoplus_{r\left(\pi\right)=r}m\left(\pi;q\right)\pi^{I}$,
which is an orthogonal $\mathcal{H}$-invariant (and in particular
$T$-invariant) decomposition, so that
\[
\|T^{\ell}|_{V\left(r;q\right)}\|=\max\left\{ \|T^{\ell}|_{\pi^{I}}\|\,:\,\pi\in\Pi\left(G\right),\;m\left(\pi;q\right)\pi^{I}\ne0\right\} .
\]
Therefore it suffices to prove that for any $\pi\in\Pi\left(G\right)$
with $\pi^{I}\ne0$ and $r(\pi)=r$, and any $\ell$, 
\[
\|T^{\ell}|_{\pi^{I}}\|\leq c\cdot\ell^{|\mathcal{W}|}\cdot d^{-\ell/r}.
\]
By the Borel-Casselman classification of Iwahori-spherical representations
\cite{Bor76,Cas80}, we get that $\dim\pi^{I}\leq s:=|\mathcal{W}|$.
Note that $\|T|_{\pi^{I}}\|\leq\|T\|=1$, and by Lemma~\ref{lem:eigenvalue-rate}
we get $\lambda\left(T|_{\pi^{I}}\right)\leq\lambda:=d^{-1/r}$. Therefore,
we are reduced to the following claim: For any $s\in\mathbb{N}$ and
$\lambda<1$, there exists $c>0$, such that for any $A\in\mathrm{Mat}_{s\times s}\left(\mathbb{R}\right)$
with $\|A\|\leq1$ and $\lambda\left(A\right)\leq\lambda$, and any
$\ell\in\mathbb{N}$, then
\[
\|A^{\ell}\|\leq c\cdot\ell^{s}\cdot\lambda^{\ell},
\]
which follows from the proof of \cite[Proposition 4.1]{Par20}. 
\end{proof}
An important feature of the congruence complex $X\left(q\right)$
is its large group of symmetries 
\[
G\left(q\right):=\Gamma\left(q\right)\backslash\Gamma\leq\mbox{Aut}\left(X\left(q\right)\right).
\]
The following Proposition collects some useful properties of the action
of this group.
\begin{prop}
\label{prop:symmetries} (1) The action of $G\left(q\right)$ preserves
the $L^{2}$-norm.

(2) The action of $G\left(q\right)$ preserves the subspaces $V\left(r;q\right)$.

(3) There exists $c_{1}>0$, independent of $q$, such that $|Y\left(q\right)|\leq c_{1}\cdot|G\left(q\right)|$.

(4) There exists $c_{2}>0$, independent of $q$, such that $\max_{x}|\mathrm{Stab}_{G\left(q\right)}\left(x\right)|\leq c_{2}$.
\end{prop}

\begin{proof}
(1) For any $f\in V\left(q\right)$ and $g\in G\left(q\right)$, then
$\|g.f\|_{2}^{2}=\sum_{x}|f\left(g.x\right)|^{2}=\sum_{y}|f\left(y\right)|^{2}=\|f\|_{2}^{2}$. 

(2) Recall that $V\left(q\right)=L^{2}\left(\Gamma\left(q\right)\backslash G\right)^{I}=L^{2}\left(\mathcal{G}\left(F\right)\backslash\mathcal{G}\left(\mathbb{A}\right)\right)^{K^{\mathfrak{p}}\left(q\right)\cdot I}$,
where $K^{\mathfrak{p}}\left(q\right)=\mathcal{G}_{\infty}\cdot\prod_{v\nmid\mathfrak{p\cdot}q\cdot\infty}K_{v}\cdot\prod_{v\mid q}K_{v}(\mathfrak{p}_{v}^{\mathrm{ord}_{v}(q)}$).
First, observe that since $G\left(q\right)=\Gamma\left(q\right)\backslash\Gamma\leq\prod_{v\mid q}K_{v}(\mathfrak{p}_{v}^{\mathrm{ord}_{v}(q)})\backslash\prod_{v\mid q}K_{v}$
is the quotient of a subgroup of the prime-to-$\mathfrak{p}$ adelic
points of $\mathcal{G},$ it commutes with the group $G=\mathcal{G}\left(F_{\mathfrak{p}}\right)$
inside of $\mathcal{G}(\mathbb{A}).$  As such, in the decomposition
$V\left(q\right):=\bigoplus_{\pi\in\Pi\left(G\right)}m\left(\pi;q\right)\pi^{I}$,
the subspace $m\left(\pi;q\right)\pi^{I}$ is preserved under the
action of $G\left(q\right)$, for any $\pi\in\Pi\left(G\right)$,
and in particular $G\left(q\right)$ preserves $V\left(r;q\right):=\bigoplus_{r\left(\pi\right)=r}m\left(\pi;q\right)\pi^{I}$,
for any $r$. 

(3) Since $X\left(1\right)=\Gamma\backslash\tilde{X}$ and $X\left(q\right)=\Gamma\left(q\right)\backslash\tilde{X}$,
we get that $G\left(q\right)\backslash X\left(q\right)\cong X\left(1\right)$,
and in particular $G\left(q\right)\backslash Y\left(q\right)\cong Y\left(1\right)$.
Let $\Omega\subset\tilde{Y}$ be a fundamental domain for the action
of $\Gamma$ on $\tilde{Y}$, i.e. $\tilde{Y}=\Gamma\cdot\Omega$,
and observe that $|\Omega|=|Y\left(1\right)|=:c_{1}$, is independent
of $q$. Denote $\Omega_{q}=\left\{ \Gamma\left(q\right)xI\,:\,x\in\Omega\right\} \subset Y\left(q\right)$,
and note that $\Omega_{q}$ is a fundamental domain for the action
of $G\left(q\right)$ on $Y\left(q\right)$. Therefore $|Y\left(q\right)|\leq|\Omega_{q}|\cdot|G\left(q\right)|$,
and $|\Omega_{q}|\leq|\Omega|=c_{1}$ as claimed. 

(4) First observe that for any $\tilde{x}=gI\in\tilde{Y}$, its stabilizer
in $G$ is the open compact subgroup $\mathrm{Stab}_{G}\left(\tilde{x}\right)=gIg^{-1}$,
its stabilizer in the discrete subgroup $\Gamma$ is the finite group
$\mathrm{Stab}_{\Gamma}\left(\tilde{x}\right)=\Gamma\bigcap gIg^{-1}$,
and for $x=\Gamma\left(q\right)gI\in Y\left(q\right)$, its stabilizer
in $G\left(q\right)=\Gamma\left(q\right)\backslash\Gamma$ is $\mathrm{Stab}_{G\left(q\right)}\left(x\right)=\Gamma\left(q\right)\backslash\left(\Gamma\bigcap gIg^{-1}\right)$.
Clearly $\mathrm{Stab}_{G\left(q\right)}\left(x\right)$ is a quotient
of the finite group $\mathrm{Stab}_{\Gamma}\left(\tilde{x}\right)$,
whose size is independent of $q$. Second note that in $\max_{x}|\mathrm{Stab}_{G\left(q\right)}\left(x\right)|$,
it suffices to run over the $x$ in $\Omega\subset\tilde{Y}$, a fundamental
domain for the action of $\Gamma$ on $\tilde{Y}$, whose size is
finite and independent of $q$. The claim now follows from these two
observations.
\end{proof}
For any subspace $U\leq V\left(q\right)$, denote by $\mathrm{Proj}_{U}\,:\,V\left(q\right)\rightarrow U$,
the orthogonal projection operator onto $U$. 
\begin{lem}
\label{lem:norm-proj} There exists $C>0$, such that for any $x\in Y\left(q\right)$
and any $G\left(q\right)$-invariant subspace $U\leq V\left(q\right)$,
\[
\|\mathrm{Proj}_{U}\mathbf{1}_{x}\|_{2}^{2}\leq C\cdot\frac{\dim U}{|Y\left(q\right)|}.
\]
\end{lem}

\begin{proof}
By Proposition~\ref{prop:symmetries}(1) we get that for any $g\in G\left(q\right)$,
\[
\|\mathrm{Proj}_{U}\mathbf{1}_{x}\|_{2}=\|g.\mathrm{Proj}_{U}\mathbf{1}_{x}\|_{2}=\|\mathrm{Proj}_{gU}\mathbf{1}_{gx}\|_{2}=\|\mathrm{Proj}_{U}\mathbf{1}_{gx}\|_{2},
\]
combined with Proposition~\ref{prop:symmetries}(3,4), denote $C=c_{1}c_{2}$,
we get 
\[
\|\mathrm{Proj}_{U}\mathbf{1}_{x}\|_{2}^{2}=\frac{1}{|G\left(q\right)|}\sum_{g\in G\left(q\right)}\|\mathrm{Proj}_{U}\mathbf{1}_{gx}\|_{2}^{2}\leq\frac{C}{|Y\left(q\right)|}\sum_{y\in Y\left(q\right)}\|\mathrm{Proj}_{U}\mathbf{1}_{y}\|_{2}^{2}.
\]
Denote by $M$ the representing matrix of the projection operator
$\mathrm{Proj}_{U}$ w.r.t. the standard basis $\left\{ \mathbf{1}_{y}\right\} _{y\in Y\left(q\right)}$.
Note that $M=M^{2}$, $M=M^{t}$ and $M=M^{t}M$, which implies that
$M_{yy}=\mathbf{1}_{y}^{t}M^{t}M\mathbf{1}_{y}=\|M\mathbf{1}_{y}\|^{2}$,
for any $y\in Y\left(q\right)$. Therefore 
\[
\sum_{y\in Y\left(q\right)}\|\mathrm{Proj}_{U}\mathbf{1}_{x}\|_{2}^{2}=\sum_{y\in Y\left(q\right)}M_{yy}=\mathrm{Trace}\left(M\right)=\dim U.
\]
\end{proof}
\begin{prop}
\label{prop:norm-proj} Fix $r\in[2,\infty)$, let $|q|\rightarrow\infty$
and let $x\in Y\left(q\right)$. If $\left\{ X\left(q\right)\right\} _{q}$
is a family of density-Ramanujan complexes, then
\[
\sum_{t\geq r}\|\mathrm{Proj}_{V\left(t;q\right)}\left(\mathbf{1}_{x}-u\right)\|_{2}^{2}\ll|Y\left(q\right)|^{-1+\frac{2}{r}}.
\]
\end{prop}

\begin{proof}
Since $u\in V\left(\infty;q\right)$, $r\ne\infty$ and $V\left(q\right)=\bigoplus_{r\geq2}V\left(r;q\right)$
is an orthogonal decomposition, we get that $\mathrm{Proj}_{V\left(r;q\right)}\left(\mathbf{1}_{x}-u\right)=\mathrm{Proj}_{V\left(r;q\right)}\mathbf{1}_{x}$.
By Lemma~\ref{lem:norm-proj} and Proposition~\ref{prop:symmetries}(2),
\[
\|\mathrm{Proj}_{V\left(r;q\right)}\left(\mathbf{1}_{x}-u\right)\|_{2}^{2}\leq C\cdot|Y\left(q\right)|^{-1}\cdot\dim V\left(r;q\right).
\]
Since $\left\{ X\left(q\right)\right\} _{q}$ is density-Ramanujan,
i.e. $\sum_{t\geq r}\dim V\left(t;q\right)\ll|Y\left(q\right)|^{\frac{2}{r}}$,
we get
\[
\sum_{t\geq r}\|\mathrm{Proj}_{V\left(t;q\right)}\left(\mathbf{1}_{x}-u\right)\|_{2}^{2}\leq C\cdot|Y\left(q\right)|^{-1}\cdot\sum_{t\geq r}\sum_{x\in Y\left(q\right)}\|\mathrm{Proj}_{V\left(t;q\right)}\mathbf{1}_{x}\|_{2}^{2}
\]
\[
\leq C\cdot|Y\left(q\right)|^{-1}\cdot\sum_{t\geq r}\dim V\left(t;q\right)\ll|Y\left(q\right)|^{-1+\frac{2}{r}}.
\]
 
\end{proof}
We are now in a position to prove Theorem~\ref{thm:den-Ram->cutoff}.
\begin{proof}[\textbf{Proof of Theorem~\ref{thm:den-Ram->cutoff}}]
 By Lemma~\ref{lem:cutoff-lower}, we get the lower bound. Denote
$n=|Y\left(q\right)|$ and $\ell=\left(1+\epsilon\right)\log_{d}n$.
Then for any $x\in Y\left(q\right)$, we have
\[
d_{x}\left(\ell;q\right)^{2}=\|T^{\ell}\mathbf{1}_{x}-u\|_{1}^{2}=\|T^{\ell}\left(\mathbf{1}_{x}-u\right)\|_{1}^{2}\leq n\cdot\|T^{\ell}\left(\mathbf{1}_{x}-u\right)\|_{2}^{2}.
\]
By the orthogonal $T$-invariant decomposition $V\left(q\right)=\oplus_{r}V\left(r;q\right)$,
where $r$ runs over the finite set of $r\in[2,\infty)$ such that
$V\left(r;q\right)\ne0$, we get 
\[
d_{x}\left(\ell;q\right)^{2}\leq n\cdot\sum_{r}\|T^{\ell}|_{V\left(r;q\right)}\|^{2}\|\mathrm{Proj}_{V\left(r;q\right)}\left(\mathbf{1}_{x}-u\right)\|_{2}^{2}.
\]
By Proposition~\ref{prop:spec-exp}, $V\left(r;q\right)=0$ for any
$R:=2\cdot\mathrm{rank}\left(G_{\mathfrak{p}}\right)<r<\infty$. Let
$r_{i}=2\left(1+\epsilon/2\right)^{i}$ and let $K=\log_{\left(1+\epsilon/2\right)}\left(R/2\right)$.
We begin by lumping together the contributions of all $r\in[r_{i},r_{i+1})$
\[
d_{x}\left(\ell;q\right)^{2}\leq n\sum_{i=0}^{K}\sum_{r\in[r_{i},r_{i+1})}\|T^{\ell}|_{V\left(r;q\right)}\|^{2}\|\mathrm{Proj}_{V\left(r;q\right)}\left(\mathbf{1}_{x}-u\right)\|_{2}^{2}.
\]
Using Proposition~\ref{prop:norm-rate} and $r<r_{i+1}$, we can
bound the first term to get 
\[
d_{x}\left(\ell;q\right)^{2}\leq Cn\sum_{i=0}^{K}\sum_{r\in[r_{i},r_{i+1})}d^{-2\ell/r_{i+1}}\|\mathrm{Proj}_{V\left(r;q\right)}\left(\mathbf{1}_{x}-u\right)\|_{2}^{2},
\]
where $C=c\cdot\ell^{|\mathcal{W}|}$. We add non-negative terms to
the inner sum and apply Proposition~\ref{prop:norm-proj} to get
\[
d_{x}\left(\ell;q\right)^{2}\leq Cn\sum_{i=0}^{K}\sum_{t\geq r_{i}}d^{-2\ell/r_{i+1}}\|\mathrm{Proj}_{V\left(r;q\right)}\left(\mathbf{1}_{x}-u\right)\|_{2}^{2}\ll n\cdot\sum_{i=0}^{K}d^{-2\ell/r_{i+1}}n^{-1+2/r_{i}}.
\]
Finally, using $r_{i+1}=r_{i}\cdot\left(1+\epsilon/2\right)$, we
simplify each term in the sum. The final bound is 
\[
d_{x}\left(\ell;q\right)^{2}\ll n\cdot\sum_{i=0}^{K}n^{\frac{-\epsilon}{r_{i}\left(1+\epsilon/2\right)}}
\]
and each term in the finite sum goes to $0$ as $n\rightarrow\infty$.
\end{proof}

\subsection{Applications}

We end this section with the following two (closely related) applications
of the cutoff phenomenon: (i) The optimal almost diameter property
for congruence complexes associated to $G$ (Corollary~\ref{cor:OAD-comp-SO5}),
and (ii) the optimal strong approximation property for $\mathfrak{p}$-arithmetic
subgroups of $G$ (Corollary~\ref{cor:OSA-comp-SO5}). 

These applications were previously investigated in \cite{GK23} (their
optimal lifting property is what we call optimal strong approximation),
as consequences of the density hypothesis. We decided to add the following
subsection in order to clarify the different steps in the implications:
The density hypothesis implies the cutoff phenomenon which implies
the optimal almost diameter which implies the optimal strong approximation. 

Let us begin with the optimal almost diameter (OAD) property. Consider
the family of congruence complexes $\left\{ X_{\mathfrak{p}}\left(q\right)\right\} _{q}$,
and the branching operator $T=T_{\omega}=\frac{1}{|I_{\mathfrak{p}}\omega I_{\mathfrak{p}}/I_{\mathfrak{p}}|}\mathbf{1}_{I_{\mathfrak{p}}\omega I_{\mathfrak{p}}}$,
for a fixed $\omega\in\Lambda^{+}$. Consider the distance function
induced from $T$,
\[
\mbox{dist}\,:\,Y_{\mathfrak{p}}\left(q\right)\times Y_{\mathfrak{p}}\left(q\right)\rightarrow\mathbb{N},\qquad\mbox{dist}\left(x,y\right)=\min\left\{ \ell\,:\,y\in\mathrm{supp}(T^{\ell}\mathbf{1}_{x})\right\} .
\]
The diameter of the complex $X_{\mathfrak{p}}\left(q\right)$ (w.r.t.
the operator $T$), denoted $D\left(q\right)$, is defined to be the
maximal distance between any pair of chambers in the complex
\[
D\left(q\right):=\max_{x,y}\mbox{dist}\left(x,y\right).
\]
For $\epsilon>0$, the $\epsilon$-almost diameter of $X_{\mathfrak{p}}\left(q\right)$,
denoted $AD\left(\epsilon;q\right)$, is defined to be the maximal
distance between any pair of chambers in the complex, excluding a
negligible set of pairs $Z\subset Y_{\mathfrak{p}}\left(q\right)^{2}$,
of size at most $\epsilon n^{2}$, where $n=|Y_{\mathfrak{p}}\left(q\right)|$,
namely
\[
AD\left(\epsilon;q\right):=\min_{|Z|\leq\epsilon n^{2}}\max_{\left(x,y\right)\not\in Z}\mbox{dist}\left(x,y\right).
\]

The following lemma relates the diameter and the almost diameter for
congruence complexes. We note that the lemma makes essential use of
the fact that the congruence complex $X_{\mathfrak{p}}\left(q\right)$
admits a group action $G\left(q\right)=\Gamma_{\mathfrak{p}}\left(q\right)\backslash\Gamma_{\mathfrak{p}}\leq\mathrm{Aut}\left(X_{\mathfrak{p}}\left(q\right)\right)$,
such that $Y_{\mathfrak{p}}\left(q\right)/G\left(q\right)\cong Y_{\mathfrak{p}}\left(1\right)$,
and hence the number of its orbits is bounded by $|Y_{\mathfrak{p}}\left(1\right)|$,
which is independent of $q$.
\begin{lem}
\label{lem:AD-D} For any $0<\epsilon<\frac{1}{2c}$, where $c=|Y_{\mathfrak{p}}\left(1\right)|$,
then the diameter of $X_{\mathfrak{p}}\left(q\right)$ is bounded
by twice its $\epsilon$-almost diameter,
\[
D\left(q\right)\leq2\cdot AD\left(\epsilon;q\right).
\]
\end{lem}

\begin{proof}
Let $y_{1},y_{2}\in Y_{\mathfrak{p}}\left(q\right)$, and let $Z_{i}$
be the set of all $z\in Y_{\mathfrak{p}}\left(q\right)$ such that
$\mbox{dist}\left(y_{i},z\right)>AD\left(\epsilon;q\right)$, for
$i=1,2$. Assume in contradiction that $|Z_{i}|>\epsilon cn$, for
either $i=1,2$, and let $Z$ be the set of pairs $\left(g.y_{i},g.z\right)\in Y_{\mathfrak{p}}\left(q\right)^{2}$,
where $z\in Z_{i}$ and $g\in G\left(q\right)$. Recall that $Y_{\mathfrak{p}}\left(q\right)/G\left(q\right)\cong Y_{\mathfrak{p}}\left(1\right)$,
hence $|G\left(q\right)|\geq n/c$, and we get that $|Z|\geq|G\left(q\right)|\cdot|Z_{i}|>\epsilon n^{2}$.
Since $\mathrm{dist}\left(g.y_{i},g.z\right)=\mathrm{dist}\left(y_{i},z\right)$
for any $g$, and by the definition of $Z_{i}$, we get that all pairs
in $Z$ are of distance bigger than $AD\left(\epsilon;q\right)$,
contradicting the definition of $AD\left(\epsilon;q\right)$. Therefore
$|Z_{i}|\leq\epsilon cn<\frac{n}{2}$, for both $i=1,2$, and hence
$|Z_{1}^{c}\cap Z_{2}^{c}|\ne0$, where $Z_{i}^{c}=Y_{\mathfrak{p}}\left(q\right)\setminus Z_{i}$.
Pick some $z\in Z_{1}^{c}\bigcap Z_{2}^{c}$, then the claim follows
from the triangle inequality
\[
\mbox{dist}\left(y_{1},y_{2}\right)\leq\mbox{dist}\left(y_{1},z\right)+\mbox{dist}\left(y_{2},z\right)\leq2AD\left(\epsilon;q\right).
\]
\end{proof}
Let us now define the OAD property, which essentially says that ``most''
pairs of chambers in the complex are of distance $\left(1+o\left(1\right)\right)\log_{d}n$.
Note that by Lemma~\ref{lem:AD-D}, the OAD property implies that
all pairs of chambers in the complex are of distance at most $\left(2+o\left(1\right)\right)\log_{d}n$. 
\begin{defn}
The family of complexes $\left\{ X_{\mathfrak{p}}\left(q\right)\right\} _{q}$
exhibits the Optimal Almost Diameter (OAD) property, if for any fixed
$1>\epsilon>0$, and for large $q$, 
\[
\left(1-\epsilon\right)\log_{d}n\leq AD\left(\epsilon;q\right)\leq\left(1+\epsilon\right)\log_{d}n.
\]
\end{defn}

The OAD property follows from the cutoff phenomenon. 
\begin{prop}
\label{prop:cutoff->OAD} If $\left\{ X_{\mathfrak{p}}\left(q\right)\right\} _{q}$
exhibits the cutoff phenomenon, then it also exhibits the OAD property. 
\end{prop}

\begin{proof}
Let $1>\epsilon>0$, $q\in\mathbb{N}$, and denote $n=|Y_{\mathfrak{p}}\left(q\right)|$.
For any $x\in Y_{\mathfrak{p}}\left(q\right)$, the number of $y\in Y_{\mathfrak{p}}\left(q\right)$
such that $\mathrm{dist}\left(x,y\right)=\ell$ is at most $d^{\ell}$.
Therefore $d^{AD\left(\epsilon;q\right)}\geq\left(1-\epsilon\right)n$,
hence 
\[
AD\left(\epsilon;q\right)\geq\log_{d}n-\log_{d}\left(\frac{1}{1-\epsilon}\right)\geq\left(1-\epsilon\right)\log_{d}n.
\]
Denote $t=t_{\mathrm{mix}}\left(\epsilon;q\right)$ and let $Z$ be
the set of pairs $\left(x,y\right)$ such that $y\not\in\mathrm{supp}\left(T^{t}\mathbf{1}_{x}\right)$.
Then
\[
\frac{1}{n^{2}}|Z|\leq\frac{1}{n}\sum_{x,y}\,\left|T^{t}\mathbf{1}_{x}\left(y\right)-\frac{1}{n}\right|\leq\max_{x}\sum_{y}\,\left|T^{t}\mathbf{1}_{x}\left(y\right)-\frac{1}{n}\right|<\epsilon.
\]
Therefore by the cutoff phenomenon we get 
\[
AD\left(\epsilon;q\right)\leq t_{\mathrm{mix}}\left(\epsilon;q\right)\leq\left(1+\epsilon\right)\log_{d}n.
\]
\end{proof}
\begin{cor}
\label{cor:OAD-comp-SO5} Let $G$ be a definite Gross inner form
of $SO_{5}$ and $\mathfrak{p}$ a finite place. Then the family of
congruence complexes $\left\{ X_{\mathfrak{p}}\left(q\right)\right\} _{q}$
exhibits the OAD property. 
\end{cor}

\begin{proof}
Follows from Theorem~\ref{thm:cutoff-SO5} and Proposition~\ref{prop:cutoff->OAD}.
\end{proof}
Next we consider the optimal strong approximation property, which
was first discovered by Sarnak \cite{Sar15} for the arithmetic group
$SL_{2}\left(\mathbb{Z}\right)$, in the setting of $\mathfrak{p}$-arithmetic
groups. See also \cite{GK23}, where this property is called the optimal
lifting property and where it is studied in more general settings.

Let $G$ be a definite Gross inner form of a split classical group
$G^{*}$ defined over a totally real field $F$. Let $\mathfrak{p}$
be a fixed prime of $F$, $G_{\mathfrak{p}}=G\left(F_{\mathfrak{p}}\right)$
the $\mathfrak{p}$-adic group and $\Gamma_{\mathfrak{p}}=G\left(\mathcal{O}[1/\mathfrak{p}]\right)\leq G_{\mathfrak{p}}$
its principal $\mathfrak{p}$-arithmetic group. For any $\mathfrak{p}\nmid q\lhd\mathcal{O}$,
consider the modulo $q$ homomorphism, 
\[
\mod q\,:\,\Gamma_{\mathfrak{p}}\rightarrow G\left(\mathcal{O}/q\right),
\]
whose kernel is the congruence subgroup $\Gamma_{\mathfrak{p}}\left(q\right)$,
and denote its image by $\mathbf{G}\left(q\right)\leq G\left(\mathcal{O}/q\right)$.
By the strong approximation property the group $\mathbf{G}\left(q\right)$
is well understood. For example, if $G^{*}=Sp_{2n}$ then $\mathbf{G}\left(q\right)=G\left(\mathcal{O}/q\right)$,
and if $G^{*}=SO_{2n+1}$ then the index of $\mathbf{G}\left(q\right)$
in $G\left(\mathcal{O}/q\right)$ is a power of $2$ bounded by $2^{\omega\left(q\right)}$. 

Consider the following metric on $\mathbf{G}\left(q\right)$ induced
by the branching operator $T$, where $x_{0}$ is a fixed fundamental
chamber in the Bruhat-Tits building of $G_{\mathfrak{p}}$,
\[
\ell\,:\,\mathbf{G}\left(q\right)\rightarrow\mathbb{N}_{0},\qquad\ell\left(g\right)=\min\left\{ \mbox{dist}\left(x_{0},\gamma.x_{0}\right)\,:\,\gamma\in\Gamma_{\mathfrak{p}},\;\gamma\mod q=g\right\} ,
\]
and for any $r\in\mathbb{N}$, denote the ball of radius $r$ according
to the metric $\ell$ by 
\[
B\left(r\right)=\left\{ g\in\mathbf{G}\left(q\right)\,:\,\ell\left(g\right)\leq r\right\} =\left\{ \gamma\mod q\,:\,\gamma\in\Gamma_{\mathfrak{p}},\;\mbox{dist}\left(x_{0},\gamma.x_{0}\right)\leq r\right\} .
\]
 Note that $\left|B\left(r\right)\right|\leq d^{r}$, hence, if $|B\left(r\right)|=\left(1-\epsilon\right)|\mathbf{G}\left(q\right)|$
then $r\geq\log_{d}|\mathbf{G}\left(q\right)|-\log_{d}\left(\frac{1}{1-\epsilon}\right)$. 

In recent years a strengthening of the strong approximation property
for the arithmetic group $\Gamma_{\mathfrak{p}}$, called super strong
approximation (which applies also for non arithmetic groups, see \cite{Sar14}),
states that the Cayley graphs $\left\{ \mbox{Cay}\left(\mathbf{G}\left(q\right),S_{\mathfrak{p}}\bmod q\right)\right\} _{q}$
form a family of expanders, for any fixed generating set $S_{\mathfrak{p}}\subset\Gamma_{\mathfrak{p}}$.
Since expander graphs have logarithmic diameter, we get that $B\left(C\cdot\log_{d}|\mathbf{G}\left(q\right)|\right)=\mathbf{G}\left(q\right)$,
where $C$ is a constant depending only on the spectral gap of the
Cayley graphs (and in particular, independent of $q$).

Following \cite{Sar15}, we now define, in the setting of $p$-arithmetic
groups, the notions of optimal covering, almost covering exponents
and the property of optimal strong approximation.
\begin{defn}
Define the covering exponent $\kappa\left(\Gamma_{\mathfrak{p}}\right)$
of $\Gamma_{\mathfrak{p}}$ to be
\[
\kappa\left(\Gamma_{\mathfrak{p}}\right)=\liminf_{q\rightarrow\infty}\left\{ \kappa\geq1\,:\,B\left(\kappa\cdot\log_{d}|\mathbf{G}\left(q\right)|\right)=\mathbf{G}\left(q\right)\right\} .
\]
 Define the almost covering exponent $\kappa_{\mu}\left(\Gamma_{\mathfrak{p}}\right)$
of $\Gamma_{\mathfrak{p}}$ to be
\[
\kappa_{\mu}\left(\Gamma_{\mathfrak{p}}\right)=\liminf_{\epsilon\rightarrow0}\liminf_{q\rightarrow\infty}\left\{ \kappa\geq1\,:\,|B\left(\kappa\cdot\log_{d}|\mathbf{G}\left(q\right)|\right)|\geq\left(1-\epsilon\right)|\mathbf{G}\left(q\right)|\right\} .
\]
Say that $\Lambda$ has the Optimal Strong Approximation (OSA) property
if $\kappa_{\mu}\left(\Gamma_{\mathfrak{p}}\right)=1$.
\end{defn}

Note that by a simple union bound $1\leq\kappa_{\mu}\left(\Gamma_{\mathfrak{p}}\right)$,
clearly $\kappa_{\mu}\left(\Gamma_{\mathfrak{p}}\right)\leq\kappa\left(\Gamma_{\mathfrak{p}}\right)$,
and by the super strong approximation property $\kappa\left(\Gamma_{\mathfrak{p}}\right)<\infty$.
An analogous result to Lemma~\ref{lem:AD-D}, shows that $\kappa\left(\Gamma_{\mathfrak{p}}\right)\leq2\cdot\kappa_{\mu}\left(\Gamma_{\mathfrak{p}}\right)$,
for $q$ large enough. Hence if $\Gamma_{\mathfrak{p}}$ exhibits
the OSA property then $1\leq\kappa\left(\Gamma_{\mathfrak{p}}\right)\leq2$.
It is an interesting open problem to give better bounds on $\kappa\left(\Gamma_{\mathfrak{p}}\right)$.

In \cite{Sar15}, Sarnak studied the covering and almost covering
exponents of the arithmetic group $SL_{2}\left(\mathbb{Z}\right)$
and proved the following results. (See \cite{AB22,JK22} for a generalization
for $SL_{n}\left(\mathbb{Z}\right)$ for $q$ running over square
free integers.)
\begin{thm}
\cite{Sar15} (i) $\kappa_{\mu}\left(SL_{2}\left(\mathbb{Z}\right)\right)=1$,
i.e. $SL_{2}\left(\mathbb{Z}\right)$ exhibits the OSA property, and

(ii) $\kappa\left(SL_{2}\left(\mathbb{Z}\right)\right)\leq\frac{4}{3}$
with equality when $q$ runs over powers of $2$. 
\end{thm}

The optimal strong approximation property for the $\mathfrak{p}$-arithmetic
subgroup $\Gamma_{\mathfrak{p}}=G\left(\mathcal{O}[1/\mathfrak{p}]\right)$,
of a definite Gross inner form of $G$, is a group theoretic reformulation
of the optimal almost diameter property of the congruence complexes
$\left\{ X_{\mathfrak{p}}\left(q\right)\right\} _{q}$.
\begin{prop}
\label{prop:OAD->OSA} If the family of congruence complexes $\left\{ X_{\mathfrak{p}}\left(q\right)\right\} _{q}$
exhibits the OAD property, then the $\mathfrak{p}$-arithmetic subgroup
$\Gamma_{\mathfrak{p}}$ satisfies the OSA property.
\end{prop}

\begin{proof}
Let $n:=|Y_{\mathfrak{p}}\left(q\right)|$ and $c=|Y_{\mathfrak{p}}\left(1\right)|$,
and for $\epsilon>0$, denote $r:=\left(1+\epsilon/c\right)\log_{d}n$.
Denote $Z$ to be the set of pairs $\left(x,y\right)\in Y_{\mathfrak{p}}\left(q\right)^{2}$,
such that $\mbox{dist}\left(x,y\right)>r$, and by the OAD property
we get that $|Z|\leq\frac{\epsilon}{c}n^{2}$. The group $\mathbf{G}\left(q\right)$
acts isometrically on $Y_{\mathfrak{p}}\left(q\right)$, hence $Z$
is $\mathbf{G}\left(q\right)$-invariant. Let $B'\left(r\right)$
be the set of $y\in Y_{\mathfrak{p}}\left(q\right)$ such that $\mbox{dist}\left(x_{0},y\right)\leq r$.
Note that $Z$ contains all the pairs $\left(g.x_{0},g.y\right)$,
where $y\not\in B'\left(r\right)$ and $g\in\mathbf{G}\left(q\right)$.
Also note that $|B'\left(r\right)|\leq|B\left(r\right)|$ and that
$|\mathbf{G}\left(q\right)|\geq\frac{1}{c}|Y_{\mathfrak{p}}\left(q\right)|$.
Therefore
\[
\frac{n}{c}\left(n-|B\left(r\right)|\right)\leq|\mathbf{G}\left(q\right)|\cdot|Y_{\mathfrak{p}}\left(q\right)\setminus B'\left(r\right)|\leq|Z|\leq\frac{\epsilon}{c}n^{2},
\]
which implies $|B\left(r\right)|\geq\left(1-\epsilon\right)n\geq\left(1-\epsilon\right)|\mathbf{G}\left(q\right)|$
for $r=\left(1+\epsilon/c\right)\log_{d}n$, hence $\kappa_{\mu}\left(\Gamma_{\mathfrak{p}}\right)=\liminf_{\epsilon\rightarrow0}\left(1+\epsilon/c\right)=1$,
which proves the OSA property.
\end{proof}
The following proposition shows that the optimal strong approximation
property for definite Gross inner forms follows from the Sarnak--Xue
Density Hypothesis.
\begin{prop}
\label{prop:DSXDH->OSA} Let $G$ be a definite Gross inner form of
a split classical group. If Conjecture~\ref{conj:DSXDH-shape} holds
for $G$, then for any finite prime $\mathfrak{p}$, 
\[
\kappa_{\mu}\left(G\left(\mathcal{O}[1/\mathfrak{p}]\right)\right)=1,
\]
i.e. its $\mathfrak{p}$-arithmetic subgroup satisfies the OSA property.
\end{prop}

\begin{proof}
By Proposition~\ref{prop:DSXDH-shape->SXDH-p+den-Ram}, assuming
Conjecture~\ref{conj:DSXDH-shape}, the congruence complexes $\left\{ X_{\mathfrak{p}}\left(q\right)\right\} _{q}$
are density-Ramanujan, by Theorem~\ref{thm:den-Ram->cutoff}, the
complexes $\left\{ X_{\mathfrak{p}}\left(q\right)\right\} _{q}$ exhibit
the cutoff phenomenon, by Proposition~\ref{prop:cutoff->OAD}, the
cutoff phenomenon implies the optimal almost diameter property, and
by Proposition~\ref{prop:OAD->OSA}, the optimal almost diameter
property implies the optimal strong approximation property.
\end{proof}
In particular, we get the following new instances of the optimal strong
approximation for definite Gross inner forms of $SO_{5}$.
\begin{cor}
\label{cor:OSA-comp-SO5} Let $G$ be a definite Gross inner form
of $SO_{5}$. Then for any finite prime $\mathfrak{p}$, the principal
$\mathfrak{p}$-arithmetic subgroup of $G$ satisfies the OSA property.
\end{cor}

\begin{proof}
Follows from Theorem~\ref{thm:CSXDH-SO5-shape} and Proposition~\ref{prop:DSXDH->OSA}.
\end{proof}
\bibliographystyle{plain}
\bibliography{bibfile}

\end{document}